\PassOptionsToPackage{unicode}{hyperref}
\PassOptionsToPackage{hyphens}{url}
\PassOptionsToPackage{dvipsnames,svgnames,x11names}{xcolor}
\documentclass[
  11pt]{article}

\usepackage{amsmath,amssymb}
\usepackage{amsthm}
\usepackage{amsmath,amssymb,amsfonts,amsthm,subfloat,algorithm,url,float,epsf,psfrag,bm}
\usepackage{mathtools}
\usepackage{color,graphicx}
\usepackage[colorlinks,citecolor=blue,urlcolor=blue]{hyperref}
\usepackage[noabbrev,capitalize]{cleveref}
\usepackage{algorithm, algpseudocode, verbatim, float}
\usepackage{marvosym}
\usepackage{titlesec}
\usepackage{xr}
\usepackage[pagewise]{lineno}
\usepackage{subcaption}
\usepackage[utf8]{inputenc}
\usepackage{mathrsfs}
\usepackage{amsfonts}
\usepackage{enumerate}
\usepackage{latexsym}
\usepackage{multirow}

\allowdisplaybreaks

\usepackage{yhmath}
\usepackage{iftex}
\ifPDFTeX
  \usepackage[T1]{fontenc}
  \usepackage[utf8]{inputenc}
  \usepackage{textcomp} 
\else 
  \usepackage{unicode-math}
  \defaultfontfeatures{Scale=MatchLowercase}
  \defaultfontfeatures[\rmfamily]{Ligatures=TeX,Scale=1}
\fi
\usepackage{lmodern}
\ifPDFTeX\else  
\fi
\IfFileExists{upquote.sty}{\usepackage{upquote}}{}
\IfFileExists{microtype.sty}{
  \usepackage[]{microtype}
  \UseMicrotypeSet[protrusion]{basicmath} 
}{}
\makeatletter
\@ifundefined{KOMAClassName}{
  \IfFileExists{parskip.sty}{%
    \usepackage{parskip}
  }{
    \setlength{\parindent}{0pt}
    \setlength{\parskip}{6pt plus 2pt minus 1pt}}
}{
  \KOMAoptions{parskip=half}}
\makeatother
\usepackage{xcolor}
\makeatletter
\ifx\paragraph\undefined\else
  \let\oldparagraph\paragraph
  \renewcommand{\paragraph}{
    \@ifstar
      \xxxParagraphStar
      \xxxParagraphNoStar
  }
  \newcommand{\xxxParagraphStar}[1]{\oldparagraph*{#1}\mbox{}}
  \newcommand{\xxxParagraphNoStar}[1]{\oldparagraph{#1}\mbox{}}
\fi
\ifx\subparagraph\undefined\else
  \let\oldsubparagraph\subparagraph
  \renewcommand{\subparagraph}{
    \@ifstar
      \xxxSubParagraphStar
      \xxxSubParagraphNoStar
  }
  \newcommand{\xxxSubParagraphStar}[1]{\oldsubparagraph*{#1}\mbox{}}
  \newcommand{\xxxSubParagraphNoStar}[1]{\oldsubparagraph{#1}\mbox{}}
\fi
\makeatother

\usepackage{longtable,booktabs,array}
\usepackage{calc} 
\usepackage{etoolbox}
\makeatletter
\patchcmd\longtable{\par}{\if@noskipsec\mbox{}\fi\par}{}{}
\makeatother
\IfFileExists{footnotehyper.sty}{\usepackage{footnotehyper}}{\usepackage{footnote}}
\makesavenoteenv{longtable}
\usepackage{graphicx}
\makeatletter
\def\maxwidth{\ifdim\Gin@nat@width>\linewidth\linewidth\else\Gin@nat@width\fi}
\def\maxheight{\ifdim\Gin@nat@height>\textheight\textheight\else\Gin@nat@height\fi}
\makeatother
\setkeys{Gin}{width=\maxwidth,height=\maxheight,keepaspectratio}
\makeatletter
\def\fps@figure{htbp}
\makeatother

\usepackage{geometry}
\makeatletter
\@ifpackageloaded{caption}{}{\usepackage{caption}}
\AtBeginDocument{%
\ifdefined\contentsname
  \renewcommand*\contentsname{Table of contents}
\else
  \newcommand\contentsname{Table of contents}
\fi
\ifdefined\listfigurename
  \renewcommand*\listfigurename{List of Figures}
\else
  \newcommand\listfigurename{List of Figures}
\fi
\ifdefined\listtablename
  \renewcommand*\listtablename{List of Tables}
\else
  \newcommand\listtablename{List of Tables}
\fi
\ifdefined\figurename
  \renewcommand*\figurename{Figure}
\else
  \newcommand\figurename{Figure}
\fi
\ifdefined\tablename
  \renewcommand*\tablename{Table}
\else
  \newcommand\tablename{Table}
\fi
}
\@ifpackageloaded{float}{}{\usepackage{float}}
\floatstyle{ruled}
\@ifundefined{c@chapter}{\newfloat{codelisting}{h}{lop}}{\newfloat{codelisting}{h}{lop}[chapter]}
\floatname{codelisting}{Listing}

\makeatother
\makeatletter
\@ifpackageloaded{caption}{}{\usepackage{caption}}
\@ifpackageloaded{subcaption}{}{\usepackage{subcaption}}
\makeatother

\makeatletter
\newcommand{\appendixtableofcontents}{%
  \section*{Table of contents}%
  \@starttoc{apc}%
}
\newcommand{\appsection}[1]{%
  \section{#1}%
  \addcontentsline{apc}{section}{\protect\numberline{\thesection}#1}%
}
\newcommand{\appsubsection}[1]{%
  \subsection{#1}%
  \addcontentsline{apc}{subsection}{\protect\numberline{\thesubsection}#1}%
}
\newcommand{\appsubsubsection}[1]{%
  \subsubsection{#1}%
  \addcontentsline{apc}{subsubsection}{\protect\numberline{\thesubsubsection}#1}%
}
\makeatother

\ifLuaTeX
  \usepackage{selnolig}  
\fi
\usepackage[numbers]{natbib}
\usepackage{bookmark}

\IfFileExists{xurl.sty}{\usepackage{xurl}}{} 
\hypersetup{
  pdftitle={Empirical Bayes Estimation and Inference via Smooth Nonparametric Maximum Likelihood},
  pdfauthor={Taehyun Kim; Bodhisattva Sen},
  pdfkeywords={Deconvolution, g-modeling, optimal marginal coverage sets, posterior density estimation, one-sided inference},
  colorlinks=true,
  linkcolor={blue},
  filecolor={Maroon},
  citecolor={Blue},
  urlcolor={Blue},
  pdfcreator={LaTeX via pandoc}}

\newtheoremstyle{customgapstyle}
{8pt}      
{8pt}      
{\itshape}  
{}          
{\bfseries} 
{.}         
{.5em}      
{}          

\newcommand{\anon}{1}

\usepackage{bbm}
\theoremstyle{customgapstyle}
\newtheorem{theorem}{Theorem}

\newtheorem{lemma}{Lemma}

\newtheorem{remark}{Remark}

\newtheorem{prop}{Proposition}

\newtheorem{assumption}{Assumption}

\numberwithin{equation}{section}
\numberwithin{example}{section}
\numberwithin{theorem}{section}
\numberwithin{lemma}{section}
\numberwithin{corollary}{section}
\numberwithin{prop}{section}
\numberwithin{definition}{section}
\numberwithin{remark}{section}
\numberwithin{assumption}{section}

\def\argmin{\mathop{\rm argmin}}
\def\argmax{\mathop{\rm argmax}}
\def\1v{\mathbf 1}
\def\Var{\mbox{\textup{Var}}}

\def\E{\mathbb{E}}
\def\F{\mathbb{F}}

\def\H{\mathfrak{H}}
\def\P{\mathbb{P}}

\def\fv{\mathbf f}

\def\hv{\mathbf h}

\def\Av{\mathbf A}

\newcommand{\gammav}{\mbox{\boldmath{$\gamma$}}}

\newcommand{\Dc}{\mathcal{D}}

\newcommand{\Fc}{\mathcal{F}}

\newcommand{\Hc}{\mathcal{H}}
\newcommand{\Ic}{\mathcal{I}}
\newcommand{\Jc}{\mathcal{J}}
\newcommand{\Kc}{\mathcal{K}}

\newcommand{\Mc}{\mathcal{M}}
\newcommand{\Nc}{\mathcal{N}}

\newcommand{\Pc}{\mathcal{P}}

\newcommand{\Rc}{\mathcal{R}}
\newcommand{\Sc}{\mathcal{S}}

\newcommand{\Xc}{\mathcal{X}}

\newcommand{\KS}{d_{\textup{KS}}}
\newcommand{\aprior}{H}
\newcommand{\trueprior}{{\aprior^*}}
\newcommand{\npmle}{{\widehat{\aprior}_n}}

\newcommand{\czero}{c_0}

\newcommand*\diff{\mathop{}\!\mathrm{d}}
\newcommand{\abs}[1]{\left\vert#1\right\vert}
\newcommand{\Real}{\mathbb R}
\allowdisplaybreaks

\begin{document}

\def\spacingset#1{\renewcommand{\baselinestretch}%
{#1}\small\normalsize} \spacingset{1}


\if1\anon
{
  \title{\bf Empirical Bayes Estimation and Inference via \\Smooth Nonparametric Maximum Likelihood}
  \author{Taehyun Kim\thanks{
    E-mail: tk3036@columbia.edu}\hspace{.2cm}\\
    Department of Statistics, Columbia University\\
    and \\
    Bodhisattva Sen\thanks{The author gratefully acknowledges support from NSF grants DMS-2311062 and DMS-2515520. E-mail: b.sen@columbia.edu} \\
    Department of Statistics, Columbia University}
  \maketitle
} \fi

\if0\anon
{
  \bigskip
  \bigskip
  \bigskip
  \begin{center}
    {\LARGE\bf Empirical Bayes Estimation and Inference via Smooth Nonparametric Maximum Likelihood}
\end{center}
  \medskip
} \fi

\bigskip
\begin{abstract}
\noindent 
The empirical Bayes $g$-modeling approach based on the nonparametric maximum likelihood estimator (NPMLE) has been central to large-scale estimation and inference in the normal means problem. However, theoretical guarantees for uncertainty quantification remain scarce. A key obstacle is that the NPMLE is necessarily discrete, yielding discrete posterior credible sets, while its deconvolution rate can also be logarithmically slow. We address both limitations by studying a hierarchical Gaussian smoothing layer that restricts the mixing distribution to a Gaussian location mixture, with the smoothing parameter estimated from data. Our smooth NPMLE inherits the favorable properties of the classical NPMLE: it is computable via convex optimization and retains near-parametric rates for marginal density estimation and denoising. Moreover, it achieves polynomial convergence rates for prior and posterior density estimation, with minimax optimal leading polynomial exponents under suitable conditions. Even when the true prior is not a Gaussian location mixture, our smooth NPMLE and the posterior density estimator converge to the true prior and posterior densities under additional regularity conditions. Further, we characterize marginal coverage sets that are optimal in expected length, construct plug-in estimators, and establish theoretical guarantees for coverage probability and expected length. We also extend the theory to heteroscedastic Gaussian observations.
\end{abstract}

\noindent%
{\it Keywords:} 
Deconvolution, $g$-modeling, optimal marginal coverage sets, posterior density estimation, one-sided inference.

\spacingset{1} 

\section{Introduction}\label{sec:intro}
Consider the prototypical normal location mixture model:
\begin{equation}\label{eq:original model}
X_i \mid \theta_i \overset{ind}\sim N(\theta_i,1), \qquad \text{and} \qquad 
\theta_i \overset{iid}\sim G^*,\qquad \mbox{for }\;\; i = 1, \ldots, n,
\end{equation}
where we observe $X_1,\ldots,X_n$, the $\theta_1,\ldots,\theta_n$ are latent (unobserved), and the mixing distribution (or ``prior'') $G^*$ is unknown and belongs to $\Pc(\Real)$, the collection of all probability measures on $\Real$. Model~\eqref{eq:original model} is a central workhorse in empirical Bayes (EB), and it has been studied extensively from several complementary perspectives. The main objectives pursued in this setting include: (i) estimation of the marginal density of the observations $\{X_i\}_{i=1}^n$ \citep{Wong1995,Ghosal2001,Zhang2009,Saha2020,Soloff2025}; (ii) denoising $\{X_i\}_{i=1}^n$ to estimate (or predict) $\{\theta_i\}_{i=1}^n$ \citep{Jiang2009,Saha2020,Soloff2025,ghosh2025steinsunbiasedriskestimate}; (iii) estimation of the mixing distribution $G^*$ itself \citep{efron2016,Soloff2025}; and (iv) uncertainty quantification and inference for $\{\theta_i\}_{i=1}^n$ \citep{Morris1983,Laird1987,Jiang2009,Armstrong2022}.

A common approach to EB estimation proceeds via \emph{$g$-modeling}: one estimates the unknown mixing distribution $G^*$ directly from the data and then plugs this estimate into downstream procedures \citep{Laird1978,Efron2014,Jiang2009}. A widely used choice is the nonparametric maximum likelihood estimator (NPMLE) \citep{Kiefer1956}, defined as the maximizer of the marginal likelihood of $\{X_i\}_{i=1}^n$ over $\Pc(\Real)$. It is well-known that the NPMLE exhibits strong theoretical and empirical performance---it achieves near-parametric convergence rates (under suitable conditions) for marginal density estimation and for denoising in the normal means problem; see e.g., \cite{Jiang2009,Zhang2009,Saha2020,Soloff2025}. 
Despite these strengths, the NPMLE is not well suited for \emph{deconvolution}, i.e., recovering $G^*$ itself from the noisy observations. In particular, convergence of the NPMLE to $G^*$ can be logarithmically slow (see e.g., Theorem~10 of~\citet{Soloff2025}); such slow rates reflect the intrinsic ill-posedness of deconvolution and are, in fact, minimax-optimal over Sobolev classes \citep{Carroll1988,Fan1991}. A further limitation is that the NPMLE is necessarily discrete \citep{Lindsay1995,Jiang2009,Shen1999}, which can be undesirable when the true $G^*$ is smooth. This discreteness also complicates inference for $\{\theta_i\}_{i=1}^n$: the implied plug-in posterior for each $\theta_i$ given $ X_i$ is discrete, leading to posterior credible sets with an inherently discrete structure.

To address these limitations, it is natural to invoke the ``\emph{bet on smoothness}'' principle \citep{efron2016,efron2019}, which recommends restricting attention to suitably smooth classes of mixing distributions for $G^*$. Several influential proposals follow this philosophy. \citet{efron2016}, for example, models the prior density within a (quasi-)parametric exponential-family form using a natural spline basis, producing a smooth and often accurate estimate of $G^*$. 
\citet{Stephens2016} instead posits Gaussian scale-mixture priors (typically with a mode at zero) to obtain stabilized, smoothed posterior distributions. 
\citet{Bovy2011} propose \emph{Extreme Deconvolution} (XD), which models the unknown prior as a finite Gaussian location-scale mixture, i.e., $G^*$ is assumed to have density $g$ given by $g(\cdot)=\sum_{k=1}^{K} w_k\,\phi_{\tau_k}(\cdot-\xi_k)$,
where $\phi_{\sigma}$ denotes the $N(0,\sigma^2)$ density. The parameters $\{(w_k,\xi_k,\tau_k^2)\}_{k=1}^K$ are typically fit via the Expectation-Maximization (EM) algorithm. XD has been effective in astronomy applications~\citep{Anderson2018}, but it has two key limitations. 
Because joint optimization over the component weights, locations, and variances is non-convex, EM can converge to local rather than global maxima. In addition, the number of mixture components $K$ must be chosen, with limited theory for the end-to-end procedure.

Motivated by the \emph{bet on smoothness} principle, we introduce an additional Gaussian latent layer in~\eqref{eq:original model}. 
Specifically, consider the hierarchical normal location mixture model
\begin{equation}\label{eq:hierarchical model}
    X_i \mid \theta_i \overset{ind}\sim N(\theta_i,1),
    \qquad
    \theta_i \mid \xi_i \overset{ind}\sim N(\xi_i,\czero^2),
    \qquad
    \xi_i \overset{iid}\sim \trueprior,
\end{equation}
for $i=1,\ldots,n$, where $\czero\geq0$ and $\trueprior\in\Pc(\Real)$ are unknown. 
As we explain below, this framework combines flexible prior modeling and tractable computation with statistical guarantees for estimation and inference. Under \eqref{eq:hierarchical model}, the marginal distribution of $\theta_i$ is
\begin{equation}\label{eq:smoothed true prior}
    G^*=\trueprior\star N(0,\czero^2),
\end{equation}
where $\star$ denotes convolution on $\Real$. Note that the Gaussian-convolution decomposition is not identifiable without further restrictions: multiple pairs $(H,c)$ can induce the same $G^*$. For example, $G^*=N(0,4)$ can be written as $H\star N(0,c^2)$ with $H=N(0,4-c^2)$ for any $c^2\in[0,4]$. We resolve this ambiguity by defining $\czero$ to be the scale of the \emph{largest Gaussian component} of $G^*$:
\begin{equation}\label{eq:c0}
    \czero\equiv\czero(G^*)
    :=\sup\left\{
        c\geq0:
        G^*=H\star N(0,c^2)
        \text{ for some }H\in\Pc(\Real)
    \right\}.
\end{equation}
Intuitively, $\czero$ is the largest amount of Gaussian smoothing that can be extracted from $G^*$. 
When $\czero>0$, $G^*$ has a smooth Gaussian location mixture density. Allowing $\czero=0$ also accommodates priors beyond Gaussian location mixtures, including discrete priors and continuous priors such as the Laplace distribution.

We first estimate $\czero$ using a neighborhood procedure in the spirit of \citet{Donoho1988}; see Section~\ref{sec:identifiability}. Since the observations have distribution $F^*=\trueprior\star N(0,1+\czero^2)$, we search for the largest Gaussian scale $\sigma$ for which some mixture $H\star N(0,\sigma^2)$ lies within a prescribed Kolmogorov--Smirnov distance of the empirical distribution of $X_1,\ldots,X_n$. This leads to the neighborhood estimator $\hat c_{\mathrm{KS}}^2$ of $\czero^2$, whose consistency is known \citep{Donoho1988,Donoho2013}. 
We establish a logarithmic convergence rate for estimating $\czero^2$ (Proposition~\ref{thm:estimated-c-rate}) and show that this rate is minimax optimal (Proposition~\ref{prop:scale-minimax-lower}). 
However, this estimator overestimates $\czero$ with high probability, and its accuracy does not by itself ensure nearly parametric marginal density estimation.
Indeed, the smooth NPMLE fitted at $\hat c_{\mathrm{KS}}$ can fail to attain the nearly parametric marginal density estimation rate uniformly over the assumed class (Theorem~\ref{thm:ks-only-density}). We therefore introduce a modified estimator $\hat c_0$ that retains the same convergence rate for $\czero^2$ while keeping the fitted marginal log-likelihood sufficiently close to that of the classical NPMLE (Theorem~\ref{thm:modified-c-rate}). This additional likelihood guarantee is used to establish the subsequent statistical guarantees for the smooth NPMLE.

Using the estimated smoothing parameter $\hat c_0$, we estimate $\trueprior$ via the NPMLE. Under \eqref{eq:hierarchical model},
\begin{align}\label{eq:induced-model}
    X_i\mid\xi_i\overset{ind}{\sim}N(\xi_i,1+\czero^2), \qquad \xi_i\overset{iid}{\sim}\trueprior, \qquad \mbox{for }\;\; i = 1, \ldots, n,
\end{align}
which has the same form as the Gaussian location mixture model \eqref{eq:original model}, with $\xi_i$ as the latent effects and noise variance $1+\czero^2$. 
Given $\hat c_0$, we estimate $\trueprior$ by the NPMLE $\npmle$ fitted at noise variance $1+\hat c_0^2$, and obtain the estimate of $G^*$ by Gaussian convolution:
\begin{align}\label{eq:estimated-prior-distribution}
    \npmle &\in\argmax_{H\in\Pc(\Real)} \sum_{i=1}^n\log\left\{ \int\phi_{\sqrt{1+\hat c_0^2}}(X_i-\xi) \,\diff H(\xi) \right\}, \qquad \widehat G_n:=\npmle\star N(0,\hat c_0^2).
\end{align}
For a fixed positive smoothing parameter, this construction coincides with the \emph{smooth NPMLE} of \citet{Magder1996}; see also \citep{Cordy1997,Shen1999}. Here, we estimate the scale of the largest Gaussian component of the true prior and establish statistical guarantees for the resulting smooth NPMLE and empirical Bayes inference.
In contrast to Efron’s spline-based approach and the joint parameter optimization in the XD approach~\cite{Bovy2011}, computing the NPMLE $\npmle$ in \eqref{eq:estimated-prior-distribution} amounts to solving a \emph{convex} optimization problem, albeit infinite-dimensional; this can be easily approximated by a finite-dimensional convex problem and efficiently solved using off-the-shelf convex programming methods 
\citep{Koenker2014,Koenker2017}.

A primary task in EB is to denoise the observations $X_i$ to estimate (or predict) the latent effects $\theta_i$. Under squared error loss, the oracle decision rule is the posterior mean $\E_{G^*} [\theta_i \mid X_i]$, and the $g$-modeling approach replaces the unknown prior distribution $G^*$ in \eqref{eq:smoothed true prior} by the plug-in estimator $\widehat G_n$ in \eqref{eq:estimated-prior-distribution}. 
A key advantage of the hierarchical normal-normal model \eqref{eq:hierarchical model} is that Gaussian conjugacy yields a tractable expression for the posterior mean:
\begin{align}\label{eq:oracle posterior mean}
    \theta_i^*:=\E_{G^*}[\theta_i\mid X_i] =\frac{\czero^2}{1+\czero^2}X_i +\frac{1}{1+\czero^2}\E_{\trueprior}[\xi_i\mid X_i],
\end{align}
making the plug-in EB estimate straightforward to compute. 
Further, our smooth NPMLE achieves the nearly parametric rates for marginal density estimation and denoising established for the classical NPMLE \citep{Jiang2009, Saha2020} (Theorem~\ref{thm:estimated-c-density-denoising}) under Assumption~\ref{ass:moment-fourier}, which requires either compact support or suitable moment and Fourier conditions on $H^*$ in \eqref{eq:hierarchical model}. These guarantees hold for all $\czero\ge0$, despite the logarithmic convergence rate of $\hat c_0$. In particular, when $\czero=0$, they apply to both discrete and continuous priors $G^*$ satisfying Assumption~\ref{ass:moment-fourier}.

Our smooth NPMLE retains the advantages of the classical NPMLE, while addressing its limitations in deconvolution and posterior density estimation. 
Note that, when $\hat c_0=0$, the fitted prior is discrete. Hence, we impose a vanishing positive floor on $\hat c_0$, yielding the stabilized smoothing parameter $\hat c_0^+:=\hat c_0\vee\sqrt{\rho_n}>0$, where $\rho_n=(\log\log n)/\log n$. We obtain the smooth prior density $\widehat g_n^+$ by Gaussian convolution of the NPMLE fitted at $\hat c_0^+$; see~\eqref{eq:NPMLE} and~\eqref{eq:estimated prior density}. 
When $\czero > 0$, we show that $\widehat g_n^+$ converges to the true prior density $g^*$ of $G^*$ at a fast \emph{polynomial} rate under Assumption~\ref{ass:moment-fourier} (Theorem~\ref{thm:moment-deconvolution}), in sharp contrast to the logarithmically slow behavior of the classical NPMLE under the original model \eqref{eq:original model}~\citep{Soloff2025}. Under suitable conditions, the leading polynomial exponent in this convergence rate is minimax optimal (Theorem~\ref{thm:moment-deconvolution-lower}).
Furthermore, we show that the posterior density estimated using the smooth NPMLE converges to the true posterior density at a polynomial rate in weighted total variation distance, with a leading polynomial exponent that is minimax optimal under the stated conditions (Theorems~\ref{thm:estimated-c-posterior} and~\ref{thm:moment-posterior-lower}).

\begin{figure}[t]
    \centering
    \begin{subfigure}[t]{0.32\textwidth}
        \centering
        \includegraphics[width=\linewidth]{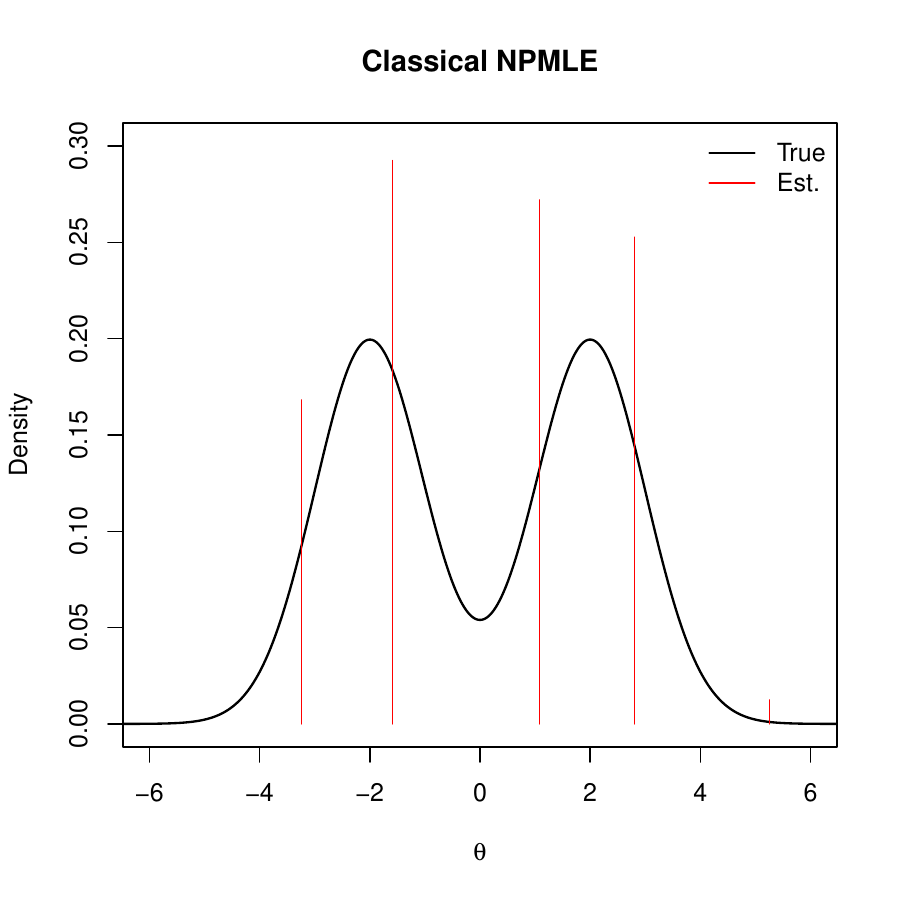}
    \end{subfigure}
    \hfill
    \begin{subfigure}[t]{0.32\textwidth}
        \centering
        \includegraphics[width=\linewidth]{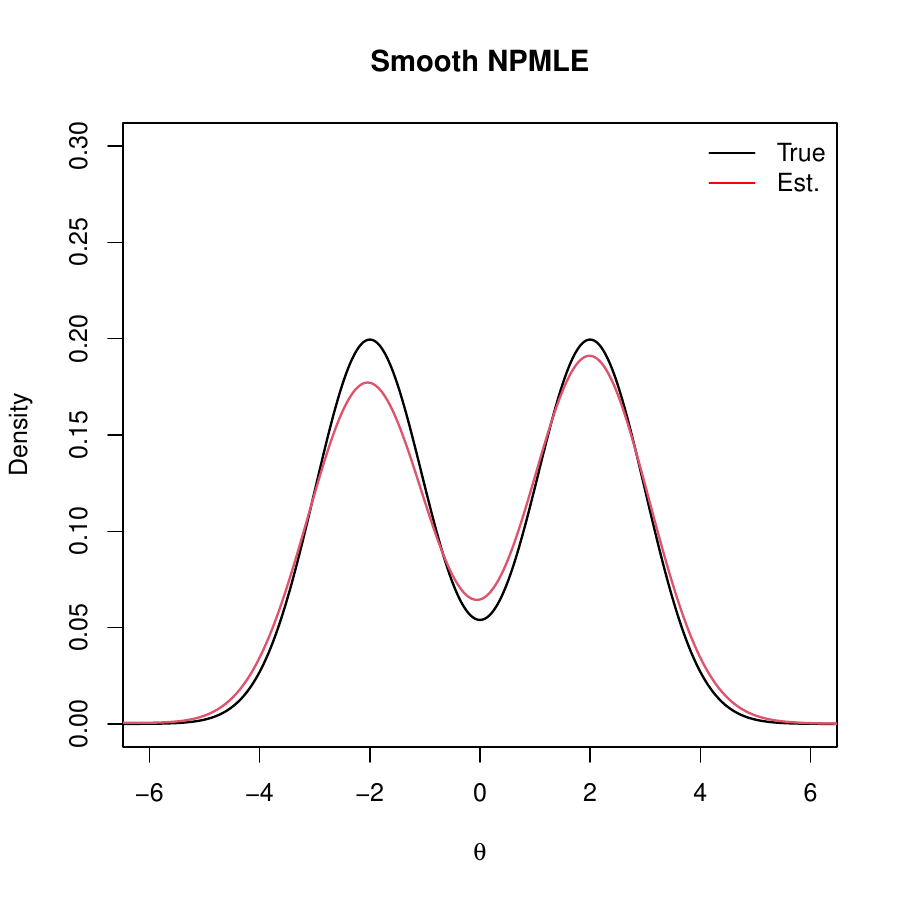}
    \end{subfigure}
    \hfill
    \begin{subfigure}[t]{0.32\textwidth}
        \centering
        \includegraphics[width=\linewidth]{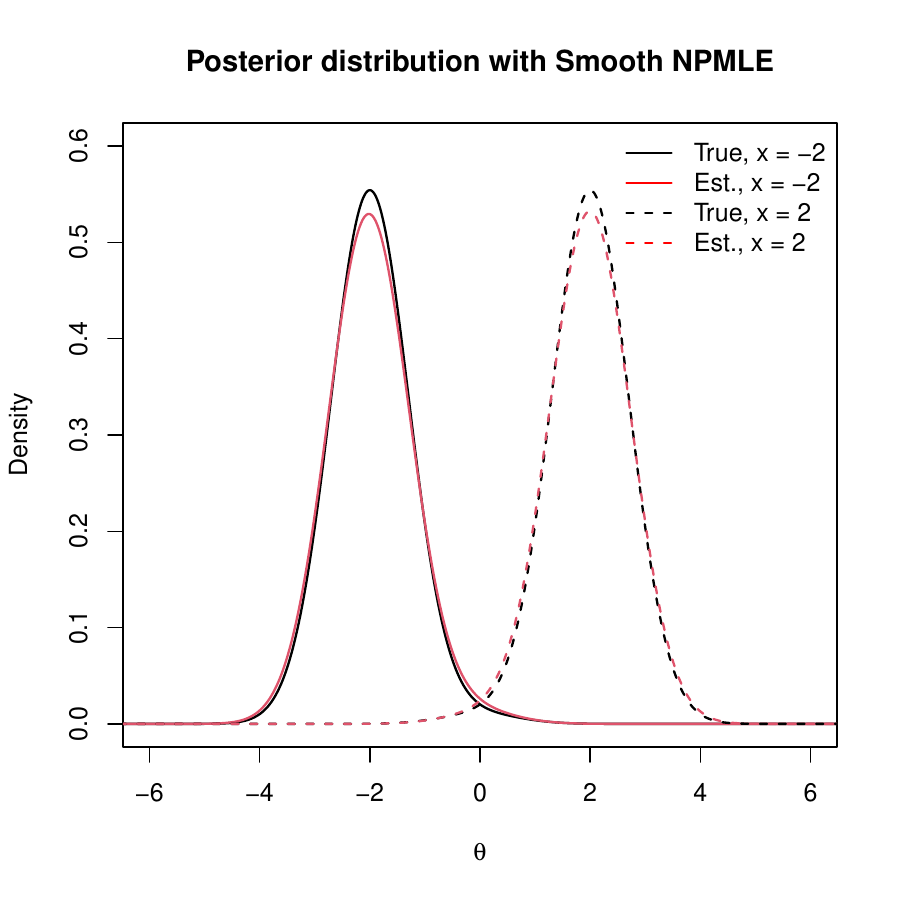}
    \end{subfigure}
    \caption{\footnotesize 
    We consider $G^* = N(-2,1)/2+N(2,1)/2$ in \eqref{eq:original model} (equivalently, $\trueprior = \delta_{-2}/2 + \delta_{2}/2$ and $\czero = 1$ in \eqref{eq:hierarchical model}); here $\delta_x$ denotes the Dirac delta measure at $x$. The classical NPMLE and the smooth NPMLE $\widehat g_n^+$ are computed using $n = 1000$ observations and shown in the left and center plots along with the true prior density $g^*$. For the smooth NPMLE, $\czero$ is also estimated by $\hat c_0^+$. The true posterior densities at $x = \pm 2$ and the estimated posterior densities based on the smooth NPMLE are shown in the rightmost plot.
    }
    \label{fig:SNPML-TwoComp}
\end{figure}

We illustrate the advantages of using the smooth NPMLE in Figures~\ref{fig:SNPML-TwoComp} and~\ref{fig:SNPML-laplace}. 
In both examples, the (nonsmooth) classical NPMLE (based on model~\eqref{eq:original model}) is discrete, while the smooth NPMLE provides a continuous density estimate that approximates the true prior density.
While the two-component normal mixture prior in Figure~\ref{fig:SNPML-TwoComp} has $\czero = 1$, the Laplace prior in Figure~\ref{fig:SNPML-laplace} has $\czero=0$.
Nevertheless, Gaussian mixtures are dense in broad classes of continuous probability densities, so smooth NPMLEs can approximate a wide range of true priors under model misspecification (i.e., $\czero =0$). 
When $\czero=0$, our smooth NPMLE density $\widehat g_n^+$ is misspecified, but we show that it and the corresponding estimated posterior densities converge to the true prior and posterior densities at logarithmic rates in squared $L_2$ risk and expected weighted total variation distance, respectively, under suitable regularity conditions (Theorem~\ref{thm:misspecification-true}).
These results extend our guarantees to a broader class of priors, including the Laplace prior in Figure~\ref{fig:SNPML-laplace}.

\begin{figure}[t]
    \centering
    \begin{subfigure}[t]{0.32\textwidth}
        \centering
        \includegraphics[width=\linewidth]{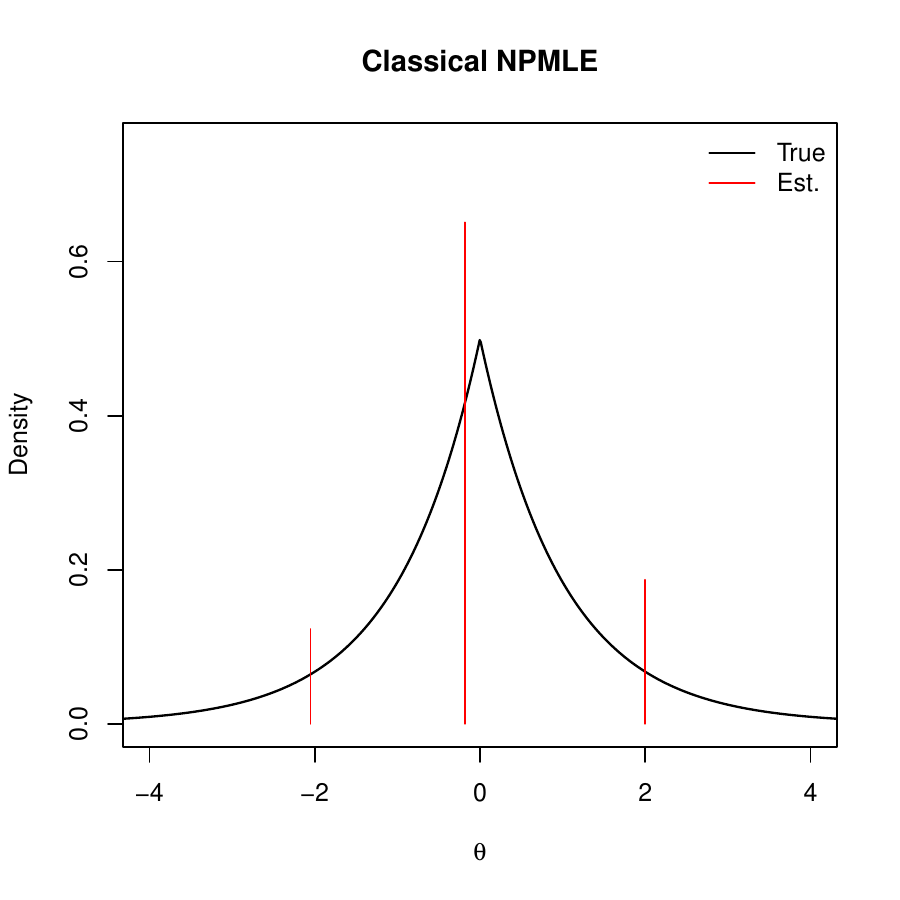}
    \end{subfigure}
    \hfill
    \begin{subfigure}[t]{0.32\textwidth}
        \centering
        \includegraphics[width=\linewidth]{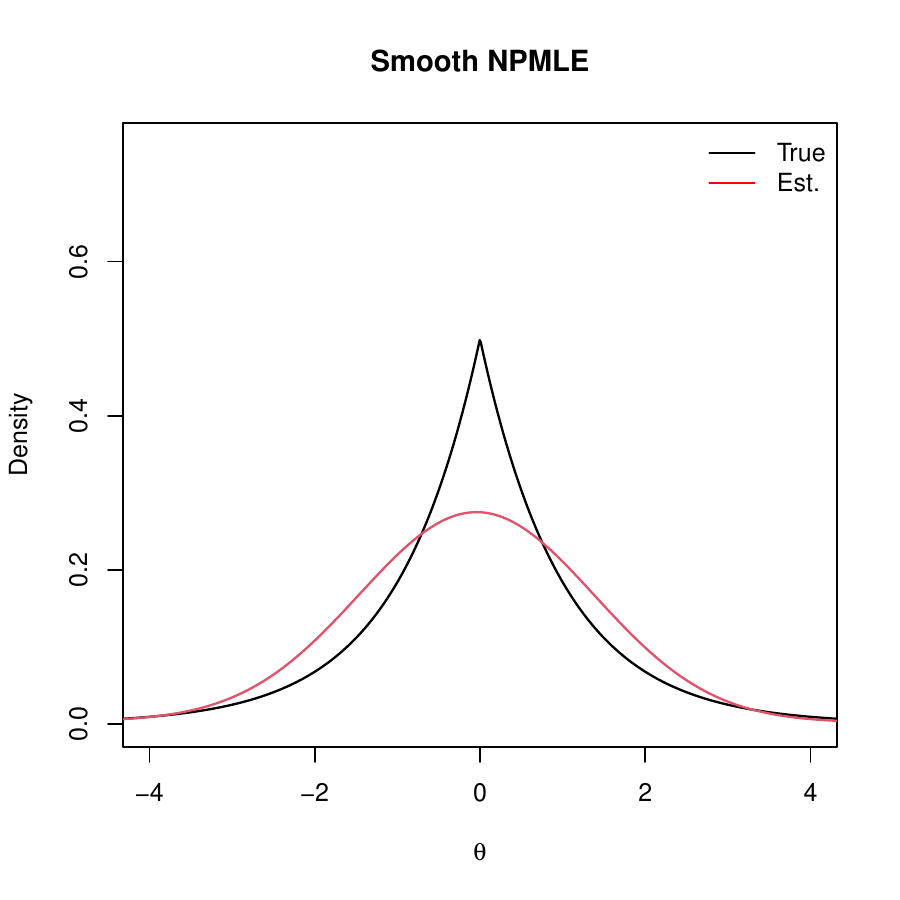}
    \end{subfigure}
    \hfill
    \begin{subfigure}[t]{0.32\textwidth}
        \centering
        \includegraphics[width=\linewidth]{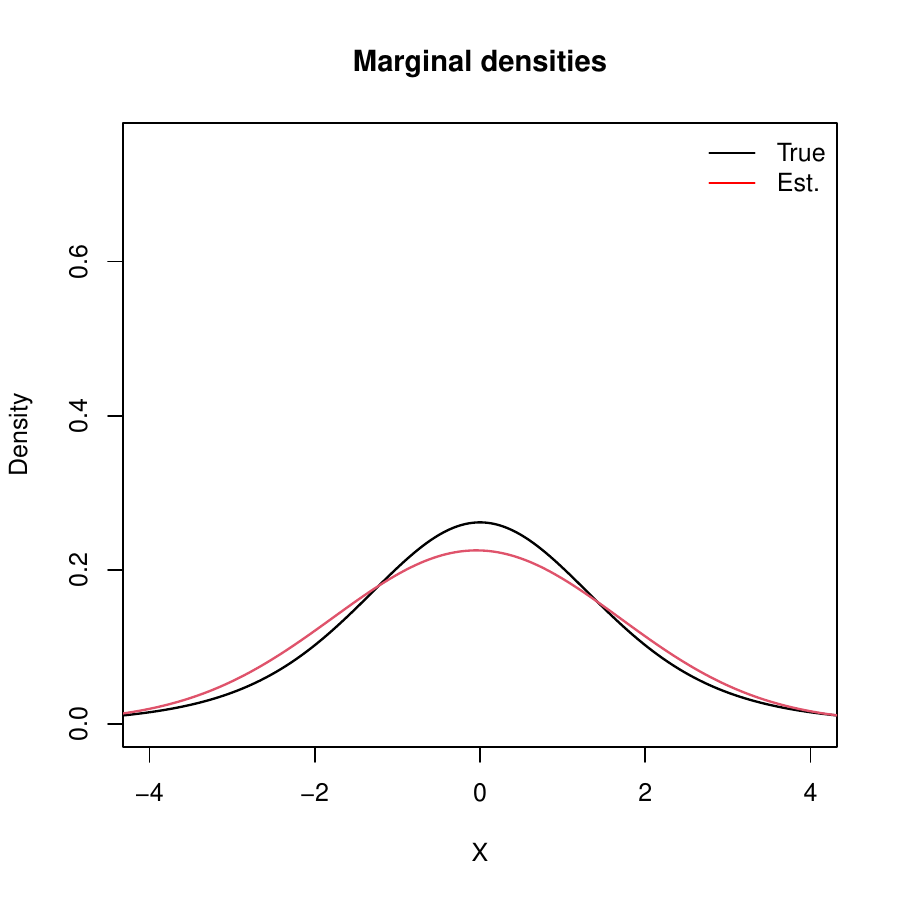}
    \end{subfigure}
    \caption{\footnotesize The setup is the same as in Figure~\ref{fig:SNPML-TwoComp}, except that the true prior is $G^* = \mathrm{Laplace}(0,1)$. The true marginal density of the observations and the estimated marginal density based on the smooth NPMLE are shown in the rightmost plot. For the smooth NPMLE, $\czero$ is estimated by $\hat c_0^+$ defined in \eqref{eq:deconvolution-floor}.}
    \label{fig:SNPML-laplace}
\end{figure}

While EB methods have been extensively developed for estimating the latent effects $\{\theta_i\}_{i=1}^n$ \citep{James1961,Jiang2009,Saha2020,Soloff2025}, comparatively less is known about uncertainty quantification in this setting \citep{Morris1983,Jiang2019,Armstrong2022}. 
In this paper, we focus on the notion of \emph{marginal} coverage: a confidence set $\Jc_i \equiv \Jc_i(X_i)$ is said to have $(1-\beta)$ marginal coverage for $\theta_i$ if
\begin{equation}\label{eq:marginal coverage}
    \P_{G^*} \!\left(\theta_i \in \Jc_i \right) \ge 1-\beta,
\end{equation}
where the probability is taken with respect to the joint law of $(X_i,\theta_i)$ in \eqref{eq:hierarchical model}.
Marginal coverage is natural in EB problems because latent effects are modeled as draws from a common prior~\citep{Morris1983, Armstrong2022}; moreover, average coverage is more relevant in many applications than coverage for any specific effect~\citep{Hoff2009, Hoff2025}.  Marginal coverage~\eqref{eq:marginal coverage} is also less stringent than requiring $\Jc_i$ to be a (conditional) frequentist confidence interval for each fixed $\theta_i$ (e.g., the usual $(1-\beta)$ normal-theory interval $\Sc_i = X_i \pm z_{1-\beta/2}$) or an exact Bayesian credible set (e.g., the $(1-\beta)$ highest posterior density set). 
Despite its importance, there has not been a rigorous study of optimal marginal coverage sets in a nonparametric setting.

Our approach targets marginal coverage sets $\Ic^*$ that are optimal in expected length among all procedures that satisfy \eqref{eq:marginal coverage}. The construction is closely related to highest posterior density (HPD) sets: while HPD sets use an $x$-dependent posterior density threshold to achieve posterior probability $1-\beta$, optimal marginal coverage sets use a \emph{constant} threshold chosen to satisfy \eqref{eq:marginal coverage} (Theorem~\ref{thm:Opt-Marginal}). 
Compared with HPD credible sets, optimal marginal coverage sets can have shorter expected length under the less stringent marginal coverage requirement and require calibration of only a single threshold.  We note that this characterization holds for general likelihoods, beyond the normal model~\eqref{eq:original model} discussed so far (see Section~\ref{sec:OPT construction}).

Based on the smooth NPMLE, we construct \emph{EB confidence sets} $\Jc_i(X_1,\ldots,X_n)$ that target the oracle sets $\Ic^*(X_i)$ while borrowing strength across observations.
In contrast to the classical NPMLE, the estimated smooth prior $\widehat g_n^+$ yields smooth posterior densities and hence natural (non-discrete) plug-in confidence sets. 
Building on the polynomial deconvolution rate of the smooth NPMLE, we show that their marginal coverage approaches $1-\beta$ and their expected length converges to the oracle optimal value, both at polynomial rates (Theorems~\ref{thm:estimated-c-coverage} and~\ref{thm:estimated-c-length}).
Our procedure provides a substantial gain compared to the standard normal-theory interval $\Sc_i = X_i \pm z_{1-\beta/2}$. 
For example, under the two-component Gaussian mixture prior used in Figure~\ref{fig:SNPML-TwoComp} with $\beta=0.05$, the estimated optimal marginal coverage sets have average length around $3.31$, while $\Sc_i$ have constant length approximately $3.92$; see Appendix~\ref{sec:simulation} for additional simulation results.

In the normal location mixture model~\eqref{eq:original model}, we assume homoscedastic noise with a common variance, while in many practical applications the observations exhibit {\it heterogeneity} (i.e., varying noise levels) \citep{Anderson2018,Banerjee2023,Soloff2025}.
In Section~\ref{sec:hetero model} and Appendix~\ref{sec:hetero model details}, we discuss statistical properties of the smooth NPMLE and extensions of optimal marginal coverage sets to heterogeneous settings. We show that analogous results hold for estimation of $\czero$, marginal density estimation, denoising, and prior and posterior density estimation (Theorems~\ref{thm:hetero-estimated-c-rate}--\ref{thm:het-posterior}).

The full smooth $g$-modeling procedure for estimation and inference developed in this paper is summarized in Algorithm~\ref{alg:summary}. Using the smooth $g$-modeling procedure, we analyze the 2002 Educational Longitudinal Study (ELS) dataset and the prostate dataset of \citet{Singh2002} in Appendices~\ref{sec:applications-ELS} and~\ref{sec:applications-prostate}, respectively. Our results show that, in the ELS dataset, the average length of the optimal marginal coverage sets is $4.44$, compared to $10.07$ for the standard frequentist confidence intervals. In the prostate dataset, the corresponding average lengths are $1.91$ and $3.92$, respectively.

\begin{algorithm}
\caption{Workflow for empirical Bayes estimation and inference via smooth NPMLE}\label{alg:summary}
\small
\begin{algorithmic}[1]
    \Require Observations $\{X_i\}_{i=1}^{n}$ and target coverage level $1-\beta$
    \State Compute $\hat c_0$ in \eqref{eq:estimated-c} and the NPMLE $\npmle$ in \eqref{eq:estimated-prior-distribution} using $\hat{c}_0$
    \State Compute the EB posterior means $\{\hat\theta_i\}_{i=1}^{n}$ in \eqref{eq:EB posterior mean} \textbf{(Output I)}
    \State Set $\hat c_0^+=\hat c_0\vee\sqrt{\rho_n}$ with $\rho_n = \log\log n /\log n$ as in \eqref{eq:deconvolution-floor} and compute the NPMLE $\npmle^+$ in \eqref{eq:NPMLE} using $\hat c_0^+$
    \State Form the smooth NPMLE $\widehat g_n^+=\npmle^+\star N(0,(\hat c_0^+)^2)$ and posterior density $\widehat\pi_n^+(\cdot\mid X_i)$
    \State Solve \eqref{eq:estimated threshold with sampling} for the threshold $\hat k_n^B$ and construct the marginal coverage sets $\{\hat{\Ic}_n^B(X_i)\}_{i=1}^{n}$ in \eqref{eq:estimated optimal marginal sets with sampling} \textbf{(Output II)}
\end{algorithmic}
\end{algorithm}

Throughout, we write $x \gtrsim_{p,q} y$ or $x \lesssim_{p,q} y$ to mean that $x \ge C_{p,q}\,y$ or $x \le C_{p,q}\,y$ for a constant $C_{p,q}>0$ depending only on the parameters $p$ and $q$. 
Unless otherwise stated, subscripts on $\E$ and $\P$ indicate the prior distribution used in the model, such as $G^*$ or $\widehat G_n$ for $\theta_i$ and $\trueprior$ or $\npmle$ for $\xi_i$.

\section{Estimation of the largest Gaussian component $c_0$}\label{sec:identifiability}

In this section, we study estimation of the largest Gaussian component $\czero \equiv c_0(G^*)$ of the true prior $G^*$, defined in~\eqref{eq:c0}. The functional $G\mapsto\czero(G)$ is upper semicontinuous under weak convergence, but it is not lower semicontinuous \citep{Donoho1988}. 
Without further restrictions on the prior, one cannot construct an informative lower confidence bound for $\czero$ that is uniformly valid over $G^*\in\Pc(\Real)$; see \citep{Donoho1988, Patra2016} (see Appendix~\ref{sec:neighborhood-upper-bound} for the precise impossibility argument). 
An upper confidence bound, however, can be constructed by searching for the largest Gaussian scale among distributions close to the empirical distribution of the data. Building on this construction, we estimate $\czero$ using a modification of the neighborhood procedure of~\citet{Donoho1988}.

To construct the neighborhood estimator, note that $X_1, \ldots, X_n \overset{iid}{\sim} F^* := G^* \star N(0,1) = \trueprior \star N(0,\sigma_0^2)$ under model \eqref{eq:hierarchical model} where $\sigma_0^2:=\czero^2+1$. Thus, inference on $\czero$ can be recast as inference on the observable marginal distribution $F^*$. For any $F\in\Pc(\Real)$, define its largest Gaussian scale by
\begin{equation}\label{eq:largest-gaussian-scale}
    \sigma_0(F):=\sup\left\{\sigma\geq0:F=H\star N(0,\sigma^2) \text{ for some }H\in\Pc(\Real)\right\}.
\end{equation}
Note that the supremum is attained (see Section~III.A of \citet{Donoho2013}). By the definitions in \eqref{eq:c0} and \eqref{eq:largest-gaussian-scale}, we have $\sigma_0(F^*)=\sigma_0$, which is the largest standard deviation of a Gaussian component that can be factored out of $F^*$. Thus, $\czero = \sqrt{\sigma_0^2 - 1}$. 
Working with $\sigma_0(F^*)$ rather than $\czero(G^*)$ is convenient because $F^*$ is directly estimable from the data via the empirical distribution function $\F_n(x) := n^{-1}\sum_{i=1}^{n}\1v(X_i\le x)$.

The key idea behind the neighborhood procedure is a robustification of $\sigma_0$: instead of evaluating $\sigma_0$ at the empirical distribution $\F_n$ directly, which is discrete and therefore satisfies $\sigma_0(\F_n)=0$, we ask for the largest $\sigma_0$ achievable by any distribution within Kolmogorov--Smirnov (KS) distance $\eta$ of $\F_n$. For any $F \in \Pc(\Real)$ and $\eta > 0$, we define the $\eta$-upper envelope of $\sigma_0$ in \eqref{eq:largest-gaussian-scale} as
\begin{align}\label{eq:envelope}
    \sigma_0(F; \eta) := \sup\{\sigma_0(\tilde{F}): \tilde{F} \in \Pc(\Real),\;
    \KS(F, \tilde{F}) \leq \eta\}
\end{align}
where $\KS(F, \tilde{F}) := \sup_{x \in \Real} |F(x) - \tilde{F}(x)|$ is the KS distance between $F, \tilde{F} \in \Pc(\Real)$. Here, by abuse of notation, we write $F(x) \equiv F((-\infty, x])$ for any $F \in \Pc(\Real)$. 
The implication
\begin{align}\label{eq:neighborhood}
    \KS(\F_n, F^*) \leq \eta \quad \Longrightarrow \quad
    \sigma_0 \leq \sigma_0(\F_n; \eta) \leq \sigma_0(F^*; 2\eta)
\end{align}
then says: if $\F_n$ is within $\eta$ of $F^*$, then the true $\sigma_0$ lies below $\sigma_0(\F_n; \eta)$. So $\sigma_0(\F_n; \eta)$ is a valid upper bound for $\sigma_0$ on the event $\{\KS(\F_n, F^*) \le \eta\}$. For the neighborhood radius, set
\begin{equation}\label{eq:eta-radius}
    \eta_n:=\sqrt{\frac{\log n}{n}}.
\end{equation}
Using \eqref{eq:neighborhood} with this choice of $\eta_n$, we have $\sigma_0(\F_n;\eta_n)\xrightarrow{a.s.}\sigma_0(F^*)$; see Theorem~3.4 of \citet{Donoho1988} and Lemma~III.2 of \citet{Donoho2013}.
The corresponding neighborhood estimator of $\czero$ is given by
\begin{equation}\label{eq:estimated-KS-c}
    \hat c_{\mathrm{KS}} :=\sqrt{(\sigma_0^2(\F_n;\eta_n)-1)_+} \xrightarrow{a.s.} \czero.
\end{equation}

While the preceding result establishes consistency, to our knowledge no convergence rate has been established for this neighborhood estimator. To quantify its convergence rate, we impose the following conditions on $\trueprior$ in \eqref{eq:hierarchical model}.

\begin{assumption}\label{ass:moment-fourier}
Suppose that at least one of the following conditions on $\trueprior$ holds.
\begin{itemize} 
    \item[(i)] (Moment and Fourier conditions) There are constants $K_0\geq1$, $r\in(0,2]$, $a\in(0,1]$, $\vartheta>0$, and $t_0\geq1$ such that \begin{align}\label{eq:simultaneous-moment-control}
        \left\{\int|\xi|^p\,\diff\trueprior(\xi)\right\}^{1/p} \leq K_0p^{1/r}, \qquad \mbox{for all }\;\; p\geq1,
    \end{align}
    and, writing $\varphi_{\trueprior}(t) :=\int e^{it\xi}\,\diff\trueprior(\xi)$,
    \begin{align}\label{eq:fourier-lower-bound}
        |\varphi_{\trueprior}(t)| \geq a(1+|t|)^{-\vartheta}, \qquad |t|\geq t_0.
    \end{align}
    In this case, write $\boldsymbol{\gamma}:=(K_0,r,a,\vartheta,t_0)$.

    \item[(ii)] (Compact support) There is a constant $R\geq0$ such that $\operatorname{supp}(\trueprior)\subseteq[-R,R]$. In this case, $r=2$ in \eqref{eq:simultaneous-moment-control} and set $\boldsymbol{\gamma}:=(R,2)$.
\end{itemize}
\end{assumption}

In Assumption~\ref{ass:moment-fourier} $(i)$, the moment condition \eqref{eq:simultaneous-moment-control} controls the tails of $H^*$, with exponential tails for $r=1$ and sub-Gaussian tails for $r=2$. This condition is commonly used in EB estimation \citep{Jiang2009, Saha2020,Soloff2025}. The Fourier lower bound \eqref{eq:fourier-lower-bound} is closely related to Assumption~B of \citet{Butucea2005}, one of the conditions used in their work on deconvolution with an unknown noise scale. This is analogous to our induced model~\eqref{eq:induced-model}, in which the Gaussian noise variance $1+\czero^2$ is unknown. Assumption~\ref{ass:moment-fourier} $(i)$ permits unbounded priors $\trueprior$; e.g., a Laplace distribution satisfies this assumption with $r=1$ and $\vartheta=2$. 
Assumption~\ref{ass:moment-fourier} $(ii)$ is a compact support assumption on $\trueprior$, which is also common in EB estimation and includes any discrete distribution $\trueprior$ with bounded support. In this case, \eqref{eq:simultaneous-moment-control} holds with $K_0=1\vee R$ and $r=2$, without requiring the Fourier lower bound \eqref{eq:fourier-lower-bound}. While the hierarchical model \eqref{eq:hierarchical model} also permits heavy-tailed $\trueprior$ considered in earlier EB work \citep{Jiang2009, Saha2020, Soloff2025}, our theoretical analysis in this paper focuses on the two conditions in Assumption~\ref{ass:moment-fourier}.

Now, let
\begin{equation}\label{eq:rho_n}
    \rho_n:=\frac{\log\log n}{\log n}, \qquad n\geq3.
\end{equation}
Under Assumption~\ref{ass:moment-fourier}, Proposition~\ref{thm:estimated-c-rate} shows that there is a constant $C>0$, depending only on $\czero$ and $\boldsymbol{\gamma}$, such that
\begin{equation}\label{eq:sigma2-rate-event}
    \P_{G^*}\!\left\{
        \czero^2\leq\hat c_{\mathrm{KS}}^2
        \leq\czero^2+C\rho_n
    \right\}\geq1-2n^{-2}
\end{equation}
for all sufficiently large $n$. Thus, with high probability, $\hat c_{\mathrm{KS}}^2$ overestimates
$\czero^2$ by at most $C\rho_n$.
Moreover, the rate $\rho_n$ cannot be improved uniformly over a class satisfying Assumption~\ref{ass:moment-fourier}~$(i)$; see Proposition~\ref{prop:scale-minimax-lower}.

While the bound \eqref{eq:sigma2-rate-event} controls the error in estimating $\czero^2$, it does not ensure sufficiently accurate marginal density estimation. Indeed, the smooth NPMLE fitted at $\hat c_{\mathrm{KS}}$ need not attain a nearly parametric marginal density rate uniformly over the class of models satisfying Assumption~\ref{ass:moment-fourier} $(i)$; see Theorem~\ref{thm:ks-only-density}. That is, a logarithmically small scale error can still produce too much oversmoothing for nearly parametric marginal density estimation. We therefore modify $\hat c_{\mathrm{KS}}$ to retain its scale estimation rate while ensuring that the fitted marginal log-likelihood is sufficiently close to that of the classical NPMLE. This likelihood guarantee is the key property needed for the subsequent smooth NPMLE analysis.

First, to implement the modified neighborhood procedure in practice, we restrict the range of candidate smoothing parameters over which we search. A useful observation is that $\Var(X_i) =1+\czero^2+\Var(\xi_i) \ge 1+\czero^2$ under \eqref{eq:hierarchical model}. Thus, $\sqrt{\{\Var(X_i)-1\}_+}$ is an upper bound for $\czero$, and the sample variance provides a consistent estimator of this upper bound. By slightly inflating the sample variance, we can further construct an upper bound for $\czero$ that is valid with high probability. Specifically, let
\begin{align}\label{eq:cap}
    S_n^2 :=\frac1{n-1}\sum_{i=1}^n(X_i-\bar X)^2, \qquad C_n :=\sqrt{\left(S_n^2-1+\frac{1}{\log n}\right)_+}.
\end{align}
We cap the original neighborhood estimator at $C_n$, replacing $\hat c_{\mathrm{KS}}$ with $\hat c_{\mathrm{KS}}\wedge C_n$.

Next, we adjust the capped estimator $\hat c_{\mathrm{KS}}\wedge C_n$ by choosing a data-dependent reduction on the squared scale. We balance the resulting gain in marginal log-likelihood against the size of this reduction. For this, define
\begin{align}\label{eq:true data density}
    p_{H,\sigma}(x):=\int\phi_\sigma(x-\xi)\,\diff H(\xi), \qquad H \in \Pc(\Real), \qquad \sigma >0.
\end{align}
That is, $p_{H,\sigma}$ is the marginal density of $X_i$ induced by $(H,c)$ under \eqref{eq:hierarchical model} when $\sigma^2 = c^2+1$. For each $c\ge 0$, define
\begin{align}\label{eq:likelihood}
    L_n(c):=\max_{H\in\Pc(\Real)}\sum_{i=1}^n\log p_{H,\sqrt{1+c^2}}(X_i), \qquad b_n:=\left(1+\frac{\max_i|X_i|}{\sqrt{\log n}}\right)(\log n)^2.
\end{align}
Here, $L_n(c)$ is the maximized marginal log-likelihood for a fixed smoothing parameter $c$. 
Then we set
\begin{equation}\label{eq:estimated-c}
    \hat c_0:=\sqrt{(\hat c_{\mathrm{KS}}^2\wedge C_n^2)-\widehat C\rho_n},
\end{equation}
where $\hat c_{\mathrm{KS}}$ is defined in \eqref{eq:estimated-KS-c}, $\rho_n$ is defined in \eqref{eq:rho_n} and the coefficient $\widehat C$ is given by
\begin{equation}\label{eq:estimated-c-buffer}
    \widehat C\in \argmin_{0\leq C\leq(\hat c_{\mathrm{KS}}^2\wedge C_n^2)/\rho_n} \left\{ L_n(0) -L_n\!\left(\sqrt{(\hat c_{\mathrm{KS}}^2\wedge C_n^2)-C\rho_n}\right) +b_nC\right\}.
\end{equation}
Since $L_n(c)$ is nonincreasing in $c$, increasing $C$ can only reduce the likelihood loss in \eqref{eq:estimated-c-buffer}; without the penalty $b_nC$, 
the largest admissible $C$ would always be a minimizer and would yield $\hat c_0=0$. The term $b_nC$ discourages unnecessary reductions in smoothing.  The following theorem establishes the scale accuracy and likelihood guarantee of the modified estimator $\hat c_0$. Its proof is given in Appendix~\ref{app:proof-modified-c-rate}.

\begin{theorem}[Convergence rate of the modified neighborhood estimator]
\label{thm:modified-c-rate}
Under Assumption~\ref{ass:moment-fourier}, let $\hat c_0$ and $\widehat C$ be defined in \eqref{eq:estimated-c} and \eqref{eq:estimated-c-buffer}, respectively. 
There is a constant $K>0$, depending only on $\czero$ and $\boldsymbol{\gamma}$, such that, for all sufficiently large $n$, the event
\begin{equation}\label{eq:modified-c-good-event}
    \Sc_n:=\left\{(\czero^2-K\rho_n)_+\leq\hat c_0^2\leq\czero^2+K\rho_n,\quad \widehat C\leq K,\quad0\leq L_n(0)-L_n(\hat c_0)\leq Kb_n\right\}
\end{equation}
satisfies
\begin{align}\label{eq:one-sided-control}
    \P_{G^*}(\Sc_n) \geq 1-3n^{-2}.
\end{align}
Consequently, $\hat c_0^2-\czero^2=O_P(\rho_n)$ and $\widehat C=O_P(1)$. If $\czero>0$, then $\hat c_0-\czero=O_P(\rho_n)$; if $\czero=0$, then $\hat c_0=O_P(\sqrt{\rho_n})$.
\end{theorem}

Theorem~\ref{thm:modified-c-rate} retains the scale estimation rate of the original neighborhood estimator $\hat c_{\mathrm{KS}}$ while additionally ensuring $L_n(\hat c_0)\geq L_n(0)-Kb_n$ with high probability. This likelihood bound allows the fitted marginal density to be treated as an approximate maximum likelihood estimator in the subsequent analysis.

In Algorithm~\ref{alg:summary}, we use the estimator $\hat{c}_0$ as input to the smooth NPMLE computation. We discuss computation of the neighborhood procedure in Appendix~\ref{sec:neighborhood procedure-computation}.

\section{Statistical properties of the smooth NPMLE}\label{sec:theory}
In this section, we study the statistical properties of the smooth NPMLE with an estimated smoothing parameter $\hat c_0$.
Under the hierarchical model \eqref{eq:hierarchical model}, a direct calculation yields: 
\begin{align}\label{eq:full conditional theta}
    \theta_i \mid X_i, \xi_i \sim N\left(\alpha_0 X_i + (1-\alpha_0)\xi_i,  \alpha_0\right), \qquad \mbox{for }\;\; i = 1, \ldots, n,
\end{align}
where (recall $\sigma_0^2 := \czero^2+1$) 
\begin{align}\label{eq:alphastar}
    \alpha_0 := \frac{\czero^2}{ \sigma_0^2} = \frac{\czero^2}{ \czero^2 +1}.
\end{align}
The quantity $\alpha_0\in[0,1)$ measures the effective smoothing induced by $\czero>0$, and satisfies $\alpha_0 \uparrow 1$ as $\czero\to\infty$. 
If $\trueprior$ were known, the oracle posterior mean of $\theta_i$ would be $\theta_i^*:=\E_{G^*}[\theta_i \mid X_i] = \alpha_0X_i + (1-\alpha_0)\xi_i^*$, where $\xi_i^* := \E_{\trueprior}[\xi_i \mid X_i]$ is the oracle posterior mean of $\xi_i$ given $X_i$, as in \eqref{eq:oracle posterior mean}. Thus, the hierarchical normal-normal structure allows the oracle posterior mean to be expressed as a convex combination of $X_i$ and $\E_{\trueprior}[\xi_i \mid X_i]$. 

In our smooth $g$-modeling approach, we estimate $\trueprior$ by the NPMLE $\npmle$ in \eqref{eq:estimated-prior-distribution} fitted at the estimated smoothing parameter $\hat c_0$ in \eqref{eq:estimated-c}. Consequently, the induced estimate of the prior distribution is $\widehat G_n = \npmle \star N(0,\hat c_0^2)$.
The EB estimate of $\theta_i^*$ is then given by
\begin{align}\label{eq:EB posterior mean}
    \hat{\theta}_i := \E_{\widehat G_n}[\theta_i \mid X_i] = \widehat \alpha_0 X_i + (1-\widehat \alpha_0) \,\E_{\npmle}[\xi_i \mid X_i] = \widehat \alpha_0 X_i + (1-\widehat \alpha_0) \hat\xi_i, 
\end{align}
where $\hat\sigma_0:=\sqrt{1+\hat c_0^2}$, $\widehat \alpha_0 := \hat c_0^2/ \hat \sigma_0^2$, and
\begin{equation}\label{eq:Post-Mean}
    \hat\xi_i := \E_{\npmle}[\xi_i \mid X_i] = \frac{\int \xi \phi_{\hat \sigma_0}(X_i-\xi)\diff \npmle(\xi)}{\widehat f_n(X_i)}
\end{equation}
is the EB estimate of $\xi_i^*= \E_{\trueprior}[\xi_i \mid X_i]$.
Here, $\widehat f_n$ is the estimated marginal density of $X_i$:
\begin{equation}\label{eq:estimated-density}
    \widehat f_n(x) := \int \phi_{\hat \sigma_0}(x-\xi)\diff \npmle(\xi), \qquad x\in \Real.
\end{equation}
This representation makes posterior mean computation simple in the normal-normal setting. Although computing $\npmle$ involves an infinite-dimensional optimization problem, it can be accurately approximated by a finite-dimensional convex program on a fine grid~\citep{Koenker2014, Soloff2025}. 
Since the NPMLE has at most $n$ atoms \citep{Lindsay1995} (and $O(\log n)$ atoms with high probability when $H^*$ has sub-Gaussian tails \citep{polyanskiy2020selfregularizingpropertynonparametricmaximum}), the integral in~\eqref{eq:Post-Mean} reduces to a finite sum, making the EB posterior means easy to compute. In general, the smooth NPMLE and its corresponding posterior density are finite Gaussian mixtures, making posterior computations tractable.

\subsection{Denoising}\label{sec:denoising}
One of the main uses of the estimated prior $\widehat G_n$ in \eqref{eq:estimated-prior-distribution} is in denoising: constructing accurate estimates of the posterior means. 
We quantify the performance of our EB approach via the (excess) mean-squared error relative to the oracle Bayes rule. Since $\theta_i^*=\E_{G^*}[\theta_i\mid X_i]$ is the $L_2$-projection of $\theta_i$ onto $\sigma(X_i)$, the Pythagorean/orthogonality identity yields
\begin{equation}\label{eq:regret}
    \Rc_n(\hat\theta,\theta^*) 
    :=\E\!\left[\frac1n\sum_{i=1}^n(\theta_i-\hat\theta_i)^2\right]
    -\E\!\left[\frac1n\sum_{i=1}^n(\theta_i-\theta_i^*)^2\right]
    =\E\!\left[\frac1n\sum_{i=1}^n(\hat\theta_i-\theta_i^*)^2\right].
\end{equation}
To build intuition for why the smooth NPMLE retains nearly parametric denoising performance, suppose for now that $\czero$ is known and set $\hat c_0=\czero$. The posterior mean formulas \eqref{eq:oracle posterior mean} and \eqref{eq:EB posterior mean} give
\begin{align}\label{eq:regret-xi}
    \Rc_n(\hat\theta,\theta^*)
    =\E\!\left[\frac{(1-\alpha_0)^2}{n}\sum_{i=1}^n(\hat\xi_i-\xi_i^*)^2\right].
\end{align}
Equation \eqref{eq:regret-xi} is particularly revealing: up to the factor $(1-\alpha_0)^2=(1+\czero^2)^{-2}$, the excess Bayes risk for estimating $\theta_i$ is exactly the excess Bayes risk for estimating $\xi_i$ in the induced Gaussian mixture model \eqref{eq:induced-model}. 
The denoising problem for $\{\xi_i\}_{i=1}^{n}$ using $\{\hat{\xi}_i \}_{i=1}^{n}$ under \eqref{eq:induced-model} is well-understood in the literature \citep{Jiang2009, Saha2020, Soloff2025}. 
Thus, the additional Gaussian hierarchical layer preserves nearly parametric denoising rates under suitable conditions on $\trueprior$.

We now return to the estimated smoothing parameter $\hat c_0$ in \eqref{eq:estimated-c}. The following theorem shows that estimating $\czero$ does not worsen the nearly parametric rates of the smooth NPMLE for either marginal density estimation or denoising. For densities $f$ and $g$, write $\H^2(f,g):=\frac12\int(\sqrt f-\sqrt g)^2\,\diff x$ for the squared Hellinger distance between $f$ and $g$. We provide the proof of Theorem~\ref{thm:estimated-c-density-denoising} in Appendix~\ref{app:proof-estimated-c-density-denoising}.

\begin{theorem}[Near-parametric rates for marginal density estimation and denoising]\label{thm:estimated-c-density-denoising}
Suppose that Assumption~\ref{ass:moment-fourier} holds. Let $\widehat f_n$ be as in \eqref{eq:estimated-density}, and write $f^*$ for the true marginal density of $X_i$. Then, for all sufficiently large $n$,
\begin{align}\label{eq:estimated-c-density-rate}
    \E_{G^*}[\H^2(\widehat f_n, f^*)]
    &\lesssim_{\czero,\boldsymbol{\gamma}}\frac{(\log n)^{3/2+1/r}}{n}.
\end{align}
Moreover, let $\theta_i^*$, $\hat\theta_i$, and $\Rc_n(\hat\theta,\theta^*)$ be as in \eqref{eq:oracle posterior mean}, \eqref{eq:EB posterior mean}, and \eqref{eq:regret}, respectively. Then
\begin{align}\label{eq:estimated-c-denoising-rate}
    \Rc_n(\hat\theta,\theta^*)
    &\lesssim_{\czero,\boldsymbol{\gamma}}\frac{(\log n)^{9/2+1/r}}{n}.
\end{align}
\end{theorem}

\begin{remark}[Marginal density estimation under misspecification]
Theorem~\ref{thm:estimated-c-density-denoising} also covers $\czero=0$, in which case the model reduces to the classical normal location mixture model~\eqref{eq:original model}. This includes both discrete priors and smooth priors outside the Gaussian location mixture class, provided they satisfy Assumption~\ref{ass:moment-fourier}. Under the moment condition~\eqref{eq:simultaneous-moment-control}, the classical NPMLE also achieves the same nearly parametric rates for marginal density estimation and denoising \citep{Saha2020,Soloff2025}. Figure~\ref{fig:misspecification-marginal-comparison} in Appendix~\ref{sec:numerical results} illustrates the smooth NPMLE's comparable empirical performance in marginal density estimation under a misspecified Laplace prior.
\end{remark}

Theorem~\ref{thm:estimated-c-density-denoising} is an important benchmark for our hierarchical model: the smoothing used to define the smooth NPMLE does not sacrifice classical EB denoising performance, even though the estimated smoothing parameter $\hat c_0$ converges at a logarithmic rate.
It preserves near-oracle denoising guarantees while laying the foundation for the sharper deconvolution and uncertainty quantification results developed in the following sections. Note that a similar separation between the rates for estimating the Gaussian scale and the marginal density appears in Proposition 2 of \citet{bariletto2026convergencerateslatentmixing}.

\subsection{Deconvolution}\label{sec:deconvolution}
Next, suppose that $\czero > 0$ and consider estimation of the true prior density of $\theta_i \sim G^*$, given by
\begin{equation}\label{eq:true prior density}
    g^*(\theta) \;:=\; \int \phi_{\czero}(\theta-\xi)\, \diff\trueprior(\xi), \qquad \theta\in\Real.
\end{equation}
It is natural to ask how well the smooth NPMLE estimates $g^*$. As mentioned earlier, in estimating $f^*$, which is the true marginal density of $X_i$, the plug-in estimator $\widehat f_n$ achieves almost parametric rates under suitable assumptions on $\trueprior$. In contrast, the deconvolution error between $\npmle$ and $\trueprior$ is upper bounded by the slow logarithmic rate in 2-Wasserstein distance (see Theorem 10 of \citet{Soloff2025}). In fact, the logarithmic rate is minimax optimal over Sobolev classes in the classical deconvolution setting, where no additional hierarchical structure is available \citep{Fan1991, Meister2009}. 

Our setting differs fundamentally: under the hierarchical model~\eqref{eq:hierarchical model} with $\czero>0$, an additional smoothing effect links the estimation of $g^*$, the density of $\trueprior \star N(0,\czero^2)$, to the near-parametric estimation of $f^*$, the density of $\trueprior \star N(0,\czero^2+1)$. We will show that this allows polynomial rates of recovery in $L_2$ distance. 
Before this, note that the estimated smoothing parameter $\hat c_0$ in \eqref{eq:estimated-c} can be zero, in which case the fitted prior $\widehat G_n = \npmle \star N(0, \hat c_0^2)$ is discrete. Thus, we impose a vanishing floor $\sqrt{\rho_n} \downarrow 0$ on $\hat c_0$ in \eqref{eq:estimated-c}, where $\rho_n$ defined in \eqref{eq:rho_n} is the convergence rate of $\hat c_0^2$. For $n\geq3$, let
\begin{equation}\label{eq:deconvolution-floor}
    \hat c_0^+:=\hat c_0\vee\sqrt{\rho_n}.
\end{equation}
For deconvolution and posterior density estimation, we estimate $\trueprior$ by $\npmle^+$, which is the NPMLE fitted at the stabilized scale $\hat c_0^+$. Specifically, for each $c\geq0$, let
\begin{equation}\label{eq:NPMLE}
    \npmle^{\,c}\in\argmax_{H\in\Pc(\Real)}\sum_{i=1}^n\log p_{H,\sqrt{1+c^2}}(X_i), \qquad \npmle^+:=\npmle^{\,\hat c_0^+}.
\end{equation}
By \eqref{eq:estimated-prior-distribution} and \eqref{eq:NPMLE}, $\npmle^+=\npmle$ when $\hat c_0\geq\sqrt{\rho_n}$.
Otherwise, $\npmle^+$ is the NPMLE fitted at the deterministic scale $\sqrt{\rho_n}$. Then $(\npmle^+, \hat c_0^+)$ induces a plug-in estimator of the smooth prior density $g^*$ in \eqref{eq:true prior density}:
\begin{equation}\label{eq:estimated prior density}
    \widehat g_n^+(\theta) := \int \phi_{\hat c_0^+}(\theta-\xi)\, \diff\npmle^+(\xi), \qquad \theta\in\Real.
\end{equation}

\begin{remark}[$\hat c_0$ versus $\hat c_0^+$]
For fixed $\czero>0$, the stabilized estimator $\hat c_0^+ = \hat c_0 \vee \sqrt{\rho_n}$ coincides with $\hat c_0$ in \eqref{eq:estimated-c} with high probability by Theorem~\ref{thm:modified-c-rate}. 
However, when $\czero=0$, fitting at $\hat c_0^+$ can yield marginal density estimation and denoising errors of order at least $\rho_n^2$, a logarithmic rate that is much slower than the nearly parametric rates obtained using $\hat c_0$; see Proposition~\ref{prop:floor-density-denoising}.
Therefore, we retain $\hat c_0$ for marginal density estimation and denoising. 
\end{remark}

The following theorem shows that, when $\czero > 0$, the smooth NPMLE based on this stabilized $\hat c_0^+$ achieves a polynomial deconvolution rate. Recall that $\alpha_0$ in \eqref{eq:alphastar} is the ratio of the largest Gaussian component variance $\czero^2$ in the true prior $G^*$ to its counterpart $\sigma_0^2 = \czero^2+1$ in the true marginal distribution $F^*$.

\begin{theorem}[Polynomial deconvolution rate]\label{thm:moment-deconvolution}
Suppose that $\czero>0$ and Assumption~\ref{ass:moment-fourier} holds. Let $\hat c_0^+$ be defined in \eqref{eq:deconvolution-floor}, and let $\widehat g_n^+$ be defined in \eqref{eq:estimated prior density}. 
There exists a constant $k_1>0$, depending only on $\czero$ and $\gammav$, such that, for all sufficiently large $n$,
\begin{align}\label{eq:moment-deconvolution-upper}
    \E_{G^*}\|\widehat g_n^+-g^*\|_{L_2}^2 &\lesssim_{\czero,\gammav} n^{-\alpha_0}\exp\{-k_1(\log n)^{r/2}\}.
\end{align}
When $r=2$, the constant $k_1$ satisfies $\alpha_0+k_1<1$.
\end{theorem}

When $g^*$ is sufficiently smooth and $\czero>0$ is large, the gain from smoothing becomes substantial. In the extreme case where $\czero$ is large and $\alpha_0\approx 1$, $\widehat g_n^+$ achieves an almost parametric rate of convergence. Hence, the rate interpolates between the parametric regime ($\czero^2\to\infty$) and the severely ill-posed regime ($\czero^2\to0$). This aligns with our observation that the smooth NPMLE $\widehat g_n^+$ approximates $g^*$ well in the center panel of Figure~\ref{fig:SNPML-TwoComp}. To the best of our knowledge, Theorem~\ref{thm:moment-deconvolution} is the first such result that provides a polynomial $L_2$ convergence rate for the smooth NPMLE in the hierarchical Gaussian smoothing framework. 

We provide the proof of Theorem~\ref{thm:moment-deconvolution} in Appendix~\ref{app:proof-moment-deconvolution}. The key idea of the proof is to use Plancherel's theorem to relate $\|\widehat g_n^+-g^*\|_{L_2}^2$ to $\|\widehat f_n^+-f^*\|_{L_2}^2$ (where $\widehat f_n^+:=\widehat g_n^+\star N(0,1)$ denotes
the corresponding estimated marginal density) and obtain the leading $n^{-\alpha_0}$ rate. A careful analysis of the Fourier frequencies and Gaussian tail bounds yields the sharper factor $\exp\{-k_1(\log n)^{r/2}\}$.

We next examine the optimality of this polynomial deconvolution rate. It has been known that the minimax rate of deconvolution with Gaussian errors is logarithmic over the Sobolev class \citep{Fan1991, Meister2009}. However, note that the true prior density $g^* = \trueprior \star N(0, \czero^2)$ in \eqref{eq:true prior density} is supersmooth as its Fourier transform decays exponentially as long as $\czero > 0$. It implies that $g^*$ is contained in a more restricted class compared to the Sobolev class. The following result (Theorem~\ref{thm:moment-deconvolution-lower} proved in Appendix~\ref{app:proof-moment-deconvolution-lower}) shows that, under Assumption~\ref{ass:moment-fourier} (i) with $0<r<2$, the leading polynomial factor $n^{-\alpha_0}$ in Theorem~\ref{thm:moment-deconvolution} is sharp. Moreover, the minimax lower bound matches the form of \eqref{eq:moment-deconvolution-upper} up to constants in its exponential factor.

\begin{theorem}[Minimax lower bound for deconvolution]\label{thm:moment-deconvolution-lower}
Fix $\czero>0$ and $0<r<2$. Let $\mathcal H_r$ be the class of distributions $H\in\Pc(\Real)$ satisfying Assumption~\ref{ass:moment-fourier} (i) with $H$ in place of $\trueprior$ and a common parameter tuple $\boldsymbol{\gamma}$ across the class. For $H\in\mathcal H_r$, let $g_H=H\star N(0,\czero^2)$ and $f_H=g_H\star N(0,1)$, and define
\begin{align*}
    \mathfrak R_{n,r}:=\inf_{\widetilde g_n}\sup_{H\in\mathcal H_r}\E_H\|\widetilde g_n-g_H\|_{L_2}^2,
\end{align*}
where the infimum is over all estimators based on $n$ independent observations from $f_H$. There exists a constant $k_2>0$, depending only on $\czero$ and $\gammav$, such that, for all sufficiently large $n$,
\begin{align}\label{eq:moment-deconvolution-lower}
    \mathfrak R_{n,r}\gtrsim_{\czero,\gammav} n^{-\alpha_0}\exp\{-k_2(\log n)^{r/2}\}.
\end{align}
\end{theorem}

In Appendix~\ref{app:deconvolution rate illustration}, we further discuss polynomial deconvolution rates and present simulations illustrating the effect of Gaussian smoothing on the empirical squared $L_2$ risk (Figure~\ref{fig:L2-distance}).

\subsection{Posterior density estimation}\label{sec:estimation of posterior density}
Our result on polynomial convergence in the deconvolution problem has an important implication for posterior density estimation. While the posterior distribution is central to achieving the goals of EB methods, much of the literature establishes theoretical guarantees for only a few functionals such as the posterior mean; it falls short of providing theoretical guarantees for the posterior distribution itself. Denote by
\begin{align}\label{eq:true posterior density}
    \pi^*(\theta\mid x) := \frac{\phi(x - \theta)g^*(\theta)}{f^*(x)}, \qquad \mbox{for }\;\; \theta,x\in\Real,
\end{align}
the posterior density of $\theta_i$ given $X_i=x$ where $g^*$ is defined in \eqref{eq:true prior density} and $f^* = g^* \star N(0,1)$. With the plug-in estimators $\widehat f_n^+$ and $\widehat g_n^+$ based on $\npmle^+$ fitted at the stabilized scale $\hat c_0^+$ in \eqref{eq:deconvolution-floor}, we have a natural plug-in estimator of $\pi^*(\cdot\mid x)$:
\begin{equation}\label{eq:estimated posterior density}
    \widehat\pi_{n}^+(\theta \mid x) := \frac{\phi(x - \theta)\,\widehat g_n^+(\theta)}{\widehat f_n^+(x)}, \qquad \mbox{for }\;\;\theta,x\in\Real.
\end{equation}
Note that $\widehat \pi_{n}^+(\cdot \mid x)$ is smooth with $\hat c_0^+ >0$ in \eqref{eq:deconvolution-floor}. Our metric for quantifying the distance between $\widehat \pi_{n}^+$ and $\pi^*$ is the weighted total variation distance:
\begin{align}\label{eq:weighted total variation distance}
    \mathrm{wTV}(\widehat \pi_n^+, \pi^*) := \int \mathrm{TV}\{\widehat\pi_{n}^+(\cdot \mid x), \pi^*(\cdot \mid x)\} f^*(x) \diff x.
\end{align}
Here, $\mathrm{TV}(f, g) = \frac{1}{2}\int |f(t) - g(t)|\diff t$ denotes the usual total variation distance between densities $f$ and $g$. It can be shown that 
$$ \mathrm{wTV}(\widehat \pi_{n}^+, \pi^*) \le \mathrm{TV}(\widehat f_n^+, f^*) + \mathrm{TV}(\widehat g_n^+, g^*)$$
(see the proof of Theorem~\ref{thm:estimated-c-posterior}). Thus, the quality of the posterior approximation depends on how well the estimated marginal density $\widehat f_n^+$ and the estimated prior density $\widehat g_n^+$ approximate $f^*$ and $g^*$, respectively. In the following theorem, we provide the rate of convergence of $\widehat\pi_{n}^+(\cdot \mid \cdot)$ to $\pi^*(\cdot \mid \cdot)$ in the weighted total variation distance \eqref{eq:weighted total variation distance}. We provide the proof of Theorem~\ref{thm:estimated-c-posterior} in Appendix~\ref{app:proof-estimated-c-posterior}.

\begin{theorem}[Polynomial rate for posterior density estimation]\label{thm:estimated-c-posterior}
Suppose that $\czero>0$ and Assumption~\ref{ass:moment-fourier} holds. Let $\hat c_0^+$ be defined in \eqref{eq:deconvolution-floor}, and let $\widehat\pi_n^+$ be defined in \eqref{eq:estimated posterior density}. 
There is a constant $k_3>0$, depending only on $\czero$ and $\gammav$, such that, for all sufficiently large $n$,
\begin{align}\label{eq:moment-posterior-rate}
    \E_{G^*}[\mathrm{wTV}(\widehat\pi_n^+,\pi^*)]\lesssim_{\czero,\gammav} n^{-\alpha_0/2}\exp\{-k_3(\log n)^{r/2}\}.
\end{align}
\end{theorem}
To the best of our knowledge, this is the first result establishing a polynomial convergence rate for posterior density estimation using the smooth NPMLE. This is consistent with the observation that $\widehat \pi_{n}^+(\cdot \mid x)$ approximates $\pi^*(\cdot \mid x)$ well in the rightmost panel of Figure~\ref{fig:SNPML-TwoComp}.

We next establish a minimax lower bound for posterior density estimation. The following theorem shows that, for $0<r<2$, the leading polynomial exponent $\alpha_0/2$ in Theorem~\ref{thm:estimated-c-posterior} is minimax optimal over the class of mixing distributions satisfying Assumption~\ref{ass:moment-fourier} $(i)$. To the best of our knowledge, this minimax lower bound for posterior density estimation is new. We provide the proof of Theorem~\ref{thm:moment-posterior-lower} in Appendix~\ref{app:proof-moment-posterior-lower}.

\begin{theorem}[Minimax lower bound for posterior density estimation]\label{thm:moment-posterior-lower}
Fix $\czero>0$ and $0<r<2$. Let $\mathcal H_r$, $g_H$, and $f_H$ be as in Theorem~\ref{thm:moment-deconvolution-lower}, and define $\pi_H(\theta\mid x):=\phi(x-\theta)g_H(\theta)/f_H(x)$. Let
\begin{align*}
    \mathfrak R^\pi_{n,r} :=\inf_{\widetilde\pi_n}\sup_{H\in\mathcal H_r}\E_H\left[\mathrm{wTV}(\widetilde\pi_n,\pi_H)\right],
\end{align*}
where the infimum is over all posterior density estimators based on $n$ independent observations from $f_H$ (i.e., measurable functions $\widetilde\pi_n(\theta\mid x)\geq0$ satisfying $\int_{\Real}\widetilde\pi_n(\theta\mid x)\,\diff\theta=1$ for every $x\in\Real$). There exists a constant $k_4>0$, depending only on $\czero$ and $\gammav$, such that, for all sufficiently large $n$,
\begin{align}\label{eq:moment-posterior-lower}
    \mathfrak R^\pi_{n,r} \gtrsim_{\czero,\gammav} n^{-\alpha_0/2}\exp\{-k_4(\log n)^{r/2}\}.
\end{align}
\end{theorem}

\subsection{Model misspecification}\label{sec:model misspecification}
In Sections~\ref{sec:deconvolution} and \ref{sec:estimation of posterior density}, we have assumed that $\czero > 0$ in the hierarchical model \eqref{eq:hierarchical model}. We now turn our attention to the case where $\czero=0$, which encompasses any distribution on $\Real$ beyond Gaussian location mixtures.
Recall that the nearly parametric guarantees for marginal density estimation and denoising based on $\hat c_0$ in Theorem~\ref{thm:estimated-c-density-denoising} remain valid even when $\czero=0$, provided that $G^* \equiv H^*$ satisfies Assumption~\ref{ass:moment-fourier}. Thus, in this subsection we focus on studying deconvolution and posterior density estimation. Note that our smooth NPMLE $\widehat g_n^+$ is misspecified in this case because we use the smoothing parameter $\hat c_0^+\geq\sqrt{\rho_n}>0$ in \eqref{eq:deconvolution-floor}, which excludes $G^*$ from the fitted Gaussian location mixture class. However, under suitable regularity conditions, we will show that the estimated prior and posterior densities converge to the true prior and posterior densities, respectively. 

For this, we impose an additional regularity condition on the true prior $G^*$. Specifically, suppose that $G^*$ admits a strictly positive density $g^*$ on $\Real$ and that $\log g^*$ is locally absolutely continuous. Also, let $s_{g^*}:\Real\to\Real$ denote its log-density slope, with $s_{g^*}(x)=(\log g^*)'(x)$ for almost every $x$. We assume that,
for some $L>0$,
\begin{align}\label{eq:log-density-slope-variation}
    \sum_{j=1}^m |s_{g^*}(x_j)-s_{g^*}(x_{j-1})| \leq L, \qquad \text{for every }m\geq1 \text{ and }x_0<\cdots<x_m\text{ in }\Real.
\end{align}
When $\log g^*$ is twice continuously differentiable, \eqref{eq:log-density-slope-variation} is equivalent to $\int_{\Real}|(\log g^*)''(x)|\,\diff x\leq L$.  Condition~\eqref{eq:log-density-slope-variation} controls the roughness of the true prior density by bounding the total variation of its log-density slope, which has been used for regularization in density estimation; see \citet{KoenkerMizera2007TV,KoenkerMizera2018Renyi}. Under these regularity conditions, the following theorem establishes convergence of the density $\widehat g_n^+$ of the smooth NPMLE and the corresponding estimated posterior density $\widehat\pi_n^+$ to the true prior and posterior densities, respectively. We provide the proof of Theorem~\ref{thm:misspecification-true} in Appendix~\ref{app:proof-misspecification-true}.

\begin{theorem}[Convergence to the true prior and posterior densities under misspecification]\label{thm:misspecification-true}
Suppose that $\czero=0$ and Assumption~\ref{ass:moment-fourier} (i) holds with $r=1$ and $\vartheta\geq2$. Suppose also that $G^*$ admits a strictly positive density $g^*$ on $\Real$ such that $\log g^*$ is locally absolutely continuous and \eqref{eq:log-density-slope-variation} holds.
Let $\widehat g_n^+$ and $\widehat\pi_n^+$ be defined in \eqref{eq:estimated prior density} and \eqref{eq:estimated posterior density}, respectively. 
There exists $n_0$, depending only on $\gammav$ and $L$, such that, for $n\geq n_0$,
\begin{align}\label{eq:misspecification-true-prior-rate}
    \E_{G^*}\bigl[\|\widehat g_n^+-g^*\|_{L_2}^2\bigr] \lesssim_{\gammav,L}\frac{\log\log n}{\log n}.
\end{align}
Moreover,
\begin{align}\label{eq:misspecification-true-posterior-rate}
     \E_{G^*}\bigl[\mathrm{wTV}(\widehat\pi_n^+,\pi^*)\bigr] \lesssim_{\gammav,L}\frac{\log\log n}{\sqrt{\log n}}.
\end{align}
\end{theorem}

Here, Assumption~\ref{ass:moment-fourier} $(i)$ applies directly to the true prior $G^*$ because $\trueprior=G^*$ when $\czero=0$ in \eqref{eq:hierarchical model}.
The choice $r=1$ corresponds to subexponential tail control for $G^*$. The assumptions of Theorem~\ref{thm:misspecification-true} accommodate a broad class of prior densities, including both log-concave and multimodal densities. Writing $g^*(\theta)\propto\exp\{-V(\theta)\}$, the regularity condition bounds the total variation of $V'$ without requiring $V$ to be convex. A convex $V$ gives a log-concave density, while a nonconvex $V$ can give a multimodal density. In either case, the moment and Fourier conditions must also hold. The Laplace prior in Figure~\ref{fig:SNPML-laplace} corresponds to $V(\theta)=|\theta|$ and satisfies the assumptions of Theorem~\ref{thm:misspecification-true}. This explains why the smooth NPMLE approximates the true prior reasonably well in Figure~\ref{fig:SNPML-laplace}.

\begin{remark}[Deterministic smoothing]\label{rem:misspecification-fixed-floor}
The Fourier lower bound \eqref{eq:fourier-lower-bound} in Assumption~\ref{ass:moment-fourier} (i) is used to establish the convergence rate of the estimated smoothing parameter $\hat c_0^+$.
If $\czero=0$ and the NPMLE is instead fitted at the deterministic scale $\sqrt{\rho_n}$, the same prior and posterior convergence rates of Theorem~\ref{thm:misspecification-true} hold under the moment condition \eqref{eq:simultaneous-moment-control} on $G^*$ with $r=1$ and the stated density regularity conditions, without requiring the Fourier lower bound; see Remark~\ref{rem:misspecification-fixed-floor-justification} for justification.
\end{remark}

Therefore, with the estimated smoothing parameter $\hat c_0^+$, the smooth NPMLE adaptively achieves polynomial convergence rates for prior and posterior density estimation when $\czero>0$, while consistently estimating the true prior and posterior densities for a broad class of priors even when $\czero=0$ under suitable regularity conditions.

\section{Optimal marginal coverage sets}\label{sec:Opt-conf-sets}
In this section, we first construct the optimal set under the marginal coverage constraint \eqref{eq:marginal coverage}. It is well-known that the highest posterior density (HPD) set is the optimal credible set in terms of length (see Appendix~\ref{sec:HPD} for a review). We will show that the optimal marginal coverage set can also be characterized using a posterior density, similar to the HPD set, but with a different threshold (Theorem~\ref{thm:Opt-Marginal}). We then discuss how to estimate the optimal marginal coverage set using the smooth NPMLE for the hierarchical model \eqref{eq:hierarchical model}. We also provide theoretical guarantees for the estimated set in terms of coverage probability (Theorem~\ref{thm:estimated-c-coverage}) and excess length (Theorem~\ref{thm:estimated-c-length}).

\subsection{Optimality among marginal coverage sets}\label{sec:OPT construction}
 In this subsection we study a general mixture model as the conclusions hold more generally. Let $X\in\mathcal X$ be an observable with conditional density $p(\cdot\mid\theta)$, and let the prior distribution of $\theta\in\Theta\subset\mathbb R$ be $G$ with Lebesgue density $g(\cdot)$. Denote by
\[
 \pi(\cdot \mid x)\;:=\frac{p(x\mid\cdot)\,g(\cdot)}{p_{G}(x)}
\]
the posterior density of $\theta$ given $X=x$ where $p_{G}(x) = \int p(x \mid \theta)\,g(\theta)\diff \theta$ is the marginal density of $X$. For a set-valued rule $\Ic\!:x\mapsto \Ic(x)\subset\Theta$ let $|\Ic(x)|$ be its Lebesgue length, i.e., $|\Ic(x)| = \int \mathbf{1}_{\Ic(x)}(t) \diff t$. 
We wish to solve the following optimization problem (for $\beta \in (0,1)$):
\begin{equation}
\label{eq:Opt-Marg-Set}
 \min_{\Ic(\cdot)} \;
 \mathbb E_G\!\bigl[|\Ic(X)|\bigr]
 \qquad \;\;\text{ subject to }\;\;
 \mathbb P_G\bigl(\theta\in \Ic(X)\bigr)\;\ge\;1-\beta,
\end{equation}
where $\E_{G}$ and $\P_{G}$ denote the expectation and probability under $\theta \sim G$ and $X \mid \theta \sim p(\cdot \mid \theta)$. 
For any set-valued rule $\Ic$, define $A := \{(x,\theta) \in \mathcal{X} \times \Theta: \theta \in \Ic(x)\}.$ Then
\[
 \mathbb E_G\bigl[|\Ic(X)|\bigr] = 
\iint_{A} p_G(x)\,\diff\theta\,\diff x, \qquad \mathbb P_G\bigl(\theta\in \Ic(X)\bigr) = \iint_{A} p(x,\theta)\,\diff \theta\,\diff x,
\]
where $p(x,\theta)$ is the joint density of $(X,\theta)$. The following theorem (proved in Appendix~\ref{app:proof of thm:Opt-Marginal}) characterizes the optimizer of \eqref{eq:Opt-Marg-Set} by thresholding the posterior density $\pi(\theta\mid x)=p(x,\theta)/p_G(x)$ at a \emph{constant} level $k^*$ chosen to satisfy the marginal coverage constraint. This characterization follows from a generalized Neyman--Pearson argument. 
\begin{theorem}[Oracle characterization of optimal marginal coverage sets]\label{thm:Opt-Marginal}
Suppose that $G$ has a density $g$ with respect to the Lebesgue measure on $\Theta \subset \mathbb{R}$. For $0 < \beta < 1$, define
\begin{align}\label{eq:Marg-Const}
    k^* := \sup \Big\{k \ge 0: \P_{G}\big(\pi(\theta \mid X) \ge k\big) \ge 1-\beta \Big\}.
\end{align}
Then, $
    \P_{G}(\pi(\theta \mid X )> k^*) \le 1-\beta \le \P_{G}(\pi(\theta \mid X) \ge k^*).$
Moreover, there exists a measurable set $A^* \subseteq \{(x,\theta)\in \Xc\times \Theta: \pi(\theta \mid x) = k^* \}$ such that the measurable set-valued rule
\begin{equation}\label{eq:Opt-Marg-Cov}
    \Ic^{*}(x) :=\{\theta \in \Theta: \pi(\theta \mid x) > k^*\} \cup \{\theta \in \Theta: (x,\theta) \in A^* \}
\end{equation}
solves~\eqref{eq:Opt-Marg-Set} and $\P_{G}(\theta \in \Ic^*(X)) = 1-\beta$. 
\end{theorem}

\begin{remark}
If $\P_{G}(\pi(\theta \mid X) = k^*) = 0$, then we may take $A^* = \emptyset$ in \eqref{eq:Opt-Marg-Cov}, and the optimal marginal coverage set becomes
\begin{align}\label{eq:Opt-Marg-Cov-simple}
        \Ic^*(x) = \{\theta \in \Theta: \pi(\theta \mid x) \ge k^* \},
    \end{align}
where $k^*$ is obtained from the following equation: $\P_{G}(\pi(\theta \mid X) \ge k^*) = 1-\beta$. This holds for the posterior density \eqref{eq:true posterior density} under \eqref{eq:hierarchical model} with $\czero>0$, since $\pi^*(\cdot\mid x)$ is nonconstant and analytic for every $x$. 
\end{remark}

Let us compare the HPD set (see \eqref{eq:HPD} in Appendix~\ref{sec:HPD}) and the optimal marginal coverage set in \eqref{eq:Opt-Marg-Cov}. Under the stronger \emph{conditional} coverage requirement $\mathbb P_G \bigl(\theta\in \Ic(X)\mid X=x\bigr)\ge 1-\beta$, the optimal set is the HPD set with {\it $x$-dependent threshold} given in~\eqref{eq:HPD} (when there is no tie at the threshold). Relaxing to the unconditional constraint, as in Theorem~\ref{thm:Opt-Marginal}, allows a {\it single} global threshold (see~\eqref{eq:Opt-Marg-Cov}) which can shorten the expected length. For an illustration comparing the HPD set and the optimal marginal coverage set see Appendix~\ref{sec:comparison}.

\subsection{Estimation of optimal marginal coverage sets}\label{sec:OPT estimation}
We now focus on estimating the optimal marginal coverage set $\Ic^*$ under model~\eqref{eq:hierarchical model} with {\it Gaussian likelihood}. Estimating $\Ic^*$ via the classical NPMLE would be ill-posed, since the resulting posterior distribution would be discrete (as the classical NPMLE is itself discrete). Our smooth NPMLE, by contrast, yields a smooth posterior density, making estimation of $\Ic^*$ tractable. Specifically, as the likelihood is Gaussian, $\Ic^*$ admits the characterization~\eqref{eq:Opt-Marg-Cov-simple} in terms of the posterior density $\pi^*(\cdot \mid x)$, which we estimate by its EB analogue:
\begin{align}\label{eq:estimated optimal marginal sets}
    \hat{\Ic}_n(x) := \bigl\{\theta \in \Real : \widehat\pi_{n}^+(\theta \mid x) \ge \hat{k}_n\bigr\},
\end{align}
where $\widehat \pi_{n}^+(\cdot \mid \cdot)$ is the estimated posterior density~\eqref{eq:estimated posterior density} based on $(\npmle^+, \hat c_0^+)$. For $u\geq0$, define the fitted coverage curve and choose $\hat{k}_n$ by
\begin{equation}\label{eq:Est-CI}
    \widehat C_n(u):=\iint \1v({\widehat\pi_n^+(\theta\mid x)\geq u})\, \widehat g_n^+(\theta)\phi(x-\theta)\,\diff\theta\,\diff x,\qquad \widehat C_n(\hat k_n)=1-\beta.
\end{equation}
A solution $\hat{k}_n$ exists because $\hat c_0^+>0$ by \eqref{eq:deconvolution-floor}, making the fitted posterior a nonconstant real-analytic density in $\theta$ whose positive level sets have Lebesgue measure zero. Thus, the left-hand side of~\eqref{eq:Est-CI} is continuous in $u>0$, with limits $1$ as $u\downarrow0$ and $0$ as $u\to\infty$. In~\eqref{eq:Est-CI} we use the plug-in distribution $\npmle^+$ with $\hat c_0^+$ twice---first to approximate the true posterior density $\pi^*(\cdot \mid x)$ by $\widehat \pi_n^+(\cdot \mid x)$, and then again to approximate the true coverage of the obtained confidence set using the estimated joint density of $(X,\theta)$, i.e., $\widehat g_n^+(\theta)\phi(x-\theta)$, instead of $g^*(\theta)\phi(x-\theta)$. Thus, \eqref{eq:Est-CI} is a natural plug-in analogue of $k^*$ for \eqref{eq:Opt-Marg-Cov-simple}. Even if the NPMLE $\npmle^+$ is discrete, we can use the smooth estimated posterior density $\widehat \pi_{n}^+(\cdot \mid x)$, which approximates the true posterior $\pi^*(\cdot \mid x)$ well under the hierarchical model \eqref{eq:hierarchical model}.

A natural question is how well the coverage probability of $\hat\Ic_n$ in~\eqref{eq:estimated optimal marginal sets} approximates the nominal confidence level $1-\beta$.
We provide the proof of Theorem~\ref{thm:estimated-c-coverage} in Appendix~\ref{app:proof-estimated-c-coverage}, where we reduce the problem to bounding the expected total variation distance between the estimated prior density $\widehat g_n^+$ and the true prior density $g^*$.

\begin{theorem}[Marginal coverage of the estimated set]\label{thm:estimated-c-coverage}
Suppose that the conditions of Theorem~\ref{thm:estimated-c-posterior} hold. Let $\Dc_n=\sigma(X_1,\ldots,X_n)$ and let $(X,\theta)$ be an independent copy of $(X_i,\theta_i)$, independent of $\Dc_n$. 
For every fixed $\beta\in(0,1)$, there exists a constant $k>0$, depending only on $\czero$ and $\gammav$, such that, for all sufficiently large $n$,
\begin{align}\label{eq:estimated-c-coverage-rate}
    \E_{G^*}\left|\P_{G^*}\{\theta\in\hat\Ic_n(X)\mid\Dc_n\}-(1-\beta)\right|\lesssim_{\czero,\gammav} n^{-\alpha_0/2}\exp\{-k(\log n)^{r/2}\}.
\end{align}
Here, the conditional coverage is evaluated at the independent pair $(X,\theta)$.
\end{theorem}

Next, we show that the expected length of $\hat{\Ic}_n$ in \eqref{eq:estimated optimal marginal sets} converges to that of $\Ic^*$ in \eqref{eq:Opt-Marg-Cov-simple}. That is, $\hat{\Ic}_n$ asymptotically achieves the shortest expected length among all marginal coverage sets. For some small $\delta_0 \in (0,k^*/2)$, let $\Kc = [k^* - \delta_0, k^* + \delta_0]$ where $k^*$ is the threshold for the optimal marginal coverage set $\Ic^*$ in \eqref{eq:Opt-Marg-Cov-simple}. For $(X, \theta)$ drawn from the true joint distribution (i.e., under model \eqref{eq:hierarchical model}), we assume that there exist positive constants $C_1, C_2 > 0$ and $t_1\in[\delta_0,k^*-\delta_0)$ satisfying:
\begin{equation}
    \begin{aligned}\label{eq:excess length condition}
        \mathrm{(C1)} & \,\,  \P(|\pi^*(\theta\mid X) - u| \le t) \le C_1 t, \quad \forall u \in \Kc,\quad \forall t \in (0,t_1];\\
        \mathrm{(C2)} &\,\, |C^*(u) - C^*(k^*)| \ge C_2|u-k^*|, \;\; \forall u \in \Kc, \quad \;\; C^*(u):= \P_{G^*}(\pi^*(\theta \mid X) \ge u).
    \end{aligned}    
\end{equation}
Note that the above conditions are posterior analogues of standard regularity conditions from density level set estimation \citep{Baillo2001, Cadre2006, Chen2017density, Qiao2020}. First, (C1) is a non-flatness condition on the posterior surface near the relevant contour levels.
Next, (C2) ensures that the coverage curve $C^*(u)$ is not flat near the threshold $k^*$. For any two sets $A$ and $B$, let $A \Delta B$ be their symmetric difference, i.e., $A \Delta B = (A \setminus B) \cup (B \setminus A)$. Also, recall that $|\Ic(x)|$ is the Lebesgue length of the set-valued rule $\Ic\!:x\mapsto \Ic(x)\subset\Real$. We prove Theorem~\ref{thm:estimated-c-length} in Appendix~\ref{app:proof-estimated-c-length}.

\begin{theorem}[Threshold and expected length of the estimated set]\label{thm:estimated-c-length}
Suppose that the conditions of Theorem~\ref{thm:estimated-c-posterior} and (C1)--(C2) in \eqref{eq:excess length condition} hold. Let $\Dc_n=\sigma(X_1,\ldots,X_n)$ and let $X$ be an independent copy of $X_i$, independent of $\Dc_n$. Define
\begin{align}\label{eq:coverage-rate}
    \zeta_n:=n^{-\alpha_0/2}\exp\{-k_1(\log n)^{r/2}/4\},
\end{align}
where $k_1$ is the constant in \eqref{eq:moment-deconvolution-upper}. Then, for all sufficiently large $n$,
\begin{align}\label{eq:estimated-c-threshold-rate}
    \E_{G^*}|\hat k_n-k^*|\leq K\zeta_n^{1/2},
\end{align}
where $K$ depends only on the parameters in Theorem~\ref{thm:estimated-c-posterior} and on $\beta,C_1,C_2,\delta_0$, and $t_1$. Moreover,
\begin{align}\label{eq:estimated-c-length-rate}
    \left|\E_{G^*}\{|\hat\Ic_n(X)|\mid\Dc_n\}-\E_{G^*}|\Ic^*(X)|\right| \leq\E_{G^*}\{|\hat\Ic_n(X)\Delta\Ic^*(X)|\mid\Dc_n\} =O_P(\zeta_n^{1/2}).
\end{align}
\end{theorem}
Theorem~\ref{thm:estimated-c-length} shows that the plug-in EB marginal coverage set $\hat\Ic_n$ is asymptotically optimal in expected length. The square-root loss relative to the coverage rate arises from controlling sets near their posterior contours (cf.~\citep{Ignatiadis2025}). 

In Appendix~\ref{sec:sensitivity}, we numerically illustrate how the coverage and average length of the estimated sets change as the smoothing parameter varies around $\czero$. With the smoothing parameter estimated by our procedure, the resulting marginal coverage sets are asymptotically optimal in expected length.

\begin{remark}[Monte Carlo computation of the threshold]\label{rem:mc-threshold}
While the threshold $\hat{k}_n$ in \eqref{eq:estimated optimal marginal sets} can be obtained via numerical integration, we can also approximate $\hat{k}_n$ using Monte Carlo simulation. Conditional on the data, generate $\{(\tilde{\theta}_i, \tilde{X}_i)\}_{i=1}^{B}$ with $\tilde{\theta}_i \overset{iid}{\sim}  \widehat g_n^+$ and $\tilde{X}_i \mid \tilde{\theta}_i \overset{ind}{\sim} N(\tilde{\theta}_i, 1)$. Then we can use
\begin{align}\label{eq:estimated optimal marginal sets with sampling}
    \hat{\Ic}_n^B(x) := \{\theta \in \Real: \widehat\pi_n^+(\theta \mid x) \ge \hat{k}_n^{B} \} 
\end{align}
where the threshold $\hat{k}_n^B$ is defined as
\begin{align}\label{eq:estimated threshold with sampling}
    \hat{k}_n^B := \sup_{k \ge 0} \left\{\frac{1}{B}\sum_{i=1}^{B} \1v(\widehat\pi_n^+(\tilde{\theta}_i \mid \tilde{X}_i) \ge k) \ge 1-\beta \right\}.
\end{align}
Monte Carlo calibration contributes an additional expected coverage error of order $B^{-1/2}$. Taking $B=B_n\asymp n$ makes this error asymptotically negligible relative to the statistical coverage error in Theorem~\ref{thm:estimated-c-coverage}, preserving the stated rate.
\end{remark}

\begin{remark}[EB credible sets with the smooth NPMLE]
When controlling posterior coverage is desired rather than the marginal coverage, the smooth posterior density estimated by the smooth NPMLE can also be used to construct plug-in credible or HPD sets with posterior probability $1-\beta$ under $\widehat\pi_n^+(\cdot\mid x)$.
\end{remark}

\section{Generalization to the heteroscedastic setting}\label{sec:hetero model}
So far our results have been discussed under the hierarchical model \eqref{eq:hierarchical model}, which assumes that all the observations $X_1, \ldots, X_n$ are i.i.d. We extend model \eqref{eq:hierarchical model} to the more practical setting where we have {\it heterogeneity} \citep{Banerjee2023, Ignatiadis2025, Soloff2025, Anderson2018}:
\begin{equation}\label{eq:hetero hierarchical model}
    X_i \mid \theta_i \overset{ind}\sim N(\theta_i, \sigma_i^2), \quad \qquad 
\theta_i \mid \xi_i \overset{ind}\sim N(\xi_i,\czero^2), \quad \qquad 
\xi_i \overset{iid}\sim \trueprior;
\end{equation}
here $\czero$ is the largest Gaussian component of $G^*$ as defined in \eqref{eq:c0}, $\trueprior \in \Pc(\Real)$ is unknown, and $\sigma_i^2 \in [\underline{k}, \bar{k}] \subset (0,\infty)$ are assumed to be known but need not be equal. While $\theta_i$ still has the common prior $G^*$ as in the i.i.d.~case, the $i$th observation $X_i$ has density
\begin{align*}
    f^*_{i}(x) := \int \phi_{\sigma_i}(x-\theta) \diff G^*(\theta) = \int \phi_{\sigma_{0,i}}(x-\xi) \diff \trueprior(\xi), \quad x \in \Real, \quad \mbox{for }\;\; i = 1, \ldots, n,
\end{align*}
where we write $\sigma_{0,i}^2 := \czero^2 + \sigma_i^2$. In Appendix~\ref{sec:hetero-scale-rate}, we estimate $\czero$ using the heteroscedastic version of the modified estimator $\hat c_0$ in \eqref{eq:hetero-estimated-c}, for which Theorem~\ref{thm:hetero-estimated-c-rate} establishes the same logarithmic convergence rate on the squared scale as in the homoscedastic setting. Consequently, we estimate $\trueprior$ via a NPMLE with $\hat c_0$, defined as any maximizer
\begin{equation}\label{eq:hetero NPMLE}
    \npmle \in\argmax_{H\in\Pc(\Real)}\sum_{i=1}^n\log p_{H,\sqrt{\sigma_i^2+\hat c_0^2}}(X_i),
\end{equation}
where $p_{H,\sqrt{\sigma_i^2+\hat c_0^2}}$ is the marginal density of $X_i$ induced by $H$ and $\hat c_0$, as defined in \eqref{eq:true data density}. Note that if $\sigma_i = 1$ for all $i$, then this reduces to the NPMLE defined in \eqref{eq:estimated-prior-distribution} for the i.i.d~setting. In Appendix~\ref{sec:hetero-smooth-NPMLE}, we show that the main theoretical guarantees from the i.i.d~setting, discussed in Section~\ref{sec:theory}, extend to the heteroscedastic setting. In particular, the smooth NPMLE still achieves nearly parametric regret rates up to logarithmic factors with estimated smoothing parameter $\hat c_0$ (see Theorem~\ref{thm:hetero denoising}). 
The smooth NPMLE also achieves a polynomial deconvolution rate in the heteroscedastic model \eqref{eq:hetero hierarchical model} (see Theorem~\ref{thm:hetero deconvolution upper bound}). Here, we note that the deconvolution rate is governed by $\bar{\alpha}_0 = \czero^2 / (\czero^2 +\bar{k})$ where $\bar{k}$ is the least favorable noise level, in contrast to $\alpha_0 = \czero^2 /(\czero^2+1)$ in Theorem~\ref{thm:moment-deconvolution} for the i.i.d~setting. We also show that the posterior density based on the smooth NPMLE achieves a polynomial convergence rate as in Theorem~\ref{thm:estimated-c-posterior} (see Theorem~\ref{thm:hetero posterior convergence rate}).

\section{Discussion}\label{sec:Discussion}
In this paper, we propose a smooth $g$-modeling approach to EB estimation and inference under a hierarchical Gaussian location mixture model. The resulting smooth NPMLE leads to a practical procedure for point estimation, prior estimation, posterior approximation and uncertainty quantification, while achieving strong theoretical guarantees. 

There are several natural directions for future work. First, it would be interesting to develop confidence intervals that are explicitly tied to EB estimates, e.g., by centering them at EB posterior means. Such a result would provide a more direct link between point estimation and uncertainty quantification. However, this connection remains only partially understood in the EB literature. Recent work by \citet{Armstrong2022} shows that one can construct confidence intervals centered at linear EB estimates while maintaining coverage guarantees by systematically adjusting the critical value. Extending the smooth $g$-modeling framework in this direction would be an interesting problem. Next, many EB problems arise in models beyond the Gaussian likelihood considered in this paper. In particular, our smooth $g$-modeling framework can be extended more broadly to settings with conjugate priors and more general likelihood models. For example, recent work by \citet{kim2026poissonempiricalbayesgammasmoothed} shows that the smooth NPMLE can also be computed for Poisson mixture models with Gamma mixture priors and enjoys analogous theoretical guarantees. This suggests that broadening the framework beyond the Gaussian setting is a promising direction for future work.

\vspace{-2mm}
\spacingset{1}
\section*{Acknowledgements}
We would like to thank Adityanand Guntuboyina, Jiaying Gu, Nikos Ignatiadis, Jiafeng (Kevin) Chen and Seunghyun Lee for helpful conversations.


\spacingset{.6}
\bibliographystyle{plainnat}
\bibliography{bib}

\clearpage
\phantomsection\label{supplementary-material}
\bigskip
\begin{center}
{\LARGE \bf Supplementary Materials}
\end{center}
\appendix
\spacingset{0.9}
\appendixtableofcontents
\spacingset{1}

\appsection{Details on estimation and inference for the smoothing parameter}\label{sec:identifiability-details}

\appsubsection{Finite-sample confidence bounds for $\czero$}\label{sec:neighborhood-upper-bound}
In Section~\ref{sec:identifiability}, we introduced the modified neighborhood estimator of $\czero$ used in the smooth NPMLE. In this subsection, we discuss inference on $\czero$. We first examine the limitations of uniformly valid lower confidence bounds for $\czero$ and then construct a finite-sample upper confidence bound.
As mentioned in Section~\ref{sec:identifiability}, the functional $G\mapsto\czero(G)$ is upper semicontinuous under weak convergence, but it is not lower semicontinuous \citep{Donoho1988}. The latter has the following consequence for lower confidence bounds. Fix a sample size $n$ and suppose that a statistic $\ell_n$ satisfies
\[
    \P_G\{\ell_n\leq\czero(G)\}\geq1-\beta \qquad\text{for every }G\in\Pc(\Real),
\]
where $\P_G$ denotes the law of $n$ independent observations from $G\star N(0,1)$. Fix any $G\in\Pc(\Real)$ and choose finitely supported probability measures $G_m\Rightarrow G$. Since $\czero(G_m)=0$, uniform coverage implies $\P_{G_m}(\ell_n>0)\leq\beta$. Moreover, writing
\[
    f_m(x):=\int\phi_1(x-\theta)\,\diff G_m(\theta),
    \qquad
    f(x):=\int\phi_1(x-\theta)\,\diff G(\theta),
\]
weak convergence gives $f_m(x)\to f(x)$ for every $x$. Scheff\'e's lemma therefore yields $\|f_m-f\|_{L_1}\to0$, and hence
\[
    \operatorname{TV}(f_m^{\otimes n},f^{\otimes n})
    \leq \frac n2\|f_m-f\|_{L_1}\longrightarrow0.
\]
In particular, $|\P_{G_m}(\ell_n>0)-\P_G(\ell_n>0)| \leq\operatorname{TV}(f_m^{\otimes n},f^{\otimes n})\to0$. Thus, it follows that
\[
    \P_G(\ell_n>0)
    =\lim_{m\to\infty}\P_{G_m}(\ell_n>0)\leq\beta.
\]
Thus, over the unrestricted class of priors, a uniformly valid $1-\beta$ lower confidence bound can be positive with probability at most $\beta$, under any $G$.

In contrast, we can construct a finite-sample upper confidence bound $\hat{c}_U$ satisfying
\begin{align}\label{eq:c0 upper confidence bound}
    \P_{G^*}(\czero \leq \hat{c}_U) \ge 1-\beta
\end{align}
for a confidence level $1-\beta$ ($0 < \beta < 1$). 
To this end, we choose the KS neighborhood around $\F_n$ to contain $F^*$ with probability at least $1-\beta$.
Note that $F^*$ is continuous and thus the distribution of $\KS(\F_n, F^*)$ is universal, i.e., $\KS(\F_n, F^*)$ is distribution-free. Then, for any given confidence level $1-\beta \in (0,1)$, we can choose $\eta$ such that 
    $\P(\KS(\F_n, F^*) \leq \eta) \ge 1-\beta$
independently of $F^*$. For example, we may choose $\eta = \eta_n := \sqrt{\frac{\log(2/\beta)}{2n}}$ based on the Dvoretzky–Kiefer–Wolfowitz (DKW) inequality with Massart's tight constant \citep{Massart1990}. Hence, for a fixed confidence level $1-\beta$, we use a KS neighborhood radius of order $n^{-1/2}$, which shrinks faster than the radius in~\eqref{eq:eta-radius} used to establish the convergence rate of $\hat c_0$ in Section~\ref{sec:identifiability}. The following proposition states the result; its proof is given in Appendix~\ref{app:proof of prop:upper confidence bound neighborhood}.

\begin{prop}[Finite-sample upper confidence bound for $\czero$]\label{prop:upper confidence bound neighborhood}
Let $\hat\sigma_U := \sigma_0(\F_n; \eta_n)$, where $\sigma_0(\cdot;\cdot)$ is defined in~\eqref{eq:envelope} and $\eta_n:=\sqrt{\frac{\log(2/\beta)}{2n}}$.
\begin{enumerate}
    \item[(i)] If $\eta_n < 1/2$ (equivalently, $n > 2\log(2/\beta)$), then
          $\hat\sigma_U$ is finite almost surely.
    \item[(ii)] \eqref{eq:c0 upper confidence bound} holds with
          $\hat{c}_U = \sqrt{\max(\hat{\sigma}_U^2 - 1, 0)}$.
\end{enumerate}
\end{prop}

\appsubsection{Convergence rate and minimax optimality for scale estimation}\label{app:convergence rate of the neighborhood estimator}
In this subsection, we establish a logarithmic convergence rate for estimating $\czero$ with the original neighborhood estimator $\hat c_{\mathrm{KS}}$ in \eqref{eq:estimated-KS-c} and show that its rate on the squared scale is minimax optimal over a suitable class of priors. First, the following proposition quantifies the convergence of the neighborhood estimator $\hat c_{\mathrm{KS}}$ in \eqref{eq:estimated-KS-c} with the choice of $\eta_n$ in \eqref{eq:eta-radius} under Assumption~\ref{ass:moment-fourier}. 

\begin{prop}[Convergence rate of the original neighborhood estimator]\label{thm:estimated-c-rate}
Suppose that Assumption~\ref{ass:moment-fourier} holds. Let $\hat c_{\mathrm{KS}}$ and $\rho_n$ be defined in \eqref{eq:estimated-KS-c} and \eqref{eq:rho_n}, respectively. There is a constant $C>0$, depending only on $\czero$ and $\boldsymbol{\gamma}$, such that \eqref{eq:sigma2-rate-event} holds for all sufficiently large $n$. 
Consequently, $\hat c_{\mathrm{KS}}^2-\czero^2=O_P(\rho_n)$. If $\czero>0$, then $\hat c_{\mathrm{KS}}-\czero=O_P(\rho_n)$; if $\czero=0$, then $\hat c_{\mathrm{KS}}=O_P(\sqrt{\rho_n})$.
\end{prop}

We prove Proposition~\ref{thm:estimated-c-rate} in Appendix~\ref{app:proof-estimated-c-rate} using the inverse inequality in Lemma~\ref{lem:ks-inverse}, which bounds how much the squared Gaussian scale of a location mixture can exceed $\sigma_0^2$ when the mixture is close to $F^*$ in KS distance.

Next, the following proposition shows that the logarithmic rate $\rho_n$ in \eqref{eq:rho_n} for estimating $\czero^2$ cannot be improved uniformly over a fixed class satisfying Assumption~\ref{ass:moment-fourier} $(i)$. We provide the proof of Proposition~\ref{prop:scale-minimax-lower} in Appendix~\ref{app:proof-scale-minimax-lower}.

\begin{prop}[Minimax lower bound for scale estimation]
\label{prop:scale-minimax-lower}
Let $\mathcal H_{\gammav}$ be the class of distributions $H\in\Pc(\Real)$ satisfying Assumption~\ref{ass:moment-fourier} (i) with $H$ in place of $\trueprior$ and $\gammav=(K_0,r,a,\vartheta,t_0)=(2,1,1/2,2,1)$. Define $\mathcal G :=\left\{ G=H\star N(0,c^2): 0\leq c\leq1, H\in\mathcal H_{\gammav} \right\}$. Then, for all sufficiently large $n$,
\begin{align}\label{eq:scale-minimax-lower}
    \inf_{\widetilde c_n}\sup_{G\in\mathcal G}\E_G\!\left[|\widetilde c_n^2-\czero^2(G)|\right] \gtrsim\rho_n,
\end{align}
where the infimum is over all nonnegative estimators $\widetilde c_n$ based on $n$ independent observations from $p_{G,1} = G \star N(0,1)$.
\end{prop}

\appsubsection{Failure of the uniform near-parametric marginal density rate for $\hat c_{\mathrm{KS}}$}\label{sec:ks-only-density}
While the estimator $\hat c_{\mathrm{KS}}$ in \eqref{eq:estimated-KS-c} has the same $O_p(\rho_n)$ rate on the squared scale as the modified $\hat c_0$ in \eqref{eq:estimated-c}, the smooth NPMLE fitted at $\hat c_{\mathrm{KS}}$ need not attain the near-parametric marginal density rate in Theorem~\ref{thm:estimated-c-density-denoising} uniformly over a fixed class satisfying Assumption~\ref{ass:moment-fourier}~$(i)$. The proof is given in Appendix~\ref{app:ks-only-lower-bound}.

\begin{theorem}\label{thm:ks-only-density}
There exists a sequence of priors $G_n^*$ with $\czero(G_n^*)=0$ satisfying Assumption~\ref{ass:moment-fourier}~(i) for a fixed $\gammav$ with $r=1$, such that for all sufficiently large $n$, the marginal density estimator $\widehat f_n^{\mathrm{KS}}$ constructed as in~\eqref{eq:estimated-density} using the smooth NPMLE fitted at $\hat c_{\mathrm{KS}}$ satisfies
\begin{align*}
    \E_{G_n^*}[\H^2(\widehat f_n^{\mathrm{KS}},f_n^*)] \gtrsim (n\log n)^{-1/2},  
\end{align*}
where $f_n^*=p_{G_n^*,1}$.
\end{theorem}

\appsubsection{Effect of the positive floor on marginal density estimation and denoising}\label{sec:positive-floor}
The following proposition shows that fitting at $\hat c_0^+$ in \eqref{eq:deconvolution-floor} need not retain nearly parametric marginal density estimation and denoising rates when $\czero=0$. This explains why we retain $\hat c_0$ for marginal density estimation and denoising. We provide the proof of Proposition~\ref{prop:floor-density-denoising} in Appendix~\ref{app:proof-floor-density-denoising}.

\begin{prop}\label{prop:floor-density-denoising}
Suppose that $\trueprior=\delta_0$ and $\czero=0$ in \eqref{eq:hierarchical model} so that $G^* = \delta_0$ and $X_i \overset{iid}{\sim} N(0,1)$ for every $i$. Let $\hat c_0^+$ be defined in \eqref{eq:deconvolution-floor}, and let $\npmle^+$ be the NPMLE fitted at $\hat c_0^+$. Let $\widehat f_n^+$ and $\hat\theta_i^+$ denote the estimated marginal density and EB posterior mean of $\theta_i$, respectively, based on $(\npmle^+,\hat c_0^+)$. Then, for $n\geq3$,
\begin{align}\label{eq:floor-density-lower}
    \E_{G^*}[\H^2(\widehat f_n^+,f^*)] \geq\frac{e^{-1}}8 \left(1-e^{-\rho_n/2}\right)^2 \asymp \rho_n^2,
\end{align}
and
\begin{align}\label{eq:floor-denoising-lower}
    \Rc_n(\hat\theta^+,\theta^*) \geq \left(\frac{\rho_n}{1+\rho_n}\right)^2 \left(1-\frac1n\right) \asymp \rho_n^2.
\end{align}
Here, $\Rc_n(\hat\theta^+,\theta^*)$ is the denoising regret defined in~\eqref{eq:regret}.
\end{prop}
Consequently, both risks are bounded below by a positive constant
times $\rho_n^2$ for all sufficiently large $n$, even though
$\trueprior=\delta_0$ satisfies Assumption~\ref{ass:moment-fourier}.

\appsubsection{Computation of the neighborhood procedure}\label{sec:neighborhood procedure-computation}
In this subsection, we discuss how to approximate the neighborhood procedure and use it to select the smoothing parameter in practice. We first explain how to approximate $\sigma_0(\F_n;\eta)$ in \eqref{eq:envelope}. 
Suppose that the neighborhood radius $\eta  > 0$ is given. For each $\sigma > 0$, define the Kolmogorov-Smirnov minimum distance
\begin{align*}
    \delta_n(\sigma) := \inf_{H \in \Pc(\Real)} \KS\!\left(\F_n,\; H \star N(0,\sigma^2)\right).
\end{align*}
Note that $\delta_n(\sigma)$ is nondecreasing in $\sigma$. Then $\sigma_0(\F_n;\eta)$ in \eqref{eq:envelope} is equivalently expressed as
\begin{align}\label{eq:sigma_0 and delta_n}
    \sigma_0(\F_n;\eta) = \sup\{\sigma \ge 0 : \delta_n(\sigma) \le \eta\}.
\end{align}
For each $\sigma > 0$, we approximate $\delta_n(\sigma)$ by restricting $H$ to a support grid and minimizing the Kolmogorov--Smirnov distance over grid-supported distributions. This yields a linear programming problem \citep{Deely1968}. Specifically, for each $\sigma > 0$, we set $\tilde{F} = H \star N(0, \sigma^2)$ where $H$ is a discrete distribution:
\begin{align*}
    H = \sum_{j=1}^{m} h_j \delta_{\theta_j}, \qquad h_j \ge 0, \qquad \sum_{j=1}^{m}h_j = 1,
\end{align*}
where $\{\theta_j\}_{j=1}^{m}$ is an equi-spaced grid on $[X_{(1)}, X_{(n)}]$ (or on $[-L, L]$ for some large $L > 0$). For each $\sigma > 0$, we check whether some $H$ supported on this grid satisfies $\KS(\F_n, \tilde{F}) \leq \eta$. If so, then $\sigma \leq \sigma_0(\F_n;\eta)$. If not, we can conclude that $\sigma$ is infeasible on the chosen grid. Write $X_{(i)}$ for the $i$th order statistic of $X_1,\ldots,X_n$. Note that
\begin{align*}
    \KS(\F_n, \tilde{F}) &= \max_{1\leq i \leq n}\left\{ \abs{\tilde{F}(X_{(i)}) - \frac{i-1}{n}},\quad  \abs{\tilde{F}(X_{(i)}) -\frac{i}{n}}\right\}.
\end{align*}
Moreover, it holds that:
\begin{align*}
    \tilde{F}(X_{(i)}) = \sum_{j=1}^{m}h_j \Phi\left(\frac{X_{(i)} -\theta_j}{\sigma} \right) = [\Av\hv]_i
\end{align*}
where $\hv = (h_1, \ldots, h_m)^\top \in \Real^{m}$ and $\Av \in \Real^{n \times m}$ such that
\begin{align*}
    [\Av]_{ij} = \Phi\left(\frac{X_{(i)} -\theta_j}{\sigma}\right).
\end{align*}
Define $\fv_{-0} = (1/n, \ldots, (n-1)/n, 1)$ and $\fv_{-1} = (0, 1/n, \ldots, (n-1)/n)$.
Then, for each $\sigma > 0$, we need to check the feasibility of the following linear problem:
\begin{equation}\label{eq:LP}
\left\{
\begin{aligned}
    \Av \hv - \fv_{-1} &\le \eta \1v_n, \\
    \Av \hv - \fv_{-0} &\ge -\eta \1v_n, \\
    \hv               &\ge \mathbf{0}_m, \\
    \hv^\top \1v_m     &= 1.
\end{aligned}
\right.
\end{equation}

As discussed in Section~\ref{sec:identifiability}, we need only search over $\sigma\in[\underline{\sigma},\bar{\sigma}]$, where $\underline{\sigma}=1$, $\bar{\sigma}=\sqrt{1+C_n^2}$, and $C_n$ is defined in \eqref{eq:cap}. We can search this interval on a grid. Alternatively, motivated by the monotonicity of the exact $\delta_n(\sigma)$ in \eqref{eq:sigma_0 and delta_n}, we use bisection as a numerical heuristic that can reduce the number of linear programming evaluations.
That is, we iterate bisection by setting $\sigma_{\mathrm{mid}} = (\underline{\sigma} + \bar{\sigma})/2$ and updating $\underline{\sigma} = \sigma_{\mathrm{mid}}$ if \eqref{eq:LP} is feasible (else $\bar{\sigma} = \sigma_{\mathrm{mid}}$), repeating until $\bar{\sigma} - \underline{\sigma} < \epsilon$ for a prescribed $\epsilon > 0$ and using $\underline{\sigma}$ in place of $\sigma_0(\F_n;\eta)$.

While the theoretical analysis uses $\eta_n=\sqrt{\log n/n}$, we consider $\eta_n(A)=\sqrt{A\log n/n}$ in numerical implementation and choose $A$ by cross-validation to improve finite-sample fit. This is because, although the scale estimation rate in Proposition~\ref{thm:estimated-c-rate} remains $\rho_n$ for any fixed $A>0$, the constant $C$ in the corresponding error bound (see \eqref{eq:sigma2-rate-event}) can depend on $A$ and need not be small. 
To cross-validate $A$ in $\eta_n= \sqrt{A\log n / n}$, fix a finite candidate grid $\mathcal A\subset[A_-,A_+]$, where $0<A_-<A_+<\infty$, and for a training sample of size $m$ set $\eta_m(A)=\sqrt{A\log m/m}$. For a $K$-fold cross validation, we randomly partition the data into $K$ sets, say $\Dc_1, \ldots, \Dc_K$. Let $m_k=n-|\Dc_k|$ and let $\F_{n,-k}$ be the empirical distribution of $\{X_i:i\notin\Dc_k\}$. For each $A\in\mathcal A$, we use $\F_{n,-k}$ and radius $\eta_{m_k}(A)$ in the bounded search described above to obtain $\hat\sigma_{0,-k}(A)$. Let $\hat H_{-k}(A)$ be a corresponding Kolmogorov--Smirnov minimum distance fit on the training observations at scale $\hat\sigma_{0,-k}(A)$. We select $A$ by the held-out marginal likelihood:
\begin{equation}\label{eq:eta_cv}
\begin{aligned}
    \hat A\in\argmax_{A\in\mathcal A}\frac1n\sum_{k=1}^K\sum_{i\in\Dc_k} \log p_{\hat H_{-k}(A),\hat\sigma_{0,-k}(A)}(X_i), \qquad \hat\eta_n:=\sqrt{\frac{\hat A\log n}{n}}.
\end{aligned}
\end{equation}
We then use the selected radius $\hat\eta_n$ to compute the capped neighborhood estimate $\hat c_{\mathrm{KS}}\wedge C_n$ from all observations, where $\hat c_{\mathrm{KS}}=\sqrt{(\sigma_0^2(\F_n;\hat\eta_n)-1)_+}$. Since $\eta_n(A_-)\leq\hat\eta_n\leq\eta_n(A_+)$, the argument of Proposition~\ref{thm:estimated-c-rate} gives $0\leq\sigma_0^2(\F_n;\hat\eta_n)-\sigma_0^2\leq C_\mathcal A\rho_n$ with probability at least $1-2n^{-2A_-}$ for a sufficiently large constant $C_\mathcal A$. 
When $A_- < 1$, this bound alone does not establish the theoretical guarantees proved in Sections~\ref{sec:theory} and~\ref{sec:Opt-conf-sets}. Nevertheless, we allow $A<1$ and use cross-validation for finite-sample tuning, since the multiplicative constants in theoretical error bounds need not be small.

So far, we have described how to approximate the capped estimator $\hat c_{\mathrm{KS}}\wedge C_n$ at radius $\hat\eta_n$. It remains to explain how we use this estimate to compute the modified neighborhood estimator $\hat c_0$ described in Section~\ref{sec:identifiability}. We compute $\hat c_0$ by numerically minimizing the criterion in \eqref{eq:estimated-c-buffer}. Specifically, we evaluate candidate squared scales $c^2$ in $[0,\hat c_{\mathrm{KS}}^2\wedge C_n^2]$ on a grid with spacing $\kappa\rho_n$ for some small fixed $\kappa > 0$, including both endpoints. At each candidate $c$, we approximate $L_n(c)$ by solving the NPMLE problem \eqref{eq:NPMLE} on a support grid and evaluate this criterion with $C=\{(\hat c_{\mathrm{KS}}^2\wedge C_n^2)-c^2\}/\rho_n$. The selected scale gives $\hat c_0$, and we retain its corresponding NPMLE for marginal density estimation and denoising. For prior and posterior density estimation and construction of optimal marginal coverage sets, we use the stabilized scale $\hat c_0^+ = \hat c_0 \vee \sqrt{\rho_n}$ in \eqref{eq:deconvolution-floor}, recomputing the NPMLE only if the floor changes the selected scale.

For the simulations and real data analyses in Appendix~\ref{sec:numerical results}, we select $A$ by five-fold cross-validation over $\mathcal A=\{0.01,0.02,\ldots,0.1\}\cup \{0.15,0.2,\ldots,2\}$. To implement the neighborhood procedure for estimating $\czero$, we use 100 equally spaced support points on $[X_{(1)},X_{(n)}]$ for the KS minimum-distance optimization. We initialize the NPMLE fits with 1000 support points and use $\kappa=1/5$ for the scale search. Coverage thresholds are calibrated separately for each distinct noise level using $B=5000$ Monte Carlo draws, following Remark~\ref{rem:mc-threshold} with the corresponding noise variance.

\appsection{Discussion of polynomial deconvolution rates}\label{app:deconvolution rate illustration}
In this section, we illustrate how the Gaussian smoothing parameter $\czero$ affects deconvolution. To separate this effect from smoothing parameter estimation, we treat $\czero$ as known and fit the smooth NPMLE at $c=\czero$. Figure~\ref{fig:L2-distance} plots its empirical squared $L_2$ risk for estimating $g^*$. In our example, we set $\trueprior = \mathrm{Laplace}(0,4)$ and $\czero \in \{1/4, 1/2, 1, 2\}$ under \eqref{eq:hierarchical model}. Note that the upper bound on the squared $L_2$ risk in Theorem~\ref{thm:moment-deconvolution} has leading polynomial factor $n^{-\alpha_0}=n^{-\czero^2/(\czero^2+1)}$, giving $n^{-1/17}, n^{-1/5}, n^{-1/2}$, and $n^{-4/5}$ for the four values of $\czero$, respectively. The results show that larger values of $\czero$ yield faster empirical risk decay over the sample sizes considered.

\begin{figure}[t]
    \centering
    \begin{tabular}{cc}
    \includegraphics[width=0.4\textwidth]{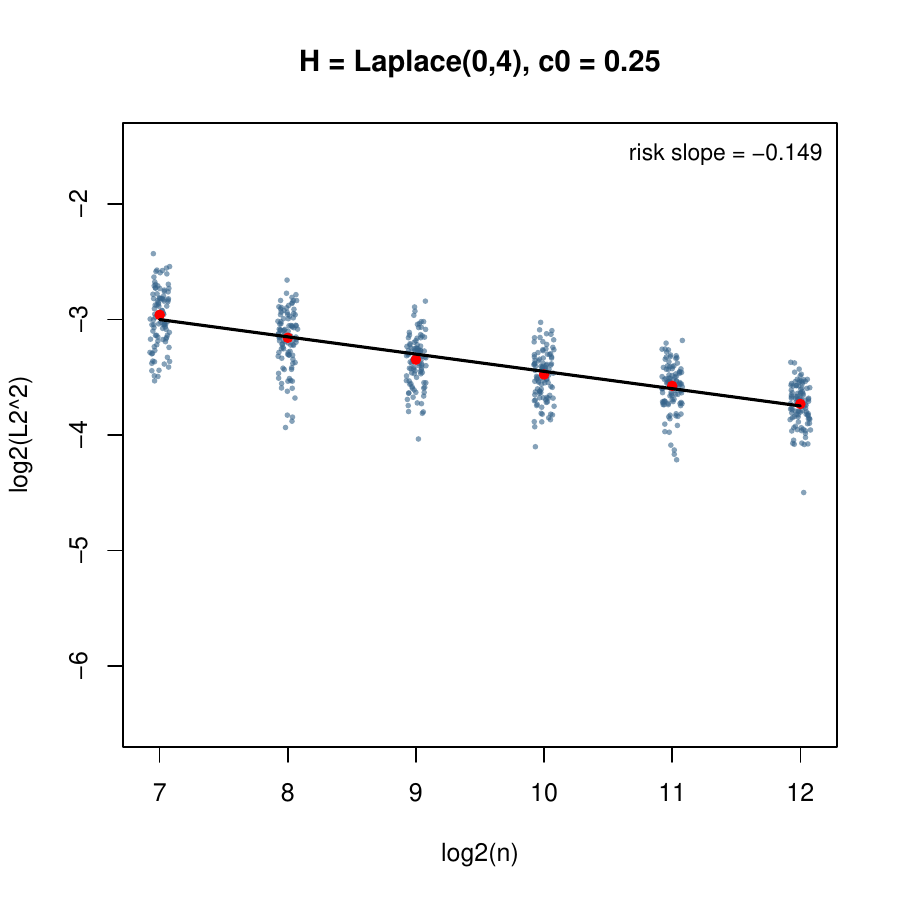} &
    \includegraphics[width=0.4\textwidth]{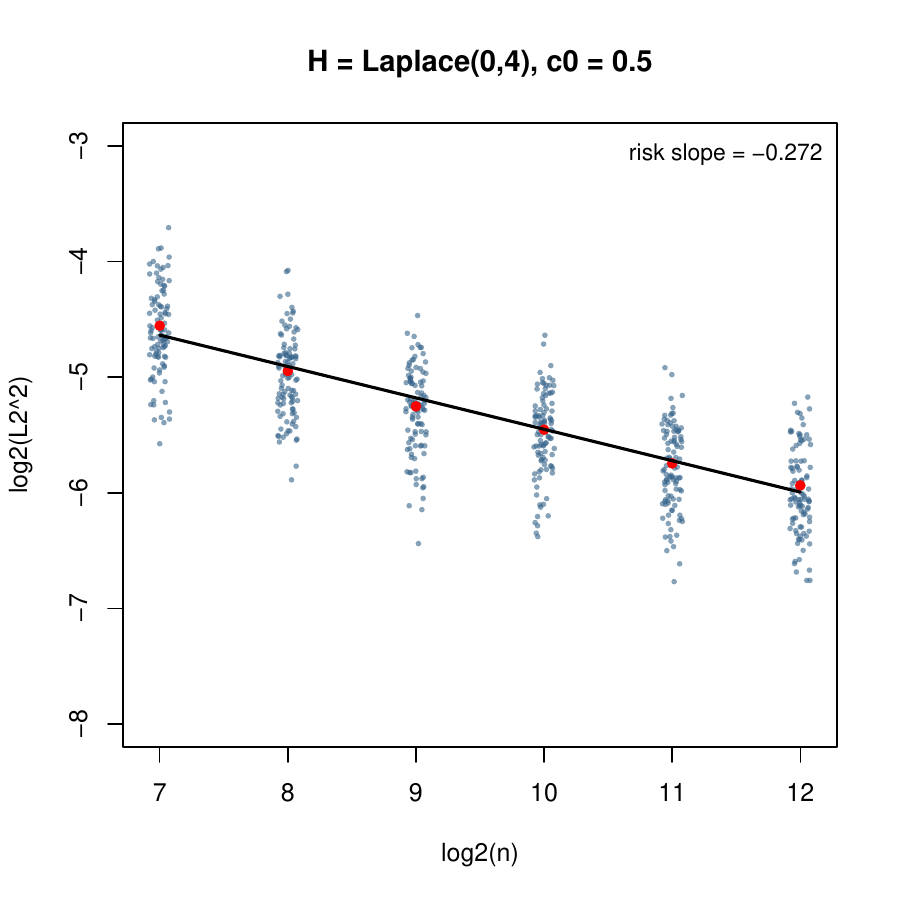}  \\
    \includegraphics[width=0.4\textwidth]{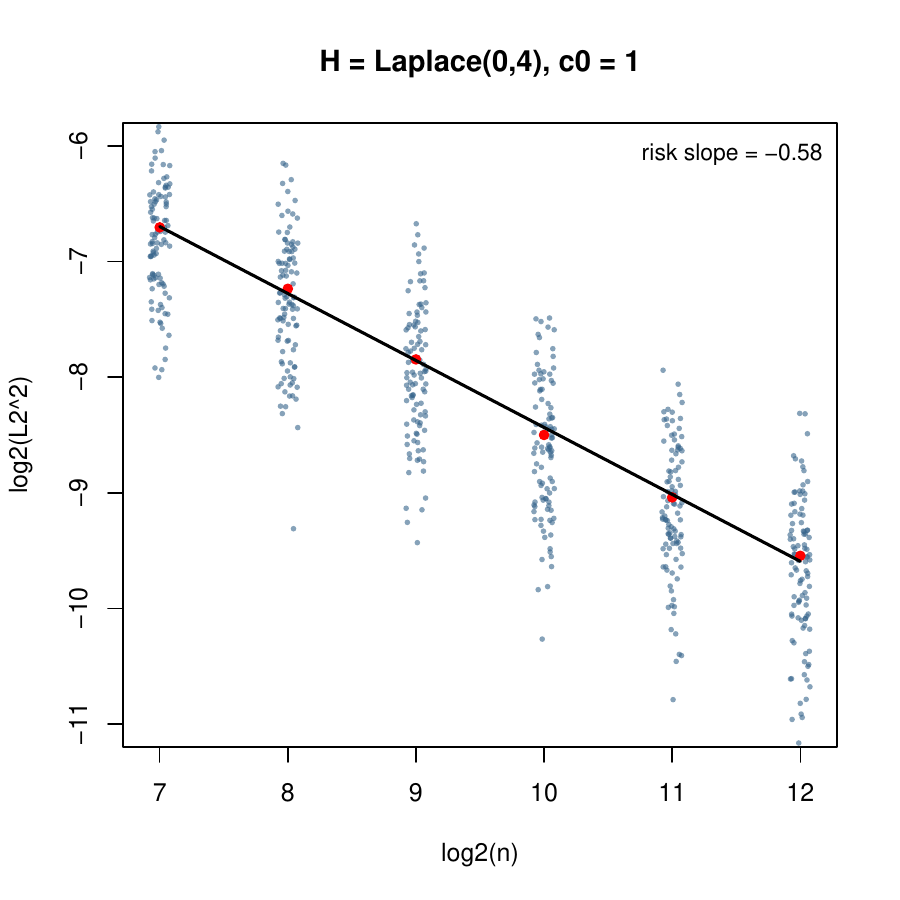} &
    \includegraphics[width=0.4\textwidth]{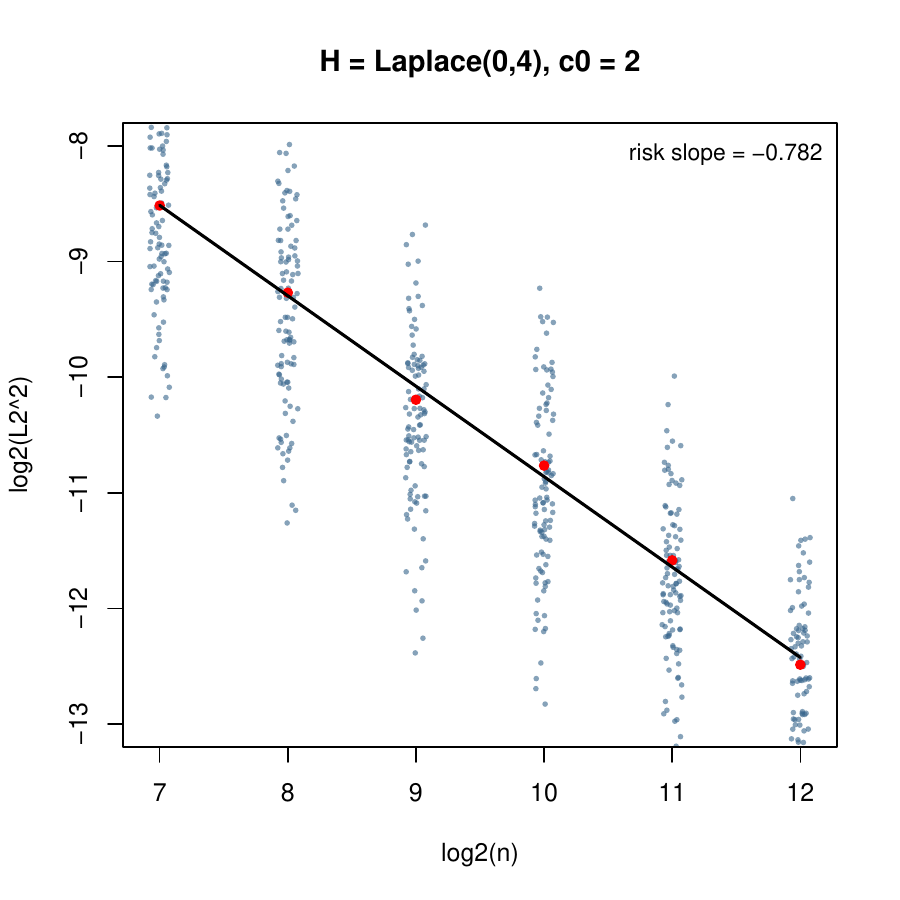}
    \end{tabular}

    \caption{\footnotesize Squared $L_2$ distance between the smooth NPMLE fitted at $c=\czero$ and the true prior density $g^*$ in \eqref{eq:true prior density}. We set $\trueprior = \mathrm{Laplace}(0,4)$ with $\czero \in \{1/4, 1/2, 1, 2\}$. The red dots represent the logarithms of the average squared $L_2$ distances over 100 independent replications at each sample size, and the regression lines are fitted to the red dots.}
    \label{fig:L2-distance}
\end{figure}

\begin{remark}[Deconvolution rate]
It has been known that imposing stronger smoothness conditions on the target density (i.e., $g^*$) can yield non-logarithmic rates in deconvolution; see e.g., pp.~44--45 in \citet{Meister2009} (also see \cite{Butucea2008a, Butucea2008b}).
\citet{Butucea2008a,Butucea2008b} obtain matching upper and lower bounds for kernel-based estimators when the noise has a higher supersmoothness order than the target density. In the Gaussian case, \citet{Lacour2006} prove only an upper bound for kernel-based estimators, which yields the same $n^{-\alpha_0}$ rate up to logarithmic factors (see Theorem 3.1 therein). However, standard kernel deconvolution estimates can take negative values and may require modification before being used as prior densities. Our results establish the optimality of the smooth NPMLE, which directly produces a nonnegative, normalized prior density and corresponding posterior distributions. For $0<r<2$ in \eqref{eq:simultaneous-moment-control}, we also prove a corresponding lower bound in this Gaussian setting that matches the leading polynomial exponent and the order of the exponential term (Theorems~\ref{thm:moment-deconvolution} and~\ref{thm:moment-deconvolution-lower}).
\end{remark}

\appsection{Review of HPD sets}\label{sec:HPD}
In this section, we review the fact that the HPD set is the optimal credible set in terms of expected length. Consider the general setting in Section~\ref{sec:OPT construction}, and recall that for a set-valued rule $\Ic\!:x\mapsto \Ic(x)\subseteq\Real$ we denote by $|\Ic(x)|$ its Lebesgue length. We wish to solve the following optimization problem:
\begin{equation}
\label{eq:Opt-Cred-Set-Problem}
  \min_{\Ic(\cdot)} \; \mathbb E_G\bigl[|\Ic(X)|\bigr]
  \qquad \;\; \text{subject to}\quad
  \mathbb P_G\bigl(\theta\in \Ic(X)\mid X=x\bigr)\;\ge\;1-\beta,\;\forall x\in\mathcal X.
\end{equation}
Because the expectation above factorizes as $ \mathbb E_G\bigl[|\Ic(X)|\bigr] =\int_{\mathcal X} |\Ic(x)|\,p_G(x)\,\diff x, $ problem~\eqref{eq:Opt-Cred-Set-Problem} separates pointwise. Fix $x \in \mathcal X$ and let $\mathcal{C}_{x}$ be the collection of all measurable sets with posterior probability at least $1-\beta$, i.e., 
\[
 \mathcal C_{x} :=\Bigl\{\Ic\subset\Theta: \mathbb{P}_G(\Ic \mid x)\equiv\int_\Ic \pi (\theta \mid x)\,\diff \theta \ge1-\beta\Bigr\}.
\]
Then~\eqref{eq:Opt-Cred-Set-Problem} reduces to finding $\Ic_x\in\argmin_{\Ic\in\mathcal C_{x}} |\Ic|.$ The optimal $\Ic_x$ can be characterized in terms of the {\it highest posterior density (HPD) set} at $x$ with posterior probability $1-\beta$. We describe the details below.

Define for any threshold $k>0$ the \emph{$k$-level set} of the posterior distribution $\pi (\cdot \mid x)$ as 
\[
  L_k(x) :=\{\theta\in\Theta:\,\pi (\theta \mid x)\ge k\}.
\]
The level sets are nested, have finite Lebesgue measure, and their posterior probability $\mathbb{P}_G(L_k(x)\mid x)$ is non-increasing in $k$.  Define $$k(x) := \sup \left\{k > 0:  \mathbb{P}_G\bigl(L_{k}(x)\mid x \bigr)  \;\ge \;1-\beta \right\}.$$  Call 
\begin{equation}\label{eq:HPD}
    L_{k(x)}(x) = \{\theta\in\Theta:\,\pi (\theta \mid x)\ge k(x)\}
\end{equation} 
the HPD set at $x$. When $\pi (\cdot \mid x)$ is continuous, unimodal, and vanishes at the tails, $L_{k(x)}(x) $ is a contiguous interval $[\theta_L(x),\theta_U(x)]$ with
$\pi(\theta_L(x) \mid x)=\pi (\theta_U(x) \mid x)=k(x)$. We provide the proof of Theorem~\ref{thm:Opt-credible} in Appendix~\ref{app:proof of thm:Opt-credible}.

\begin{theorem}\label{thm:Opt-credible}
Suppose that $G$ has a density $g$ with respect to the Lebesgue measure on $\Theta \subset \mathbb{R}$. 
Then, for $0 < \beta < 1$ and for any fixed $x \in \mathcal{X}$,
\[
\int_{\{\theta:\pi(\theta\mid x)>k(x)\}} \pi(\theta\mid x)\,\diff\theta
\le 1-\beta
\le
\mathbb P_G(L_{k(x)}(x)\mid x).
\]
Hence, there exists a measurable set $B_x \subseteq \{\theta\in\Theta:\pi(\theta\mid x)=k(x)\}$ such that
\begin{equation}\label{eq:Opt-Cred-Set}
    \Ic_x := \{\theta\in\Theta:\pi(\theta\mid x)>k(x)\}\cup B_x.
\end{equation}
Then $\Ic_x$ in~\eqref{eq:Opt-Cred-Set} solves the pointwise optimization problem~$\min_{\Ic\in\mathcal C_x} |\Ic| $ and has posterior probability exactly $1-\beta$.
\end{theorem}
Since \eqref{eq:Opt-Cred-Set-Problem} separates pointwise in $x$, the rule $x\mapsto \Ic_x$ is globally optimal.
\begin{remark}
    We ideally want $k(x)$ to satisfy $\mathbb{P}_G (L_{k(x)}(x)\mid x) = 1-\beta$ which holds under mild conditions on the posterior density. 
    If $\{\theta\in\Theta:\pi(\theta\mid x)=k(x)\}$ is a measure zero set, then we may take $B_x=\{\theta\in\Theta:\pi(\theta\mid x)=k(x)\}$, in which case the HPD set \eqref{eq:HPD} itself is the optimal credible set:
\begin{equation}\label{eq:HPD-simple}
    \Ic_x = \{\theta\in\Theta:\,\pi (\theta \mid x)\ge k(x)\}  = L_{k(x)}(x).
\end{equation} 
This holds, e.g., for the posterior density \eqref{eq:true posterior density} under \eqref{eq:hierarchical model} with $\czero > 0$.
\end{remark}
Thus, the shortest highest posterior density credible set with posterior probability $1-\beta$ is Bayes-optimal with respect to expected length under the stipulated conditional coverage constraint.

\appsection{Comparison of HPD sets and optimal marginal coverage sets}\label{sec:comparison}
Consider the hierarchical model \eqref{eq:hierarchical model}. If $\trueprior = \delta_a$ for some $a \in \Real$, then both the HPD set and the optimal marginal coverage set at $X_i = x$ coincide with the equal-tailed $(1-\beta)$ credible interval $\alpha_0 x + (1-\alpha_0)a \pm z_{1-\beta/2}\sqrt{\alpha_0}$ where $\alpha_0 = \czero^2 / (\czero^2+1)$. In general, HPD sets and optimal marginal coverage sets are different from each other. See Figure~\ref{fig:twocomp} for an illustration with a two-component normal mixture prior. The rightmost panel of Figure~\ref{fig:twocomp} shows the lengths of each set at $X_i \in (-10, 10)$ and their expected lengths. As discussed in Section~\ref{sec:OPT construction}, the optimal marginal coverage sets achieve the shortest expected length among all marginal coverage sets. Note that $(1-\beta)$ credible sets are also $(1-\beta)$ marginal coverage sets, so HPD sets are $95\%$ marginal coverage sets. While HPD sets guarantee $95\%$ coverage for every $X_i = x$, optimal marginal coverage sets achieve the shortest length on average at the expense of giving up coverage guarantees for some $X_i = x$ (see $|x| \approx 0$ in the right panel of Figure~\ref{fig:twocomp}).

\begin{figure}[t]
  \centering
  \begin{subfigure}[t]{0.32\linewidth}
    \centering
    \includegraphics[width=\linewidth]{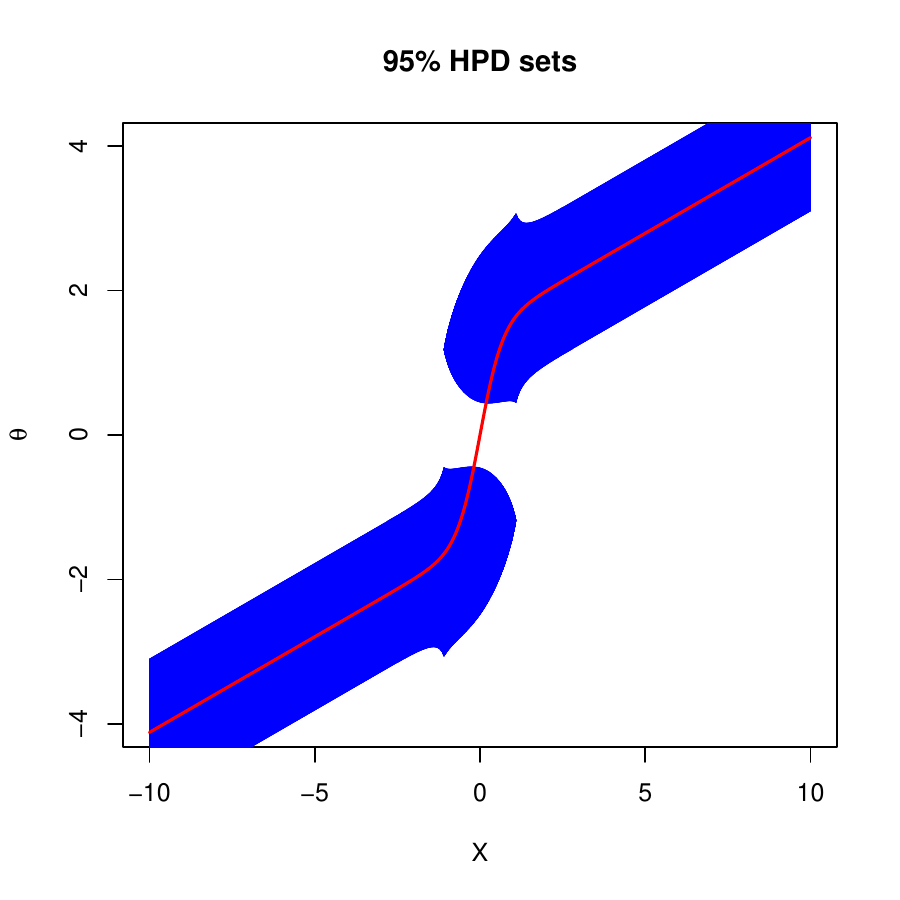}
    \label{fig:twocomp-hpd}
  \end{subfigure}
  \begin{subfigure}[t]{0.32\linewidth}
    \centering
    \includegraphics[width=\linewidth]{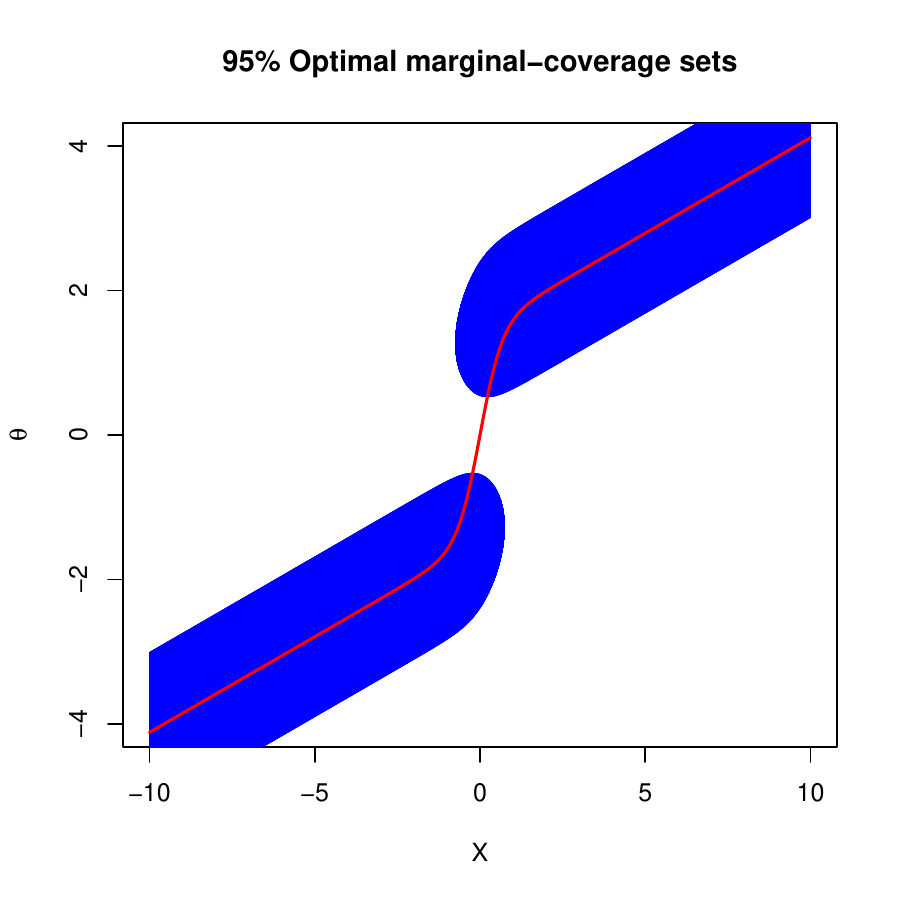}
    \label{fig:twocomp-opt}
  \end{subfigure}\hfill
  \begin{subfigure}[t]{0.32\linewidth}
    \centering
    \includegraphics[width=\linewidth]{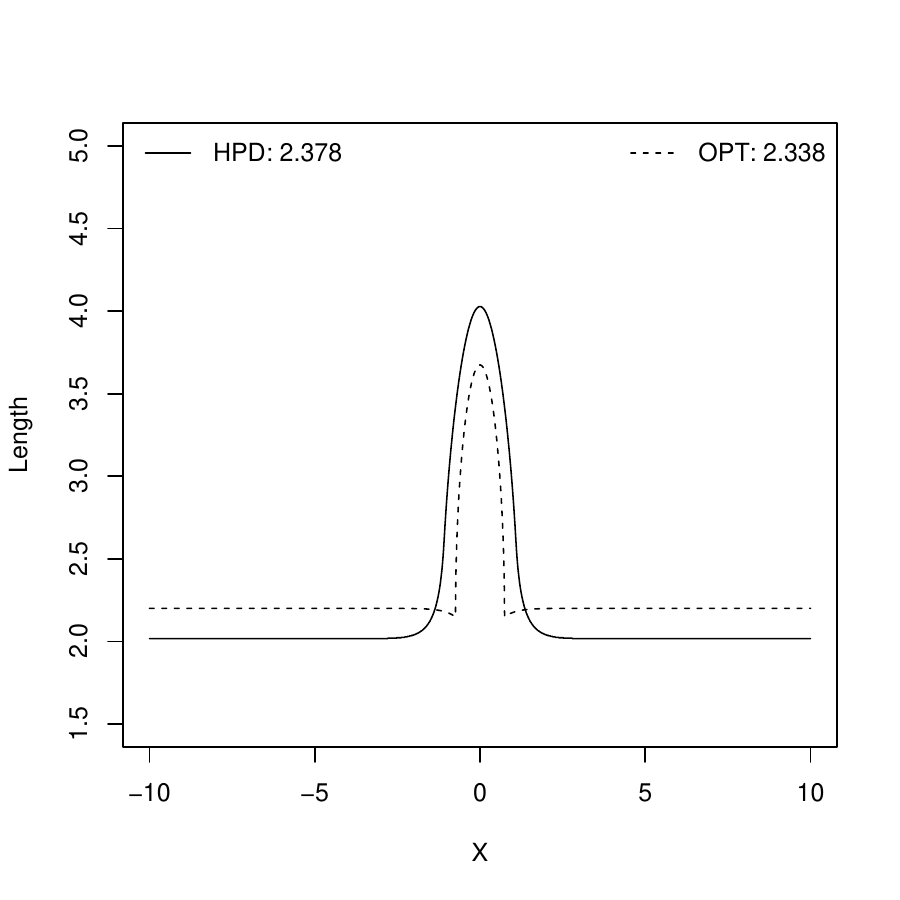}
    \label{fig:twocomp-length}
  \end{subfigure}

  \caption{\footnotesize Comparison of HPD sets (left) and optimal marginal coverage sets (center) under $\trueprior = \delta_{-2}/2 + \delta_{2}/2$ and $\czero = 3/5$ in \eqref{eq:hierarchical model} and confidence level $1-\beta = 0.95$. The red line represents the oracle posterior mean $\E_{G^*}[\theta_i \mid X_i = x]$. (Right) Lengths of the HPD sets (solid) and those of the optimal marginal coverage sets (dashed) as functions of $x$ (the expected lengths of both methods are also noted).} 
  \label{fig:twocomp}
\end{figure}

\appsection{Sensitivity to the smoothing parameter}\label{sec:sensitivity}
In this section, we investigate how using a smoothing parameter $c\ne\czero$ to compute the smooth NPMLE affects the estimated optimal marginal coverage sets. 
Theoretically, if $0 < c \le \czero$, the oracle optimal marginal coverage sets constructed with such $c$ remain optimal marginal coverage sets. This is because, for each $c \le \czero$, there exists $H_{c} \in \Pc(\Real)$ such that $G^* = \trueprior \star N(0,\czero^2) = H_c \star N(0, c^2)$.
In contrast, if $c > \czero$, then the estimated optimal marginal coverage sets are no longer guaranteed to achieve $(1-\beta)$ marginal coverage nor to have the shortest expected length among all sets with $(1-\beta)$ marginal coverage.

In practice, when we estimate the optimal marginal coverage sets using the NPMLE, the choice of $c$ affects the difficulty of estimating $g^*$. This estimation error, in turn, affects their coverage and length. In Figure~\ref{fig:OPT-misspecification}, we can see that the estimated optimal marginal coverage sets based on the NPMLE are very short and have low coverage when $c$ is very small. 
When $c$ is much larger than $\czero$, the estimated optimal marginal coverage sets still achieve $(1-\beta)$ marginal coverage in our example, but their average length is larger than when we use the true $\czero$. This illustrates the importance of accurately estimating $\czero$, as discussed in Section~\ref{sec:identifiability} and Appendix~\ref{sec:identifiability-details}.

\begin{figure}[t]
    \centering
    \begin{subfigure}[b]{0.48\textwidth}
        \centering
        \includegraphics[width=0.95\linewidth]{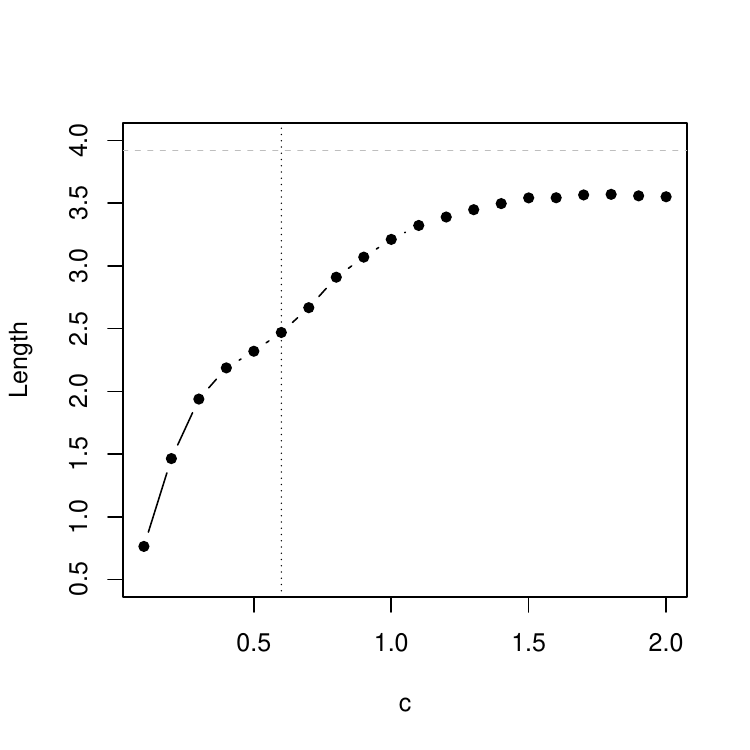}
    \end{subfigure}
    \hfill
    \begin{subfigure}[b]{0.48\textwidth}
        \centering
        \includegraphics[width=0.95\linewidth]{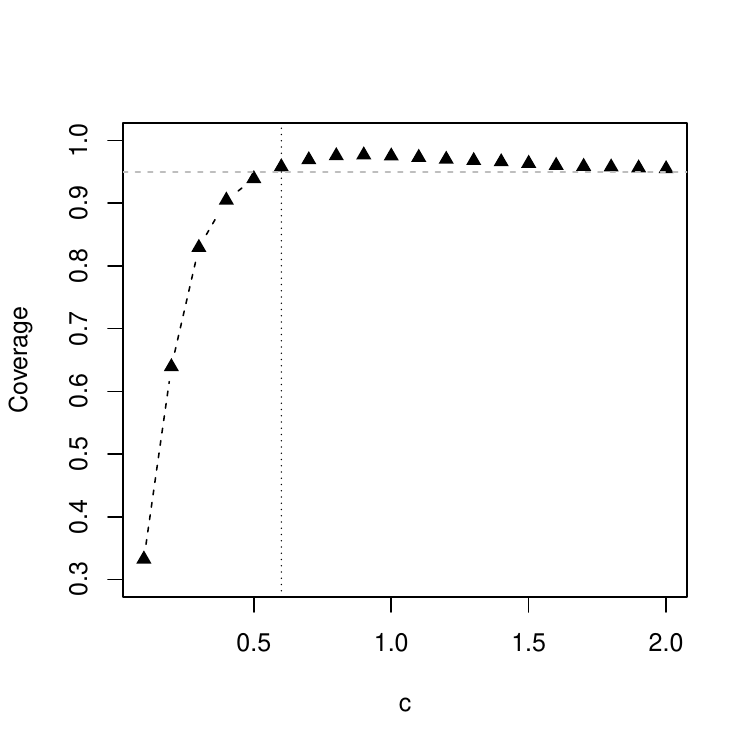}
    \end{subfigure}
    \caption{\footnotesize The true model is $\trueprior = \delta_{-2}/2 + \delta_{2}/2$ with $\czero = 3/5$ under the hierarchical model \eqref{eq:hierarchical model}. (Left) Average length and (right) coverage of the estimated optimal marginal coverage sets based on the smooth NPMLE using $n = 1000$ observations, when $c \in (0, 2)$ is used instead of $\czero$, with $1-\beta = 0.95$. The dashed horizontal lines represent the length and coverage of the $(1-\beta)$ standard confidence interval $X_i \pm z_{1-\beta/2}$. The dashed vertical lines represent $\czero = 3/5$.}
    \label{fig:OPT-misspecification}
\end{figure}

\appsection{Details for the heteroscedastic setting}\label{sec:hetero model details}
In this section, we study the theoretical guarantees of the smooth NPMLE under the heteroscedastic setting of Section~\ref{sec:hetero model}.

\appsubsection{Estimation of the largest Gaussian component}\label{sec:hetero-scale-rate}
We first extend the neighborhood estimator of $\czero$ to the heteroscedastic model \eqref{eq:hetero hierarchical model}. Because the observations are independent but not identically distributed, we compare their empirical distribution with the average true marginal distribution. Let $\F_n$ be the empirical distribution of $X_1,\ldots,X_n$, and define
\begin{align}\label{eq:hetero-pooled-distribution}
    M_n:=\frac1n\sum_{i=1}^nN(0,\sigma_i^2),\qquad
    \bar F_n:=\frac1n\sum_{i=1}^nF_i=G^*\star M_n,
\end{align}
where $F_i=G^*\star N(0,\sigma_i^2)$. The Bretagnolle--Dvoretzky--Kiefer--Wolfowitz inequality gives
\begin{align}\label{eq:hetero-KS-concentration}
    \P\{\KS(\F_n,\bar F_n)>\eta\}\leq 2e\exp(-2n\eta^2),\qquad \eta>0;
\end{align}
see \citet{Bretagnolle1981} and Lemma~7.1 of \citet{Donoho2006}. Thus, the radius $\eta_n$ in \eqref{eq:eta-radius} satisfies $\P\{\KS(\F_n,\bar F_n)\leq\eta_n\}\geq1-2e n^{-2}$.

In the heteroscedastic setting, we retain the known noise mixture $M_n$ in the neighborhood constraint and define the neighborhood estimator by
\begin{align}\label{eq:hetero-neighborhood-estimator}
    \hat c_{\mathrm{KS}}:=\sup\left\{c\geq0:\KS\!\left(\F_n,H\star N(0,c^2)\star M_n\right)\leq\eta_n
    \text{ for some }H\in\Pc(\Real)\right\},
\end{align}
with value zero if the set is empty. Note that, when $\sigma_i^2=1$ for every $i$, this reduces to $\hat c_{\mathrm{KS}}$ in \eqref{eq:estimated-KS-c}. As in the homoscedastic setting, we cap the estimator to restrict the search range and subtract a small data-driven correction to address potential oversmoothing while retaining a sufficiently high fitted marginal likelihood. Let $S_n^2$ be the sample variance in \eqref{eq:cap}, and define
\begin{align*}
    \tilde C_n^2:= \left(S_n^2-\frac1n\sum_{i=1}^n\sigma_i^2 +\frac1{\log n}\right)_+.
\end{align*}
Following the homoscedastic construction, we define the modified estimator by
\begin{equation}\label{eq:hetero-estimated-c}
    \hat c_0:=\sqrt{{(\hat c_{\mathrm{KS}}^2\wedge\tilde C_n^2)}-\widehat C\rho_n},
\end{equation}
where $\rho_n$ is defined in \eqref{eq:rho_n} and $\widehat C$ is selected by
\begin{equation}\label{eq:hetero-estimated-c-buffer}
    \widehat C\in \argmin_{0\leq C\leq{(\hat c_{\mathrm{KS}}^2\wedge\tilde C_n^2)}/\rho_n}\left\{L_n(0)-L_n\!\left(\sqrt{{(\hat c_{\mathrm{KS}}^2\wedge\tilde C_n^2)}-C\rho_n}\right)+b_nC\right\}.
\end{equation}
Here, we retain $b_n$ from Section~\ref{sec:identifiability}, while $L_n(c)$ denotes the maximized marginal log-likelihood under the heteroscedastic model:
\begin{align*}
    L_n(c):=\max_{H\in\Pc(\Real)} \sum_{i=1}^n\log p_{H,\sqrt{\sigma_i^2+c^2}}(X_i), \qquad c\geq0.
\end{align*} 
The next theorem shows that this estimator has the same convergence rate as in the homoscedastic setting. We provide the proof of Theorem~\ref{thm:hetero-estimated-c-rate} in Appendix~\ref{app:proof-hetero-estimated-c-rate}.

\begin{theorem}[Convergence rate of the modified heteroscedastic neighborhood estimator]\label{thm:hetero-estimated-c-rate}
Suppose that Assumption~\ref{ass:moment-fourier} and \eqref{eq:hetero hierarchical model} hold with $\underline{k}\leq\sigma_i^2\leq\bar{k}$ for all $i$. 
There is a constant $K>0$ such that, for all sufficiently large $n$, the event
\begin{align*}
    \Sc_n:=\left\{(\czero^2-K\rho_n)_+\leq\hat c_0^2\leq\czero^2+K\rho_n,\quad\widehat C\leq K,\quad 0\leq L_n(0)-L_n(\hat c_0)\leq Kb_n\right\}
\end{align*}
satisfies
\begin{align}\label{eq:hetero-one-sided-control}
    \P(\Sc_n)\geq 1-3e n^{-2}.
\end{align}
Consequently, $\hat c_0^2-\czero^2=O_P(\rho_n)$ and $\widehat C=O_P(1)$. If $\czero>0$, then $\hat c_0-\czero=O_P(\rho_n)$; if $\czero=0$, then $\hat c_0=O_P(\sqrt{\rho_n})$.
\end{theorem}

\begin{remark}[Computation of the heteroscedastic neighborhood estimator]
We can also apply the linear programming approximation in Appendix~\ref{sec:neighborhood procedure-computation} to the heteroscedastic neighborhood estimator in \eqref{eq:hetero-neighborhood-estimator}. For a grid $\{\theta_j\}_{j=1}^m$ and each candidate $c\in[0,\tilde C_n]$, where $\tilde C_n$ is defined above, replace the matrix in \eqref{eq:LP} by
\begin{align*}
    [\Av(c)]_{\ell j}=\frac1n\sum_{i=1}^n\Phi\!\left(\frac{X_{(\ell)}-\theta_j}{\sqrt{c^2+\sigma_i^2}}\right),
    \qquad \ell=1,\ldots,n,\quad j=1,\ldots,m.
\end{align*}
With this matrix, the linear program in \eqref{eq:LP} checks whether a grid-supported mixing distribution satisfies the heteroscedastic neighborhood constraint. We then approximate the capped estimator $\hat c_{\mathrm{KS}}\wedge\tilde C_n$ by searching over $c\in[0,\tilde C_n]$. We then select $\widehat C$ by \eqref{eq:hetero-estimated-c-buffer} and obtain $\hat c_0$ from \eqref{eq:hetero-estimated-c}.
\end{remark}

\appsubsection{Statistical properties of smooth NPMLE}\label{sec:hetero-smooth-NPMLE}
In this section, we study the smooth NPMLE with the estimated smoothing parameter $\hat c_0$ in \eqref{eq:hetero-estimated-c} under heteroscedastic observations. Under \eqref{eq:hetero hierarchical model}, it can be shown that
\begin{align}\label{eq:hetero full conditional theta}
    \theta_i \mid X_i, \xi_i \sim N\left(\alpha_{0, i} X_i + (1-\alpha_{0,i})\xi_i,  \alpha_{0, i}\sigma_i^2\right) \qquad \mbox{for }\;\; i = 1, \ldots, n,
\end{align}
where (recall $\sigma_{0,i}^2 := \czero^2+ \sigma_i^2$)
\begin{align*}
    \alpha_{0,i} := \frac{\czero^2}{\sigma_{0,i}^2} = \frac{\czero^2}{\czero^2 + \sigma_i^2}.
\end{align*}
Then the oracle posterior mean of $\theta_i$ given $X_i$, under model~\eqref{eq:hetero hierarchical model}, is given by:
\begin{align*}
    \theta_i^* := \E_{G^*}[\theta_i \mid X_i] = \alpha_{0,i}X_i + (1-\alpha_{0,i})\E_{\trueprior}[\xi_i \mid X_i] = \alpha_{0,i}X_i + (1-\alpha_{0,i})\xi_i^*
\end{align*}
where $\xi_i^* := \E_{\trueprior}[\xi_i \mid X_i]$. Write 
\begin{align*}
    \widehat G_n:=\npmle\star N(0,\hat c_0^2)
\end{align*}
where $\hat c_0$ is defined in \eqref{eq:hetero-estimated-c} and $\npmle$ is defined in \eqref{eq:hetero NPMLE}. Then the EB estimate of $\theta_i^*$ is obtained via:
\begin{align*}
    \hat\theta_i := \E_{\widehat G_n}[\theta_i \mid X_i] = \widehat \alpha_{0,i}X_i + (1-\widehat \alpha_{0,i})\E_{\npmle}[\xi_i \mid X_i] = \widehat \alpha_{0,i}X_i + (1-\widehat\alpha_{0,i})\hat\xi_i
\end{align*}
where $\widehat \alpha_{0,i}:= \hat c_0^2/(\hat c_0^2+\sigma_i^2)$ and $\hat\xi_i := \E_{\npmle}[\xi_i \mid X_i]$.
For $i=1,\ldots,n$, define the fitted marginal density and the averaged squared Hellinger error by
\begin{align}\label{eq:hetero-fitted-marginals}
    \widehat f_{n,i}:=p_{\widehat G_n,\sigma_i}
    =p_{\npmle,\sqrt{\sigma_i^2+\hat c_0^2}},\qquad
    \overline\H^2(\widehat f_n,f^*):=\frac1n\sum_{i=1}^n\H^2(\widehat f_{n,i},f_i^*),
\end{align}
where $f_{i}^*= \trueprior \star N(0,\czero^2+\sigma_i^2)$. The following theorem extends the marginal density estimation and denoising guarantees to the heteroscedastic setting with an estimated smoothing parameter. We provide the proof of Theorem~\ref{thm:hetero denoising} in Appendix~\ref{app:proof-hetero-density-denoising}.

\begin{theorem}[Near-parametric rates for marginal density estimation and denoising with heteroscedastic observations]\label{thm:het-density-denoising}\label{thm:hetero denoising}
Suppose that Assumption~\ref{ass:moment-fourier} and \eqref{eq:hetero hierarchical model} hold with $\underline{k}\leq\sigma_i^2\leq\bar{k}$ for all $i$. Let $\hat c_0$ be defined in \eqref{eq:hetero-estimated-c}. Then, for all sufficiently large $n$,
\begin{align}\label{eq:hetero density estimation}
    \E[\overline\H^2(\widehat f_n,f^*)]
    \lesssim_{\czero,\boldsymbol{\gamma},\underline{k},\bar{k}}\frac{(\log n)^{3/2+1/r}}{n}.
\end{align}
Moreover, for the posterior means above and the regret in \eqref{eq:regret},
\begin{align}\label{eq:hetero regret bound}
    \Rc_n(\hat\theta,\theta^*)
    \lesssim_{\czero,\boldsymbol{\gamma},\underline{k},\bar{k}}\frac{(\log n)^{9/2+1/r}}{n}.
\end{align}
\end{theorem}

We next estimate the prior density. Suppose $\czero>0$. As in \eqref{eq:deconvolution-floor}, we ensure that the estimated smoothing parameter is positive by imposing a vanishing lower bound. For $n\geq3$, define
\begin{equation}\label{eq:hetero-deconvolution-floor}
    \hat c_0^+:=\hat c_0\vee\sqrt{\rho_n},
\end{equation}
where $\rho_n$ is defined in \eqref{eq:rho_n}. For prior and posterior density estimation, define the stabilized NPMLE by
\begin{equation}\label{eq:hetero-stabilized-NPMLE}
    \npmle^+\in\argmax_{H\in\Pc(\Real)}\sum_{i=1}^n\log p_{H,\sqrt{\sigma_i^2+(\hat c_0^+)^2}}(X_i).
\end{equation}
Set $\widehat g_n^+:=\npmle^+\star\phi_{\hat c_0^+}$, reusing $\npmle$ whenever $\hat c_0^+=\hat c_0$. When $\czero>0$, the lower bound in Theorem~\ref{thm:hetero-estimated-c-rate} ensures that $\hat c_0>\sqrt{\rho_n}$ with high probability for all sufficiently large $n$. Hence, $\hat c_0^+=\hat c_0$ and $\npmle^+=\npmle$ with high probability. Write
\begin{align*}
    \bar\alpha_0:=\frac{\czero^2}{\czero^2+\bar{k}}.
\end{align*}
The following theorem (proved in Appendix~\ref{app:proof-het-prior}) shows that $\widehat g_n^+$ retains a polynomial deconvolution rate. 

\begin{theorem}[Polynomial deconvolution rate with heteroscedastic observations]\label{thm:het-prior}\label{thm:hetero deconvolution upper bound}
Suppose that $\czero>0$ and that Assumption~\ref{ass:moment-fourier} and \eqref{eq:hetero hierarchical model} hold with $\underline{k}\leq\sigma_i^2\leq\bar{k}$ for all $i$. Let $\hat c_0^+$ be defined in \eqref{eq:hetero-deconvolution-floor}. 
There exists a constant $k>0$, depending only on $\czero$, $\gammav$, $\underline{k}$, and $\bar{k}$, such that, for all sufficiently large $n$,
\begin{align}\label{eq:hetero MISE convergence rate}
    \E[\|\widehat g_n^+-g^*\|_{L_2}^2]
    \lesssim_{\czero,\gammav,\underline{k},\bar{k}} n^{-\bar\alpha_0}\exp\{-k(\log n)^{r/2}\}.
\end{align}
When $r=2$, $k$ can be chosen so that $\bar\alpha_0+k<1$.
\end{theorem}

In the heteroscedastic setting, the upper bound $\bar{k}$ on the observation variances determines the polynomial exponent $\bar\alpha_0=\czero^2/(\czero^2+\bar{k})$, in place of $\alpha_0$ in \eqref{eq:alphastar} for homoscedastic observations.

We use the preceding prior density result to establish a convergence rate for posterior density estimation under \eqref{eq:hetero hierarchical model}. For $i=1,\ldots,n$, define the oracle posterior density by
\begin{align}\label{eq:hetero posterior distribution}
    \pi_i^*(\theta\mid x):=\frac{\phi_{\sigma_i}(x-\theta)g^*(\theta)}{f_i^*(x)},
\end{align}
and its plug-in estimator, based on the stabilized $\hat c_0^+$ in \eqref{eq:hetero-deconvolution-floor}, by
\begin{align}\label{eq:hetero estimated posterior density}
    \widehat\pi_{n,i}^+(\theta\mid x):=\frac{\phi_{\sigma_i}(x-\theta)\widehat g_n^+(\theta)}{\widehat f_{n,i}^+(x)},
\end{align}
where $\widehat f_{n,i}^+:=p_{\npmle^+,\sqrt{\sigma_i^2+(\hat c_0^+)^2}}=\widehat g_n^+\star\phi_{\sigma_i}$ is the corresponding fitted marginal density based on the stabilized scale $\hat c_0^+$. We measure their error by the averaged weighted total variation distance
\begin{align}\label{eq:hetero-posterior-loss}
    \overline{\mathrm{wTV}}(\widehat\pi_n^+,\pi^*)
    :=\frac1n\sum_{i=1}^n\int
    \mathrm{TV}\{\widehat\pi_{n,i}^+(\cdot\mid x),\pi_i^*(\cdot\mid x)\}
    f_i^*(x)\,\diff x.
\end{align}
The posterior density estimation rate follows from the squared $L_2$ risk bound for prior density estimation in Theorem~\ref{thm:het-prior}. We give the proof in Appendix~\ref{app:proof-het-posterior}.

\begin{theorem}[Polynomial rate for posterior density estimation with heteroscedastic observations]\label{thm:het-posterior}\label{thm:hetero posterior convergence rate}
Under the conditions of Theorem~\ref{thm:het-prior}, there exists a constant $k>0$, depending only on $\czero$, $\gammav$, $\underline{k}$, and $\bar{k}$, such that, for all sufficiently large $n$,
\begin{align}\label{eq:hetero-posterior-rate}
    \E[\overline{\mathrm{wTV}}(\widehat\pi_n^+,\pi^*)]
    \lesssim_{\czero,\gammav,\underline{k},\bar{k}} n^{-\bar\alpha_0/2}\exp\{-k(\log n)^{r/2}\}.
\end{align}
\end{theorem}
Further, the optimal marginal coverage sets can be estimated using $\widehat\pi_{n,i}^+(\theta\mid x)$ for each $i=1,\ldots,n$. For each $i$, we consider the marginal coverage problem separately: among set-valued rules $\Ic_i$ satisfying $\P_{G^*}\{\theta_i\in \Ic_i(X_i)\}\geq 1-\beta$, we seek to minimize $\E_{G^*}|\Ic_i(X_i)|$. Similar to Section~\ref{sec:OPT estimation}, we use the EB marginal coverage set based on the NPMLE:
\begin{align*}
    \hat{\Ic}_{n,i}(x):=\{\theta\in\Real:\widehat\pi_{n,i}^+(\theta\mid x)\ge\hat{k}_{n,i}\},
    \qquad i=1,\ldots,n,
\end{align*}
where $\widehat\pi_{n,i}^+$ is defined in \eqref{eq:hetero estimated posterior density}. Analogously to \eqref{eq:Est-CI}, $\hat{k}_{n,i}$ is obtained from
\begin{align*}
    \iint\1v\{\widehat\pi_{n,i}^+(\theta\mid x)\ge\hat{k}_{n,i}\}
    \widehat g_n^+(\theta)\phi_{\sigma_i}(x-\theta)\,\diff\theta\,\diff x=1-\beta,
\end{align*}
where the fitted joint model has prior density $\widehat g_n^+$ and $X\mid\theta\sim N(\theta,\sigma_i^2)$. Note that $\hat{\Ic}_{n,i}$ depends on $\sigma_i$ and may differ across $i$ when the observation variances are unequal.

\appsection{Numerical studies}\label{sec:numerical results}
\appsubsection{Simulation}\label{sec:simulation}
We first conduct a simulation study to evaluate the performance of optimal marginal coverage sets. We consider three scenarios. First, we consider a two-component normal mixture prior, i.e., $G^* = \frac{1}{2} \cdot N(-a, \czero^2)  + \frac{1}{2} \cdot N(a,\czero^2)$ in \eqref{eq:original model} (equivalently, $\trueprior = \delta_{-a} /2 + \delta_{a}/2$ in \eqref{eq:hierarchical model}) for $a = 0, 1, 2, 3$, and $\czero = 1$. Next, we consider the same two-component normal mixture prior, but we assume that the observations do not have equal variance, i.e., $X_i \mid \theta_i \overset{ind}\sim N(\theta_i, \sigma_i^2)$, as in Section~\ref{sec:hetero model}. Specifically, we set $\sigma_i \in \{\sqrt{1/2}, \sqrt{3/4}, 1, \sqrt{2}\}$, each with probability $1/4$ (drawn independently of $\theta_i$). Lastly, we consider the prior $G^*$ to be a Laplace distribution and a Gamma distribution. As mentioned in the Introduction, these priors have $\czero = 0$. However, in Figure~\ref{fig:SNPML-laplace}, we saw that the Laplace prior is well approximated by the smooth NPMLE. We examine whether optimal marginal coverage sets perform well in such misspecified settings. In all settings, we use $n = 1000$ and estimate $\czero$ based on the modified neighborhood procedure described in Section~\ref{sec:identifiability}.

\begin{table}[!htbp]\footnotesize
    \begin{center}
        \begin{tabular}{c|cc|cc|cc|cc}
            \multirow{2}{*}{$\sigma_i = 1$} & \multicolumn{2}{c}{$a = 0$} & \multicolumn{2}{|c}{$a = 1$} & \multicolumn{2}{|c}{$a = 2$} & \multicolumn{2}{|c}{$a = 3$} \\
            \cline{2-9}
            & Cov. & Len. & Cov. & Len. & Cov. & Len. & Cov. & Len. \\
            \hline                        
            \multirow{2}{*}{Oracle($\czero$)} & 0.950 & 2.776 & 0.951 & 3.177 & 0.950 & 3.250 & 0.950 & 2.930 \\
                               & (0.008) & (0.038) & (0.007) & (0.043) & (0.007) & (0.051) & (0.008) & (0.049) \\
            \hline
            \multirow{2}{*}{NPMLE($\czero$)}  & 0.952 & 2.810 & 0.950 & 3.184 & 0.952 & 3.285 & 0.952 & 2.983 \\
                               & (0.008) & (0.052) & (0.008) & (0.051) & (0.008) & (0.063) & (0.008) & (0.075) \\
            \hline
            \multirow{2}{*}{NPMLE($\hat c_0^+$)} & 0.953 & 2.831 & 0.952 & 3.217 & 0.952 & 3.310 & 0.955 & 3.026 \\
                                & (0.010) & (0.081) & (0.009) & (0.063) & (0.012) & (0.110) & (0.010) & (0.111) \\              
        \end{tabular}
    \end{center}
    \caption{\footnotesize Averages of coverage probabilities (left in each cell) and lengths (right) for optimal marginal coverage sets (with nominal level $0.95$) under the normal hierarchical model \eqref{eq:hierarchical model} with $\trueprior = \delta_{-a}/2 + \delta_{a}/2$, $\czero = 1$ and $n = 1000$. Standard deviations are presented in parentheses. 
    }
    \label{table:simulation-homo}
\end{table}

\begin{table}[!htbp]\footnotesize
    \begin{center}
        \begin{tabular}{c|cc|cc|cc|cc}
            \multirow{2}{*}{$\sigma_i \in \{\sqrt{1/2}, \sqrt{3/4}, 1, \sqrt{2}\}$} & \multicolumn{2}{c}{$a = 0$} & \multicolumn{2}{|c}{$a = 1$} & \multicolumn{2}{|c}{$a = 2$} & \multicolumn{2}{|c}{$a = 3$} \\
            \cline{2-9}
            & Cov. & Len. & Cov. & Len. & Cov. & Len. & Cov. & Len. \\
            \hline
            \multirow{2}{*}{Oracle($\czero$)} &  0.949 & 2.702 & 0.949 & 3.092 & 0.949 & 3.204 & 0.948 & 2.888 \\
                                     & (0.007) & (0.020) & (0.008) & (0.025) & (0.007) & (0.033) & (0.007) & (0.027) \\
            \hline
            \multirow{2}{*}{NPMLE($\czero$)} & 0.951 & 2.729 & 0.949 & 3.101 & 0.950 & 3.237 & 0.951 & 2.936 \\
                                     & (0.007) & (0.035) & (0.007) & (0.036) & (0.007) & (0.046) & (0.007) & (0.054) \\
            \hline
            \multirow{2}{*}{NPMLE($\hat c_0^+$)} & 0.952 & 2.751 & 0.951 & 3.139 & 0.952 & 3.264 & 0.953 & 2.977 \\
                                     & (0.008) & (0.065) & (0.008) & (0.050) & (0.008) & (0.073) & (0.008) & (0.083) \\          
        \end{tabular}
    \end{center}
    \caption{\footnotesize Averages of coverage probabilities (left in each cell) and lengths (right) for optimal marginal coverage sets (with nominal level $0.95$) under the normal heteroscedastic hierarchical model \eqref{eq:hetero hierarchical model} with $\trueprior = \delta_{-a}/2 + \delta_{a}/2$, $\czero = 1$, $\sigma_i \in \{\sqrt{1/2}, \sqrt{3/4}, 1, \sqrt{2}\}$ and $n = 1000$. Standard deviations are presented in parentheses. }
    \label{table:simulation-hetero}
\end{table}

\begin{table}[!htbp]\footnotesize
    \begin{center}
        \begin{tabular}{c|cc|cc}
            \multirow{2}{*}{$\sigma_i = 1$} & \multicolumn{2}{c}{$\mathrm{Laplace}(0,1)$} & \multicolumn{2}{|c}{$\mathrm{Gamma}(1,1)$}  \\
            \cline{2-5}
            & Cov. & Len. & Cov. & Len. \\
            \hline             
            \multirow{2}{*}{NPMLE($\hat c_0^+$)} & 0.950 & 3.188 & 0.947 & 2.484 \\
                                  & (0.009) & (0.083) & (0.014) & (0.147) \\
        \end{tabular}
    \end{center}
    \caption{\footnotesize Averages of coverage probabilities (left in each cell) and lengths (right) for optimal marginal coverage sets (with nominal level $0.95$) under the normal location mixture model \eqref{eq:original model} with $G^* = \mathrm{Laplace}(0,1)$ or $G^* = \mathrm{Gamma}(1,1)$ and $n = 1000$. Standard deviations are presented in parentheses. }
    \label{table:simulation-misspec}
\end{table}

\begin{table}[!htbp]\footnotesize
\begin{center}
\begin{tabular}{c|cccc}
     & $a=0$ & $a=1$ & $a=2$ & $a=3$ \\
    \hline
    Table~\ref{table:simulation-homo}
    & \begin{tabular}[c]{@{}c@{}}1.030\\(0.081)\end{tabular}
    & \begin{tabular}[c]{@{}c@{}}1.311\\(0.235)\end{tabular}
    & \begin{tabular}[c]{@{}c@{}}1.052\\(0.123)\end{tabular}
    & \begin{tabular}[c]{@{}c@{}}1.054\\(0.078)\end{tabular} \\
    \hline
    Table~\ref{table:simulation-hetero}
    & \begin{tabular}[c]{@{}c@{}}1.026\\(0.066)\end{tabular}
    & \begin{tabular}[c]{@{}c@{}}1.340\\(0.213)\end{tabular}
    & \begin{tabular}[c]{@{}c@{}}1.051\\(0.090)\end{tabular}
    & \begin{tabular}[c]{@{}c@{}}1.044\\(0.059)\end{tabular}
\end{tabular}
\end{center}
\caption{\footnotesize Estimates of $\hat{c}_0$ for the homoscedastic and heteroscedastic simulation settings. Each entry reports the mean with standard deviation in parentheses.}
\label{table:chat0-homo-hetero}
\end{table}

\begin{table}[!htbp]\footnotesize
\begin{center}
\begin{tabular}{c|cc}
 & $\mathrm{Laplace}(0,1)$ & $\mathrm{Gamma}(1,1)$ \\
\hline
Table~\ref{table:simulation-misspec}
& \begin{tabular}[c]{@{}c@{}}1.306\\(0.149)\end{tabular}
& \begin{tabular}[c]{@{}c@{}}0.645\\(0.133)\end{tabular}
\end{tabular}
\end{center}
\caption{\footnotesize Estimates of $\hat{c}_0$ for the misspecified simulation settings. Each entry reports the mean with standard deviation in parentheses.}
\label{table:chat0-misspec}
\end{table}

Tables~\ref{table:simulation-homo}--\ref{table:simulation-misspec} report the empirical coverage probabilities and lengths of the oracle and estimated optimal marginal coverage sets averaged over 100 repetitions. Tables~\ref{table:chat0-homo-hetero} and~\ref{table:chat0-misspec} report the corresponding estimates of $\hat{c}_0$ for each setting using the modified neighborhood procedure. Our methods perform well, both under the homoscedastic and heteroscedastic settings (Tables~\ref{table:simulation-homo} and~\ref{table:simulation-hetero}). In particular, the performance of the optimal marginal coverage sets constructed using the NPMLE fitted at the stabilized smoothing parameter $\hat c_0^+$ in \eqref{eq:deconvolution-floor} is comparable to that of the oracle optimal marginal coverage sets. Interestingly, optimal marginal coverage sets also perform reasonably well under Laplace and Gamma priors (Table~\ref{table:simulation-misspec}). This suggests that smoothing yields reliable uncertainty quantification even when the true prior does not belong to the assumed model class. 

\begin{figure}[t]
    \centering
    \begin{subfigure}[t]{0.49\textwidth}
        \centering
        \includegraphics[width=\linewidth]{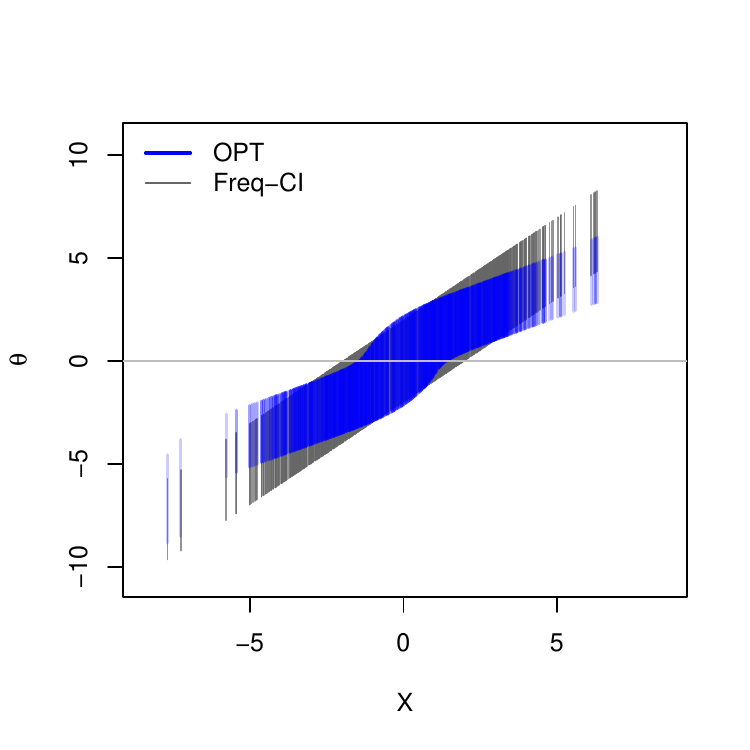}
        \caption{$G^* = N(-2,1)/2 + N(2,1)/2$}
        \label{fig:simulation-twocomp}
    \end{subfigure}
    \hfill
    \begin{subfigure}[t]{0.49\textwidth}
        \centering
        \includegraphics[width=\linewidth]{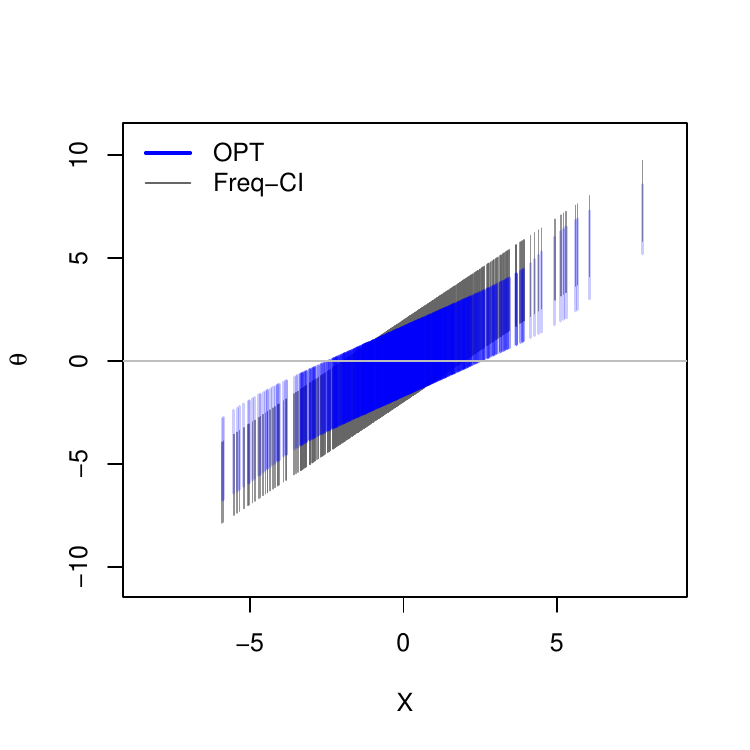}
        \caption{$G^* = \mathrm{Laplace}(0,1)$}
        \label{fig:simulation-laplace}
    \end{subfigure}
    \caption{\footnotesize Comparison of 95\% estimated optimal marginal coverage sets and standard frequentist confidence intervals $X_i \pm z_{0.975}$.}
    \label{fig:simulation}
\end{figure}

Figure~\ref{fig:simulation} illustrates a sharp contrast between the optimal marginal coverage sets and the standard frequentist confidence intervals $X_i \pm z_{1-\beta/2}$ for the examples used in Tables~\ref{table:simulation-homo} and~\ref{table:simulation-misspec}. As mentioned earlier, the optimal marginal coverage sets tend to be shorter than the standard frequentist confidence intervals. For example, in Table~\ref{table:simulation-homo} with nominal level $0.95$, when $G^*$ is the standard normal distribution (i.e., $a = 0$), the standard frequentist confidence interval $X_i \pm z_{0.975}$ has fixed length $2z_{0.975} \approx 3.92$, while the optimal marginal coverage set has average length approximately $2.83$. That is, in this case, the optimal marginal coverage set achieves roughly a $28\%$ reduction in average length. Moreover, when $a =3$ so that $G^*$ is highly bimodal, the optimal marginal coverage set has average length approximately $3.03$, which is about a $23\%$ reduction relative to the standard frequentist confidence interval. The optimal marginal coverage sets adapt to the shape of the underlying prior density (see also the center panels of Figures~\ref{fig:SNPML-TwoComp} and~\ref{fig:SNPML-laplace}), while the standard frequentist confidence intervals always have constant length. Under the bimodal prior in Figure~\ref{fig:SNPML-TwoComp}, when $X_i$ is near zero, it is harder to tell where $\theta_i$ came from. Thus, the posterior is more diffused and the optimal marginal coverage set is longer. By contrast, when $X_i$ is far from zero, it is easier to tell where $\theta_i$ came from and thus the posterior is more concentrated and the optimal set becomes shorter. This illustrates how the optimal marginal coverage sets adapt to the underlying prior structure, while still achieving the shortest expected length among all sets satisfying the marginal coverage constraint.

\begin{figure}[!t]
    \centering
    \includegraphics[width=\textwidth]{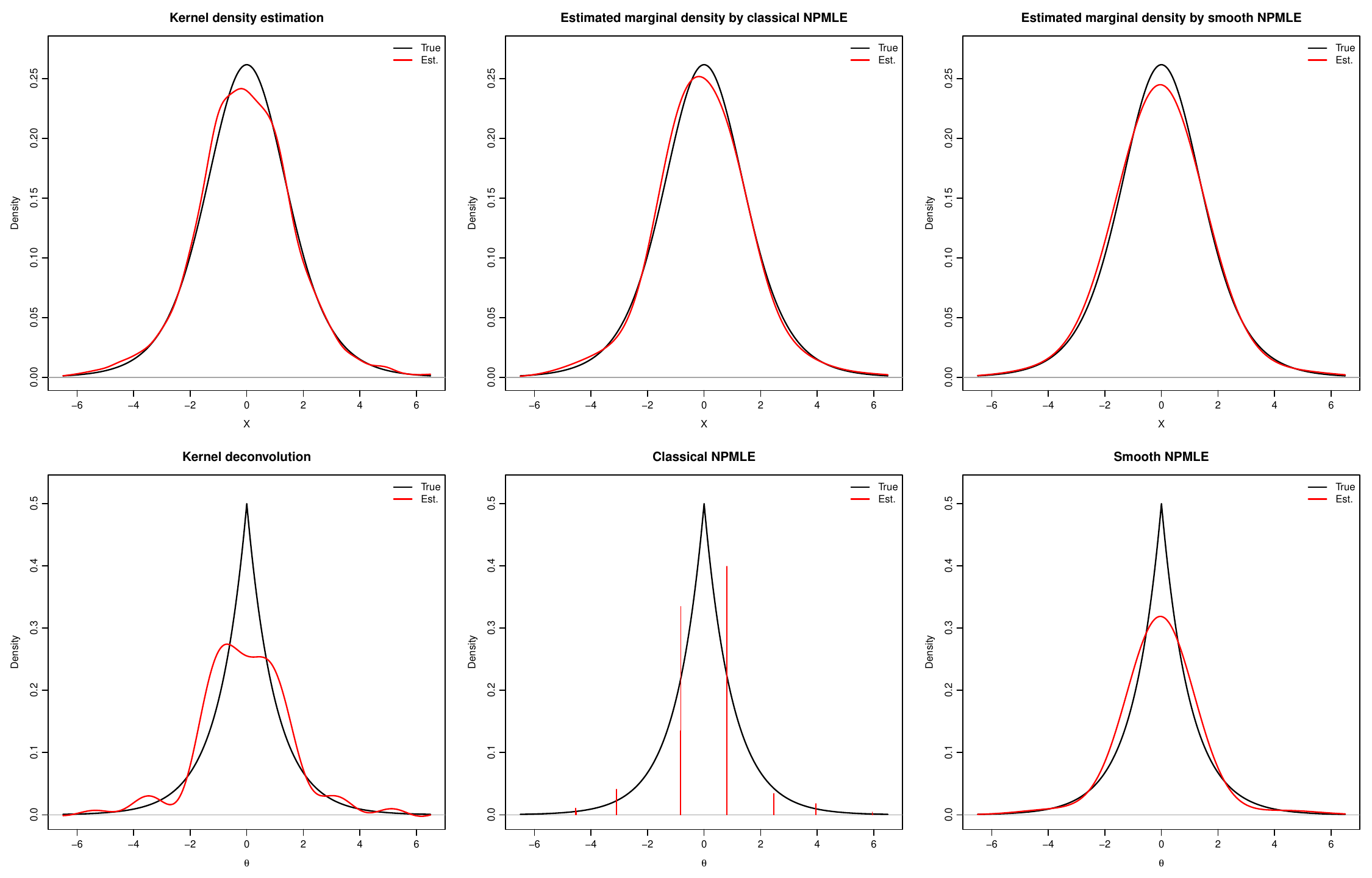}
    \caption{\footnotesize Marginal and prior density estimation under misspecification with $G^*=\mathrm{Laplace}(0,1)$ in model~\eqref{eq:original model} and $n=2000$. The top row compares marginal density estimates from a Gaussian kernel density estimator with a Sheather--Jones bandwidth \citep{Sheather1991} (left), the classical NPMLE (center), and the smooth NPMLE fitted at $\hat c_0\approx1.195$ (right). The bottom row compares prior density estimates from a kernel deconvolution estimator with the plug-in bandwidth of \citet{DelaigleGijbels2004} (left), the classical NPMLE (center), and the smooth NPMLE fitted at $\hat c_0^+$ (right). The true densities are shown in black, and the estimates are shown in red.}
    \label{fig:misspecification-marginal-comparison}
\end{figure}

We further illustrate marginal and prior density estimation under misspecification in Figure~\ref{fig:misspecification-marginal-comparison}, comparing the smooth NPMLE with kernel methods and the classical NPMLE. The top row compares the marginal density estimates obtained from a Gaussian kernel density estimator, the classical NPMLE, and the smooth NPMLE under $G^*=\mathrm{Laplace}(0,1)$ with $n=2000$. Although $c_0=0$, the smooth NPMLE fitted at $\hat c_0$ closely approximates the true marginal density in this example, as do the two other estimators. The bottom row compares the prior density estimates for this example. Our smooth NPMLE provides smooth prior density estimates by fitting at the stabilized smoothing parameter $\hat c_0^+$. In contrast, the classical NPMLE produces a discrete prior estimate, while the kernel deconvolution estimate exhibits some oscillation and asymmetry in this example and takes slightly negative values in the tails, requiring modification before being used as a prior density.

\appsubsection{Application to an Educational Longitudinal Study}\label{sec:applications-ELS}
We apply our smooth NPMLE and estimated optimal marginal coverage sets to the dataset from the 2002 Educational Longitudinal Study (ELS). 
We use a survey from $n= 100$ different schools including math test scores and normalized socioeconomic status (SES) of $\sum_{i=1}^{n} N_i = 1,993$ 10th grade students across the United States. The number of students $N_i$ surveyed in each school varies ranging from 4 to 32 with a median of 20 students. As in \citet{Soloff2025}, we can apply the NPMLE to hierarchical linear models when the heteroscedastic variances $\sigma_i^2$ are known or can be accurately estimated. Specifically, we consider the following working linear model:
\begin{align*}
    y_{ij} = X_{ij}\beta_i + \epsilon_{ij},\qquad\text{where}\quad \beta_i \overset{iid}\sim G^*\quad  \text{and}\quad \epsilon_{ij} \overset{iid}\sim N(0, \sigma^2).
\end{align*}
where the response $y_{ij}$ represents the centered math score of student $j$ in school $i$, and $X_{ij}$ is the centered SES score of student $j$ in school $i$. Writing $y_i = (y_{i1}, \ldots, y_{iN_i}) \in \Real^{N_i}$ and $X_i = (X_{i1}, \ldots, X_{iN_i}) \in \Real^{N_i}$, we can write the model as
\begin{align*}
    y_i \mid \beta_i \overset{ind}\sim N(X_i\beta_i, \sigma^2 I_{N_i}),\qquad \text{with}\quad \beta_i \overset{iid}\sim G^*, \qquad \mbox{for }\;\; i = 1, \ldots, n.
\end{align*}
Using the ordinary least squares (OLS) solution $b_i = (X_i^\top X_i)^{-1}X_i^\top y_i$, we can write
\begin{align*}
b_i \mid \beta_i \overset{ind}\sim N(\beta_i, \sigma_i^2), \qquad\text{with}\quad \beta_i \overset{iid}\sim G^*, \qquad \mbox{for }\;\; i = 1, \ldots, n
\end{align*}
where $\sigma_i^2 = (X_i^\top X_i)^{-1}\sigma^2$. We can estimate $\sigma^2$ with
\begin{align*}
    \hat\sigma^2 = \frac{1}{\sum_{i=1}^{n}(N_i - 1)} \sum_{i=1}^{n} \|y_i - X_i b_i \|_{2}^2.
\end{align*}

\begin{figure}[t]
    \centering
    \begin{tabular}{cc}
    \includegraphics[width=0.4\textwidth]{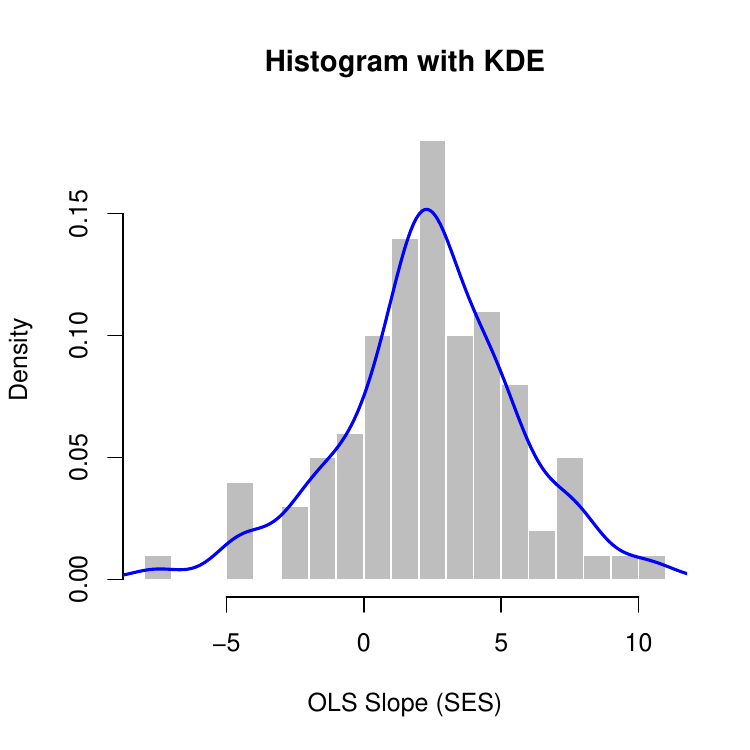} &
    \includegraphics[width=0.4\textwidth]{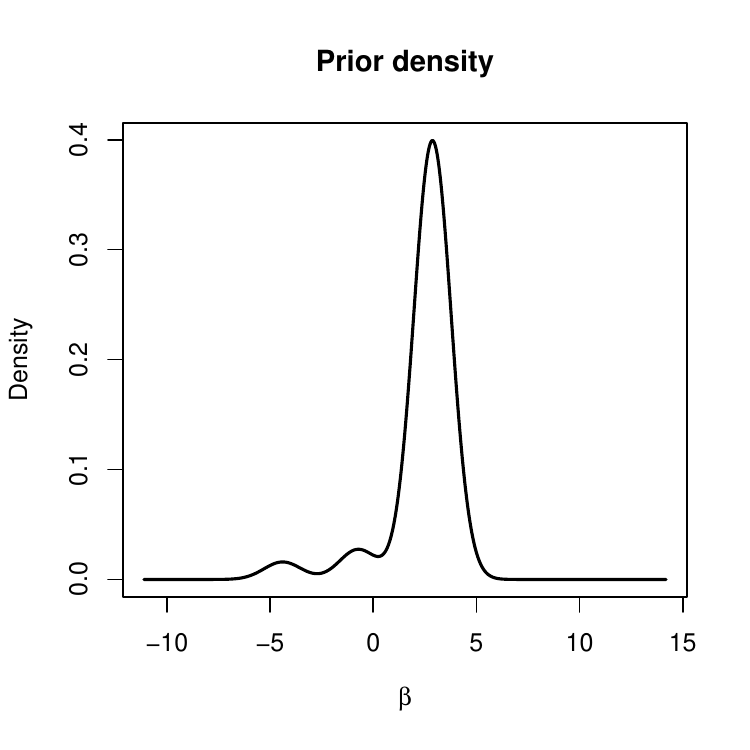}  \\
    \includegraphics[width=0.4\textwidth]{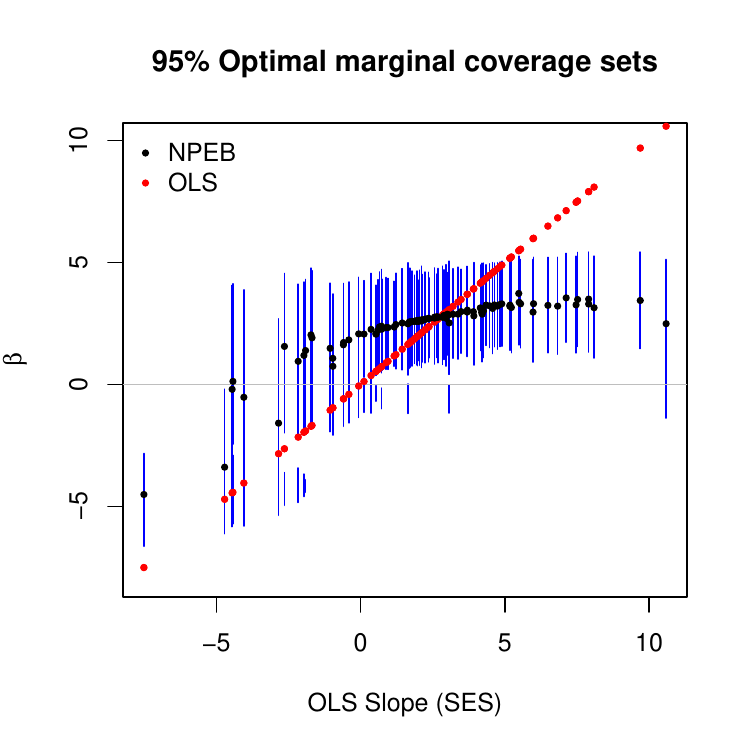} &
    \includegraphics[width=0.4\textwidth]{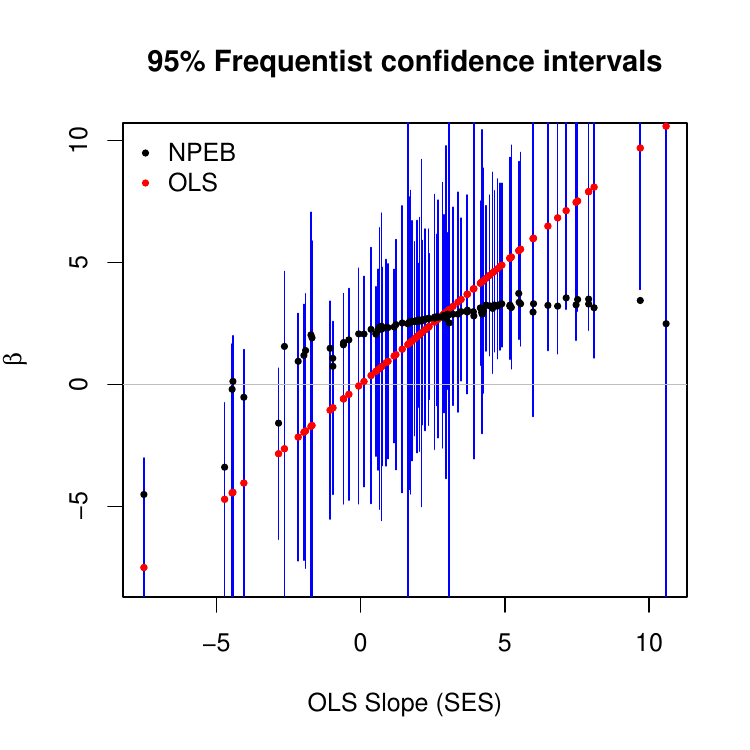}
    \end{tabular}

    \caption{\footnotesize Empirical Bayes analysis of school-specific regression coefficients in the math scores dataset. The panels show the histogram of observations, the smooth NPMLE, estimated 95\% optimal marginal coverage sets and standard frequentist confidence intervals $b_i \pm \sigma_iz_{0.975}$ along with empirical Bayes estimates.}
    \label{fig:SES_all}
\end{figure}

\citet{Soloff2025} provide EB estimates of separate regression coefficients $\beta_i$ for each school $i= 1, \ldots, n$, which are more reasonable than the corresponding OLS estimates. We go beyond point estimation and provide confidence sets for $\beta_i$. We provide an illustration of $95\%$ optimal marginal coverage sets for $\beta_i$ in Figure~\ref{fig:SES_all}. Using the modified neighborhood procedure in Section~\ref{sec:identifiability}, with $5$-fold cross-validation, we obtain $\hat{c}_0 = 0.90$. We can see that the smooth NPMLE is multimodal with small modes around $-4$ and $-1$ and a large mode around $3$. From this, we obtain non-trivial confidence sets for the school-specific coefficients: the average length of the estimated 95\% optimal marginal coverage sets is $4.44$, and $79$ of the optimal marginal coverage sets do not include zero. In contrast, the average length of $95\%$ frequentist confidence intervals $b_i \pm \sigma_iz_{0.975}$ is $10.07$, and $28$ of the confidence intervals do not include zero.

\appsubsection{Application to prostate data}\label{sec:applications-prostate}
Next, we apply our smooth NPMLE and optimal marginal coverage sets to the prostate dataset from \citet{Singh2002}, which has been widely used in EB literature (see, e.g., \citet{Efron2009, efron2016, Ignatiadis2022}). This dataset contains microarray gene-expression measurements for $n = 6033$ genes from 52 healthy men and 50 prostate cancer patients. We compute a two-sample $t$-statistic $T_i$ for each gene $i$, and transform it to a $z$-score via $X_i = \Phi^{-1}(F_{100}(T_i))$ where $F_{100}(\cdot)$ is the cdf of the $t$-distribution with 100 degrees of freedom. We then consider the hierarchical normal mixture model \eqref{eq:hierarchical model} where $X_i$ is the $z$-score for gene $i$ and $\theta_i$ is the standardized effect size. 

\begin{figure}[t]
    \centering
    \begin{tabular}{ccc}
    \includegraphics[width=0.4\textwidth]{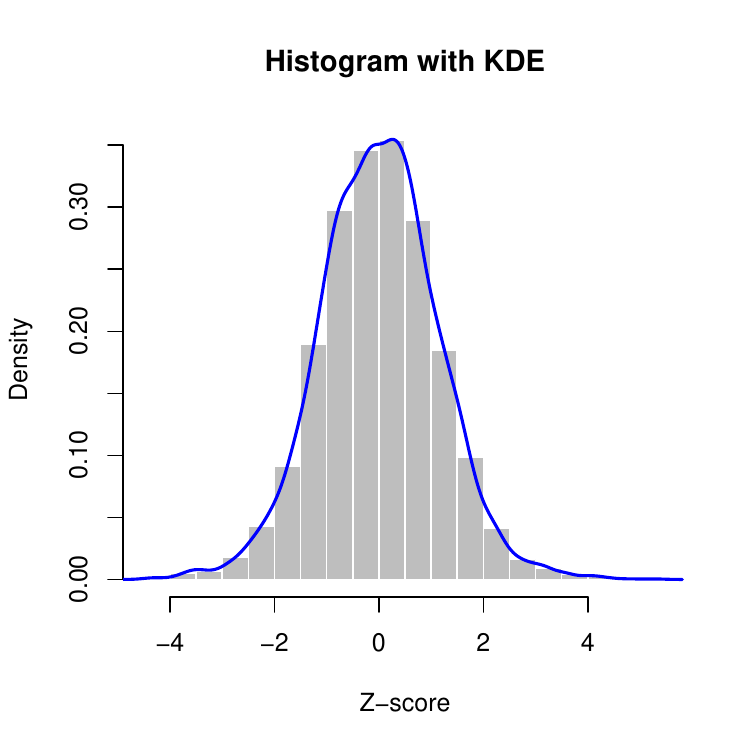} &
    \includegraphics[width=0.4\textwidth]{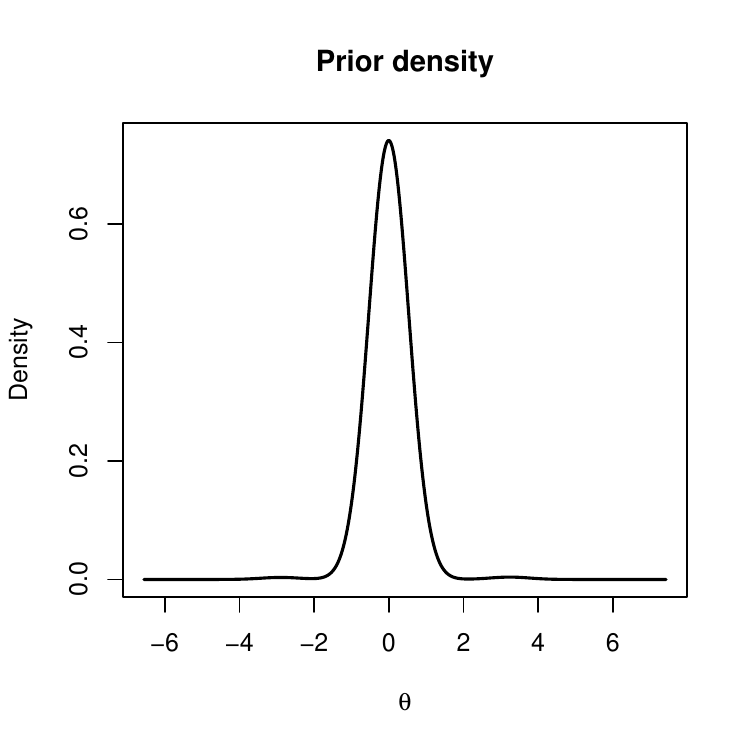} \\
    \includegraphics[width=0.4\textwidth]{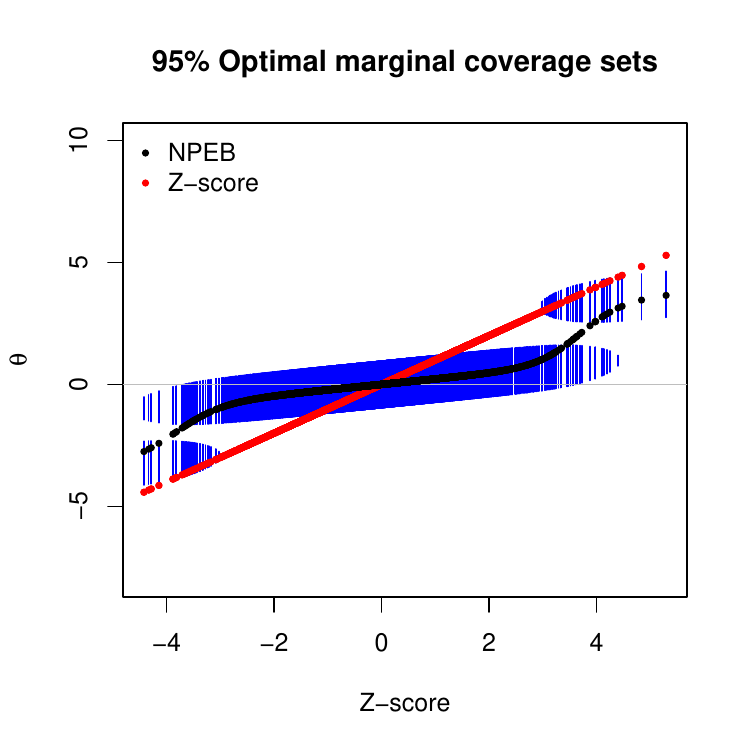} &
    \includegraphics[width=0.4\textwidth]{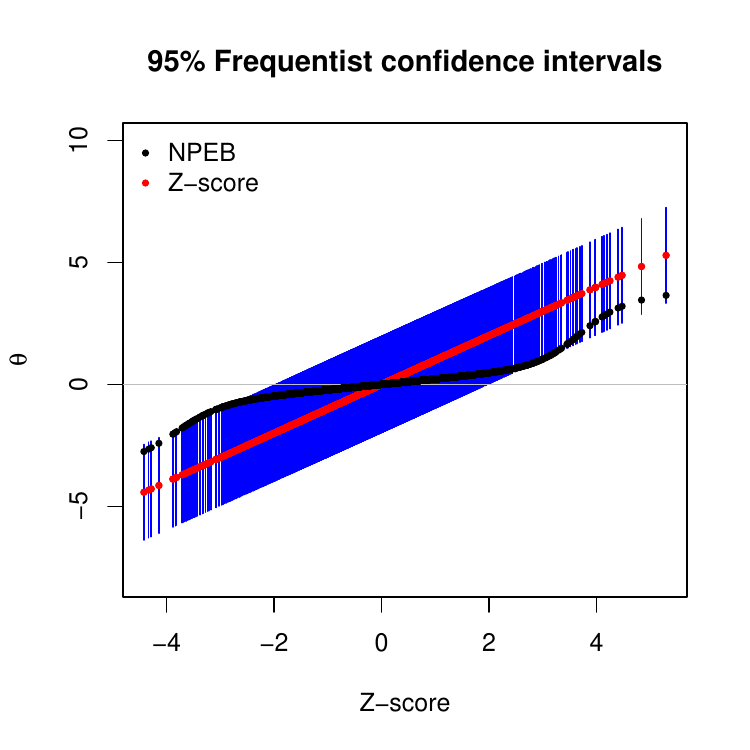}
    \end{tabular}

    \caption{\footnotesize Empirical Bayes analysis of standardized effect sizes in the prostate dataset. The panels show the histogram of observations, the smooth NPMLE, 95\% optimal marginal coverage sets and standard frequentist confidence intervals $X_i \pm z_{0.975}$ along with empirical Bayes estimates.}
    \label{fig:prostate_all}
\end{figure}

We provide an illustration of optimal marginal coverage sets for $\theta_i$ in Figure~\ref{fig:prostate_all}. We use an estimate $\hat{c}_0 = 0.53$ obtained using the modified neighborhood procedure in Section~\ref{sec:identifiability} with 5-fold cross-validation. It turns out that most of the mass of the NPMLE $\npmle$ is concentrated at zero, but very small amount of mass is placed at $\pm 3$. Consequently, the smooth prior looks like a normal distribution, but it has a small bump near $\pm 3$. This leads to a difference between the linear shrinkage estimators and the EB estimates, especially for extreme $z$-scores. Also, the optimal marginal coverage sets for extreme $z$-scores are disjoint unions of intervals. The averaged length of $95\%$ optimal marginal coverage sets is $1.91$, and $27$ of them do not include zero. In contrast, the length of the $95\%$ frequentist confidence intervals $X_i \pm z_{0.975}$ is $3.92$, and $478$ of them do not include zero.

\appsection{Proofs of main results}\label{app:proof}
\appsubsection{Proof of Proposition~\ref{thm:estimated-c-rate}}\label{app:proof-estimated-c-rate}
To prove Proposition~\ref{thm:estimated-c-rate}, we first establish an inverse inequality in KS distance: a Gaussian mixture close to $F^*$ cannot have a Gaussian scale substantially larger than $\sigma_0$.

\begin{lemma}[Inverse inequality in KS distance]\label{lem:ks-inverse}
Let
\begin{align*}
    F^*=\trueprior\star N(0,\sigma_0^2), \qquad F=H\star N(0,\sigma^2),
\end{align*}
where $H\in\Pc(\Real)$ and $\sigma\geq\sigma_0$. Under Assumption~\ref{ass:moment-fourier}, there exist constants $C>0$ and $\delta_0\in(0,e^{-e})$ such that, whenever $0<\delta\leq\delta_0$ and $\KS(F,F^*)\leq\delta$,
\begin{align}\label{eq:ks-inverse}
    0\leq\sigma^2-\sigma_0^2\leq C\frac{\log\log(1/\delta)}{\log(1/\delta)}.
\end{align}
The constants depend only on $\sigma_0$ and the parameters in Assumption~\ref{ass:moment-fourier}.
\end{lemma}
\begin{proof}[Proof of Lemma~\ref{lem:ks-inverse} under Assumption~\ref{ass:moment-fourier} (i)]
If $X^*=\xi+\sigma_0Z\sim F^*$, where $\xi\sim\trueprior$ and $Z\sim N(0,1)$ are independent, then \eqref{eq:simultaneous-moment-control} with $r\leq2$ and $(\E|Z|^p)^{1/p} \lesssim \sqrt{p}$ give 
\begin{align*}
    (\E|X^*|^p)^{1/p}\leq K_0p^{1/r}+\sigma_0(\E|Z|^p)^{1/p}\lesssim_{\sigma_0,\gammav} p^{1/r}, \qquad \mbox{for all }\;\; p\geq1,
\end{align*}
by Minkowski's inequality.
Take $p=2\log(1/\delta)$ and $M=K\{\log(1/\delta)\}^{1/r}$, where $K>0$ is sufficiently large so that $(\E|X^*|^p)^{1/p}\leq M/e$. Then, Markov's inequality gives 
\begin{equation}\label{eq:ks-inverse-true-tail}
\begin{aligned}
    F^*([-M,M]^c) &=\P(|X^*|>M)\leq\frac{\E|X^*|^p}{M^p}\leq e^{-p} =e^{-2\log(1/\delta)} =\delta^2.
\end{aligned}
\end{equation}
Then since $\KS(F,F^*)\leq\delta$, we have
\begin{equation}\label{eq:ks-inverse-comparison-tail}
    F([-M,M]^c)\leq F^*([-M,M]^c)+2\delta\leq\delta^2+2\delta.
\end{equation}
Also, Stieltjes integration by parts on $[-M,M]$ yields
\begin{align*}
    \int_{-M}^M e^{itx}\,\diff(F-F^*)(x)
    &=e^{itM}\{F(M)-F^*(M)\}-e^{-itM}\{F(-M)-F^*(-M)\}\\
    &\quad-it\int_{-M}^M e^{itx}\{F(x)-F^*(x)\}\,\diff x.
\end{align*}
Since $\KS(F,F^*)\leq\delta$, this gives
\begin{align*}
    \left|\int_{-M}^M e^{itx}\,\diff(F-F^*)(x)\right|
    &\leq |F(M)-F^*(M)|+|F(-M)-F^*(-M)|\\
    &\quad+|t|\int_{-M}^M |F(x)-F^*(x)|\,\diff x\\
    &\leq2\delta+2M|t|\delta.
\end{align*}
Combining this with the tail bounds \eqref{eq:ks-inverse-true-tail} and~\eqref{eq:ks-inverse-comparison-tail}, we obtain
\begin{equation}\label{eq:estimated-c-characteristic-bound}
\begin{aligned}
    |\varphi_F(t)-\varphi_{F^*}(t)|
    &\leq\left|\int_{-M}^M e^{itx}\,\diff(F-F^*)(x)\right|
       +F([-M,M]^c)+F^*([-M,M]^c)\\
    &\leq2\delta+2M|t|\delta+F([-M,M]^c)+F^*([-M,M]^c)\\
    &\leq4\delta+2M|t|\delta+2\delta^2\\
    &\lesssim_{\sigma_0,\gammav} \delta\left\{1+\{\log(1/\delta)\}^{1/r}|t|\right\}, \qquad t\in\Real.
\end{aligned}
\end{equation}
Put $q=1/r+(\vartheta+1)/2+1$ with $r$ and $\vartheta$ in Assumption~\ref{ass:moment-fourier} and choose $t>0$ by
\begin{align*}
    t^2=\frac{2}{\sigma_0^2}\left\{\log(1/\delta)-q\log\log(1/\delta)\right\}.
\end{align*}
By \eqref{eq:estimated-c-characteristic-bound} and the choice of $t$,
\begin{equation}\label{eq:ks-inverse-frequency-difference}
\begin{aligned}
    |\varphi_F(t)-\varphi_{F^*}(t)| &\lesssim_{\sigma_0,\gammav} \delta\left\{1+\{\log(1/\delta)\}^{1/r}|t|\right\}\\
    &\lesssim_{\sigma_0,\gammav} \delta\{\log(1/\delta)\}^{1/r+1/2}.
\end{aligned}
\end{equation}

Now, for sufficiently small $\delta$, $t\geq t_0$, $t\asymp\{\log(1/\delta)\}^{1/2}$, and $e^{-\sigma_0^2t^2/2}=\delta\{\log(1/\delta)\}^q$. Also,
\begin{align*}
    1+t\leq\left(1+\frac{\sqrt2}{\sigma_0}\right)\{\log(1/\delta)\}^{1/2}.
\end{align*}
Thus, the Fourier lower bound \eqref{eq:fourier-lower-bound} gives
\begin{align*}
    |\varphi_{F^*}(t)|
    &=e^{-\sigma_0^2t^2/2}|\varphi_{\trueprior}(t)|\\
    &\geq a(1+t)^{-\vartheta}e^{-\sigma_0^2t^2/2}\\
    &\geq c e^{-\sigma_0^2t^2/2}
       \{\log(1/\delta)\}^{-\vartheta/2},
\end{align*}
where $c=a(1+\sqrt2/\sigma_0)^{-\vartheta}>0$ for some $a \in (0,1]$. Hence, by \eqref{eq:ks-inverse-frequency-difference} and the reverse triangle inequality, for sufficiently small $\delta$,
\begin{equation}\label{eq:ks-inverse-characteristic-lower}
\begin{aligned}
    |\varphi_F(t)| \geq\frac12|\varphi_{F^*}(t)| \geq\frac c2 e^{-\sigma_0^2t^2/2}\{\log(1/\delta)\}^{-\vartheta/2}.
\end{aligned}
\end{equation}
Since $F=H\star N(0,\sigma^2)$ and $|\varphi_H(t)|\leq1$, we also have the following upper bound:
\begin{equation}\label{eq:ks-inverse-characteristic-upper}
    |\varphi_F(t)|=|\varphi_H(t)|e^{-\sigma^2t^2/2}\leq e^{-\sigma^2t^2/2}.
\end{equation}
Combining \eqref{eq:ks-inverse-characteristic-lower} and~\eqref{eq:ks-inverse-characteristic-upper} and taking logarithms gives
\begin{align*}
    \frac{(\sigma^2-\sigma_0^2)t^2}{2} \leq\frac{\vartheta}{2}\log\log(1/\delta)+\log(2/c).
\end{align*}
Substituting the definition of $t$ yields
\begin{align*}
    \sigma^2-\sigma_0^2 &\leq\sigma_0^2 \frac{(\vartheta/2)\log\log(1/\delta)+\log(2/c)} {\log(1/\delta)-q\log\log(1/\delta)} \leq C\frac{\log\log(1/\delta)}{\log(1/\delta)},
\end{align*}
where the last inequality holds for all sufficiently small $\delta$. This completes the proof.
\end{proof}

\begin{proof}[Proof of Lemma~\ref{lem:ks-inverse} under Assumption~\ref{ass:moment-fourier} (ii)]
We adapt the proof of Theorem~5 of \citet{bariletto2026convergencerateslatentmixing} to KS distance (see also Remark~\ref{rem:bariletto-inverse-inequality} below).

\emph{(i) Scale control and localization.} Under Assumption~\ref{ass:moment-fourier}~$(ii)$, we have $\operatorname{supp}(H^*)\subseteq[-R,R]$. We first show that there exist constants $\delta_0>0$ and $\bar\sigma<\infty$, depending only on $R$ and $\sigma_0$, such that $\KS(F,F^*)\leq\delta_0$ implies $\sigma\leq\bar\sigma$. Put $A_0=R+2\sigma_0$ and $p_0=\P(|Z|\leq2)$. Since $F^*([-A_0,A_0])\geq p_0$ and 
\begin{align*}
    |F([-A_0,A_0])-F^*([-A_0,A_0])|\leq 2\KS(F,F^*)\leq2\delta,
\end{align*}
we have $F([-A_0,A_0])\geq p_0-2\delta$. Also, note that every $H\star N(0,\sigma^2)$ has density bounded by $(\sqrt{2\pi}\sigma)^{-1}$, and thus $F([-A_0,A_0])\le (\sqrt{2\pi}\sigma)^{-1}2A_0$. Choose $\delta_0>0$ sufficiently small that $2\delta_0\leq p_0/2$. Then
\begin{equation}\label{eq:ks-inverse-scale-bound}
    \sigma\leq \bar\sigma :=\frac{4A_0}{\sqrt{2\pi}p_0}.
\end{equation}

We next show that most of the mass of the mixing distribution $H$ of $F$ must lie in a bounded interval. To see this, write
\begin{align*}
    X=\xi+\sigma Z\sim F, \qquad X^*=\xi^*+\sigma_0 Z^*\sim F^*,
\end{align*}
where $\xi\sim H$, $\xi^*\sim H^*$, and $Z,Z^*\sim N(0,1)$, with independence within each representation. Let
\begin{align}\label{eq:ks-inverse-localization-levels}
    L=\bar\sigma\sqrt{2\log\frac{1}{\delta}}, \qquad M=R+2L.
\end{align}
Since $\{\xi>M,\ \sigma Z\geq-L\}\subseteq\{X>R+L\}$ and $\P(\sigma Z\geq-L)\geq1/2$, we have
\begin{align*}
    H((M,\infty))\leq2\P(X>R+L).
\end{align*}
The KS bound $\KS(F,F^*)\leq\delta$ and the compact support condition $H^*\in\Pc([-R,R])$ imply
\begin{align*}
    \P(X>R+L) \le \P(X^*>R+L)+\delta \le \P\left(Z^*>\frac{L}{\sigma_0}\right)+\delta \leq 2\delta,
\end{align*}
where the last inequality follows from $\sigma_0\leq\bar\sigma$, the definition of $L$, and the standard Gaussian tail bound. Thus, we have $H((M,\infty)) \le 4\delta$. Analogously, we have $H((-\infty, -M)) \le 4\delta$ and 
\begin{equation}\label{eq:ks-inverse-competitor-localization}
    q:=H([-M,M]^c)\leq8\delta.
\end{equation}
Choose $\delta_0$ sufficiently small so that $q<1$. Let $H^{(M)}$ be $H$ conditioned on $[-M,M]$ and put
\begin{align}\label{eq:ks-inverse-localized-mixture}
    F^{(M)}=H^{(M)}\star N(0,\sigma^2).
\end{align}
Because $F=(1-q)F^{(M)}+qF^{(T)}$ for another distribution $F^{(T)}$, we have
\begin{equation}\label{eq:ks-inverse-localized-ks}
    \KS(F^{(M)},F^*) \leq\KS(F^{(M)},F)+\KS(F,F^*) \leq q+\delta\leq9\delta.
\end{equation}

\emph{(ii) A lower bound for the moment difference.}
Put $\gamma^2=\sigma^2-\sigma_0^2\geq0$. Since $N(0,\sigma^2)=N(0,\gamma^2)\star N(0,\sigma_0^2)$, $F^{(M)}$ in \eqref{eq:ks-inverse-localized-mixture} can equivalently be expressed as
\begin{align*}
    F^{(M)}=\{H^{(M)}\star N(0,\gamma^2)\} \star N(0,\sigma_0^2).
\end{align*}
For an integer $m\geq0$ and $s\geq0$, let 
\begin{align*}
    \mathsf H_m(x;s^2):=\E[(x+isZ)^m], \qquad Z \sim N(0,1),
\end{align*}
be the Hermite polynomial of degree $m$. Note that $H^{(M)}$ is supported on $[-M,M]$, while $H^*$ is supported on $[-R,R]\subseteq[-M,M]$, so all moments are finite. For an integer $k\geq1$, following the derivation leading to Eq.~(39) in \citet{bariletto2026convergencerateslatentmixing} and using Lemma~6 therein, we obtain
\begin{equation}\label{eq:ks-inverse-moment-difference}
    \begin{aligned}
        \Delta_k &:=\int \mathsf H_{2k}(x;\sigma_0^2)\,\diff(F^{(M)}-F^*)(x) \\
        &=\int\theta^{2k}\,\diff\{H^{(M)}\star N(0,\gamma^2)\}(\theta) -\int\theta^{2k}\,\diff H^*(\theta).  
    \end{aligned}
\end{equation}
Let $U\sim H^{(M)}$ and $Z\sim N(0,1)$ be independent. Since $U+\gamma Z\sim H^{(M)}\star N(0,\gamma^2)$ and $\E[Z^{2j}]=(2j-1)!!$, we obtain
\begin{align*}
    \int\theta^{2k}\,\diff\{H^{(M)}\star N(0,\gamma^2)\}(\theta) =\E[U^{2k}] +\sum_{j=1}^k\binom{2k}{2j}\gamma^{2j}(2j-1)!!\E[U^{2k-2j}] \ge\gamma^{2k}(2k-1)!!,
\end{align*}
where the inequality retains only the $j=k$ term; every term in the sum is nonnegative. Moreover, the support condition on $H^*$ gives
\begin{align*}
    \int\theta^{2k}\,\diff H^*(\theta)\leq R^{2k}.
\end{align*}
Substituting these two bounds into \eqref{eq:ks-inverse-moment-difference} yields
\begin{equation}\label{eq:ks-inverse-moment-lower}
    \Delta_k\geq\gamma^{2k}(2k-1)!!-R^{2k}.
\end{equation}

\emph{(iii) An upper bound for the moment difference.} 
To obtain an upper bound for $|\Delta_k|$, we split the integral in \eqref{eq:ks-inverse-moment-difference} into the central region $[-\Lambda,\Lambda]$ and the tail region $\{|x|>\Lambda\}$, where $\Lambda:=M+L$ and $M$ and $L$ are defined in \eqref{eq:ks-inverse-localization-levels}. First, we consider the central region $[-\Lambda,\Lambda]$. By the definition of the Hermite polynomial, uniformly for $s\leq\bar\sigma$, we have
\begin{equation}\label{eq:ks-inverse-hermite-bound}
    \begin{aligned}
        |\mathsf H_{2k}(x;s^2)| &\le C^k\{|x|^{2k}+k^k\},\\
        |\mathsf H_{2k-1}(x;s^2)| &\le C^k\{|x|^{2k-1}+k^{k-1/2}\},
    \end{aligned}
\end{equation}
where $C>0$ depends only on $R$ and $\sigma_0$. Integration by parts on $[-\Lambda,\Lambda]$ gives
\begin{align*}
    \int_{-\Lambda}^{\Lambda}\mathsf H_{2k}(x;\sigma_0^2)\,\diff(F^{(M)}-F^*)(x) &=
    \mathsf H_{2k}(\Lambda;\sigma_0^2) \{F^{(M)}(\Lambda)-F^*(\Lambda)\} \\
    & -\mathsf H_{2k}(-\Lambda;\sigma_0^2) \{F^{(M)}(-\Lambda)-F^*(-\Lambda)\}\\
    &-\int_{-\Lambda}^{\Lambda}\{F^{(M)}(x)-F^*(x)\}\partial_x\mathsf H_{2k}(x;\sigma_0^2)\,\diff x.
\end{align*}
Using \eqref{eq:ks-inverse-localized-ks} and $\partial_x\mathsf H_{2k}(x;s^2)=2k\mathsf H_{2k-1}(x;s^2)$, we have
\begin{align*}
    &\left|\int_{-\Lambda}^{\Lambda}\mathsf H_{2k}(x;\sigma_0^2)\,\diff(F^{(M)}-F^*)(x)\right|\\
    &\quad\leq9\delta\bigg\{|\mathsf H_{2k}(\Lambda;\sigma_0^2)|+|\mathsf H_{2k}(-\Lambda;\sigma_0^2)|+2k\int_{-\Lambda}^{\Lambda}|\mathsf H_{2k-1}(x;\sigma_0^2)|\,\diff x\bigg\}.
\end{align*}
The bounds for Hermite polynomials in \eqref{eq:ks-inverse-hermite-bound} and $\Lambda k^{k-1/2}\leq\Lambda^{2k}+k^k$ give
\begin{align*}
    2k\int_{-\Lambda}^{\Lambda} |\mathsf H_{2k-1}(x;\sigma_0^2)|\,\diff x \le C^k k\{\Lambda^{2k}+\Lambda k^{k-1/2}\} \le C^k k\{\Lambda^{2k}+k^k\}.
\end{align*}
Combining the preceding bound with \eqref{eq:ks-inverse-hermite-bound} yields
\begin{align}\label{eq:ks-inverse-central-region}
    \left|\int_{-\Lambda}^{\Lambda}\mathsf H_{2k}(x;\sigma_0^2)\,\diff(F^{(M)}-F^*)(x)\right|\le C_1^k\delta k\{\Lambda^{2k}+k^k\},
\end{align}
where $C_1>0$ depends only on $R$ and $\sigma_0$.

Next, we consider the tail region $\{|x|>\Lambda\}$. Note that both $F^{(M)}$ and $F^*$ are Gaussian mixtures with mixing distributions supported on $[-M,M]$ and scales bounded above by $\bar\sigma$. Applying parts~(a) and~(b) of Lemma~7 of \citet{bariletto2026convergencerateslatentmixing} gives
\begin{align}\label{eq:ks-inverse-tail-region}
    \int_{|x|>\Lambda} |\mathsf H_{2k}(x;\sigma_0^2)|\,\diff(F^{(M)}+F^*)(x) \le C_2^k e^{-L^2/(2\bar\sigma^2)} \left(M^{2k}+L^{2k}+k^k\right),
\end{align}
where $C_2>0$ depends only on $R$ and $\sigma_0$. Here, we used
\begin{align*}
    \bar\sigma^{2k}(2k-1)!!\leq C^k k^k
\end{align*}
for some constant $C>0$ depending only on $R$ and $\sigma_0$.

Finally, combining \eqref{eq:ks-inverse-central-region} and \eqref{eq:ks-inverse-tail-region} gives
\begin{equation}\label{eq:ks-inverse-moment-upper}
    |\Delta_k| \le C_3^k\left[ \delta k\{\Lambda^{2k}+k^k\}+e^{-L^2/(2\bar\sigma^2)}\{M^{2k}+L^{2k}+k^k\}\right],
\end{equation}
where $C_3>0$ depends only on $R$ and $\sigma_0$.

\emph{(iv) Choice of parameters.}
By the definitions of $L$ and $M$ in \eqref{eq:ks-inverse-localization-levels} and $\Lambda = M+L$, we have
\begin{align*}
    e^{-L^2/(2\bar\sigma^2)}=\delta, \qquad M=R+2L, \qquad \Lambda=R+3L.
\end{align*}
Consequently, for all sufficiently small $\delta$,
\begin{align*}
    M^2+L^2+\Lambda^2 \lesssim_{R,\sigma_0}R^2+L^2 \lesssim_{R,\sigma_0}\log(1/\delta).
\end{align*}
Thus, we have 
\begin{align*}
    M^{2k}+L^{2k}+\Lambda^{2k} \leq C^k\{\log(1/\delta)\}^k
\end{align*}
for some constant $C > 0$. Now, take
\begin{align*}
    k=\left\lfloor \frac{\log(1/\delta)}{2\log\log(1/\delta)} \right\rfloor.
\end{align*}
Then we have $\delta\{\log(1/\delta)\}^k\leq\delta^{1/2}$ and $\delta k^k\leq\delta^{1/2}$. Consequently, for the choice of $k$ made above, \eqref{eq:ks-inverse-moment-upper} gives, after increasing $C_3$ if necessary,
\begin{align*}
    |\Delta_k| &\le C_3^k\left[ \delta k\Lambda^{2k}+\delta k k^k +\delta M^{2k}+\delta L^{2k}+\delta k^k\right] \\
    &\le C_3^k\left[\delta k\{\Lambda^{2k}+M^{2k}+L^{2k}\} +\delta(k+1)k^k\right] \\
    &\le C_3^k k\left[\delta\{\log(1/\delta)\}^k+\delta k^k\right] \\
    &\le 2C_3^k k\delta^{1/2} \le C_4^k,
\end{align*}
where $C_4 > 0$ depends only on $R$ and $\sigma_0$; the last inequality uses $2k\leq4^k$ for $k\geq1$. Combining this with \eqref{eq:ks-inverse-moment-lower} and $(2k-1)!!\geq(2k/e^2)^k$ yields
\begin{align*}
    \gamma^{2k}(2k/e^2)^k\leq R^{2k}+C_4^k,
\end{align*}
where we recall $\gamma^2 = \sigma^2 - \sigma_0^2\geq0$. Taking $k$th roots and using $(R^{2k}+C_4^k)^{1/k}\leq R^2+C_4$ give
\begin{align*}
    \gamma^2 \le \frac{e^2}{2k}(R^{2k}+C_4^k)^{1/k} \le \frac{e^2(R^2+C_4)}{2k} \lesssim_{R,\sigma_0}\frac{1}{k} \asymp \frac{\log\log(1/\delta)}{\log(1/\delta)}.
\end{align*}
This proves the result.
\end{proof}

We now use Lemma~\ref{lem:ks-inverse} to prove Proposition~\ref{thm:estimated-c-rate}.
\begin{proof}[Proof of Proposition~\ref{thm:estimated-c-rate}]
Let
\begin{equation}\label{eq:estimated-c-neighborhood-event}
    \mathcal A_n=\{\KS(\F_n,F^*)\leq\eta_n\},
\end{equation}
where $\eta_n$ is defined in \eqref{eq:eta-radius}. Note that, by the Dvoretzky--Kiefer--Wolfowitz (DKW) inequality \citep{Massart1990}, we have 
\begin{align}\label{eq:estimated-c-KS-event}
    \P_{G^*}(\mathcal A_n^c)\leq2e^{-2n\eta_n^2}=2n^{-2}.
\end{align}
On $\mathcal A_n$, we have $\sigma_0(\F_n;\eta_n)\geq\sigma_0(F^*)=\sigma_0$. Moreover, every $\widetilde F\in\Pc(\Real)$ satisfying $\KS(\widetilde F,\F_n)\leq\eta_n$ satisfies $\KS(\widetilde F,F^*)\leq\KS(\widetilde F,\F_n)+\KS(\F_n,F^*)\leq2\eta_n$. For any such $\widetilde F$ and any $s\geq\sigma_0$ such that $\widetilde F=H\star N(0,s^2)$, Lemma~\ref{lem:ks-inverse} with $\delta=2\eta_n$ gives, for all sufficiently large $n$,
\begin{align*}
    0\leq s^2-\sigma_0^2 \leq C\frac{\log\log\{1/(2\eta_n)\}}{\log\{1/(2\eta_n)\}} \leq C\rho_n.
\end{align*}
Taking the supremum in the definition of $\sigma_0(\F_n;\eta_n)$ yields
\begin{align*}
    0\leq\sigma_0^2(\F_n;\eta_n)-\sigma_0^2\leq C\rho_n
\end{align*}
on $\mathcal A_n$. On the same event, we have $\hat c_{\mathrm{KS}}^2-\czero^2=\sigma_0^2(\F_n;\eta_n)-\sigma_0^2$, so this proves \eqref{eq:sigma2-rate-event} and gives $\hat c_{\mathrm{KS}}^2-\czero^2=O_P(\rho_n)$. If $\czero>0$, dividing by $\hat c_{\mathrm{KS}}+\czero\geq\czero$ gives $\hat c_{\mathrm{KS}}-\czero=O_P(\rho_n)$. If $\czero=0$, the upper bound gives $\hat c_{\mathrm{KS}}=O_P(\sqrt{\rho_n})$.
\end{proof}

\begin{remark}\label{rem:bariletto-inverse-inequality}
Theorem~5 of \citet{bariletto2026convergencerateslatentmixing} gives a related inverse inequality in $L_1$ distance under compact support assumptions on the location mixing distributions. Lemma~\ref{lem:ks-inverse} instead uses KS distance and allows the mixing distribution $H$ in $F=H\star N(0,\sigma^2)$ to have unbounded support. Under their Bayesian mixture model, the inverse inequality yields a $(\log\log n)/\log n$ posterior contraction rate for the squared Gaussian scale (Proposition~2 therein), matching the order of our rate under different assumptions.
\end{remark}

\appsubsection{Proof of Theorem~\ref{thm:modified-c-rate}}\label{app:proof-modified-c-rate}
To prove Theorem~\ref{thm:modified-c-rate}, we first show that the sample variance cap $C_n$ in \eqref{eq:cap} remains above $\czero$ with high probability.
\begin{lemma}[Validity of the sample variance cap]\label{lem:estimated-c-cap}
Under the moment condition \eqref{eq:simultaneous-moment-control}, the cap $C_n$ in \eqref{eq:cap} satisfies, for some $k>0$ and all sufficiently large $n$,
\begin{align}\label{eq:cap-probability}
    \P_{G^*}(C_n<\czero)\leq2\exp\left\{-k\frac{n}{(\log n)^2}\right\}.
\end{align}
\end{lemma}
\begin{proof}[Proof of Lemma~\ref{lem:estimated-c-cap}]
We first note that, under \eqref{eq:hierarchical model}, $\Var(X_i)=1+\czero^2+\Var(\xi_i)\geq1+\czero^2$. Thus, 
\begin{align}\label{eq:estimated-c-cap-event-inclusion}
    \{C_n<\czero\}\subseteq\{S_n^2<1+\czero^2-a_n\} \subseteq\{S_n^2<\Var(X_i)-a_n\},
\end{align}
where $a_n:=1/\log n$. Now, put $V=(X_1-X_2)^2/2$. By the moment condition \eqref{eq:simultaneous-moment-control}, $\E_{G^*}[V^2]$ is bounded by a constant depending only on $K_0$, $r$, and $\czero$. Moreover, $\E_{G^*}[V]=\Var(X_1)$ and
\begin{align*}
    S_n^2=\binom n2^{-1}\sum_{i<j}\frac{(X_i-X_j)^2}{2}.
\end{align*}
Following \citet{Hoeffding1963}, we express $S_n^2$ as an average over permutations. That is, let $m=\lfloor n/2\rfloor$, and for a permutation $\pi$ of $\{1,\ldots,n\}$,
\begin{align*}
    A_\pi:=\frac1m\sum_{j=1}^m
    \frac{(X_{\pi(2j-1)}-X_{\pi(2j)})^2}{2}.
\end{align*}
Each unordered pair $\{i,j\}$ is selected with probability $m/\binom n2$ under a uniformly random permutation. Consequently,
\begin{align*}
    \frac1{n!}\sum_\pi A_\pi &=\frac1m\sum_{i<j}\frac{(X_i-X_j)^2}{2} \frac{m}{\binom n2}\\
    &=\binom n2^{-1}\sum_{i<j}\frac{(X_i-X_j)^2}{2} =S_n^2.
\end{align*}
Now, since $S_n^2=n!^{-1}\sum_\pi A_\pi$, Jensen's inequality gives, for every $\lambda\geq0$,
\begin{align*}
    \exp\{-\lambda m(S_n^2-\Var(X_1))\} &\leq\frac1{n!}\sum_\pi\exp\{-\lambda m(A_\pi-\Var(X_1))\}.
\end{align*}
For a fixed permutation $\pi$, the variables
\begin{align*}
    V_{\pi,j}:=\frac{(X_{\pi(2j-1)}-X_{\pi(2j)})^2}{2},
    \qquad j=1,\ldots,m,
\end{align*}
are independent and each has the same distribution as $V = (X_1-X_2)^2/2$. Taking expectations and using $\E_{G^*}[V]=\Var(X_1)$, we obtain
\begin{align*}
    \E_{G^*}[\exp\{-\lambda m(S_n^2-\Var(X_1))\}]
    &\leq\frac1{n!}\sum_\pi\E_{G^*}\!\left[\exp\left\{-\lambda\sum_{j=1}^m (V_{\pi,j}-\E_{G^*}[V])\right\}\right]\\
    &=\frac1{n!}\sum_\pi \left\{\E_{G^*}[e^{-\lambda(V-\E_{G^*}[V])}]\right\}^{m}\\
    &=\left\{\E_{G^*}[e^{-\lambda(V-\E_{G^*}[V])}]\right\}^{m}\\
    &\leq\exp\left\{\frac{m\lambda^2\E_{G^*}[V^2]}{2}\right\}.
\end{align*}
The last inequality uses $V\geq0$, $e^{-x}\leq1-x+x^2/2$ for $x\geq0$, and $1+x\leq e^x$. Finally, Markov's inequality with $\lambda=a_n/\E_{G^*}[V^2]$ gives, for some $k>0$ and all sufficiently large $n$,
\begin{align*}
    \P_{G^*}\{S_n^2<\Var(X_1)-a_n\}
    &\leq e^{-\lambda m a_n}
       \E_{G^*}[\exp\{-\lambda m(S_n^2-\Var(X_1))\}]\\
    &\leq\exp\left\{-\frac{m a_n^2}{2\E_{G^*}[V^2]}\right\}\\
    &\leq2\exp\left\{-k\frac{n}{(\log n)^2}\right\}.
\end{align*}
Combining this with \eqref{eq:estimated-c-cap-event-inclusion} proves \eqref{eq:cap-probability}.
\end{proof}

Next, we bound the log-likelihood ratio of the estimated marginal density based on the classical NPMLE under \eqref{eq:original model} to the true marginal density. This bound will be used to control the likelihood loss $L_n(0)-L_n(\hat c_0)$ of \eqref{eq:modified-c-good-event} in the proof of Theorem~\ref{thm:modified-c-rate}.

\begin{lemma}[Log-likelihood ratio bound for the classical NPMLE]\label{lem:estimated-c-likelihood-gain}
For each $c \ge 0$, let $L_n(c)$ and $b_n$ be defined in \eqref{eq:likelihood}. Under the moment condition \eqref{eq:simultaneous-moment-control}, for any sufficiently large fixed constant $B>0$ and all sufficiently large $n$,
\begin{align}\label{eq:estimated-c-likelihood-gain}
    \P_{G^*}\left\{L_n(0)-\sum_{i=1}^n\log f^*(X_i)>Bb_n\right\}\leq2n^{-3}.
\end{align}
\end{lemma}
\begin{proof}
Let $\widehat f_n^{\,0}$ be the estimated marginal density based on the classical NPMLE under \eqref{eq:original model} (equivalently, fitted at $c=0$ in \eqref{eq:likelihood}), so that $L_n(0)=\sum_i\log\widehat f_n^{\,0}(X_i)$. Note that 
\begin{align*}
    \sum_{i=1}^n\log\frac{\widehat f_n^{\,0}(X_i)}{f^*(X_i)} =L_n(0)-\sum_{i=1}^n\log f^*(X_i).
\end{align*}
We first bound the log-likelihood ratio on $\{\max_i|X_i|\leq R\}$ for a fixed $R\geq\sqrt{\log n}$. For functions $f$ and $g$, define their uniform distance on $[-R,R]$ by
\begin{align*}
    \|f-g\|_{\infty,[-R,R]}:=\sup_{x\in[-R,R]}|f(x)-g(x)|.
\end{align*}
Let $N(\varepsilon,\mathcal F,\|\cdot\|)$ denote the minimum cardinality of an $\varepsilon$-net of $\mathcal F$ under the distance induced by $\|\cdot\|$. Set $\Mc:= \{p_{G,1}:G\in\Pc(\Real)\}$ and
\begin{align*}
    N_R:=N\!\left( \eta, \Mc, \|\cdot\|_{\infty,[-R,R]} \right).
\end{align*}
By the covering bound of Supplementary Lemma~5 of \citet{Soloff2025}, for $\eta=n^{-3}$ and each fixed $R\geq\sqrt{\log n}$, we have
\begin{align}\label{eq:estimated-c-likelihood-net}
    \log N_R\leq C\left(1+\frac{R}{\sqrt{\log n}}\right)(\log n)^2,
\end{align}
for an absolute constant $C > 0$. Take a fixed $\eta$-net of $\Mc$ with $N_R$ members. Since $\widehat f_n^{\,0}\in\Mc$, some member $g$ of this net satisfies $\|\widehat f_n^{\,0}-g\|_{\infty,[-R,R]}\leq\eta$. On $\{\max_i|X_i|\leq R\}$, it follows that $|\widehat f_n^{\,0}(X_i)-g(X_i)|\leq\eta$ for $i=1,\ldots,n$. Now, note that the general maximum likelihood theorem (see e.g., (B.8)--(B.9) in the supplementary material of \citet{Saha2020}) gives $\min_i\widehat f_n^{\,0}(X_i)\geq(n\sqrt{2\pi})^{-1}$. Consequently, for each $i$,
\begin{align*}
    g(X_i) &\geq\widehat f_n^{\,0}(X_i)-n^{-3}\\
    &\geq\widehat f_n^{\,0}(X_i) \left(1-\frac{\sqrt{2\pi}}{n^2}\right).
\end{align*}
Then, it follows that, for all sufficiently large $n$,
\begin{align*}
    L_n(0)-\sum_i\log g(X_i) &\leq-n\log\left(1-\frac{\sqrt{2\pi}}{n^2}\right) \leq\frac{2\sqrt{2\pi}}n\leq1.
\end{align*}
Thus, $L_n(0)\leq\sum_i\log g(X_i)+1$. Now, write the fixed net as $\{g_1,\ldots,g_{N_R}\}$. For each member $g_j$, we have $\E_{G^*}\prod_i g_j(X_i)/f^*(X_i)=1$. Using the likelihood approximation with a union bound and Markov's inequality, we get
\begin{equation}\label{eq:estimated-c-likelihood-fixed-radius}
\begin{aligned}
    &\P_{G^*}\left\{\max_i|X_i|\leq R,\quad L_n(0)-\sum_i\log f^*(X_i)>1+\log N_R+5\log n\right\}\\
    &\quad\leq\P_{G^*}\left\{ \max_{1\leq j\leq N_R}\sum_{i=1}^n \log\frac{g_j(X_i)}{f^*(X_i)}>\log N_R+5\log n\right\}\\
    &\quad=\P_{G^*}\left\{ \max_{1\leq j\leq N_R}\prod_{i=1}^n  \frac{g_j(X_i)}{f^*(X_i)}>N_Rn^5\right\}\\
    &\quad\leq\sum_{j=1}^{N_R}\P_{G^*}\left\{ \prod_{i=1}^n\frac{g_j(X_i)}{f^*(X_i)}>N_Rn^5\right\}\\
    &\quad\leq\frac1{N_Rn^5}\sum_{j=1}^{N_R} \E_{G^*}\left[\prod_{i=1}^n\frac{g_j(X_i)}{f^*(X_i)}\right] =n^{-5}.
\end{aligned}
\end{equation}

We now derive the probability bound \eqref{eq:estimated-c-likelihood-gain}. We first note that the moment condition \eqref{eq:simultaneous-moment-control} and the Gaussian moment bound imply $(\E_{G^*}|X_i|^p)^{1/p}\lesssim p^{1/r}$ for all $p\geq1$. Taking $p=4\log n$, we can choose a sufficiently large fixed $K_R$ such that
\begin{align*}
    (\E_{G^*}|X_1|^p)^{1/p} \leq\frac{K_R}{e}(\log n)^{1/r}.
\end{align*}
Then Markov's inequality and a union bound give
\begin{equation}\label{eq:estimated-c-observation-tail}
\begin{aligned}
    \P_{G^*}\{\max_i|X_i|>K_R(\log n)^{1/r}\} &\leq\sum_{i=1}^n\P_{G^*}\{|X_i|>K_R(\log n)^{1/r}\}\\
    &\leq n\frac{\E_{G^*}|X_1|^p} {\{K_R(\log n)^{1/r}\}^{p}}\\
    &\leq ne^{-p}=ne^{-4\log n}=n^{-3}.
\end{aligned}
\end{equation}
Choose $R_j=2^j\sqrt{\log n}$, $0\leq j\leq J_n$, where $J_n=O(\log\log n)$ and $R_{J_n}\geq K_R(\log n)^{1/r}$. Note that $b_n$ in \eqref{eq:likelihood} satisfies
\begin{align*}
    b_n&\geq(\log n)^2, &&\text{on }\{\max_i|X_i|\leq R_0\},\\
    b_n&\geq(1+2^{j-1})(\log n)^2, &&\text{on }\{R_{j-1}<\max_i|X_i|\leq R_j\},\quad 1\leq j\leq J_n.
\end{align*}
Since $R_j=2^j\sqrt{\log n}$ and $1+5\log n\leq6(\log n)^2$ for $n\geq3$, \eqref{eq:estimated-c-likelihood-net} gives
\begin{align*}
    1+\log N_{R_j}+5\log n \leq\{C(1+2^j)+6\}(\log n)^2.
\end{align*}
Hence, for $j=0$, on $\{\max_i|X_i|\leq R_0\}$, we have
\begin{align*}
    1+\log N_{R_0}+5\log n \leq(2C+6)(\log n)^2\leq(2C+6)b_n.
\end{align*}
For $1\leq j\leq J_n$, using $1+2^j\leq2(1+2^{j-1})$, we obtain on $\{R_{j-1}<\max_i|X_i|\leq R_j\}$,
\begin{align*}
    1+\log N_{R_j}+5\log n &\leq(2C+6)(1+2^{j-1})(\log n)^2 \leq(2C+6)b_n.
\end{align*}
Thus, any fixed $B\geq2C+6$ satisfies $1+\log N_{R_j}+5\log n\leq Bb_n$ on these events for every $0\leq j\leq J_n$. Applying \eqref{eq:estimated-c-likelihood-fixed-radius} on $\{\max_i|X_i|\leq R_0\}$ and on $\{R_{j-1}<\max_i|X_i|\leq R_j\}$ for $1\leq j\leq J_n$, and using a union bound, we obtain
\begin{align*}
    &\P_{G^*}\left\{L_n(0)-\sum_i\log f^*(X_i)>Bb_n\right\}\\
    &\quad\leq\sum_{j=0}^{J_n}\P_{G^*}\left\{ \max_i|X_i|\leq R_j,\quad L_n(0)-\sum_i\log f^*(X_i)>1+\log N_{R_j}+5\log n \right\}\\
    &\qquad+\P_{G^*}\{\max_i|X_i|>R_{J_n}\}\\
    &\quad\leq(J_n+1)n^{-5}+n^{-3} \leq2n^{-3}.
\end{align*}
Here, the tail probability is at most $n^{-3}$ by \eqref{eq:estimated-c-observation-tail} because $R_{J_n}\geq K_R(\log n)^{1/r}$. The last inequality follows from $J_n=O(\log\log n)$, which gives $J_n+1\leq n^2$ for all sufficiently large $n$. This proves \eqref{eq:estimated-c-likelihood-gain}.
\end{proof}

We now prove Theorem~\ref{thm:modified-c-rate} using Lemmas~\ref{lem:estimated-c-cap} and~\ref{lem:estimated-c-likelihood-gain}.
\begin{proof}[Proof of Theorem~\ref{thm:modified-c-rate}]
Fix a sufficiently large $B$ as in Lemma~\ref{lem:estimated-c-likelihood-gain}. We work on the event
\begin{align*}
    \mathcal A_n\cap\{C_n\geq\czero\}
    \cap\left\{L_n(0)-\sum_i\log f^*(X_i)\leq Bb_n\right\},
\end{align*}
where $\mathcal A_n$ and $C_n$ are defined in \eqref{eq:estimated-c-neighborhood-event} and \eqref{eq:cap}, respectively. On this event, the proof of Proposition~\ref{thm:estimated-c-rate} gives, for a fixed constant $K>0$,
\begin{align*}
    \czero^2\leq(\hat c_{\mathrm{KS}}^2\wedge C_n^2)\leq\czero^2+K\rho_n.
\end{align*}
Set $\tilde C:=\{(\hat c_{\mathrm{KS}}^2\wedge C_n^2)-\czero^2\} /\rho_n\in[0,K]$, which is an admissible choice in \eqref{eq:estimated-c-buffer}. Substituting $C=\tilde C$ in \eqref{eq:estimated-c-buffer} gives
\begin{align*}
    \sqrt{(\hat c_{\mathrm{KS}}^2\wedge C_n^2)-\tilde C\rho_n} =\sqrt{(\hat c_{\mathrm{KS}}^2\wedge C_n^2) -\{(\hat c_{\mathrm{KS}}^2\wedge C_n^2)-\czero^2\}} =\czero.
\end{align*}
Also, since $L_n(\czero)\geq\sum_i\log f^*(X_i)$, we have $L_n(0)-L_n(\czero)\leq Bb_n$ on this event. Hence, since $\widehat C$ minimizes the criterion in \eqref{eq:estimated-c-buffer}, by the definition of $\hat c_0$ in \eqref{eq:estimated-c}, we have
\begin{align*}
    0\leq L_n(0)-L_n(\hat c_0)+b_n\widehat C &\leq L_n(0)-L_n(\czero)+b_n\tilde C \leq(B+K)b_n.
\end{align*}
Note that both $L_n(0)-L_n(\hat c_0)$ and $\widehat C$ are nonnegative. Hence, we have $\widehat C\leq B+K$, $L_n(0)-L_n(\hat c_0)\leq(B+K)b_n$, and
\begin{align*}
    (\czero^2-(B+K)\rho_n)_+\leq\hat c_0^2 \leq\czero^2+K\rho_n.
\end{align*}
Taking $B+K$ as the constant in \eqref{eq:modified-c-good-event} and combining \eqref{eq:estimated-c-KS-event}, \eqref{eq:cap-probability}, and \eqref{eq:estimated-c-likelihood-gain} proves \eqref{eq:one-sided-control}, since
\begin{align*}
    2n^{-2}+2\exp\left\{-k\frac{n}{(\log n)^2}\right\}+2n^{-3}
    \leq 3n^{-2}
\end{align*}
for all sufficiently large $n$. The remaining conclusions follow by the same argument as in the proof of Proposition~\ref{thm:estimated-c-rate}.
\end{proof}

\appsubsection{Proof of Theorem~\ref{thm:estimated-c-density-denoising}}\label{app:proof-estimated-c-density-denoising}
\begin{proof}[Proof of Theorem~\ref{thm:estimated-c-density-denoising}]
We first show that the smooth NPMLE has sufficiently high marginal likelihood relative to the true marginal density to apply Theorem~7 and adapt the proof of Theorem~9 of \citet{Soloff2025}. Let
\begin{equation}\label{eq:estimated-c-denoising-event}
    \mathcal T_n:=\Sc_n\cap\{R_n\leq K_R(\log n)^{1/r}\},
\end{equation}
where $R_n:=\max_i|X_i|$ and $\Sc_n$ is the event in Theorem~\ref{thm:modified-c-rate}. As shown in the proof of Lemma~\ref{lem:estimated-c-likelihood-gain}, a sufficiently large fixed $K_R>0$ satisfies
\begin{align*}
    \P_{G^*}\{R_n>K_R(\log n)^{1/r}\}\leq n^{-3}
\end{align*}
for all sufficiently large $n$. Combining this with \eqref{eq:one-sided-control} and a union bound gives
\begin{align}\label{eq:estimated-c-denoising-good-event}
    \P_{G^*}(\mathcal T_n^c)
    \leq 3n^{-2}+n^{-3}.
\end{align}
Note that, on $\mathcal T_n$, we have
\begin{align*}
    b_n =\left(1+\frac{R_n}{\sqrt{\log n}}\right)(\log n)^2 \lesssim_{\czero,\gammav}(\log n)^{3/2+1/r}=nr_n,
\end{align*}
where
\begin{equation}\label{eq:estimated-c-rn}
    r_n:=\frac{(\log n)^{3/2+1/r}}{n}.
\end{equation}
Also, we have $L_n(0)\geq\sum_i\log f^*(X_i)$, since $L_n(0)$ is the maximized marginal log-likelihood under \eqref{eq:original model}. Hence, the likelihood bound in \eqref{eq:modified-c-good-event} gives
\begin{align}\label{eq:estimated-c-density-likelihood-comparison}
    \sum_{i=1}^n\log\widehat f_n(X_i)=L_n(\hat c_0) \geq\sum_{i=1}^n\log f^*(X_i)-K_1nr_n, \qquad\text{on }\mathcal T_n,
\end{align}
for a fixed $K_1>0$. Let $\widehat G_n^{\mathrm{KW}}$ be the classical NPMLE over $\Pc(\Real)$ under \eqref{eq:original model}, and define the auxiliary prior
\begin{align*}
    \overline G_n:=\{\npmle\star N(0,\hat c_0^2)\}\1v_{\mathcal T_n} +\widehat G_n^{\mathrm{KW}}\1v_{\mathcal T_n^c}.
\end{align*}
On $\mathcal T_n^c$, since the classical NPMLE $\widehat G_n^{\mathrm{KW}}$ has likelihood at least that of $G^*$, we have
\begin{align}\label{eq:estimated-c-auxiliary-likelihood}
    \sum_{i=1}^n\log p_{\overline G_n,1}(X_i) \geq\sum_{i=1}^n\log f^*(X_i)-K_1nr_n,
\end{align}
with $r_n$ defined in \eqref{eq:estimated-c-rn}.

Based on \eqref{eq:estimated-c-auxiliary-likelihood}, we first prove the marginal density bound. If $\Theta=\xi+\czero Z\sim G^*$ with $Z\sim N(0,1)$ independent of $\xi\sim\trueprior$, then by the moment condition \eqref{eq:simultaneous-moment-control} with $r\leq2$, for every $p\geq1$,
\begin{align}\label{eq:estimated-c-true-prior-moments}
    \left\{\int|\theta|^p\,\diff G^*(\theta)\right\}^{1/p} \leq K_0p^{1/r}+\czero(\E|Z|^p)^{1/p} \leq(K_0+C_Z\czero)p^{1/r},
\end{align}
where $C_Z>0$ is an absolute constant. Now, in the notation of Theorem 7 of \citet{Soloff2025}, take $S=\{0\}$ and $M=M_0(\log n)^{1/r}$ with sufficiently large $M_0$. By \eqref{eq:estimated-c-true-prior-moments}, the corresponding rate function is of order $r_n$. Thus, \eqref{eq:estimated-c-auxiliary-likelihood} (divided by $n$) satisfies the approximate likelihood condition in Eq.~(14) of \citet{Soloff2025}. Then Theorem~7 of \citet{Soloff2025} gives
\begin{align*}
    \E_{G^*}[\H^2(p_{\overline G_n,1},f^*)] \lesssim_{\czero,\gammav}r_n, \qquad \P_{G^*}\{\H^2(p_{\overline G_n,1},f^*)>K_2t^2r_n\} \leq2n^{-t^2},\qquad t\geq1,
\end{align*}
for a fixed $K_2>0$. Since $p_{\overline G_n,1}=\widehat f_n$ on $\mathcal T_n$ and $\H^2\leq1$, we have
\begin{align*}
    \E_{G^*}[\H^2(\widehat f_n,f^*)] &\leq \E_{G^*}\!\left[ \H^2(p_{\overline G_n,1},f^*)\1v_{\mathcal T_n} \right]+\P_{G^*}(\mathcal T_n^c)\\
    &\leq \E_{G^*}[\H^2(p_{\overline G_n,1},f^*)] +\P_{G^*}(\mathcal T_n^c)\\
    &\lesssim_{\czero,\gammav}r_n+n^{-2} \lesssim_{\czero,\gammav}r_n,
\end{align*}
where we used \eqref{eq:estimated-c-denoising-good-event}. This proves \eqref{eq:estimated-c-density-rate}.

Next, for denoising, we first note that, by the general maximum likelihood theorem used in the proof of Lemma~\ref{lem:estimated-c-likelihood-gain}, for every $c\geq0$,
\begin{align*}
    \min_{1\leq j\leq n} p_{\npmle^{\,c},\sqrt{1+c^2}}(X_j) \geq\frac{1}{n\sqrt{2\pi(1+c^2)}}.
\end{align*}
On $\mathcal T_n$, \eqref{eq:modified-c-good-event} gives $\hat c_0^2\leq\czero^2+K\rho_n\leq\czero^2+1$ for all sufficiently large $n$. Since $\overline G_n$ uses the classical NPMLE $\widehat G_n^{\mathrm{KW}}$ on $\mathcal T_n^c$, we obtain
\begin{align}\label{eq:estimated-c-auxiliary-floor}
    \min_{1\leq j\leq n}p_{\overline G_n,1}(X_j) \geq\frac{1}{n\sqrt{2\pi(2+\czero^2)}}.
\end{align}
On $\mathcal T_n$, $\overline G_n=\widehat G_n$, so the EB posterior mean under \eqref{eq:original model} satisfies
\begin{align*}
    \E_{\overline G_n}[\theta_i\mid X_i] &=\frac{\int_{\Real}\theta\phi_1(X_i-\theta)\,\diff\overline G_n(\theta)} {p_{\overline G_n,1}(X_i)}\\
    &=\E_{\widehat G_n}[\theta_i\mid X_i]=\hat\theta_i, \qquad i=1,\ldots,n.
\end{align*}
Using \eqref{eq:estimated-c-auxiliary-floor} and the Hellinger bound above, we follow the regularized Tweedie and covering argument in the proof of Theorem~9 of \citet{Soloff2025} to obtain
\begin{align}\label{eq:estimated-c-good-event-regret}
    \E_{G^*}\!\left[\frac1n\sum_{i=1}^n (\hat\theta_i-\theta_i^*)^2\1v_{\mathcal T_n}\right] \lesssim_{\czero,\gammav}r_n(\log n)^3.
\end{align}
Here, $r_n$ is defined in \eqref{eq:estimated-c-rn}. It remains to control $\mathcal T_n^c$. Note that the NPMLE $\npmle$ is supported on $[X_{(1)},X_{(n)}]$ \citep{Lindsay1995}, so \eqref{eq:EB posterior mean} implies $|\hat\theta_i|\leq\max_j|X_j|$. By \eqref{eq:estimated-c-true-prior-moments}, the observations satisfy $(\E_{G^*}|X_i|^p)^{1/p}\lesssim_{\czero,\gammav}p^{1/r}$. Taking $p=\log n\geq6$ yields
\begin{align*}
    \left\{\E_{G^*}\left[\max_{1\leq j\leq n}|X_j|^6\right]\right\}^{1/6}
    \leq\left\{n\E_{G^*}|X_1|^p\right\}^{1/p}
    \lesssim_{\czero,\gammav}(\log n)^{1/r}.
\end{align*}
Also, conditional Jensen's inequality gives $\E_{G^*}|\theta_i^*|^6\leq\E_{G^*}|\theta_i|^6\lesssim_{\czero,\gammav}1$. Hence, H\"older's inequality and \eqref{eq:estimated-c-denoising-good-event} give
\begin{equation}\label{eq:estimated-c-bad-event-regret}
    \begin{aligned}
    \E_{G^*}\!\left[\frac1n\sum_{i=1}^n (\hat\theta_i-\theta_i^*)^2\1v_{\mathcal T_n^c}\right]
    &\leq\frac1n\sum_{i=1}^n \left\{\E_{G^*}|\hat\theta_i-\theta_i^*|^6\right\}^{1/3} \left\{\P_{G^*}(\mathcal T_n^c)\right\}^{2/3} \\
    &\lesssim_{\czero,\gammav}(\log n)^{2/r} \left\{\P_{G^*}(\mathcal T_n^c)\right\}^{2/3} \\
    &\lesssim_{\czero,\gammav}(\log n)^{2/r}n^{-4/3} \\
    &\lesssim_{\czero,\gammav}n^{-1}.
    \end{aligned}
\end{equation}
Here, we used \eqref{eq:estimated-c-denoising-good-event} to get $\P_{G^*}(\mathcal T_n^c)\lesssim_{\czero,\gammav}n^{-2}$ in the third inequality.  Combining this with \eqref{eq:estimated-c-good-event-regret} proves \eqref{eq:estimated-c-denoising-rate}.
\end{proof}

\appsubsection{Proof of Theorem~\ref{thm:moment-deconvolution}}\label{app:proof-moment-deconvolution}
\begin{proof}[Proof of Theorem~\ref{thm:moment-deconvolution}]
First, note that since $\czero>0$ and $\rho_n\to0$, the lower bound $\hat c_0^2\geq(\czero^2-K\rho_n)_+$ in \eqref{eq:modified-c-good-event} implies that $\hat c_0^+=\hat c_0$ and $\npmle^+=\npmle$ on $\Sc_n$ for all sufficiently large $n$. As in the proof of Theorem~\ref{thm:estimated-c-density-denoising}, let $\mathcal T_n:=\Sc_n\cap\{\max_i|X_i|\leq K_R(\log n)^{1/r}\}$. Using the Hellinger tail bound established in the proof of Theorem~\ref{thm:estimated-c-density-denoising} with $t=\sqrt3$ and enlarging $K$ if necessary, we obtain
\begin{align*}
    \P_{G^*}\{\mathcal T_n,\ \H(\widehat f_n,f^*)>K\sqrt{r_n}\}\leq2n^{-3},
\end{align*}
where $r_n$ is defined in \eqref{eq:estimated-c-rn}. We will work on the event
\begin{equation}\label{eq:moment-deconvolution-event}
    \mathcal E_n=\mathcal T_n\cap\{\H(\widehat f_n,f^*)\leq K\sqrt{r_n}\}.
\end{equation}
By \eqref{eq:estimated-c-denoising-good-event} and the preceding Hellinger bound, 
we have $\P_{G^*}(\mathcal E_n^c)\leq 3n^{-2}+3n^{-3}\lesssim n^{-2}$. On $\mathcal E_n$, define
\begin{equation}\label{eq:moment-deconvolution-Q-definitions}
\begin{aligned}
    \underline c_0&:=\hat c_0\wedge\czero, & \underline\sigma_0^2&:=1+\underline c_0^2,\\
    \widehat Q_n&:=\npmle\star N\!\left(0,(\hat c_0^2-\czero^2)_+\right), & Q_n^*&:=\trueprior\star N\!\left(0,(\czero^2-\hat c_0^2)_+\right).
\end{aligned}
\end{equation}
Then, on $\mathcal E_n$, we have $\czero/\sqrt2\leq \underline c_0\leq\czero$ for all sufficiently large $n$. Since $\hat c_0^+=\hat c_0$ on $\mathcal E_n$, we have
\begin{align*}
    \widehat f_n&=\widehat Q_n\star\phi_{\underline\sigma_0},\qquad f^*=Q_n^*\star\phi_{\underline\sigma_0},\\
    \widehat g_n^+&=\widehat Q_n\star\phi_{\underline c_0},\qquad g^*=Q_n^*\star\phi_{\underline c_0}.
\end{align*}
Then, by Plancherel's theorem, we have
\begin{equation}\label{eq:moment-deconvolution-plancherel}
    \|\widehat g_n^+-g^*\|_{L_2}^2=\frac1{2\pi}\int e^{-\underline c_0^2t^2}|\varphi_{\widehat Q_n}(t)-\varphi_{Q_n^*}(t)|^2\,\diff t.
\end{equation}
Let $\varepsilon>0$ be a sufficiently small fixed constant and define
\begin{equation}\label{eq:moment-deconvolution-low-cutoff}
    m:=\lfloor\varepsilon(\log n)^{r/2}\rfloor, \qquad T_n^2:=\frac{\log(1/r_n)-m}{\underline\sigma_0^2}.
\end{equation}
We decompose $\|\widehat g_n^+-g^*\|_{L_2}^2$ into the two terms:
\begin{align*}
    \|\widehat g_n^+-g^*\|_{L_2}^2 &=\underbrace{\frac1{2\pi}\int_{|t|\leq T_n}e^{-\underline c_0^2t^2} \left|\varphi_{\widehat Q_n}(t)-\varphi_{Q_n^*}(t)\right|^2\,\diff t}_{\mathrm{(I)}}\\
    &+\underbrace{\frac1{2\pi}\int_{|t|>T_n}e^{-\underline c_0^2t^2}\left|\varphi_{\widehat Q_n}(t)-\varphi_{Q_n^*}(t)\right|^2\,\diff t}_{\mathrm{(II)}}.
\end{align*}
On $|t|\leq T_n$, we have
\begin{align*}
    e^{-\underline c_0^2t^2}|\varphi_{\widehat Q_n}(t)-\varphi_{Q_n^*}(t)|^2 =e^{t^2}|\varphi_{\widehat f_n}(t)-\varphi_{f^*}(t)|^2.
\end{align*}
Moreover, $\|\widehat f_n-f^*\|_{L_2}^2\lesssim \H^2(\widehat f_n,f^*)$, since both marginal densities are bounded by $\|\phi_1\|_\infty$. Therefore, on $\mathcal E_n$,
\begin{align*}
    \mathrm{(I)} \leq e^{T_n^2}\|\widehat f_n-f^*\|_{L_2}^2 \lesssim r_ne^{T_n^2} =r_n^{\underline\alpha_0}e^{-m/\underline\sigma_0^2},
\end{align*}
where $\underline\alpha_0:=\underline c_0^2/\underline\sigma_0^2$ and the equality follows from the definition of $T_n$ in \eqref{eq:moment-deconvolution-low-cutoff}.  Also, by Lemma~\ref{lem:moment-deconvolution-high-frequency} below, on $\mathcal E_n$,
\begin{align*}
    \mathrm{(II)}\lesssim r_n^{\underline\alpha_0}e^{-m/\underline\sigma_0^2}.
\end{align*}
Combining the bounds for $\mathrm{(I)}$ and $\mathrm{(II)}$, we have
\begin{align*}
    \|\widehat g_n^+-g^*\|_{L_2}^2\1v_{\mathcal E_n}
    \lesssim r_n^{\underline\alpha_0}e^{-m/\underline\sigma_0^2}.
\end{align*}
Note that, on $\mathcal E_n\subseteq\Sc_n$, we have
\begin{align*}
    0\leq\alpha_0-\underline\alpha_0 =\frac{\czero^2-\underline c_0^2}{(1+\czero^2)(1+\underline c_0^2)} =\frac{(\czero^2-\hat c_0^2)_+}{(1+\czero^2)(1+\underline c_0^2)} \leq(\czero^2-\hat c_0^2)_+\leq K\rho_n,
\end{align*}
where the last inequality follows from the lower bound on $\hat c_0^2$ in \eqref{eq:modified-c-good-event}. Combining this with $\underline\sigma_0^2\leq\sigma_0^2$, we have
\begin{align*}
    r_n^{\underline\alpha_0}e^{-m/\underline\sigma_0^2}
    &\leq r_n^{\alpha_0}\exp\{K\rho_n\log(1/r_n)\}e^{-m/\sigma_0^2}\\
    &\lesssim r_n^{\alpha_0}(\log n)^K e^{-m/\sigma_0^2},
\end{align*}
where $K>0$ depends only on $\czero$ and $\gammav$.

Finally, $\|\widehat g_n^+-g^*\|_{L_2}^2\lesssim ((\hat c_0^+)^{-1}+\czero^{-1})\lesssim \rho_n^{-1/2}$ by the floor $\sqrt{\rho_n}$ in \eqref{eq:deconvolution-floor}. Since $\P(\mathcal E_n^c)\lesssim n^{-2}$, $\E_{G^*}[\|\widehat g_n^+-g^*\|_{L_2}^2\1v_{\mathcal E_n^c}] \lesssim n^{-2}\rho_n^{-1/2}\lesssim n^{-1}$. Thus, we have
\begin{equation}\label{eq:moment-deconvolution-intermediate-rate}
    \E[\|\widehat g_n^+-g^*\|_{L_2}^2]\lesssim r_n^{\alpha_0}(\log n)^K e^{-m/\sigma_0^2}+n^{-1}.
\end{equation}

For $0<r<2$, substituting the definitions of $m$ in \eqref{eq:moment-deconvolution-low-cutoff} and $r_n$ in \eqref{eq:estimated-c-rn}, and using $m\geq\varepsilon(\log n)^{r/2}-1$, gives
\begin{align*}
    r_n^{\alpha_0}(\log n)^K e^{-m/\sigma_0^2} &\lesssim n^{-\alpha_0}(\log n)^{\alpha_0(3/2+1/r)+K} \exp\left\{-\frac{\varepsilon}{\sigma_0^2}(\log n)^{r/2}\right\}\\
    &\leq n^{-\alpha_0}\exp\left\{-\frac{\varepsilon}{2\sigma_0^2}(\log n)^{r/2}\right\}
\end{align*}
for all sufficiently large $n$, where the last inequality uses $\log\log n=o\{(\log n)^{r/2}\}$. Moreover, since $\alpha_0<1$ and $r<2$,
\begin{align*}
    \frac{n^{-1}}{n^{-\alpha_0}\exp\{-\varepsilon(\log n)^{r/2}/(2\sigma_0^2)\}} =\exp\left\{-(1-\alpha_0)\log n+\frac{\varepsilon}{2\sigma_0^2}(\log n)^{r/2}\right\} \longrightarrow0.
\end{align*}
Combining these bounds with \eqref{eq:moment-deconvolution-intermediate-rate} proves \eqref{eq:moment-deconvolution-upper} with $k_1=\varepsilon/(2\sigma_0^2)$.

When $r=2$, we have $r_n=(\log n)^2/n$ and $m=\lfloor\varepsilon\log n\rfloor$. Decrease $\varepsilon$ if necessary so that $\delta:=\varepsilon/\sigma_0^2$ satisfies $\alpha_0+\delta<1$. Substituting these definitions and using $m\geq\varepsilon\log n-1$ gives
\begin{equation}\label{eq:subgaussian-deconvolution-upper}
\begin{aligned}
    r_n^{\alpha_0}(\log n)^K e^{-m/\sigma_0^2}  \lesssim n^{-(\alpha_0+\delta)}(\log n)^{2\alpha_0+K} \leq n^{-(\alpha_0+\delta/2)}
\end{aligned}
\end{equation}
for all sufficiently large $n$, where the last inequality uses $(\log n)^{2\alpha_0+K}\leq n^{\delta/2}$. Moreover, since $\alpha_0+\delta/2<1$,
\begin{align*}
    \frac{n^{-1}}{n^{-(\alpha_0+\delta/2)}} =n^{-(1-\alpha_0-\delta/2)}\longrightarrow0.
\end{align*}
Combining these bounds with \eqref{eq:moment-deconvolution-intermediate-rate} proves \eqref{eq:moment-deconvolution-upper} at $r=2$ with $k_1=\delta/2$. This completes the proof.
\end{proof}

\begin{lemma}\label{lem:moment-deconvolution-high-frequency}
Under the assumptions of Theorem~\ref{thm:moment-deconvolution} with $0<r\leq2$, let $r_n$ be defined by \eqref{eq:estimated-c-rn}, and let $\mathcal E_n$ be defined by \eqref{eq:moment-deconvolution-event}. On $\mathcal E_n$, let $\underline c_0,\underline\sigma_0,\underline\alpha_0,\widehat Q_n$, and $Q_n^*$ be as defined in
the proof of Theorem~\ref{thm:moment-deconvolution}. For any sufficiently small fixed $\varepsilon>0$, set $m:=\lfloor\varepsilon(\log n)^{r/2}\rfloor$ and define $T_n$ by \eqref{eq:moment-deconvolution-low-cutoff}. Then, for all sufficiently large $n$,
\begin{align*}
    \mathrm{(II)}:=\frac1{2\pi}\int_{|t|>T_n}e^{-\underline c_0^2t^2} |\varphi_{\widehat Q_n}(t)-\varphi_{Q_n^*}(t)|^2\,\diff t \leq K r_n^{\underline\alpha_0}e^{-m/\underline\sigma_0^2} \qquad\text{on }\mathcal E_n.
\end{align*}
Here, $K>0$ depends only on $\czero$ and $\gammav$ in Assumption~\ref{ass:moment-fourier}.
\end{lemma}

\begin{proof}
We decompose
\begin{equation}\label{eq:moment-deconvolution-tail-decomposition}
\begin{aligned}
    \mathrm{(II)} &=\underbrace{\frac1{2\pi}\int_{T_n<|t|\leq U_n}e^{-\underline c_0^2t^2} \left|\varphi_{\widehat Q_n}(t)-\varphi_{Q_n^*}(t)\right|^2\,\diff t}_{\mathrm{(A)}}\\
    &+\underbrace{\frac1{2\pi}\int_{|t|>U_n}e^{-\underline c_0^2t^2} \left|\varphi_{\widehat Q_n}(t)-\varphi_{Q_n^*}(t)\right|^2\,\diff t}_{\mathrm{(B)}},
\end{aligned}
\end{equation}
where 
\begin{equation}\label{eq:moment-deconvolution-upper-cutoff}
    U_n^2:=\frac{\underline\alpha_0\log(1/r_n)+m/\underline\sigma_0^2} {\underline c_0^2}=T_n^2+\frac{m}{\underline c_0^2}.
\end{equation}

\noindent\emph{(i) Control of $\mathrm{(B)}$.}
Since $\underline c_0\geq\czero/\sqrt2$ on $\mathcal E_n$ and $|\varphi_{\widehat Q_n}-\varphi_{Q_n^*}|\leq2$, we have
\begin{align}\label{eq:moment-deconvolution-B-bound}
    \mathrm{(B)} \lesssim \int_{U_n}^{\infty}e^{-(\underline c_0^2)t^2}\,\diff t \lesssim \frac {1}{U_n}e^{-(\underline c_0^2)U_n^2} =\frac {1}{U_n}r_n^{\underline\alpha_0}e^{-m/\underline\sigma_0^2} \lesssim r_n^{\underline\alpha_0}e^{-m/\underline\sigma_0^2}.
\end{align}
Here, the second inequality follows from the Gaussian tail bound
\begin{align}\label{eq:Mill's inequality}
    \P\{|N(0,s^2)|>T\}\leq\sqrt{\frac{2}{\pi}}\frac{s}{T}e^{-T^2/(2s^2)},\qquad s,T>0,
\end{align}
the equality follows from the definition of $U_n$, and the last inequality uses $U_n\asymp\sqrt{\log n}$.

\noindent\emph{(ii) Interpolation formula and choice of nodes.}
Next, we control $\mathrm{(A)}$. Note that, for every $t\in\Real$, 
\begin{equation}\label{eq:moment-deconvolution-hellinger-characteristic}
\begin{aligned}
    &|\varphi_{\widehat Q_n}(t)-\varphi_{Q_n^*}(t)|=\exp\left\{\frac{\underline\sigma_0^2}{2}t^2\right\}|\varphi_{\widehat f_n}(t)-\varphi_{f^*}(t)|\\
    &=\exp\left\{\frac{\underline\sigma_0^2}{2}t^2\right\}\left|\int e^{itx}\{\sqrt{\widehat f_n(x)}-\sqrt{f^*(x)}\}\{\sqrt{\widehat f_n(x)}+\sqrt{f^*(x)}\}\,\diff x\right|\\
    &\lesssim \H(\widehat f_n,f^*)\exp\left\{\frac{\underline\sigma_0^2}{2}t^2\right\} \lesssim \sqrt{r_n}\exp\left\{\frac{\underline\sigma_0^2}{2}t^2\right\}
\end{aligned}
\end{equation}
on $\mathcal E_n$. 
Applying the Gaussian tail argument for $\mathrm{(B)}$ to $\mathrm{(A)}$ would give a bound too weak to yield the desired rate. Therefore, we take an alternative approach to bounding $\varphi_{\widehat Q_n}-\varphi_{Q_n^*}$ on $[T_n,U_n]$. We use Lagrange interpolation at $m$ points in $[0,T_n]$ at which \eqref{eq:moment-deconvolution-hellinger-characteristic} gives an exponentially small bound, with $m$ defined in~\eqref{eq:moment-deconvolution-low-cutoff}. This argument yields
\begin{equation}\label{eq:moment-deconvolution-band}
    \sup_{T_n\leq |t|\leq U_n}
    |\varphi_{\widehat Q_n}(t)-\varphi_{Q_n^*}(t)|
    \lesssim_{\czero,\gammav}e^{-2m}.
\end{equation}
We prove \eqref{eq:moment-deconvolution-band} in two steps. We first give the interpolation formula and choose the nodes below. In step~$(iii)$, we bound the interpolating polynomial and the remainder.

For distinct $t_0,\ldots,t_{m-1}<T_n$, define the Lagrange basis polynomials
\begin{align}\label{eq:moment-deconvolution-lagrange-basis}
    \ell_j(t)=\prod_{\substack{0\leq k\leq m-1\\k\neq j}} \frac{t-t_k}{t_j-t_k},\qquad j=0,\ldots,m-1.
\end{align}
Note that $\ell_j(t_k)=\1v\{j=k\}$. Applying the Lagrange interpolation remainder formula (see, e.g., Theorem 3.3 of \citet{BurdenFaires2011}) separately to the real and imaginary parts gives
\begin{equation}\label{eq:moment-deconvolution-lagrange-interpolation}
\begin{aligned}
    \varphi_{\widehat Q_n}(t)-\varphi_{Q_n^*}(t) &=\sum_{j=0}^{m-1}\{\varphi_{\widehat Q_n}(t_j)-\varphi_{Q_n^*}(t_j)\}\ell_j(t)+R_m(t),\\
    |R_m(t)| &\lesssim \frac{1}{m!}\prod_{j=0}^{m-1}|t-t_j| \left\|\left(\varphi_{\widehat Q_n}-\varphi_{Q_n^*}\right)^{(m)}\right\|_\infty,
\end{aligned}
\end{equation}
for any $t \in [T_n, U_n]$. To bound the approximation error $R_m(t)$ uniformly over $t\in[T_n,U_n]$, we will choose $d_n>0$ and interpolation nodes satisfying
\begin{align}\label{eq:moment-deconvolution-node-conditions}
    U_n-T_n\leq d_n, \qquad 0 < T_n-2d_n\leq t_j\leq T_n-d_n,\qquad j=0,\ldots,m-1.
\end{align}
These inequalities ensure that $|t-t_j|\leq3d_n$ for every $t\in[T_n,U_n]$. For this, we first bound $U_n-T_n$ to determine a suitable choice of $d_n$. For all sufficiently large $n$, both $T_n$ in \eqref{eq:moment-deconvolution-low-cutoff} and $U_n$ in \eqref{eq:moment-deconvolution-upper-cutoff} are of order $\sqrt{\log n}$. Moreover, their definitions give
\begin{align*}
    U_n^2-T_n^2 =\frac{m}{\underline c_0^2}\leq\frac{2m}{\czero^2}.
\end{align*}
Consequently, we have
\begin{align}\label{eq:moment-deconvolution-cutoff-distance}
    U_n-T_n=\frac{m}{\underline c_0^2(U_n+T_n)}\lesssim_{\czero}\frac{m}{\sqrt{\log n}},
\end{align}
on $\mathcal E_n$. Let $\eta>0$ be a sufficiently small constant, and set
\begin{equation}\label{eq:moment-deconvolution-node-width}
    d_n:=\eta m^{1-1/r}.
\end{equation}
Since $m\asymp\varepsilon(\log n)^{r/2}$ (see \eqref{eq:moment-deconvolution-low-cutoff}), we have
\begin{align*}
    \frac{m}{d_n\sqrt{\log n}}\asymp\frac{\varepsilon^{1/r}}{\eta}.
\end{align*}
For sufficiently small $\varepsilon>0$ depending on $\eta$, the preceding bounds and \eqref{eq:moment-deconvolution-cutoff-distance} imply $U_n-T_n\leq d_n$ and $2d_n<T_n$ for all sufficiently large $n$.

Now, we choose the $m$ equally spaced interpolation nodes
\begin{align*}
    t_j:=T_n-2d_n+\frac{jd_n}{m-1},\qquad j=0,\ldots,m-1.
\end{align*}
These nodes satisfy the second condition in \eqref{eq:moment-deconvolution-node-conditions}.

\noindent\emph{(iii) Control of the interpolation error.}
We now bound $\varphi_{\widehat Q_n}(t)-\varphi_{Q_n^*}(t)$ uniformly over $t\in[T_n,U_n]$. Recall from \eqref{eq:moment-deconvolution-lagrange-interpolation} that, for any $t\in[T_n,U_n]$,
\begin{align*}
    \varphi_{\widehat Q_n}(t)-\varphi_{Q_n^*}(t) &=\sum_{j=0}^{m-1}\{\varphi_{\widehat Q_n}(t_j)-\varphi_{Q_n^*}(t_j)\}\ell_j(t)+R_m(t).
\end{align*}
We control the sum and the remainder separately.

By \eqref{eq:moment-deconvolution-node-conditions}, for $T_n\leq t\leq U_n$, each $|t-t_j|$ is at most $3d_n$, and the sum of the absolute Lagrange coefficients in \eqref{eq:moment-deconvolution-lagrange-basis} satisfies
\begin{align*}
    \sum_{j=0}^{m-1}|\ell_j(t)|\leq\sum_{j=0}^{m-1}\frac{\{3(m-1)\}^{m-1}}{j!(m-1-j)!} =\frac{\{6(m-1)\}^{m-1}}{(m-1)!}\leq(6e)^m.
\end{align*}
Here, we used $\prod_{k\neq j}|t_j-t_k|=\{d_n/(m-1)\}^{m-1}j!(m-1-j)!$ for the equally spaced nodes, $\sum_{j=0}^{m-1}\binom{m-1}{j}=2^{m-1}$, and $(m-1)!\geq\{(m-1)/e\}^{m-1}$. Thus, \eqref{eq:moment-deconvolution-lagrange-interpolation} gives
\begin{equation}\label{eq:moment-deconvolution-interpolation}
\begin{aligned}
    \sup_{T_n\leq t\leq U_n} |\varphi_{\widehat Q_n}(t)-\varphi_{Q_n^*}(t)| &\lesssim(6e)^m\max_j |\varphi_{\widehat Q_n}(t_j)-\varphi_{Q_n^*}(t_j)|\\
    &+\frac{(3d_n)^m}{m!}\left\|\left(\varphi_{\widehat Q_n}-\varphi_{Q_n^*}\right)^{(m)}\right\|_\infty.
\end{aligned}
\end{equation}
By Lemma~\ref{lem:moment-deconvolution-derivative}, we can control the $m$th derivative appearing in the interpolation error:
\begin{equation}\label{eq:moment-deconvolution-derivative-control}
    \left\| \left(\varphi_{\widehat Q_n}-\varphi_{Q_n^*}\right)^{(m)}\right\|_\infty \leq \{K_1m^{1/r}\}^m
\end{equation}
on $\mathcal E_n$, where $K_1>0$ depends only on $\czero$ and $\gammav$. Using \eqref{eq:moment-deconvolution-node-width} and $m!\geq(m/e)^m$, and taking $\eta$ small enough that $3eK_1\eta\leq e^{-4}$, we bound the second term in \eqref{eq:moment-deconvolution-interpolation} by
\begin{equation}\label{eq:moment-deconvolution-remainder-bound}
    \frac{(3d_n)^m}{m!}\{K_1m^{1/r}\}^m \leq (3eK_1\eta)^m\leq e^{-4m}.
\end{equation}
For the first term in \eqref{eq:moment-deconvolution-interpolation}, applying \eqref{eq:moment-deconvolution-hellinger-characteristic} at the interpolation nodes and using $t_j\leq T_n-d_n$ gives
\begin{align*}
    |\varphi_{\widehat Q_n}(t_j)-\varphi_{Q_n^*}(t_j)|
    &\lesssim_{\czero,\gammav}\sqrt{r_n}e^{\underline\sigma_0^2(T_n-d_n)^2/2}\\
    &=\exp\{-m/2-\underline\sigma_0^2T_nd_n+\underline\sigma_0^2d_n^2/2\}.
\end{align*}
Then it follows that
\begin{align*}
    (6e)^m\max_j|\varphi_{\widehat Q_n}(t_j)-\varphi_{Q_n^*}(t_j)|
    \lesssim_{\czero,\gammav}\exp\left\{m\log(6e)-\frac m2-\underline\sigma_0^2T_nd_n+\frac{\underline\sigma_0^2d_n^2}{2}\right\}.
\end{align*}
Since $d_n<T_n/2$ and $T_nd_n/m\asymp\eta\varepsilon^{-1/r}$, there exists
a constant $c>0$ such that
\begin{align*}
    m\log(6e)-\frac m2-\underline\sigma_0^2T_nd_n +\frac{\underline\sigma_0^2d_n^2}{2}
    &\leq m\log(6e)-\frac m2-\frac{3\underline\sigma_0^2}{4}T_nd_n\\
    &\leq m\left\{\log(6e)-\frac12-c\eta\varepsilon^{-1/r}\right\}\\
    &\leq-2m
\end{align*}
for every sufficiently small $\varepsilon>0$. Thus, the first term in
\eqref{eq:moment-deconvolution-interpolation} is $\lesssim_{\czero,\gammav}e^{-2m}$. Hence, combining this bound with \eqref{eq:moment-deconvolution-remainder-bound} in \eqref{eq:moment-deconvolution-interpolation} and using symmetry gives \eqref{eq:moment-deconvolution-band}.

\noindent\emph{(iv) Control of $\mathrm{(A)}$.}
Recall the definition of $\mathrm{(A)}$ in \eqref{eq:moment-deconvolution-tail-decomposition}. Then, \eqref{eq:moment-deconvolution-band} and the Gaussian tail bound \eqref{eq:Mill's inequality} give
\begin{equation}\label{eq:moment-deconvolution-A-bound}
\begin{aligned}
    \mathrm{(A)} &\lesssim_{\czero,\gammav}\frac {1}{T_n}e^{-4m-(\underline c_0^2)T_n^2}
    =\frac {1}{T_n}r_n^{\underline\alpha_0}\exp\{-(4-\underline\alpha_0)m\}\\
    &\leq r_n^{\underline\alpha_0}\exp\{-(4-\underline\alpha_0)m\}.
\end{aligned}
\end{equation}
Here, the equality uses $\underline c_0^2T_n^2 =\underline\alpha_0\{\log(1/r_n)-m\}$, where $\underline\alpha_0=\underline c_0^2/\underline\sigma_0^2 =\underline c_0^2/(1+\underline c_0^2)$, and the last inequality uses $T_n\geq1$ for all sufficiently large $n$.
Combining \eqref{eq:moment-deconvolution-B-bound} and \eqref{eq:moment-deconvolution-A-bound} gives
\begin{align*}
    \mathrm{(II)} &\lesssim_{\czero,\gammav}r_n^{\underline\alpha_0}\exp\{-(4-\underline\alpha_0)m\}
    +r_n^{\underline\alpha_0}e^{-m/\underline\sigma_0^2}
    \lesssim_{\czero,\gammav}r_n^{\underline\alpha_0}e^{-m/\underline\sigma_0^2},
\end{align*}
where the last inequality follows from $4-\underline\alpha_0-\underline\sigma_0^{-2}=3>0$.
\end{proof}

\begin{lemma}\label{lem:moment-deconvolution-derivative}
Under the assumptions of Theorem~\ref{thm:moment-deconvolution} with $0<r\leq2$, let $r_n$ and $\mathcal E_n$ be defined by \eqref{eq:estimated-c-rn} and \eqref{eq:moment-deconvolution-event}, respectively. On $\mathcal E_n$, let $\widehat Q_n$ and $Q_n^*$ be as defined in~\eqref{eq:moment-deconvolution-Q-definitions}. There exists a sufficiently small constant $\varepsilon>0$ such that, for every integer $m$ satisfying $2\leq m\leq\varepsilon(\log n)^{r/2}$, and all sufficiently large $n$,
\begin{equation}\label{eq:moment-deconvolution-derivatives}
    \left\|\left(\varphi_{\widehat Q_n}-\varphi_{Q_n^*}\right)^{(m)}\right\|_\infty\leq\{Km^{1/r}\}^m \qquad\text{on }\mathcal E_n,
\end{equation}
where $K>0$ depends only on $\czero$ and $\gammav$ in Assumption~\ref{ass:moment-fourier}.
\end{lemma}

\begin{proof}
Since $\npmle$ has bounded support and $\trueprior$ satisfies \eqref{eq:simultaneous-moment-control}, both $\widehat Q_n$ and $Q_n^*$ in \eqref{eq:moment-deconvolution-Q-definitions} have finite $m$th absolute moments. Hence, differentiating under the integral gives, for every $t\in\Real$,
\begin{align*}
    \left(\varphi_{\widehat Q_n}-\varphi_{Q_n^*}\right)^{(m)}(t)  &=\int(iu)^me^{itu}\,\diff\widehat Q_n(u)-\int(iu)^me^{itu}\,\diff Q_n^*(u).
\end{align*}
Consequently,
\begin{align*}
    \left\|\left(\varphi_{\widehat Q_n}-\varphi_{Q_n^*}\right)^{(m)}\right\|_\infty \leq\int|u|^m\,\diff(\widehat Q_n+Q_n^*)(u).
\end{align*}
Therefore, it suffices to show that
\begin{align*}
    \int|u|^m\,\diff(\widehat Q_n+Q_n^*)(u)\leq\{Km^{1/r}\}^m
\end{align*}
for a constant $K>0$ depending only on $\czero$ and $\gammav$.

First, the moment condition \eqref{eq:simultaneous-moment-control} and the definition of $Q_n^*$ imply, for every $p\geq1$,
\begin{equation}\label{eq:moment-deconvolution-Q-moment}
    \left\{\int |u|^p\,\diff Q_n^*(u)\right\}^{1/p}\lesssim_{\czero,\gammav}p^{1/r}.
\end{equation}

Next, we bound $\int |u|^m\,\diff\widehat Q_n(u)$ by first controlling $\int |u|^m\,\diff\npmle(u)$. Let $U\sim\npmle$ and $Z\sim N(0,1)$ be independent. Then $U+\sqrt{1+\hat c_0^2}Z$ has density $\widehat f_n$. Since $\E[U+\sqrt{1+\hat c_0^2}Z\mid U]=U$ and $u\mapsto|u|^m$ is convex, conditional Jensen's inequality gives
\begin{equation}\label{eq:moment-deconvolution-jensen}
    \begin{aligned}
        \int|u|^m\,\diff\npmle(u)=\E|U|^m &=\E\left|\E\left[U+\sqrt{1+\hat c_0^2}Z\mid U\right]\right|^m\\
        &\leq\E|U+\sqrt{1+\hat c_0^2}Z|^m=\int|x|^m\widehat f_n(x)\,\diff x \\
        &\leq \int |x|^m f^*(x)\,\diff x +\left|\int |x|^m\{\widehat f_n(x)-f^*(x)\}\,\diff x\right|.
    \end{aligned}
\end{equation}
We first control the true marginal moments. Since $f^*=Q_n^*\star\phi_{\underline\sigma_0}$, Minkowski's inequality and the Gaussian moment bound imply
\begin{equation}\label{eq:moment-deconvolution-true-moment}
    \left\{\int |x|^p f^*(x)\,\diff  x\right\}^{1/p}\lesssim_{\czero,\gammav}p^{1/r}, \qquad p\geq1.
\end{equation}

Next, we bound the second term in \eqref{eq:moment-deconvolution-jensen}, namely $\left|\int |x|^m\{\widehat f_n(x)-f^*(x)\}\,\diff x\right|$. Note that, by Minkowski's inequality, for $2\leq m\leq\log n$,
\begin{equation}\label{eq:moment-deconvolution-fitted-moment}
\begin{aligned}
    \left\{\int |x|^{2m}\widehat f_n(x)\,\diff x\right\}^{1/(2m)}
    &\leq(\E|U|^{2m})^{1/(2m)}+\sqrt{1+\hat c_0^2}(\E|Z|^{2m})^{1/(2m)}\\
    &\leq\max_i|X_i|+\sqrt{2m(1+\hat c_0^2)}\\
    &\lesssim_{\czero,\gammav}(\log n)^{1/r}+\sqrt m\\
    &\lesssim_{\czero,\gammav}(\log n)^{1/r}.
\end{aligned}
\end{equation}
Here, $|U|\leq\max_i|X_i| \lesssim_{\czero,\gammav}(\log n)^{1/r}$ on $\mathcal E_n$ by the definition of this event in \eqref{eq:moment-deconvolution-event}. Moreover, $1+\hat c_0^2\lesssim_{\czero,\gammav}1$ by \eqref{eq:modified-c-good-event}. The last inequality uses $m\leq\log n$ and $r\leq2$. Applying \eqref{eq:moment-deconvolution-true-moment} with $p=2m$ and combining it with \eqref{eq:moment-deconvolution-fitted-moment} yields
\begin{equation}\label{eq:moment-deconvolution-marginal-moments}
    \left\{\int |x|^{2m}\{\widehat f_n(x)+f^*(x)\}\,\diff x\right\}^{1/(2m)} \lesssim_{\czero,\gammav}(\log n)^{1/r}.
\end{equation}
Then applying the Cauchy--Schwarz inequality gives
\begin{align*}
    &\left|\int |x|^m\{\widehat f_n(x)-f^*(x)\}\,\diff x\right|\\
    &=\left|\int |x|^m\{\sqrt{\widehat f_n(x)}-\sqrt{f^*(x)}\}\{\sqrt{\widehat f_n(x)}+\sqrt{f^*(x)}\}\,\diff x\right|\\
    &\leq\left\{\int\{\sqrt{\widehat f_n(x)}-\sqrt{f^*(x)}\}^2\,\diff x\right\}^{1/2}
    \left\{\int |x|^{2m}\{\sqrt{\widehat f_n(x)}+\sqrt{f^*(x)}\}^2\,\diff x\right\}^{1/2}\\
    &\leq2\H(\widehat f_n,f^*)\left\{\int |x|^{2m}\{\widehat f_n(x)+f^*(x)\}\,\diff x\right\}^{1/2}\\
    &\leq\sqrt{r_n}\{K(\log n)^{1/r}\}^m,
\end{align*}
where the last inequality uses $\mathcal E_n$ and \eqref{eq:moment-deconvolution-marginal-moments}. Substituting this bound and \eqref{eq:moment-deconvolution-true-moment} with $p=m$ into \eqref{eq:moment-deconvolution-jensen} gives
\begin{align*}
    \int |u|^m\,\diff\npmle(u) \leq \{Km^{1/r}\}^m +\sqrt{r_n}\{K(\log n)^{1/r}\}^m.
\end{align*}
Using $r_n=(\log n)^{3/2+1/r}/n$, the logarithm of the ratio of the second term to the first is
\begin{align*}
    \log\left\{\frac{\sqrt{r_n}\{K(\log n)^{1/r}\}^m}{\{Km^{1/r}\}^m}\right\}
    = -\frac12\log n+\frac mr\log\left(\frac{\log n}{m}\right)+\left(\frac34+\frac1{2r}\right)\log\log n.
\end{align*}
This expression tends to $-\infty$ uniformly over the integers $m$ satisfying $2\leq m\leq\varepsilon(\log n)^{r/2}$ when $0<r<2$, and also when $r=2$ if $\varepsilon$ is sufficiently small. Hence, for all sufficiently large $n$,
\begin{align*}
    \left\{\int |u|^m\,\diff\npmle(u)\right\}^{1/m}\lesssim_{\czero,\gammav}m^{1/r}.
\end{align*}
Finally, recall that $\widehat Q_n=\npmle\star N(0,(\hat c_0^2-\czero^2)_+)$. Since $(\hat c_0^2-\czero^2)_+ \lesssim_{\czero,\gammav}\rho_n=o(1)$ on $\mathcal E_n$ by Theorem~\ref{thm:modified-c-rate}, Minkowski's inequality gives
\begin{align*}
    \left\{\int |u|^m\,\diff\widehat Q_n(u)\right\}^{1/m}
    &\leq \left\{\int |u|^m\,\diff\npmle(u)\right\}^{1/m}
    +\sqrt{(\hat c_0^2-\czero^2)_+}(\E|Z|^m)^{1/m}\\
    &\lesssim_{\czero,\gammav}m^{1/r}+\sqrt{\rho_n m}\\
    &\lesssim_{\czero,\gammav}m^{1/r}.
\end{align*}
The last inequality uses $\rho_n\leq1$ for all sufficiently large $n$ and $\sqrt m\leq m^{1/r}$ for $r\leq2$. Combining this with \eqref{eq:moment-deconvolution-Q-moment} proves \eqref{eq:moment-deconvolution-derivatives} after enlarging $K$.
\end{proof}

\appsubsection{Proof of Theorem~\ref{thm:moment-deconvolution-lower}}\label{app:proof-moment-deconvolution-lower}
\begin{proof}[Proof of Theorem~\ref{thm:moment-deconvolution-lower}]
We first construct two mixing distributions in $\mathcal H_r$. Let $K_0,r,a,\vartheta$, and $t_0$ be as in Assumption~\ref{ass:moment-fourier}~$(i)$. Fix $0<w_0<1-a(1+t_0)^{-\vartheta}$ and $0<\tau\leq1/2$. For $v\geq1$ and $T>0$, set
\begin{equation}\label{eq:moment-deconvolution-alternatives}
    H_\pm(\diff u)=(1-w)\delta_0(\diff u)+w\phi_{\sqrt v}(u)\{1\pm\tau\sin(Tu)\}\,\diff u, \qquad w=w_0e^{-B_0v^{r/(2-r)}},
\end{equation}
where $B_0>0$ is a sufficiently large fixed constant. We verify that $H_\pm$ satisfy the moment condition \eqref{eq:simultaneous-moment-control} and the Fourier lower bound \eqref{eq:fourier-lower-bound}. Since the sine perturbation is odd, both measures have mass one. Moreover, for every $p\geq1$,
\begin{align*}
    \int |u|^p\,\diff H_\pm(u) &=w\int |u|^p\phi_{\sqrt v}(u)\{1\pm\tau\sin(Tu)\}\,\diff u\\
    &=w\int |u|^p\phi_{\sqrt v}(u)\,\diff u =wv^{p/2}\E|Z|^p,
\end{align*}
where $Z\sim N(0,1)$. To verify the moment condition uniformly in $v$, put $q=1/r-1/2>0$, so $r/(2-r)=1/(2q)$. Since $(\E|Z|^p)^{1/p}\leq\sqrt p$ and $w_0<1$,
\begin{align*}
    \sup_{p\geq1}\frac{\{wv^{p/2}\E|Z|^p\}^{1/p}}{p^{1/r}} &\leq\sup_{p>0}\sqrt v\,p^{-q}e^{-B_0v^{1/(2q)}/p} =\sqrt v\left(\frac q{eB_0v^{1/(2q)}}\right)^q =\left(\frac q{eB_0}\right)^q\leq K_0
\end{align*}
for sufficiently large $B_0$, where the first equality follows by maximizing over $p>0$ with the maximum attained at $p=B_0v^{1/(2q)}/q$. Thus, $H_\pm$ satisfies \eqref{eq:simultaneous-moment-control}. Next, the sine perturbation also vanishes in the real part of the characteristic function, so
\begin{align*}
    \operatorname{Re}\varphi_{H_\pm}(t) &=1-w+w\int\cos(tu)\phi_{\sqrt v}(u)\,\diff u =1-w+we^{-vt^2/2}.
\end{align*}
Since $w\leq w_0<1-a(1+t_0)^{-\vartheta}$, for $|t|\geq t_0$ we obtain
\begin{align*}
    |\varphi_{H_\pm}(t)| &\geq\operatorname{Re}\varphi_{H_\pm}(t) \geq1-w_0 >a(1+t_0)^{-\vartheta} \geq a(1+|t|)^{-\vartheta}.
\end{align*}
Thus, both alternatives belong to the class $\mathcal H_r$.

We next compare the marginal and prior densities. By Lemma~\ref{lem:Gaussian-modulation} below and $\phi_{\sqrt v}\star\phi_s=\phi_{\sqrt{v+s^2}}$, for every $s>0$,
\begin{equation}\label{eq:moment-deconvolution-convolution}
\begin{aligned}
    H_\pm\star\phi_s(x)=(1-w)\phi_s(x)+w\phi_{\sqrt{v+s^2}}(x) \left[1\pm\tau e^{-vs^2T^2/\{2(v+s^2)\}}\sin\left(\frac{vTx}{v+s^2}\right)\right].
\end{aligned}
\end{equation}
Write $f_\pm=H_\pm\star\phi_{\sigma_0}$ and $g_\pm=H_\pm\star\phi_{\czero}$. By \eqref{eq:moment-deconvolution-convolution} and $0 < \tau \le 1/2$,
\begin{align*}
    f_+(x)-f_-(x) &=2w\tau e^{-v\sigma_0^2T^2/\{2(v+\sigma_0^2)\}} \phi_{\sqrt{v+\sigma_0^2}}(x) \sin\left(\frac{vTx}{v+\sigma_0^2}\right),\\
    f_-(x)&\geq w(1-\tau)\phi_{\sqrt{v+\sigma_0^2}}(x) \geq\frac w2\phi_{\sqrt{v+\sigma_0^2}}(x).
\end{align*}
Consequently,
\begin{equation}\label{eq:moment-deconvolution-distance}
\begin{aligned}
    \chi^2(f_+,f_-)
    &=\int\frac{\{f_+(x)-f_-(x)\}^2}{f_-(x)}\,\diff x\\
    &\leq8\tau^2w\exp\left\{-\frac{v\sigma_0^2T^2}{v+\sigma_0^2}\right\}
    \int\phi_{\sqrt{v+\sigma_0^2}}(x)
    \sin^2\left(\frac{vTx}{v+\sigma_0^2}\right)\,\diff x\\
    &\leq8\tau^2w\exp\left\{-\frac{v\sigma_0^2T^2}{v+\sigma_0^2}\right\}.
\end{aligned}
\end{equation}
Similarly, \eqref{eq:moment-deconvolution-convolution} gives
\begin{align*}
    g_+(\theta)-g_-(\theta)
    =2w\tau e^{-v\czero^2T^2/\{2(v+\czero^2)\}}
    \phi_{\sqrt{v+\czero^2}}(\theta)
    \sin\left(\frac{vT\theta}{v+\czero^2}\right).
\end{align*}
Hence,
\begin{equation}\label{eq:moment-deconvolution-separation}
\begin{aligned}
    \|g_+-g_-\|_{L_2}^2 &=4w^2\tau^2\exp\left\{-\frac{v\czero^2T^2}{v+\czero^2}\right\} \int\phi_{\sqrt{v+\czero^2}}^2(\theta) \sin^2\left(\frac{vT\theta}{v+\czero^2}\right)\,\diff\theta\\
    &=\frac{w^2\tau^2}{\sqrt{\pi(v+\czero^2)}} \exp\left\{-\frac{v\czero^2T^2}{v+\czero^2}\right\} \left[1-\exp\left\{-\frac{v^2T^2}{v+\czero^2}\right\}\right],
\end{aligned}
\end{equation}
where the last equality follows from Lemma~\ref{lem:Gaussian-sine-square}. Now choose $v=(\log n)^{1-r/2} \ge 1$ and
\begin{align*}
    T^2=\frac{v+\sigma_0^2}{v\sigma_0^2}\{\log(nw)+B_1\},
\end{align*}
where $B_1>0$ is a sufficiently large fixed constant. Since $\log(nw)=\log n-B_0(\log n)^{r/2}+\log w_0\to\infty$, we have $T^2>0$ for all sufficiently large $n$. By \eqref{eq:moment-deconvolution-distance}, we have
\begin{align*}
    n\chi^2(f_+,f_-)\leq8\tau^2e^{-B_1}\leq\log2.
\end{align*}
For the $n$-sample laws $P_\pm=f_\pm^{\otimes n}$, this implies
\begin{align*}
    \chi^2(P_+,P_-)\leq e^{n\chi^2(f_+,f_-)}-1\leq1, \qquad \mathrm{TV}(P_+,P_-)\leq \frac12.
\end{align*}
Thus, the standard two-point testing bound for squared $L_2$ loss gives
\begin{equation}\label{eq:moment-deconvolution-two-point}
\begin{aligned}
    \mathfrak R_{n,r} &\geq\frac18\|g_+-g_-\|_{L_2}^2 \{1-\mathrm{TV}(P_+,P_-)\}\\
    &\geq\frac1{16}\|g_+-g_-\|_{L_2}^2,
\end{aligned}
\end{equation}
see, for example, Theorem 2.2 of \citet{Tsybakov2009}.

Finally, we derive the minimax lower bound. By the choice of $T$,
\begin{align*}
    \frac{v^2T^2}{v+\czero^2}
    =\frac{v(v+\sigma_0^2)}{\sigma_0^2(v+\czero^2)}\{\log(nw)+B_1\}\to\infty.
\end{align*}
Hence, $1-\exp\{-v^2T^2/(v+\czero^2)\}\geq1/2$ for all sufficiently large $n$. Substituting the choice of $T$ into \eqref{eq:moment-deconvolution-separation} gives
\begin{align*}
    \|g_+-g_-\|_{L_2}^2
    &\geq\frac{\tau^2w^2}{2\sqrt{\pi(v+\czero^2)}}
    \exp\left[-\frac{\czero^2(v+\sigma_0^2)}{\sigma_0^2(v+\czero^2)}\{\log(nw)+B_1\}\right]\\
    &=\frac{\tau^2e^{-\gamma(v)B_1}}{2\sqrt{\pi(v+\czero^2)}}
    n^{-\gamma(v)}w^{2-\gamma(v)},
\end{align*}
where
\begin{align*}
    \gamma(v)=\frac{\czero^2(v+\sigma_0^2)}{\sigma_0^2(v+\czero^2)}
    =\alpha_0+\frac{\alpha_0}{v+\czero^2},
\end{align*}
with $\alpha_0 = \czero^2/\sigma_0^2$. Since $0<\gamma(v)<1$, \eqref{eq:moment-deconvolution-two-point} yields
\begin{align*}
    \mathfrak R_{n,r}\gtrsim_{\czero,\gammav}
    \frac{n^{-\gamma(v)}w^{2-\gamma(v)}}{\sqrt{v+\czero^2}},
\end{align*}
Recall that $w=w_0e^{-B_0v^{r/(2-r)}}$ and $v=(\log n)^{1-r/2}$. Then $v^{r/(2-r)}=(\log n)^{r/2}$ and
\begin{align*}
    \frac{n^{-\gamma(v)}w^{2-\gamma(v)}}{\sqrt{v+\czero^2}} =\frac{w_0^{2-\gamma(v)}}{\sqrt{v+\czero^2}}n^{-\alpha_0} \exp\left\{-\frac{\alpha_0\log n}{v+\czero^2} -B_0\{2-\gamma(v)\}(\log n)^{r/2}\right\}.
\end{align*}
Because $v+\czero^2\geq v$, $0<w_0<1$, and $0<\gamma(v)<1$,  we have $2-\gamma(v)\leq2$,
\begin{align*}
    \frac{\alpha_0\log n}{v+\czero^2} \leq\alpha_0(\log n)^{r/2},
    \qquad w_0^{2-\gamma(v)}\geq w_0^2.
\end{align*}
Also, $\log(v+\czero^2)=O(\log\log n)=o\{(\log n)^{r/2}\}$, so for all sufficiently large $n$,
\begin{align*}
    (v+\czero^2)^{-1/2} =\exp\left\{-\frac12\log(v+\czero^2)\right\} \geq e^{-(\log n)^{r/2}}.
\end{align*}
Consequently,
\begin{align*}
    \mathfrak R_{n,r}
    \gtrsim_{\czero,\gammav}n^{-\alpha_0}
    \exp\left\{-(\alpha_0+2B_0+1)(\log n)^{r/2}\right\}.
\end{align*}
This proves \eqref{eq:moment-deconvolution-lower} with $k_2=\alpha_0+2B_0+1$.
\end{proof}

\begin{lemma}\label{lem:Gaussian-modulation}
For any $v,s>0$ and $T\in\Real$,
\begin{align*}
    \int\phi_{\sqrt v}(u)\sin(Tu)\phi_s(x-u)\,\diff u =e^{-vs^2T^2/\{2(v+s^2)\}}\phi_{\sqrt{v+s^2}}(x) \sin\left(\frac{vTx}{v+s^2}\right),\qquad x\in\Real.
\end{align*}
\end{lemma}
\begin{proof}
Let $U\sim N(0,v)$ and $Z\sim N(0,s^2)$ be independent, and put $X=U+Z$. Then $X\sim N(0,v+s^2)$ and
\begin{align*}
    U\mid X=x\sim N\left(\frac{v}{v+s^2}x,\frac{vs^2}{v+s^2}\right).
\end{align*}
Thus,
\begin{align*}
    \int\phi_{\sqrt v}(u)\sin(Tu)\phi_s(x-u)\,\diff u &=\phi_{\sqrt{v+s^2}}(x)\E[\sin(TU)\mid X=x]\\
    &=\phi_{\sqrt{v+s^2}}(x)e^{-vs^2T^2/\{2(v+s^2)\}}\sin\left(\frac{vTx}{v+s^2}\right),
\end{align*}
where the last equality follows from the characteristic function of a Gaussian distribution.
\end{proof}

\begin{lemma}\label{lem:Gaussian-sine-square}
For any $v,s>0$ and $T\in\Real$,
\begin{align*}
    \int\phi_{\sqrt{v+s^2}}^2(x) \sin^2\left(\frac{vTx}{v+s^2}\right)\,\diff x =\frac{1}{4\sqrt{\pi(v+s^2)}} \left[1-\exp\left\{-\frac{v^2T^2}{v+s^2}\right\}\right].
\end{align*}
\end{lemma}
\begin{proof}
Put $V=v+s^2$ and $b=vT/V$. Since $\phi_{\sqrt V}^2(x)=(2\pi V)^{-1}e^{-x^2/V}$ and $\sin^2(bx)=\{1-\cos(2bx)\}/2$,
\begin{align*}
    \int\phi_{\sqrt V}^2(x)\sin^2(bx)\,\diff x &=\frac{1}{4\pi V}\left\{\int e^{-x^2/V}\,\diff x -\int e^{-x^2/V}\cos(2bx)\,\diff x\right\}\\
    &=\frac{1}{4\sqrt{\pi V}}(1-e^{-Vb^2}).
\end{align*}
Substituting $V=v+s^2$ and $b=vT/V$ proves the result.
\end{proof}

\appsubsection{Proof of Theorem~\ref{thm:estimated-c-posterior}}\label{app:proof-estimated-c-posterior}
\begin{proof}[Proof of Theorem~\ref{thm:estimated-c-posterior}]
By \eqref{eq:true posterior density}, \eqref{eq:estimated posterior density}, and the triangle inequality,
\begin{equation}\label{eq:wTV and TV}
\begin{aligned}
    \mathrm{wTV}(\widehat\pi_n^+,\pi^*)
    &=\int\mathrm{TV}\{\widehat\pi_n^+(\cdot\mid x),\pi^*(\cdot\mid x)\}f^*(x)\,\diff x\\
    &=\frac12\iint\phi(x-\theta)
    \left|\frac{\widehat g_n^+(\theta)}{\widehat f_n^+(x)}
    -\frac{g^*(\theta)}{f^*(x)}\right|f^*(x)\,\diff\theta\,\diff x\\
    &\leq\frac12\iint\phi(x-\theta)\widehat g_n^+(\theta)
    \left|\frac{f^*(x)}{\widehat f_n^+(x)}-1\right|\,\diff\theta\,\diff x\\
    &\quad+\frac12\iint\phi(x-\theta)|\widehat g_n^+(\theta)-g^*(\theta)|\,\diff\theta\,\diff x\\
    &=\frac12\int\widehat f_n^+(x)
    \left|\frac{f^*(x)}{\widehat f_n^+(x)}-1\right|\,\diff x
    +\frac12\int|\widehat g_n^+(\theta)-g^*(\theta)|\,\diff\theta\\
    &=\mathrm{TV}(\widehat f_n^+,f^*)+\mathrm{TV}(\widehat g_n^+,g^*)\\
    &\leq2\,\mathrm{TV}(\widehat g_n^+,g^*),
\end{aligned}
\end{equation}
where the first inequality follows by adding and subtracting $\phi(x-\theta)\widehat g_n^+(\theta)/f^*(x)$, the next equality uses $\widehat f_n^+=\widehat g_n^+\star\phi$ and $\int\phi(x-\theta)\,\diff x=1$, and the last inequality uses contraction of total variation under convolution with $\phi$. By Lemma~\ref{lem:l2-to-tv} below and \eqref{eq:moment-deconvolution-upper},
\begin{align*}
    \E_{G^*}[\mathrm{wTV}(\widehat\pi_n^+,\pi^*)] \leq K(\log n)^{1/(2r)}n^{-\alpha_0/2} \exp\{-k_1(\log n)^{r/2}/2\}+Kn^{-2}.
\end{align*}
Since $\log\log n=o\{(\log n)^{r/2}\}$, the logarithmic factor can be absorbed into the exponential. This proves \eqref{eq:moment-posterior-rate} with $k_3=k_1/4$.
\end{proof}

\begin{lemma}[Total variation from deconvolution error]\label{lem:l2-to-tv}
Under the conditions of Theorem~\ref{thm:estimated-c-posterior}, there exists $K>0$, depending only on $\czero$ and $\gammav$, such that, for all sufficiently large $n$,
\begin{align}\label{eq:l2-to-tv}
    \E_{G^*}[\mathrm{TV}(\widehat g_n^+,g^*)] \leq K(\log n)^{1/(2r)} \left\{\E_{G^*}\|\widehat g_n^+-g^*\|_{L_2}^2\right\}^{1/2}+Kn^{-2}.
\end{align}
\end{lemma}

\begin{proof}
Let $\mathcal E_n$ be the event in \eqref{eq:moment-deconvolution-event}. Its construction gives $\P_{G^*}(\mathcal E_n^c)\leq Kn^{-2}$. Recall that, on $\mathcal E_n$, we have $\hat c_0^+=\hat c_0$ and $\npmle^+=\npmle$. Also, $\hat c_0^2\leq\czero^2+K\rho_n\leq\czero^2+1$ for all sufficiently large $n$. Now, set $M_n=K_M(\log n)^{1/r}$, with $K_M$ sufficiently large that $\max_i|X_i|\leq M_n/2$ on $\mathcal E_n$. Since $\widehat g_n^+=\npmle\star\phi_{\hat c_0}$ and $\mathrm{supp}(\npmle)\subseteq[X_{(1)},X_{(n)}]$, a Gaussian tail bound gives, on $\mathcal E_n$,
\begin{align*}
    \int_{|\theta|>M_n}\widehat g_n^+(\theta)\,\diff\theta &\leq\P\left(|Z|>\frac{M_n}{2{\sqrt{\czero^2+1}}}\right)\leq Kn^{-3},
\end{align*}
where $Z\sim N(0,1)$. Moreover, \eqref{eq:estimated-c-true-prior-moments} and Markov's inequality with $p=4\log n$ give
\begin{align*}
    \int_{|\theta|>M_n}g^*(\theta)\,\diff\theta\leq Kn^{-3}
\end{align*}
after increasing $K_M$ if necessary. For any densities $p,q$ and $M>0$, Cauchy--Schwarz inequality gives
\begin{equation}\label{eq:tv-l2-tail-decomposition}
\begin{aligned}
    2\,\mathrm{TV}(p,q)  &=\int_{-M}^M|p-q|\,\diff\theta +\int_{|\theta|>M}|p-q|\,\diff\theta\\
    &\leq\sqrt{2M}\|p-q\|_{L_2} +\int_{|\theta|>M}(p+q)\,\diff\theta.
\end{aligned}
\end{equation}
Applying \eqref{eq:tv-l2-tail-decomposition} with $p=\widehat g_n^+$, $q=g^*$, and $M=M_n$, the preceding tail bounds give, on $\mathcal E_n$,
\begin{align*}
    2\,\mathrm{TV}(\widehat g_n^+,g^*)
       \leq\sqrt{2M_n}\|\widehat g_n^+-g^*\|_{L_2}+Kn^{-3}.
\end{align*}
On $\mathcal E_n^c$, total variation is at most one. Hence, taking expectations and applying Cauchy--Schwarz inequality once more yields
\begin{equation}\label{eq:l2-to-tv-expectation}
\begin{aligned}
    \E_{G^*}[\mathrm{TV}(\widehat g_n^+,g^*)] &=\E_{G^*}[\mathrm{TV}(\widehat g_n^+,g^*)\1v_{\mathcal E_n}] +\E_{G^*}[\mathrm{TV}(\widehat g_n^+,g^*)\1v_{\mathcal E_n^c}]\\
    &\leq K\sqrt{M_n}\left\{\E_{G^*}\|\widehat g_n^+-g^*\|_{L_2}^2\right\}^{1/2}
    +Kn^{-3}+\P_{G^*}(\mathcal E_n^c)\\
    &\leq K\sqrt{M_n}\left\{\E_{G^*}\|\widehat g_n^+-g^*\|_{L_2}^2\right\}^{1/2}+Kn^{-2},
\end{aligned}
\end{equation}
which proves \eqref{eq:l2-to-tv}.
\end{proof}

\appsubsection{Proof of Theorem~\ref{thm:moment-posterior-lower}}\label{app:proof-moment-posterior-lower}
\begin{proof}[Proof of Theorem~\ref{thm:moment-posterior-lower}]
Recall that
\begin{align*}
    \mathrm{wTV}(\widetilde\pi_n,\pi_H)=\int\mathrm{TV}\{\widetilde\pi_n(\cdot\mid x),\pi_H(\cdot\mid x)\}f_H(x)\,\diff x.
\end{align*}
We reuse the alternatives $H_\pm$ in \eqref{eq:moment-deconvolution-alternatives}, with the same choices of $v,w,T$ and fixed constants as in the proof of Theorem~\ref{thm:moment-deconvolution-lower}. Write $g_\pm=H_\pm\star\phi_{\czero}$, $f_\pm=g_\pm\star\phi$, and $\pi_\pm(\theta\mid x)=\phi(x-\theta)g_\pm(\theta)/f_\pm(x)$. The preceding proof establishes that $H_\pm\in\mathcal H_r$ and
\begin{align*}
    n\chi^2(f_+,f_-)\leq\log2,\qquad \mathrm{TV}(P_+,P_-)\leq\frac12, \qquad P_\pm=f_\pm^{\otimes n}.
\end{align*}

By \eqref{eq:moment-deconvolution-convolution}, we have
\begin{align*}
    \mathrm{TV}(g_+,g_-) =w\tau \exp\left\{-\frac{v\czero^2T^2}{2(v+\czero^2)}\right\} \E[|\sin Y_n|], \qquad Y_n\sim N\left(0,\frac{v^2T^2}{v+\czero^2}\right).
\end{align*}
Since $\Var(Y_n)\to\infty$ by the choices of $v = (\log n)^{1-r/2} \to \infty$ and $T^2 \asymp (\log n) / \sigma_0^2$,
\begin{align*}
    \E[|\sin Y_n|] \geq\E[\sin^2Y_n] =\frac{1-\E[\cos(2Y_n)]}{2}=\frac{1-\exp\left\{-2\Var(Y_n)\right\}}{2} \geq\frac14
\end{align*}
for all sufficiently large $n$. Using the same argument as in the proof of Theorem~\ref{thm:moment-deconvolution-lower}, we obtain 
\begin{equation}\label{eq:moment-posterior-prior-separation}
    \mathrm{TV}(g_+,g_-) \gtrsim_{\czero,\gammav} n^{-\alpha_0/2} \exp\{-(\alpha_0/2+B_0)(\log n)^{r/2}\}.
\end{equation}

Now, define
\begin{align*}
    D_n:=\int\min\{f_+(x),f_-(x)\} \mathrm{TV}\{\pi_+(\cdot\mid x),\pi_-(\cdot\mid x)\}\,\diff x.
\end{align*}
The identity $f_\pm(x)\pi_\pm(\theta\mid x)=\phi(x-\theta)g_\pm(\theta)$ and the triangle inequality imply
\begin{align*}
    \mathrm{TV}(g_+,g_-) &=\frac12\iint |f_+(x)\pi_+(\theta\mid x) -f_-(x)\pi_-(\theta\mid x)|\,\diff\theta\,\diff x\\
    &\leq D_n+\frac12\int\left[\{f_+(x)-\min(f_+(x),f_-(x))\}\int\pi_+(\theta\mid x)\,\diff\theta \right.\\
    &\qquad\left.+ \{f_-(x)-\min(f_+(x),f_-(x))\} \int\pi_-(\theta\mid x)\,\diff\theta \right]\diff x\\
    &=D_n+\frac12\int|f_+(x)-f_-(x)|\,\diff x\\
    &=D_n+\mathrm{TV}(f_+,f_-).
\end{align*}
Moreover, we have
\begin{align*}
    \mathrm{TV}(f_+,f_-)\leq\frac12\sqrt{\chi^2(f_+,f_-)}\leq\frac12\sqrt{\frac{\log2}{n}}.
\end{align*}
Because $\alpha_0<1$ and $r<2$, this term is negligible compared with $\mathrm{TV}(g_+,g_-)$. Thus, $D_n\geq\mathrm{TV}(g_+,g_-)/2$ for all sufficiently large $n$.

For any posterior density estimator $\widetilde\pi_n$, let
\begin{align*}
    L_\pm:=\int\mathrm{TV}\{\widetilde\pi_n(\cdot\mid x),\pi_\pm(\cdot\mid x)\}f_\pm(x)\,\diff x.
\end{align*}
Then observe that
\begin{align*}
    L_+ + L_- &\geq\int\min\{f_+(x),f_-(x)\}\left[ \mathrm{TV}\{\widetilde\pi_n(\cdot\mid x),\pi_+(\cdot\mid x)\} +\mathrm{TV}\{\widetilde\pi_n(\cdot\mid x),\pi_-(\cdot\mid x)\} \right]\,\diff x\\
    &\geq\int\min\{f_+(x),f_-(x)\}\mathrm{TV}\{\pi_+(\cdot\mid x),\pi_-(\cdot\mid x)\}\,\diff x =D_n.
\end{align*}
Let $p_\pm(\mathbf x):=\prod_{i=1}^n f_\pm(x_i)$ denote the densities of $P_\pm$, and write $L_\pm(\mathbf x)$ for the losses evaluated at the sample $\mathbf x$. Since $L_\pm\geq0$ and $L_++L_-\geq D_n$ for every sample, we have
\begin{align*}
    \E_{H_+}[L_+]+\E_{H_-}[L_-] &=\int\{L_+(\mathbf x)p_+(\mathbf x) +L_-(\mathbf x)p_-(\mathbf x)\}\,\diff\mathbf x\\
    &\geq\int\{L_+(\mathbf x)+L_-(\mathbf x)\} \min\{p_+(\mathbf x),p_-(\mathbf x)\}\,\diff\mathbf x\\
    &\geq D_n\int\min\{p_+(\mathbf x),p_-(\mathbf x)\} \,\diff\mathbf x\\
    &=D_n\{1-\mathrm{TV}(P_+,P_-)\} \geq\frac{D_n}{2}.
\end{align*}
Since $H_\pm\in\mathcal H_r$ and this holds for every estimator $\widetilde \pi_n$,
\begin{align*}
    \mathfrak R^\pi_{n,r} &\geq\frac12\inf_{\widetilde\pi_n} \{\E_{H_+}[L_+]+\E_{H_-}[L_-]\}\\
    &\geq\frac{D_n}{4}\geq\frac18\mathrm{TV}(g_+,g_-)\\
    &\gtrsim_{\czero,\gammav}n^{-\alpha_0/2}\exp\{-k_4(\log n)^{r/2}\},
\end{align*}
where the last inequality follows from \eqref{eq:moment-posterior-prior-separation} with $k_4=\alpha_0/2+B_0$. This proves \eqref{eq:moment-posterior-lower}. 
\end{proof}

\appsubsection{Proof of Theorem~\ref{thm:misspecification-true}}\label{app:proof-misspecification-true}
To prove the theorem, we control (i) the estimation error of the smooth NPMLE relative to the pseudo-true prior density and (ii) the approximation error between the pseudo-true and true prior densities. To make this precise, we first define the KL projection of the true marginal density onto the fitted model class. For $\lambda>0$, let
\begin{align*}
    \mathcal F_\lambda =\{p_{H,\lambda}\star N(0,1):H\in\Pc(\Real)\} =\{p_{H,\sqrt{1+\lambda^2}}:H\in\Pc(\Real)\},
\end{align*}
where $p_{H,\lambda}$ is defined in \eqref{eq:true data density} with $\sigma=\lambda$. That is, $\Fc_\lambda$ is the class of marginal densities under \eqref{eq:original model} whose prior distributions $G$ admit an $N(0,\lambda^2)$ convolution factor. We define the {\it pseudo-true density} $\widetilde f_\lambda \in \Fc_{\lambda}$ as the Kullback-Leibler (KL) projection of the true marginal density $f^*$ onto $\mathcal{F}_\lambda$: 
\begin{align}\label{eq:pseudo-true density condition}
    \widetilde f_\lambda = \argmin_{f \in \Fc_\lambda}\, \mathrm{KL} (f^*  \parallel f),
\end{align}
where $\mathrm{KL}(p  \parallel q)$ denotes the Kullback-Leibler (KL) divergence between two probability density functions $p$ and $q$. In general, $\widetilde f_\lambda$ is the element of $\mathcal{F}_\lambda$ closest to $f^*$ in KL divergence. The following lemma shows the existence and uniqueness of the KL projection.

\begin{lemma}[Existence and uniqueness of the KL projection]\label{lem:misspecification-projection-existence}
Suppose that $\E_{G^*}[X_1^2]<\infty$.
Then, for every $\lambda>0$, the KL projection $\widetilde f_\lambda$ in \eqref{eq:pseudo-true density condition} exists and is unique. Moreover, there is a unique mixing distribution $\widetilde H_\lambda\in\Pc(\Real)$ such that
\begin{align*}
    \widetilde f_\lambda=\widetilde H_\lambda\star\phi_{\sqrt{1+\lambda^2}}.
\end{align*}
\end{lemma}
\begin{proof}
Fix $\lambda>0$, write $\sigma^2=1+\lambda^2$, and define
\begin{align*}
    f_H:=H\star\phi_\sigma,\qquad J(H):=\int f^*(x)\log f_H(x)\,\diff x, \qquad H\in\Pc(\Real).
\end{align*}
Note that minimizing the KL divergence in \eqref{eq:pseudo-true density condition} is equivalent to maximizing $J(H)$. We first show that the optimal value is finite. The candidate $H=\delta_0$ gives
\begin{align*}
    J(\delta_0) =-\frac12\log(2\pi\sigma^2) -\frac{\E_{G^*}[X_1^2]}{2\sigma^2}>-\infty,
\end{align*}
where the finite second moment follows from the assumption. Also, for every $H\in\Pc(\Real)$,
\begin{align*}
    f_H(x)\leq\frac1{\sqrt{2\pi\sigma^2}},\qquad J(H)\leq-\frac12\log(2\pi\sigma^2).
\end{align*}
Thus $\tilde J:=\sup_{Q\in\Pc(\Real)}J(Q)$ is finite.

Next, choose probability measures $H_m$ such that $J(H_m)\to \tilde J$. By Helly's selection theorem, passing to a subsequence if necessary, we may assume that $H_m$ converges vaguely to a subprobability measure $\widetilde H$, with $\widetilde H(\Real)\leq1$. We also write $f_{\widetilde H}=\widetilde H\star\phi_\sigma$ and $J(\widetilde H)=\int f^*\log f_{\widetilde H}$ for this limit. Since Gaussian kernels are continuous and vanish at infinity,
\begin{align*}
    f_{H_m}(x)\longrightarrow f_{\widetilde H}(x),\qquad x\in\Real.
\end{align*}
Also, since $\log f_{H_m}(x)\leq-\frac12\log(2\pi\sigma^2)$ for all $m$ and $x$, reverse Fatou's lemma gives
\begin{align*}
    J(\widetilde H)\geq\limsup_{m\to\infty}J(H_m)=\tilde J.
\end{align*}
We now show that $\widetilde H$ has total mass one. If $\widetilde H(\Real)=0$, then $f_{\widetilde H}=0$ and $J(\widetilde H)=-\infty$, contradicting the preceding bound. Suppose instead that $m:=\widetilde H(\Real)\in(0,1)$. The normalized measure $\widetilde H/m$ belongs to $\Pc(\Real)$ and satisfies $f_{\widetilde H/m}=f_{\widetilde H}/m$. Hence,
\begin{align*}
    J(\widetilde H/m)=J(\widetilde H)-\log m>J(\widetilde H)\geq \tilde J,
\end{align*}
contradicting the definition of $\tilde J$. Therefore $\widetilde H(\Real)=1$ and $J(\widetilde H)=\tilde J$. Setting $\widetilde H_\lambda=\widetilde H$ and $\widetilde f_\lambda=f_{\widetilde H}$ proves existence of the KL projection.

Finally, we prove uniqueness. Suppose that $H_1$ and $H_2$ both maximize $J$. The mixture class is convex, and strict concavity of the logarithm gives
\begin{align*}
    J\bigl((H_1+H_2)/2\bigr) \geq\frac12J(H_1)+\frac12J(H_2)=\tilde J,
\end{align*}
with equality only if $f_{H_1}=f_{H_2}$ for $f^*(x)\,\diff x$-almost every $x$. The left side cannot exceed $\tilde J$, so equality holds. Since $f^*>0$ and both mixture densities are continuous, $f_{H_1}(x)=f_{H_2}(x)$ for every $x\in\Real$. Thus $\widetilde f_\lambda$ is unique. Finally,
\begin{align*}
    \varphi_{f_{H_j}}(t)=\varphi_{H_j}(t)e^{-\sigma^2t^2/2}, \qquad t\in\Real,\quad j=1,2.
\end{align*}
The Gaussian factor is never zero, so $\varphi_{H_1}=\varphi_{H_2}$ and hence $H_1=H_2$. This proves uniqueness of $\widetilde H_\lambda$.
\end{proof}

Accordingly, we define the \emph{pseudo-true prior density} as
\begin{align*}
    \widetilde g_\lambda:=\widetilde H_\lambda\star N(0,\lambda^2).
\end{align*}
The following two propositions control (i) the estimation error of the smooth NPMLE relative to the pseudo-true prior density and (ii) the approximation error between the pseudo-true and true prior densities, respectively. We prove Propositions~\ref{prop:misspecification-pseudoprior} and~\ref{prop:misspecification-true-approximation} in the next two subsubsections, respectively.

\begin{prop}[Estimation of the pseudo-true prior density]\label{prop:misspecification-pseudoprior}
Under the conditions of Theorem~\ref{thm:misspecification-true}, there exist constants $K>0$ and $n_0$, depending only on $\gammav$ and $L$, such that, for $n\geq n_0$,
\begin{equation}\label{eq:misspecification-true-pseudoprior}
    \E_{G^*}\bigl[\|\widehat g_n^+-\widetilde g_{\lambda_n}\|_{L_2}^2\bigr] \leq \frac{K}{\log n},
\end{equation}
where $\lambda_n:=\sqrt{\rho_n}$ and $\rho_n$ is defined in \eqref{eq:rho_n}.
\end{prop}

\begin{prop}[Approximation error of the pseudo-true prior density]\label{prop:misspecification-true-approximation}
Let $g$ be a strictly positive density, and assume that $\log g$ is locally absolutely continuous. Suppose that $g$ satisfies \eqref{eq:log-density-slope-variation} with $g$ in place of $g^*$ and $(\int |x|^p g(x)\,\diff x)^{1/p}\leq K_0p$ for every $p\geq1$. Let $\widetilde g_\lambda$ be the prior density associated with the KL projection of $g\star\phi_1$ onto $\Fc_\lambda$. There exist constants $K,\lambda_0>0$, depending only on $K_0$ and $L$, such that, for $0<\lambda<\lambda_0$,
\begin{equation}\label{eq:misspecification-true-projection-bias}
    \|\widetilde g_\lambda-g\|_{L_2}^2\leq K\lambda^2.
\end{equation}
\end{prop}

Using Propositions~\ref{prop:misspecification-pseudoprior} and~\ref{prop:misspecification-true-approximation}, we prove Theorem~\ref{thm:misspecification-true}.
\begin{proof}[Proof of Theorem~\ref{thm:misspecification-true}]
Put $\lambda_n:=\sqrt{\rho_n}$ where $\rho_n$ is defined in \eqref{eq:rho_n}. By Propositions~\ref{prop:misspecification-pseudoprior} and~\ref{prop:misspecification-true-approximation}, the squared triangle inequality gives
\begin{align*}
    \E_{G^*}[\|\widehat g_n^+-g^*\|_{L_2}^2] &\leq 2\E_{G^*}[\|\widehat g_n^+-\widetilde g_{\lambda_n}\|_{L_2}^2] +2\|\widetilde g_{\lambda_n}-g^*\|_{L_2}^2\\
    &\leq K\{(\log n)^{-1}+\rho_n\} \leq K\rho_n,
\end{align*}
which proves \eqref{eq:misspecification-true-prior-rate}. To prove \eqref{eq:misspecification-true-posterior-rate}, it suffices by \eqref{eq:wTV and TV} to show that $\E_{G^*}[\mathrm{TV}(\widehat g_n^+,g^*)]\lesssim_{\gammav,L}(\log\log n)/\sqrt{\log n}$. Take $M_n=B\log\log n$ with $B>0$ sufficiently large. By Cauchy--Schwarz inequality and \eqref{eq:misspecification-true-prior-rate},
\begin{align*}
    \E_{G^*}\!\left[\int_{|\theta|\leq M_n}|\widehat g_n^+(\theta)-g^*(\theta)|\,\diff\theta\right]
       &\leq\sqrt{2M_n}\left\{\E_{G^*}\|\widehat g_n^+-g^*\|_{L_2}^2\right\}^{1/2}\\
       &\lesssim_{\gammav,L}\sqrt{M_n\rho_n}
        \lesssim_{\gammav,L}\frac{\log\log n}{\sqrt{\log n}}.
\end{align*}
It remains to control the contribution from $|\theta|>M_n$. Let $\mathcal T_n$ be the event in \eqref{eq:estimated-c-denoising-event} with $r=1$. By the general maximum likelihood theorem used in the proof of Lemma~\ref{lem:estimated-c-likelihood-gain}, we obtain
\begin{align}\label{eq:NPMLE moment inequality}
    \int|u|^p\,\diff\npmle^+(u)
       \leq\frac1n\sum_{i=1}^n|X_i|^p,\qquad p\geq1.
\end{align}
Indeed, for any observations $X_1, \ldots, X_n$, put $(\hat \sigma_0^+)^2= 1+(\hat c_0^+)^2$ and define $a_i(u) := \phi_{\hat\sigma_0^+}(X_i-u)/ \{n\widehat f_n^+(X_i) \}$. Then the optimality of the NPMLE and the first-order condition in the mixing locations give
\begin{align*}
    \sum_{i=1}^n a_i(u)=1,\qquad \sum_{i=1}^n a_i(u)X_i=u, \qquad \npmle^+\text{-almost everywhere}.
\end{align*}
Then Jensen's inequality gives $|u|^p\leq\sum_i a_i(u)|X_i|^p$, and integrating against $\npmle^+$ and using $\int a_i(u)\,\diff\npmle^+(u)=1/n$ yields \eqref{eq:NPMLE moment inequality}.
Since $\czero=0$, we have $G^*=\trueprior$; the moment condition \eqref{eq:simultaneous-moment-control} and the Gaussian moment bound give $(\E_{G^*}|X_1|^p)^{1/p}\lesssim_{\gammav}p$ for $p\geq1$. By \eqref{eq:modified-c-good-event} and \eqref{eq:deconvolution-floor}, $\hat c_0^+\leq1$ on $\mathcal T_n$ for all sufficiently large $n$. Also, conditional Jensen's inequality gives $\int|\theta|^p g^*(\theta)\,\diff\theta\leq\E_{G^*}|X_1|^p$. Using these bounds and \eqref{eq:NPMLE moment inequality}, Minkowski's inequality applied to $\theta=\xi+\hat c_0^+Z$, with $Z\sim N(0,1)$ independent of $\xi$, gives
\begin{align*}
    &\E_{G^*}\left[\1v_{\mathcal T_n}\int|\theta|^p\widehat g_n^+(\theta)\,\diff\theta\right]
       +\int|\theta|^pg^*(\theta)\,\diff\theta\\
    &\quad\leq2^p\{\E_{G^*}|X_1|^p+\E|Z|^p\}
       \leq(Kp)^p,\qquad p\geq1,
\end{align*}
for some fixed $K\geq1$ depending only on $K_0$. For $M\geq eK$, take $p=M/(eK)\geq1$. Then, by Markov's inequality,
\begin{equation}\label{eq:misspecification-prior-tails}
\begin{aligned}
    &\E_{G^*}\left[\1v_{\mathcal T_n}\int_{|\theta|>M}\widehat g_n^+(\theta)\,\diff\theta\right]
       +\int_{|\theta|>M}g^*(\theta)\,\diff\theta \leq\frac{(Kp)^p}{M^p}
       =e^{-p}=e^{-M/(eK)}.
\end{aligned}
\end{equation}
For $1\leq M<eK$, the left-hand side is at most two. Thus, the bound $\lesssim_{\gammav}e^{-kM}$ holds for all $M\geq1$, with $k=1/(eK)$. Taking $M=M_n=B\log\log n$ in \eqref{eq:misspecification-prior-tails}, we obtain
\begin{align*}
    &\E_{G^*}\left[\1v_{\mathcal T_n}\int_{|\theta|>M_n}\widehat g_n^+(\theta)\,\diff\theta\right]
       +\int_{|\theta|>M_n}g^*(\theta)\,\diff\theta\\
    &\quad\lesssim_{\gammav}e^{-kM_n}
       =(\log n)^{-kB}\leq(\log n)^{-1},
\end{align*}
for $B\geq1/k$. Combining this with the preceding bound on $[-M_n,M_n]$ and \eqref{eq:estimated-c-denoising-good-event} with $r=1$ gives
\begin{align*}
    \E_{G^*}[\mathrm{TV}(\widehat g_n^+,g^*)] \leq\E_{G^*}[\mathrm{TV}(\widehat g_n^+,g^*)\1v_{\mathcal T_n}]
          +\P_{G^*}(\mathcal T_n^c) \lesssim_{\gammav,L}\frac{\log\log n}{\sqrt{\log n}}.
\end{align*}
Applying \eqref{eq:wTV and TV} proves \eqref{eq:misspecification-true-posterior-rate}.
\end{proof}

\begin{remark}[Justification of deterministic smoothing]\label{rem:misspecification-fixed-floor-justification}
To justify Remark~\ref{rem:misspecification-fixed-floor}, suppose that we replace $\hat c_0^+$ and $\npmle^+$ by $\lambda_n:=\sqrt{\rho_n}$ and $\npmle^{\,\lambda_n}$ (see \eqref{eq:NPMLE}), respectively, in the proof of Theorem~\ref{thm:misspecification-true} and the proofs of Propositions~\ref{prop:misspecification-pseudoprior} and~\ref{prop:misspecification-marginal} below. 
Then the likelihood comparison \eqref{eq:misspecification-floor-likelihood} holds for every sample with $d_n=0$. For the remaining arguments, replace the event in \eqref{eq:misspecification-estimation-event} by $\mathcal T_n:=\{\max_i|X_i|\leq K_R\log n\}$, with $K_R$ sufficiently large that the moment condition alone gives $\P_{G^*}(\mathcal T_n^c)\leq n^{-3}$ for all sufficiently large $n$. The remaining proofs give the same logarithmic rates under the moment and density regularity conditions in Remark~\ref{rem:misspecification-fixed-floor}, without requiring the Fourier lower bound.
\end{remark}

Before proving Propositions~\ref{prop:misspecification-pseudoprior} and~\ref{prop:misspecification-true-approximation}, we establish preliminary lemmas that will be used in both proofs.

\begin{lemma}[Gaussian convolution bounds]\label{lem:gaussian-convolution-bounds}
For any probability measure $H$ on $\Real$ and $U>0$,
\begin{equation}\label{eq:misspecification-true-gaussian-l2}
    \|H\star\phi_{1/U}\|_{L_2}^2\lesssim I_U(H):=\int_{-U}^U|\varphi_H(t)|\,\diff t.
\end{equation}
\end{lemma}

\begin{proof}
To prove \eqref{eq:misspecification-true-gaussian-l2}, take $\tau>0$. By Tonelli's theorem and $\phi_\tau(x)\leq\sqrt2\phi_{\sqrt2\tau}(x)$ for all $x \in \Real$,
\begin{align*}
    \|H\star\phi_{\tau/\sqrt2}\|_{L_2}^2 &=\iint\phi_\tau(u-v)\,\diff H(u)\,\diff H(v)\\
    &\leq\sqrt2\iint\phi_{\sqrt2\tau}(u-v)\,\diff H(u)\,\diff H(v) =\sqrt2\|H\star\phi_\tau\|_{L_2}^2.
\end{align*}
By Plancherel's theorem, splitting the integral at $1/\tau$ gives
\begin{align*}
    \|H\star\phi_\tau\|_{L_2}^2 &=\frac1{2\pi}\int|\varphi_H(t)|^2e^{-\tau^2t^2}\,\diff t\\
    &\leq\frac1{2\pi}\int_{|t|\leq1/\tau}|\varphi_H(t)|^2\,\diff t
       +\frac1{2\pi}\int_{|t|>1/\tau}|\varphi_H(t)|^2e^{-\tau^2t^2/2}e^{-\tau^2t^2/2}\,\diff t\\
    &\leq\frac1{2\pi}\int_{|t|\leq1/\tau}|\varphi_H(t)|^2\,\diff t +e^{-1/2}\|H\star\phi_{\tau/\sqrt2}\|_{L_2}^2\\
    &\leq\frac1{2\pi}\int_{|t|\leq1/\tau}|\varphi_H(t)|^2\,\diff t +\sqrt{2/e}\|H\star\phi_\tau\|_{L_2}^2.
\end{align*}
Since $\sqrt{2/e}<1$ and $|\varphi_H|\leq1$, taking $\tau=1/U$ gives
\begin{align*}
    \|H\star\phi_{1/U}\|_{L_2}^2
       &\leq\frac1{2\pi(1-\sqrt{2/e})}
          \int_{-U}^U|\varphi_H(t)|^2\,\diff t
       \lesssim I_U(H).
\end{align*}
This proves \eqref{eq:misspecification-true-gaussian-l2}.
\end{proof}

\begin{lemma}[$L_2$ bounds for smoothed mixing distributions]\label{lem:misspecification-smoothed-mixing}
Let $g$ be a probability density satisfying $|\varphi_g(t)|\leq K_g(1+t^2)^{-1}$ for all $t\in\Real$ and some $K_g>0$. Suppose that $H\in\Pc(\Real)$, $0<\lambda\leq1$, and $\delta\geq0$ satisfy
\begin{align*}
    \mathrm{TV}\bigl(H\star\phi_{\sqrt{1+\lambda^2}},g\star\phi_1\bigr)\leq\delta.
\end{align*}
Then, for every $U\geq1/\lambda$,
\begin{equation}\label{eq:misspecification-smoothed-mixing-bound}
    \|H\star\phi_{1/U}\|_{L_2}^2
       \lesssim_{K_g}1+\lambda e^{\lambda^2U^2/2}
          +\delta Ue^{(1+\lambda^2)U^2/2}.
\end{equation}
\end{lemma}

\begin{proof}
Write $f_H:=H\star\phi_{\sqrt{1+\lambda^2}}$ and $f_g:=g\star\phi_1$. For every $t\in\Real$,
\begin{align*}
    |\varphi_{f_H}(t)-\varphi_{f_g}(t)|
       &\leq\int|f_H(x)-f_g(x)|\,\diff x
        =2\mathrm{TV}(f_H,f_g)\leq2\delta.
\end{align*}
Since $\varphi_{f_H}(t)=\varphi_H(t)e^{-(1+\lambda^2)t^2/2}$ and $\varphi_{f_g}(t)=\varphi_g(t)e^{-t^2/2}$, the Fourier upper bound on $g$ gives
\begin{align*}
    |\varphi_H(t)|
       &=e^{(1+\lambda^2)t^2/2}|\varphi_{f_H}(t)|\\
       &\leq e^{(1+\lambda^2)t^2/2}
          \{|\varphi_{f_g}(t)|+2\delta\}\\
       &=e^{\lambda^2t^2/2}|\varphi_g(t)|
          +2\delta e^{(1+\lambda^2)t^2/2}\\
       &\lesssim_{K_g}\frac{e^{\lambda^2t^2/2}}{1+t^2}
          +\delta e^{(1+\lambda^2)t^2/2}.
\end{align*}
Recall $I_U(H)$ from \eqref{eq:misspecification-true-gaussian-l2} in Lemma~\ref{lem:gaussian-convolution-bounds}. Integrating this bound over $|t|\leq U$ gives
\begin{align*}
    I_U(H)
       &=\int_{|t|\leq U}|\varphi_H(t)|\,\diff t\\
       &\lesssim_{K_g}\int_{|t|\leq U}\frac{e^{\lambda^2t^2/2}}{1+t^2}\,\diff t
          +\delta\int_{|t|\leq U}e^{(1+\lambda^2)t^2/2}\,\diff t.
\end{align*}
Since $U\geq1/\lambda$, splitting the first integral at $|t|=1/\lambda$ yields
\begin{align*}
    \int_{|t|\leq U}\frac{e^{\lambda^2t^2/2}}{1+t^2}\,\diff t
       &\leq e^{1/2}\int_{|t|\leq1/\lambda}\frac{\diff t}{1+t^2}
          +e^{\lambda^2U^2/2}\int_{|t|>1/\lambda}\frac{\diff t}{1+t^2}\\
       &\leq\pi e^{1/2}+2\lambda e^{\lambda^2U^2/2}.
\end{align*}
The second integral is at most $2Ue^{(1+\lambda^2)U^2/2}$. Thus,
\begin{align*}
    I_U(H)\lesssim_{K_g}1+\lambda e^{\lambda^2U^2/2}
       +\delta Ue^{(1+\lambda^2)U^2/2}.
\end{align*}
Applying \eqref{eq:misspecification-true-gaussian-l2} in Lemma~\ref{lem:gaussian-convolution-bounds} proves \eqref{eq:misspecification-smoothed-mixing-bound}.
\end{proof}

\begin{lemma}[Density-ratio and Fourier bounds]\label{lem:misspecification-density-regularity}
Let $g$ be a strictly positive density, and assume that $\log g$ is locally absolutely continuous. Suppose that $g$ satisfies \eqref{eq:log-density-slope-variation} with $g$ in place of $g^*$. Then there exists a constant $K>0$, depending only on $L$, such that, for all $x,y,t\in\Real$,
\begin{equation}\label{eq:misspecification-true-envelopes}
    e^{-L|x-y|}\leq g(x)/g(y)\leq e^{L|x-y|},\qquad
    |\varphi_g(t)|\leq K(1+t^2)^{-1}.
\end{equation}
\end{lemma}
\begin{proof}
We first establish the density-ratio bounds in \eqref{eq:misspecification-true-envelopes}. Write $s_g=(\log g)'$ almost everywhere. Since $\log g$ is locally absolutely continuous, if $s_g$ were nonnegative or nonpositive almost everywhere, then $g$ would be monotone on $\Real$, contradicting its strict positivity and integrability. Thus, we can choose $a,b\in\Real$ with $s_g(a)<0<s_g(b)$. Taking $m=1$ in \eqref{eq:log-density-slope-variation} gives $|s_g(x)-s_g(y)|\leq L$. Consequently, for almost every $x$,
\begin{align*}
    -L\leq s_g(b)-L\leq s_g(x)\leq s_g(a)+L\leq L.
\end{align*}
Moreover, for all $x,y\in\Real$,
\begin{align*}
    |\log g(x)-\log g(y)|
       =\left|\int_y^x s_g(u)\,\diff u\right|
       \leq L|x-y|.
\end{align*}
Exponentiating gives the density-ratio bounds in \eqref{eq:misspecification-true-envelopes}. Note that integrating $g(y)\geq g(x)e^{-L|y-x|}$ over $y$ also gives $\|g\|_\infty\leq L/2$.

Next, to obtain the Fourier upper bound in \eqref{eq:misspecification-true-envelopes}, we use finite differences of $g'=s_g g$. Since $|s_g|\leq L$ and $\|g\|_\infty\leq L/2$, for every $h>0$,
\begin{align*}
    \int|g'(x+h)-g'(x)|\,\diff x
       &\leq L\int|g(x+h)-g(x)|\,\diff x
          +\frac L2\int|s_g(x+h)-s_g(x)|\,\diff x\\
       &\leq Lh\int|g'(x)|\,\diff x+\frac{L^2h}{2}
        \leq\frac32L^2h.
\end{align*}
Here, local absolute continuity of $g$ gives $\int|g(x+h)-g(x)|\,\diff x\leq h\int|g'|$, \eqref{eq:log-density-slope-variation} gives $\int|s_g(x+h)-s_g(x)|\,\diff x\leq Lh$, and $\int|g'|=\int|s_g|g\leq L$. Since $g$ and $g'$ are integrable, $g(x)\to0$ at both infinities. For $t\neq0$, take $h=\pi/|t|$. Integration by parts and $e^{-ith}=-1$ give
\begin{align*}
    2|t|\,|\varphi_g(t)|
       &=\left|\int e^{itx}\{g'(x+h)-g'(x)\}\,\diff x\right|\\
       &\leq\int|g'(x+h)-g'(x)|\,\diff x
        \leq\frac32L^2h.
\end{align*}
Thus, $|\varphi_g(t)|\leq3\pi L^2/(4t^2)$ for $t\neq0$. Since $|\varphi_g|\leq1$, we have $(1+t^2)|\varphi_g(t)|\leq1+3\pi L^2/4$, which proves the Fourier upper bound in \eqref{eq:misspecification-true-envelopes}.
\end{proof}

\begin{lemma}[KL approximation of the true marginal density]\label{lem:misspecification-kl-approximation}
Under the assumptions of Lemma~\ref{lem:misspecification-density-regularity}, suppose also that $(\int |x|^p g(x)\,\diff x)^{1/p}\leq K_0p$ for every $p\geq1$. Put $f:=g\star\phi_1$, and let $\widetilde f_\lambda$ be its KL projection onto $\Fc_\lambda$. Then there exist constants $K>0$ and $0<\lambda_0<1$, depending only on $K_0$ and $L$, such that, for $0<\lambda<\lambda_0$, with $V_\lambda^2:=2\log(1/\lambda)/\lambda^2$,
\begin{equation}\label{eq:misspecification-true-projection-kl-bound}
    \mathrm{KL}(f\parallel\widetilde f_\lambda)
       \leq K\lambda^4\exp\{-V_\lambda^2+KV_\lambda^{4/3}\}.
\end{equation}
\end{lemma}
\begin{proof}
We continue from the proof of Lemma~\ref{lem:misspecification-density-regularity}, retaining its notation and bounds. We will construct a probability density $k_{\lambda,g}$ such that
\begin{align*}
    \bar f_{\lambda,g}:=k_{\lambda,g}\star\phi_{\sqrt{1+\lambda^2}}
\end{align*}
closely approximates $f$. Since the KL projection $\widetilde f_\lambda$ has no larger KL error than $\bar f_{\lambda,g}$, it suffices to bound $\mathrm{KL}(f\parallel\bar f_{\lambda,g})$. Ideally, to obtain $k_{\lambda,g}\star\phi_{\sqrt{1+\lambda^2}}=f$, we would need $\varphi_{k_{\lambda,g}}(t) =e^{\lambda^2t^2/2}\varphi_g(t)$. However, $e^{\lambda^2t^2/2}\varphi_g(t)$ is a characteristic function only if $g$ admits a Gaussian convolution factor of scale $\lambda$, which is impossible in the misspecified setting $\czero=0$. Hence, we approximate it using a smooth truncation. For sufficiently small $\lambda>0$, put $V=V_\lambda$ and define $\chi_V:\mathbb C\to\mathbb C$ by
\begin{equation}\label{eq:misspecification-cutoff-definition}
    \chi_V(z)=d_V^{-1}\int_{-V}^{V}e^{-(z-v)^4}\,\diff v,\qquad
    d_V=\int_{-V}^{V}e^{-v^4}\,\diff v.
\end{equation}
The function $\chi_V$ is real and even on $\Real$, with $\chi_V(0)=1$. Then, instead of multiplying $\varphi_g(t)$ by $e^{\lambda^2t^2/2}$, we use the truncated factor $1+(e^{\lambda^2t^2/2}-1)\chi_V(t)$ and seek $k_{\lambda, g}$ satisfying
\begin{equation}\label{eq:misspecification-chf-k}
    \varphi_{k_{\lambda,g}}(t)
       =\varphi_g(t)\left\{1+(e^{\lambda^2t^2/2}-1)\chi_V(t)\right\}.
\end{equation}
To construct $k_{\lambda,g}$ satisfying \eqref{eq:misspecification-chf-k}, let
\begin{equation}\label{eq:misspecification-ell-definition}
    \ell_{u,V}(x)=\frac1{2\pi}
       \int e^{-itx}e^{ut^2}\chi_V(t)\,\diff t, \qquad 0\leq u\leq\lambda^2/2, 
\end{equation}
so that $\varphi_{\ell_{u,V}}(t) = e^{ut^2}\chi_V(t)$. Then integration by parts gives $\varphi_{g'}(t)=-it\varphi_g(t)$ and $\varphi_{\ell_{u,V}'}(t)=-it e^{ut^2}\chi_V(t)$. Hence, $-g'\star \ell_{u,V}'$ has Fourier transform $t^2\varphi_g(t)e^{ut^2}\chi_V(t)$. Since $e^{\lambda^2t^2/2}-1=t^2\int_0^{\lambda^2/2}e^{ut^2}\,\diff u$, we define
\begin{align*}
    k_{\lambda,g}(x)
       :=g(x)-\int_0^{\lambda^2/2}
                    (g'\star \ell_{u,V}')(x)\,\diff u,
\end{align*}
which satisfies \eqref{eq:misspecification-chf-k}.

\noindent\emph{(i) Verification that $k_{\lambda,g}$ is a probability density.}
We first show that $k_{\lambda,g}\geq g/2$ for all sufficiently small $\lambda$. Write $s_g=(\log g)'$ almost everywhere, so that $g'=s_g g$. For $0<h\leq1$, the density-ratio bounds in \eqref{eq:misspecification-true-envelopes} give $|g(y+h)-g(y)|\leq(e^{Lh}-1)g(y)$. Combining this with $|s_g|\leq L$ and $\int|s_g(y+h)-s_g(y)|\,\diff y\leq Lh$ from the proof of Lemma~\ref{lem:misspecification-density-regularity} gives
\begin{align*}
    \int e^{-(L+1)|x-y|}|g'(y+h)-g'(y)|\,\diff y
       &\leq\{2L(e^{Lh}-1)+Lh\}g(x)
        \lesssim_L hg(x).
\end{align*}
Consequently, for every continuously differentiable function $\ell$ with bounded derivative and $\sup_y e^{(L+1)|y|}|\ell(y)|<\infty$, changing variables and letting $h\downarrow0$ yield
\begin{equation}\label{eq:misspecification-true-curvature}
\begin{aligned}
    |(g'\star\ell')(x)|
       &=\left|\lim_{h\downarrow0}\int
          \frac{g'(y+h)-g'(y)}h\ell(x-y)\,\diff y\right|\\
       &\lesssim_L g(x)\sup_y e^{(L+1)|y|}|\ell(y)|.
\end{aligned}
\end{equation}
Here, the limit follows by dominated convergence after writing the integral as $\int g'(y)\{\ell(x+h-y)-\ell(x-y)\}/h\,\diff y$, since $g'\in L_1$ and $\ell'$ is bounded. By Lemma~\ref{lem:misspecification-cutoff-kernels} below, $\ell_{u,V}$ is continuously differentiable with bounded derivative. Applying \eqref{eq:misspecification-true-curvature} to $\ell_{u,V}$ and using \eqref{eq:misspecification-L-kernel-bound} gives
\begin{align*}
    |(g'\star \ell_{u,V}')(x)|\lesssim_L g(x)V e^{uV^2}.
\end{align*}
Integrating this bound over $0\leq u\leq\lambda^2/2$ and using the definition of $k_{\lambda,g}$, we obtain
\begin{equation}\label{eq:misspecification-candidate-relative-bound}
\begin{aligned}
    \sup_{x\in\Real}\frac{|k_{\lambda,g}(x)-g(x)|}{g(x)}
       &\lesssim_L V\int_0^{\lambda^2/2}e^{uV^2}\,\diff u\\
       &=\frac{e^{\lambda^2V^2/2}-1}{V}
        =\frac{1-\lambda}{\sqrt{2\log(1/\lambda)}}\longrightarrow0.
\end{aligned}
\end{equation}
Thus, $k_{\lambda,g}\geq g/2$ for all sufficiently small $\lambda$.

We next verify that $\int k_{\lambda,g}=1$. The preceding bounds justify absolute convergence of the integrals defining $k_{\lambda,g}$. Since $\int \ell_{u,V}'=0$, Fubini's theorem gives
\begin{align*}
    \int k_{\lambda,g}(x)\,\diff x
       &=1-\int_0^{\lambda^2/2}
          \left\{\int g'(y)\,\diff y\right\}
          \left\{\int \ell_{u,V}'(z)\,\diff z\right\}\,\diff u=1.
\end{align*}
Hence, $k_{\lambda,g}$ is a probability density.

\noindent\emph{(ii) Marginal density approximation.}
By \eqref{eq:misspecification-chf-k}, the Fourier transform of $f-\bar f_{\lambda,g}$ is
\begin{align*}
    \varphi_{f-\bar f_{\lambda,g}}(t)
       &=\varphi_g(t)e^{-t^2/2}
          -\varphi_{k_{\lambda,g}}(t)e^{-(1+\lambda^2)t^2/2}\\
       &=\varphi_g(t)e^{-(1+\lambda^2)t^2/2}
          \left[e^{\lambda^2t^2/2}-1
          -(e^{\lambda^2t^2/2}-1)\chi_V(t)\right]\\
       &=\varphi_g(t)e^{-(1+\lambda^2)t^2/2}
          (e^{\lambda^2t^2/2}-1)\{1-\chi_V(t)\}.
\end{align*}
Writing $e^{\lambda^2t^2/2}-1=t^2\int_0^{\lambda^2/2}e^{ut^2}\,\diff u$, we obtain
\begin{align*}
    \varphi_{f-\bar f_{\lambda,g}}(t)
       =\int_0^{\lambda^2/2}t^2\varphi_g(t)
          e^{-(1+\lambda^2-2u)t^2/2}\{1-\chi_V(t)\}\,\diff u.
\end{align*}
For $0\leq u\leq\lambda^2/2$, define
\begin{equation}\label{eq:misspecification-R-definition}
    R_{u,V}(x):=\frac1{2\pi}\int e^{-itx}
       e^{-(1+\lambda^2-2u)t^2/2}\{1-\chi_V(t)\}\,\diff t.
\end{equation}
Then integration by parts gives
\begin{align*}
    \varphi_{f-\bar f_{\lambda,g}}(t)
       &=\int_0^{\lambda^2/2}t^2\varphi_g(t)\varphi_{R_{u,V}}(t)\,\diff u\\
       &=-\int_0^{\lambda^2/2}
          \{-it\varphi_g(t)\}\{-it\varphi_{R_{u,V}}(t)\}\,\diff u\\
       &=-\int_0^{\lambda^2/2}
          \varphi_{g'}(t)\varphi_{R_{u,V}'}(t)\,\diff u.
\end{align*}
Therefore, we have
\begin{align*}
    f-\bar f_{\lambda,g}
       =-\int_0^{\lambda^2/2}g'\star R_{u,V}'\,\diff u.
\end{align*}
Combining \eqref{eq:misspecification-R-kernel-bound} in Lemma~\ref{lem:misspecification-cutoff-kernels} with \eqref{eq:misspecification-true-curvature} gives
\begin{equation}\label{eq:misspecification-candidate-marginal-bound}
    |f-\bar f_{\lambda,g}|
       \lesssim_L\lambda^2\exp\left(-V^2/2+KV^{4/3} \right)g.
\end{equation}
\noindent\emph{(iii) KL bound.}
Since $k_{\lambda,g}\geq g/2$, the density-ratio bound in \eqref{eq:misspecification-true-envelopes} in Lemma~\ref{lem:misspecification-density-regularity} gives
\begin{align*}
    \bar f_{\lambda,g}
       \geq\frac12g\star\phi_{\sqrt{1+\lambda^2}}\gtrsim_L g.
\end{align*}
Using this lower bound, \eqref{eq:misspecification-candidate-marginal-bound}, and $\int g=1$, we obtain
\begin{equation}\label{eq:misspecification-true-positive-approximation}
\begin{aligned}
    \mathrm{KL}(f\parallel\bar f_{\lambda,g})
       &\leq\int\frac{\{f(x)-\bar f_{\lambda,g}(x)\}^2}
                         {\bar f_{\lambda,g}(x)}\,\diff x\\
       &\leq K\lambda^4\exp\{-V^2+KV^{4/3}\}.
\end{aligned}
\end{equation}
Since $\bar f_{\lambda,g}\in\Fc_\lambda$, the definition of the KL projection gives
\begin{align*}
    \mathrm{KL}(f\parallel\widetilde f_\lambda)
       &\leq\mathrm{KL}\bigl(f
          \parallel \bar f_{\lambda,g}\bigr).
\end{align*}
Combining this and \eqref{eq:misspecification-true-positive-approximation} proves \eqref{eq:misspecification-true-projection-kl-bound}.
\end{proof}

\begin{lemma}\label{lem:misspecification-cutoff-kernels}
Let $L>0$ be the constant in \eqref{eq:log-density-slope-variation}. For $0<\lambda<1$, put $V^2=2\log(1/\lambda)/\lambda^2$, and let $\chi_V$, $\ell_{u,V}$, and $R_{u,V}$ be defined in \eqref{eq:misspecification-cutoff-definition}, \eqref{eq:misspecification-ell-definition}, and \eqref{eq:misspecification-R-definition}, respectively. There exist constants $K>0$ and $0<\lambda_0<1$, depending only on $L$, such that, for $0<\lambda<\lambda_0$, $\ell_{u,V}$ and $R_{u,V}$ are real-valued, continuously differentiable functions with bounded and integrable first derivatives, and satisfy, uniformly over $0\leq u\leq\lambda^2/2$,
\begin{equation}\label{eq:misspecification-L-kernel-bound}
    \sup_{x\in\Real}e^{(L+1)|x|}|\ell_{u,V}(x)|\lesssim_L Ve^{uV^2},
\end{equation}
and
\begin{equation}\label{eq:misspecification-R-kernel-bound}
    \sup_{x\in\Real}e^{(L+1)|x|}|R_{u,V}(x)|\lesssim_L e^{-V^2/2+KV^{4/3}}.
\end{equation}
\end{lemma}
\begin{proof}
We first verify that $\ell_{u,V}$ and $R_{u,V}$ are real-valued, continuously differentiable functions with bounded and integrable first derivatives. Note that the Fourier multipliers $e^{ut^2}\chi_V(t)$ and $e^{-(1+\lambda^2-2u)t^2/2}\{1-\chi_V(t)\}$ defining $\ell_{u,V}$ and $R_{u,V}$, respectively, are real and even. Also, they and their first two derivatives have Gaussian or faster decay. Hence, differentiation under the Fourier integrals shows that $\ell_{u,V}$ and $R_{u,V}$ are real-valued and continuously differentiable with bounded first derivatives. For example, for $x\ne0$, two integrations by parts with vanishing boundary terms give
\begin{align*}
    |\ell_{u,V}'(x)|
       &=\frac1{2\pi}\left|\int e^{-itx}t e^{ut^2}\chi_V(t)\,\diff t\right|\\
       &=\frac1{2\pi x^2}\left|\int e^{-itx}
          \frac{\diff^2}{\diff t^2}\{t e^{ut^2}\chi_V(t)\}\,\diff t\right|\\
       &\leq\frac1{2\pi x^2}\int
          \left|\frac{\diff^2}{\diff t^2}\{t e^{ut^2}\chi_V(t)\}\right|\,\diff t
        \lesssim_\lambda |x|^{-2}.
\end{align*}
The integral in the last line is bounded uniformly over $0\leq u\leq\lambda^2/2$ for each fixed $\lambda$. The same calculation gives $|R_{u,V}'(x)|\lesssim_\lambda |x|^{-2}$. Since $\ell_{u,V}'$ and $R_{u,V}'$ are bounded near zero and decay as $O(|x|^{-2})$ at infinity, both are integrable.

Next, we bound $\chi_V$ and $1-\chi_V$. Using $|e^{-(t-v+iy)^4}|\lesssim_L e^{-k(t-v)^4}$ for some $k>0$ and $d_V\asymp1$, we obtain
\begin{equation}\label{eq:misspecification-true-cutoff}
    |\chi_V(t+iy)|\lesssim_L\int_{-V}^{V} e^{-k(t-v)^4}\,\diff v,
       \qquad t\in\Real,\quad |y|\leq L+1,
\end{equation}
by \eqref{eq:misspecification-cutoff-definition}. To bound $1-\chi_V(t+iy)$, note that, for $|y|\leq L+1$, differentiating under the integral gives
\begin{align*}
    \frac{\diff}{\diff y}\int e^{-(t+iy)^4}\,\diff t
       =i\int\frac{\partial}{\partial t}e^{-(t+iy)^4}\,\diff t=0.
\end{align*}
Here, the exponential decay of the integrand and its first derivatives justifies differentiation under the integral and makes the boundary terms vanish. Thus, we have $\int e^{-(t+iy)^4}\,\diff t=\int e^{-t^4}\,\diff t$. A change of variables then gives $\int_{\Real}e^{-(t-v+iy)^4}\,\diff v=\int_{\Real}e^{-v^4}\,\diff v$, so
\begin{align*}
    d_V\{1-\chi_V(t+iy)\}
       &=\int_{|v|>V}e^{-(t-v+iy)^4}\,\diff v
          -\int_{|v|>V}e^{-v^4}\,\diff v.
\end{align*}
Here, the second integral satisfies $\int_{|v|>V}e^{-v^4}\,\diff v \leq e^{-V^4/2}\int_{\Real}e^{-v^4/2}\,\diff v \lesssim e^{-V^4/2}$. Now, let us first bound $|1-\chi_V(t+iy)|$ in the case $|t|\leq V$. Note that $|t-v|\geq V-|t|$ whenever $|v|>V$. Hence, using $d_V\asymp1$ and $|e^{-(r+iy)^4}|\lesssim_L e^{-r^4/2}$, we obtain
\begin{align*}
    |1-\chi_V(t+iy)|
       &\lesssim_L\int_{|v|>V}e^{-(t-v)^4/2}\,\diff v+e^{-V^4/2}\\
       &\leq\int_{|r|\geq V-|t|}e^{-r^4/2}\,\diff r+e^{-V^4/2}\\
       &\leq e^{-(V-|t|)^4/4}\int_{\Real}e^{-r^4/4}\,\diff r+e^{-V^4/2}\\
       &\lesssim e^{-(V-|t|)^4/4}.
\end{align*}
Next, for $|t|>V$, \eqref{eq:misspecification-true-cutoff} and the triangle inequality give $|1-\chi_V(t+iy)|\lesssim_L 1$. Combining bounds for $|t| \le V$ and $|t| > V$, we have
\begin{equation}\label{eq:misspecification-cutoff-complement}
    |1-\chi_V(t+iy)|\lesssim_L e^{-k(V-|t|)_+^4},
       \qquad t\in\Real,\quad |y|\leq L+1.
\end{equation}

Using \eqref{eq:misspecification-true-cutoff} and \eqref{eq:misspecification-cutoff-complement}, we now prove \eqref{eq:misspecification-L-kernel-bound} and \eqref{eq:misspecification-R-kernel-bound}. To prove \eqref{eq:misspecification-L-kernel-bound}, fix $x\in\Real$ and define, for real $|y|\leq L+1$,
\begin{align*}
    J(y):=\int e^{-i(t+iy)x}e^{u(t+iy)^2}\chi_V(t+iy)\,\diff t.
\end{align*}
As in the preceding arguments, differentiating under the integral gives
\begin{align*}
    J'(y) &=i\int\frac{\partial}{\partial t}
          \left\{e^{-i(t+iy)x}e^{u(t+iy)^2}\chi_V(t+iy)\right\}\,\diff t\\
       &=i\left[e^{-i(t+iy)x}e^{u(t+iy)^2}\chi_V(t+iy)\right]_{t=-\infty}^{t=\infty}
        =0.
\end{align*}
Thus $J(y)=J(0)=2\pi \ell_{u,V}(x)$, so
\begin{align*}
    \ell_{u,V}(x)
       =\frac{e^{yx}}{2\pi}\int e^{-itx}
          e^{u(t+iy)^2}\chi_V(t+iy)\,\diff t.
\end{align*}
Taking $y=-(L+1)\operatorname{sgn}(x)$ so that $e^{yx}=e^{-(L+1)|x|}$, applying \eqref{eq:misspecification-true-cutoff}, and setting $r=t-v$ yield
\begin{align*}
    e^{(L+1)|x|}|\ell_{u,V}(x)|
       &\lesssim_L\int e^{ut^2}
          \int_{-V}^{V}e^{-k(t-v)^4}\,\diff v\,\diff t\\
       &=\int_{-V}^{V}e^{uv^2}
          \int e^{2uvr+ur^2-kr^4}\,\diff r\,\diff v\\
       &\lesssim_L V e^{uV^2}.
\end{align*}
The last bound is uniform over $0\leq u\leq\lambda^2/2$, since $u\leq1$ and $uV\leq1$ for all sufficiently small $\lambda$. This proves \eqref{eq:misspecification-L-kernel-bound}.

Finally, we prove \eqref{eq:misspecification-R-kernel-bound}. The same argument as for $\ell_{u,V}$ gives, for real $|y|\leq L+1$,
\begin{align*}
    R_{u,V}(x)
       =\frac{e^{yx}}{2\pi}\int e^{-itx}
          e^{-(1+\lambda^2-2u)(t+iy)^2/2}
          \{1-\chi_V(t+iy)\}\,\diff t.
\end{align*}
Taking $y=-(L+1)\operatorname{sgn}(x)$, applying \eqref{eq:misspecification-cutoff-complement}, and using $1\leq1+\lambda^2-2u\leq2$ yield
\begin{align*}
    e^{(L+1)|x|}|R_{u,V}(x)|
       \lesssim_L\int e^{-t^2/2-k(V-|t|)_+^4}\,\diff t.
\end{align*}
On $|t|\leq V$, writing $r=V-|t|$ gives
\begin{align*}
    -t^2/2-kr^4
       \leq -V^2/2+V r-kr^4
       \leq -V^2/2+KV^{4/3}.
\end{align*}
Integrating over $|t|\leq V$ contributes a factor $2V$, which is absorbed by increasing $K$ in the exponent. Combining this with the Gaussian tail bound on $|t|>V$ yields, uniformly over $0\leq u\leq\lambda^2/2$,
\begin{align*}
    |R_{u,V}(x)|\lesssim_L e^{-V^2/2+KV^{4/3}}e^{-(L+1)|x|}.
\end{align*}
This proves \eqref{eq:misspecification-R-kernel-bound}.
\end{proof}

\appsubsubsection{Estimation of the pseudo-true prior density}\label{app:proof-misspecification-pseudoprior}
In this section, we prove Proposition~\ref{prop:misspecification-pseudoprior}. We first establish properties of the KL projection and other technical results needed for the proof. In the case of misspecification, \citet{Patilea2001} propose the divergence
\begin{equation}\label{eq:misspecification-global-divergence}
    \H_0^2(f,\widetilde f_\lambda) :=\frac12\int \left(\sqrt{\frac{f(x)}{\widetilde f_\lambda(x)} }-1\right)^2f^*(x)\,\diff x, \qquad f\in\Fc_\lambda,
\end{equation}
as a natural substitute for the squared Hellinger distance between $f$ and $\widetilde f_\lambda$. If $\widetilde f_\lambda=f^*$, then $\H_0^2(f,\widetilde f_\lambda)$ is the usual squared Hellinger distance between $f$ and $\widetilde f_\lambda$. However, $\H_0$ is not a distance in general. For $f\in\Fc_\lambda$, we have $0\leq\H_0^2(f,\widetilde f_\lambda)\leq1$; see Lemma 2.2 of \citet{Patilea2001}. We later use this notion to control the marginal density estimation error relative to $\widetilde f_\lambda$ and then establish convergence of the smooth NPMLE to the pseudo-true prior density $\widetilde g_\lambda$. The following lemma provides moment and lower density bounds for the KL projection.

\begin{lemma}[Properties of the KL projection]\label{lem:misspecification-projection-properties}
Suppose that $G^*$ satisfies \eqref{eq:simultaneous-moment-control} with $G^*$ in place of $\trueprior$ and $r=1$. For $0<\lambda\leq1$, let $\widetilde f_\lambda$ be the KL projection in \eqref{eq:pseudo-true density condition} and let $\widetilde H_\lambda$ be its mixing distribution as in Lemma~\ref{lem:misspecification-projection-existence}. There exist constants $K,k>0$, depending only on $K_0$, such that, for all $p\geq1$ and $x\in\Real$,
\begin{equation}\label{eq:misspecification-projection-properties}
\begin{aligned}
    \int|u|^p\,\diff\widetilde H_\lambda(u)\leq(Kp)^p, \qquad \widetilde f_\lambda(x)\geq ke^{-x^2}.
\end{aligned}
\end{equation}
Furthermore, if $G^*$ admits a density $g^*$ satisfying the density regularity conditions in Theorem~\ref{thm:misspecification-true}, then
\begin{equation}\label{eq:misspecification-projection-ratio}
    \widetilde f_\lambda(x)\lesssim_L f^*(x),\qquad x\in\Real,\quad 0<\lambda\leq1.
\end{equation}
\end{lemma}

\begin{proof}
We first establish the following optimality condition for $\widetilde H_\lambda$:
\begin{equation}\label{eq:misspecification-projection-optimality}
    \Psi_\lambda(u):=\int\phi_{\sqrt{1+\lambda^2}}(x-u) \frac{f^*(x)}{\widetilde f_\lambda(x)}\,\diff x\leq1, \qquad \Psi_\lambda=1\quad\widetilde H_\lambda\text{-almost everywhere}.
\end{equation}
Indeed, for each $u\in\Real$, consider $H_t=(1-t)\widetilde H_\lambda+t\delta_u$ for $0\leq t\leq1$. Its marginal density is
\begin{align*}
    p_{H_t,\sqrt{1+\lambda^2}}(x) =(1-t)\widetilde f_\lambda(x) +t\phi_{\sqrt{1+\lambda^2}}(x-u).
\end{align*}
The optimality of $\widetilde H_\lambda$ gives
\begin{align*}
    0&\geq\left.\frac{\diff}{\diff t} \int f^*(x)\log p_{H_t,\sqrt{1+\lambda^2}}(x)\,\diff x \right|_{t=0+} =\Psi_\lambda(u)-1.
\end{align*}
Moreover,
\begin{align*}
    \int\Psi_\lambda(u)\,\diff\widetilde H_\lambda(u) =\int f^*(x)\,\diff x=1.
\end{align*}
Since $\Psi_\lambda\leq1$, this implies $\Psi_\lambda=1$ $\widetilde H_\lambda$-almost everywhere. Now, define the probability measure
\begin{align*}
    \Pi_\lambda(\diff u,\diff x) :=\diff\widetilde H_\lambda(u)\, \phi_{\sqrt{1+\lambda^2}}(x-u) \frac{f^*(x)}{\widetilde f_\lambda(x)}\,\diff x.
\end{align*}
Let $(U,X)\sim\Pi_\lambda$. Then $U\sim\widetilde H_\lambda$ and $X$ has density $f^*$. Since $\Psi_\lambda\leq1$ in \eqref{eq:misspecification-projection-optimality}, any point $u$ satisfying $\Psi_\lambda(u)=1$ is a maximum. Differentiating with respect to $u$, we obtain
\begin{align*}
    0=\Psi_\lambda'(u) &=\frac1{1+\lambda^2} \int(x-u)\phi_{\sqrt{1+\lambda^2}}(x-u) \frac{f^*(x)}{\widetilde f_\lambda(x)}\,\diff x\\
    &=\frac{\E_{\Pi_\lambda}[X\mid U=u]-u}{1+\lambda^2}.
\end{align*}
Thus, $\E_{\Pi_\lambda}[X\mid U=u]=u$ for $\widetilde H_\lambda$-almost every $u$. Using this and the marginals of $\Pi_\lambda$, conditional Jensen's inequality gives, for all $p \ge 1$,
\begin{equation}\label{eq:misspecification-projection-moments}
\begin{aligned}
    \int|u|^p\,\diff\widetilde H_\lambda(u) &=\E_{\Pi_\lambda}|U|^p =\E_{\Pi_\lambda}\!\left[\left|\E_{\Pi_\lambda}[X\mid U]\right|^p\right]\\
    &\leq\E_{\Pi_\lambda}|X|^p =\E_{G^*}|X_1|^p\leq(Kp)^p,
\end{aligned}
\end{equation}
where the last inequality follows from the moment assumption on $G^*$ and the Gaussian moment bound. 

Next, note that \eqref{eq:misspecification-projection-moments} with $p=1$ and Markov's inequality give $\widetilde H_\lambda([-R,R])\geq1-K/R=1/2$ for $R=2K$. Since $1\leq1+\lambda^2\leq2$ and $(x-u)^2\leq2x^2+2R^2$ for $|u|\leq R$, we obtain
\begin{equation}\label{eq:misspecification-projection-envelope}
\begin{aligned}
    \widetilde f_\lambda(x) &\geq\int_{-R}^R\phi_{\sqrt{1+\lambda^2}}(x-u) \,\diff\widetilde H_\lambda(u)\\
    &\geq\frac{e^{-R^2}}{\sqrt{4\pi}}e^{-x^2}\,\widetilde H_\lambda([-R,R])\\
    &\geq\frac{e^{-R^2}}{4\sqrt\pi}e^{-x^2} =ke^{-x^2},\qquad x\in\Real,
\end{aligned}
\end{equation}
where $k=e^{-R^2}/(4\sqrt\pi)>0$ depends only on $K_0$. This proves \eqref{eq:misspecification-projection-properties}.

Finally, we show \eqref{eq:misspecification-projection-ratio}.  Suppose that $g^*$ satisfies the additional density regularity conditions in Theorem~\ref{thm:misspecification-true}. At points where $\Psi_\lambda(u)=1$, \eqref{eq:misspecification-projection-optimality} implies that $\Psi_\lambda$ has a maximum. Differentiating twice gives
\begin{align*}
    0\geq\Psi_\lambda''(u)
       =\frac{\E_{\Pi_\lambda}[(X-u)^2\mid U=u]-(1+\lambda^2)}
              {(1+\lambda^2)^2}
       \qquad\widetilde H_\lambda\text{-almost everywhere}.
\end{align*}
Since $\E_{\Pi_\lambda}[X\mid U]=U$, this yields $\operatorname{Var}_{\Pi_\lambda}(X\mid U)\leq1+\lambda^2$. Thus, for every $z\in\Real$,
\begin{align*}
    \E_{\Pi_\lambda}[(z-X)^2\mid U=u]
       &=(z-u)^2+\operatorname{Var}_{\Pi_\lambda}(X\mid U=u)\\
       &\leq(z-u)^2+(1+\lambda^2).
\end{align*}
Hence, conditional Jensen's inequality gives
\begin{equation}\label{eq:misspecification-projection-jensen}
\begin{aligned}
    \E_{\Pi_\lambda}\!\left[\phi_{\sqrt{1+\lambda^2}}(z-X)\mid U=u\right]
       &\geq\frac{\exp\{-\E_{\Pi_\lambda}[(z-X)^2\mid U=u]/(2+2\lambda^2)\}}
                     {\sqrt{2\pi(1+\lambda^2)}}\\
       &\geq e^{-1/2}\phi_{\sqrt{1+\lambda^2}}(z-u).
\end{aligned}
\end{equation}
By the density-ratio bound in \eqref{eq:misspecification-true-envelopes} in Lemma~\ref{lem:misspecification-density-regularity}, $g^*(z-y)\leq e^{L|y|}g^*(z)$ for all $z,y\in\Real$. Thus, we have
\begin{align*}
    f^*(z-y)
       &=\int g^*(z-y-v)\phi_1(v)\,\diff v\\
       &\leq e^{L|y|}\int g^*(z-v)\phi_1(v)\,\diff v
        =e^{L|y|}f^*(z).
\end{align*}
Integrating \eqref{eq:misspecification-projection-jensen} against $\widetilde H_\lambda$, and recalling that $U\sim\widetilde H_\lambda$ and $X$ has density $f^*$ under $\Pi_\lambda$, gives, for $Z\sim N(0,1)$,
\begin{align*}
    \widetilde f_\lambda(z)
       &=\int\phi_{\sqrt{1+\lambda^2}}(z-u)\,\diff\widetilde H_\lambda(u)\\
       &\leq e^{1/2}\int
          \E_{\Pi_\lambda}\!\left[\phi_{\sqrt{1+\lambda^2}}(z-X)\mid U=u\right]
          \,\diff\widetilde H_\lambda(u)\\
       &=e^{1/2}\E_{\Pi_\lambda}\!\left[\phi_{\sqrt{1+\lambda^2}}(z-X)\right]\\
       &=e^{1/2}\int\phi_{\sqrt{1+\lambda^2}}(y)f^*(z-y)\,\diff y\\
       &\leq e^{1/2}f^*(z)\E[e^{L\sqrt2|Z|}]
        \lesssim_L f^*(z),\qquad 0<\lambda\leq1.
\end{align*}
This proves \eqref{eq:misspecification-projection-ratio}.

\end{proof}

Now, we first control the marginal density estimation error relative to the KL projection using the divergence $\H_0$ introduced by \citet{Patilea2001} and defined in \eqref{eq:misspecification-global-divergence}. Proposition~4.1 of \citet{Patilea2001} relates convergence to the pseudo-true marginal density in the $\H_0$ divergence to a bracketing integral. To define this integral, we first introduce a bounded transformation of $\Fc_\lambda$. Let $P$ be the probability measure with density $f^*$, and define
\begin{equation}\label{eq:misspecification-bracketing-class}
    h_f(x):=\frac{2f(x)}{f(x)+\widetilde f_\lambda(x)},\qquad
    \widetilde{\Fc}_\lambda:=\{h_f:f\in\Fc_\lambda\}.
\end{equation}
Note that $0<h_f(x)\leq2$ for every $x\in\Real$ and $f\in\Fc_\lambda$. For a class $\mathcal A$, define its \emph{$L_2(P)$ bracketing entropy} as $\Hc_B(\delta,\mathcal A,L_2(P)):=\log \Nc_B(\delta,\mathcal A,L_2(P))$, where
\begin{align*}
    \Nc_B(\delta,\mathcal A,L_2(P)) :=\min\Bigl\{J\in\mathbb N:\  &\exists\,\{(h_j^L,h_j^U)\}_{j=1}^J \text{ such that}\\
    &\forall h\in\mathcal A,\ \exists j\in\{1,\ldots,J\},\  h_j^L(x)\leq h(x)\leq h_j^U(x) \quad\forall x\in\Real,\\
    &\int\{h_j^U(x)-h_j^L(x)\}^2f^*(x)\,\diff x \leq\delta^2,\quad j=1,\ldots,J \Bigr\}
\end{align*}
is the \emph{$L_2(P)$ bracketing number} of $\mathcal A$, namely the minimum number of pointwise brackets needed to cover $\mathcal A$, each with $L_2(P)$ width at most $\delta$. Finally, define the bracketing integral 
\begin{align*}
    \Jc_B(\delta,\widetilde{\Fc}_\lambda,L_2(P)) :=\int_{\delta^2/c_1}^{\delta} \sqrt{\Hc_B(u,\widetilde{\Fc}_\lambda,L_2(P))}\,\diff u\vee\delta, \qquad 0<\delta<1/2,
\end{align*}
where $c_1$ is a sufficiently large universal constant. The following lemma provides uniform bounds on the bracketing entropy and integral of $\widetilde{\Fc}_\lambda$ under the moment condition \eqref{eq:simultaneous-moment-control}.

\begin{lemma}[Bracketing entropy and integral bounds]\label{lem:misspecification-bracketing}
Suppose that $G^*$ satisfies \eqref{eq:simultaneous-moment-control} with $G^*$ in place of $\trueprior$ and $r=1$. Then, uniformly over $0<\lambda\leq1$ and $0<\delta<1/2$,
\begin{equation}\label{eq:misspecification-bracketing-entropy}
    \Hc_B(\delta,\widetilde{\Fc}_\lambda,L_2(P))\lesssim_{K_0}\{\log(1/\delta)\}^4.
\end{equation}
Moreover,
\begin{equation}\label{eq:misspecification-bracketing-integral}
    \Jc_B(\delta,\widetilde{\Fc}_\lambda,L_2(P))\lesssim_{K_0}\delta\{\log(1/\delta)\}^2.
\end{equation}
\end{lemma}

\begin{proof}
The moment condition on $G^*$ and the Gaussian moment bound give $P\{|x|>t\}\lesssim_{K_0}e^{-kt}$ for some $k>0$. Therefore, for $0<\delta<1/2$, we can choose
\begin{align*}
    B(\delta):=C_B\log(1/\delta),\qquad
    \sigma(\delta):=k_0e^{-B(\delta)^2},\qquad
    \eta(\delta):=\frac{\delta\sigma(\delta)}{16},
\end{align*}
where $C_B$ is sufficiently large that $P\{|x|>B(\delta)\}\leq\delta^2/16$ and $k_0>0$ satisfies $\widetilde f_\lambda(x)\geq k_0e^{-x^2}$ for all $x\in\Real$, as guaranteed by \eqref{eq:misspecification-projection-properties} in Lemma~\ref{lem:misspecification-projection-properties}. In particular, $\widetilde f_\lambda(x)\geq\sigma(\delta)$ for $|x|\leq B(\delta)$. We apply Theorem~3.1 of \citet{Ghosal2001} to obtain a finite cover of $\Fc_\lambda$ on $[-B(\delta),B(\delta)]$. For this, we first approximate each mixture using a compactly supported mixing distribution. Put
\begin{align*}
    a(\delta):=B(\delta)+2\sqrt{\log\{4/\eta(\delta)\}} \lesssim_{K_0}\sqrt{\log\{1/\eta(\delta)\}}.
\end{align*}
For $f=p_{H,\sqrt{1+\lambda^2}}$, move the mass of $H$ outside $[-a(\delta),a(\delta)]$ to the nearest endpoint and denote the resulting density by $f^{a}$. If $|x|\leq B(\delta)$ and $|u|>a(\delta)$, then $|x-u|\geq a(\delta)-B(\delta)$ and $|x-a(\delta)\operatorname{sgn}(u)|\geq a(\delta)-B(\delta)$. Since $1\leq1+\lambda^2\leq2$, both $\phi_{\sqrt{1+\lambda^2}}(x-u)$ and $\phi_{\sqrt{1+\lambda^2}}(x-a(\delta)\operatorname{sgn}(u))$ are bounded by $\exp[-\{a(\delta)-B(\delta)\}^2/4]$. For $|x|\leq B(\delta)$, we have
\begin{align*}
    |f(x)-f^a(x)|
    &=\left|\int_{|u|>a(\delta)}
       \left\{\phi_{\sqrt{1+\lambda^2}}(x-u)-\phi_{\sqrt{1+\lambda^2}}(x-a(\delta)\operatorname{sgn}(u))\right\}\,\diff H(u)\right|\\
    &\leq2\exp\left\{-\frac{\{a(\delta)-B(\delta)\}^2}{4}\right\}H\{|u|>a(\delta)\}\\
    &\leq2\exp\left\{-\frac{\{a(\delta)-B(\delta)\}^2}{4}\right\}.
\end{align*}
Taking the supremum over $|x|\leq B(\delta)$ and using the definition of $a(\delta)$ gives
\begin{align*}
    \|f-f^{a}\|_{\infty,[-B(\delta),B(\delta)]}
    \leq2\exp\left\{-\frac{\{a(\delta)-B(\delta)\}^2}{4}\right\}
    =\frac{\eta(\delta)}2.
\end{align*}
Here, $\|h\|_{\infty,[-B(\delta),B(\delta)]}:=\sup_{|x|\leq B(\delta)}|h(x)|$. Then, Theorem~3.1 of \citet{Ghosal2001}, applied with covering radius $\eta(\delta)/2$, gives a finite collection of approximation centers $p_1,\ldots,p_J$ such that, for every truncated mixture $f^a$, some $p_j$ satisfies
\begin{align*}
    \|f^a-p_j\|_\infty\leq\frac{\eta(\delta)}2,
    \qquad \log J\lesssim_{K_0}\{\log(1/\eta(\delta))\}^2.
\end{align*}
Combining this with the preceding truncation bound gives, for every $f\in\Fc_\lambda$,
\begin{align*}
    \|f-p_j\|_{\infty,[-B(\delta),B(\delta)]}
    &\leq\|f-f^a\|_{\infty,[-B(\delta),B(\delta)]}+\|f^a-p_j\|_\infty\\
    &\leq\frac{\eta(\delta)}2+\frac{\eta(\delta)}2=\eta(\delta).
\end{align*}
Note that the map $t\mapsto2t/(t+\widetilde f_\lambda(x))$ is nondecreasing, and, for $0\leq l\leq u$,
\begin{equation}\label{eq:misspecification-transform-bound}
    \frac{2u}{u+\widetilde f_\lambda(x)} -\frac{2l}{l+\widetilde f_\lambda(x)} =\frac{2\widetilde f_\lambda(x)(u-l)}  {\{u+\widetilde f_\lambda(x)\}\{l+\widetilde f_\lambda(x)\}} \leq\frac{2(u-l)}{\widetilde f_\lambda(x)}.
\end{equation}
For each center $p_j$ and $|x|\leq B(\delta)$, apply the map $t\mapsto2t/\{t+\widetilde f_\lambda(x)\}$ to $(p_j(x)-\eta(\delta))_+$ and $p_j(x)+\eta(\delta)$ to obtain $h_j^L(x)$ and $h_j^U(x)$, respectively. For $|x|>B(\delta)$, set $h_j^L(x)=0$ and $h_j^U(x)=2$. These form pointwise brackets $[h_j^L,h_j^U]$, $j=1,\ldots,J$, for $\widetilde{\Fc}_\lambda$. For $|x|\leq B(\delta)$, \eqref{eq:misspecification-transform-bound} gives
\begin{align*}
    h_j^U(x)-h_j^L(x)
    &\leq\frac{2}{\widetilde f_\lambda(x)}
       \left\{p_j(x)+\eta(\delta)-(p_j(x)-\eta(\delta))_+\right\}\\
    &\leq\frac{4\eta(\delta)}{\widetilde f_\lambda(x)}
     \leq\frac{4\eta(\delta)}{\sigma(\delta)},
\end{align*}
where we used $p_j(x)+\eta(\delta)-(p_j(x)-\eta(\delta))_+\leq2\eta(\delta)$ and $\widetilde f_\lambda(x)\geq\sigma(\delta)$. Consequently,
\begin{align*}
    \|h_j^U-h_j^L\|_{L_2(P)}^2
    &=\int\{h_j^U(x)-h_j^L(x)\}^2f^*(x)\,\diff x\\
    &\leq\left\{\frac{4\eta(\delta)}{\sigma(\delta)}\right\}^2
      +4P\{|x|>B(\delta)\}\\
    &\leq\frac{\delta^2}{16}+\frac{\delta^2}{4}
     \leq\delta^2.
\end{align*}
Thus, uniformly over $0<\lambda\leq1$,
\begin{align*}
    \Hc_B(\delta,\widetilde{\Fc}_\lambda,L_2(P))
    \lesssim_{K_0}\{\log(1/\eta(\delta))\}^2
    \lesssim_{K_0}\{\log(1/\delta)\}^4.
\end{align*}
This proves \eqref{eq:misspecification-bracketing-entropy}. Integrating this bound gives
\begin{align*}
    \Jc_B(\delta,\widetilde{\Fc}_\lambda,L_2(P))
    &\lesssim_{K_0}\int_0^\delta\{\log(1/u)\}^2\,\diff u\\
    &\lesssim_{K_0}\delta\{\log(1/\delta)\}^2.
\end{align*}
This proves \eqref{eq:misspecification-bracketing-integral}.
\end{proof}

We next prove a uniform exponential bound to control the expected $\H_0^2$ error.
\begin{lemma}[Uniform exponential deviation bound]\label{lem:misspecification-local-deviation}
Let $h_f$ be defined as in \eqref{eq:misspecification-bracketing-class}. Under the assumptions of Lemma~\ref{lem:misspecification-bracketing}, put
\begin{align*}
    m(f):=\int(1-h_f)\,\diff P,\qquad
    \delta_n:=\frac{(\log n)^2}{\sqrt n},
\end{align*}
and let $P_n:=n^{-1}\sum_{i=1}^n\delta_{X_i}$. There exist constants $K,k>0$, depending only on $K_0$, such that, for all sufficiently large $n$, every $0<\lambda\leq1$, and $K\delta_n^2\leq s\leq1$,
\begin{equation}\label{eq:misspecification-local-deviation}
    \P_{G^*}\left\{
       \sup_{\substack{f\in\Fc_\lambda:\ m(f)\leq2s}}
       \int(h_f-1)\,\diff(P_n-P)>s/2\right\}
       \leq Ke^{-kns}.
\end{equation}
\end{lemma}

\begin{proof}
Integrating \eqref{eq:misspecification-projection-optimality} from the proof of Lemma~\ref{lem:misspecification-projection-properties} against the mixing distribution of $f\in\Fc_\lambda$ gives $\int f/\widetilde f_\lambda\,\diff P\leq1$. Using this and Lemma~2.3 of \citet{Patilea2001}, we have
\begin{equation}\label{eq:misspecification-contrast-moments}
    \H_0^2(f,\widetilde f_\lambda)\leq m(f)\leq1,\qquad \int(h_f-1)^2\,\diff P\leq2m(f).
\end{equation}
Here, $m(f)\leq1$ follows from $h_f\geq0$, and the squared $L_2(P)$ bound follows by integrating $(h_f-1)^2\leq2(1-h_f)+f/\widetilde f_\lambda-1$ and using $\int f/\widetilde f_\lambda\,\diff P\leq1$.

By Theorem~5.11 of \citet{vandeGeer2000}, for $K\delta_n^2\leq s\leq1$, it holds that
\begin{align*}
    \P_{G^*}\left\{
       \sup_{\substack{f\in\Fc_\lambda:\ m(f)\leq2s}}
       \int(h_f-1)\,\diff(P_n-P)>s/2\right\}
    \leq C\exp\left\{-\frac{a_s^2}{CR_s^2}\right\}
       \leq Ke^{-kns},
\end{align*}
where $R_s=\sqrt{8s}$ and $a_s=\sqrt n\,s/2$, and $K,k>0$ depend only on $K_0$. Thus, it suffices to verify conditions~(5.30)--(5.34) of Theorem~5.11 of \citet{vandeGeer2000}. For this, let $\mathcal G_s:=\{h_f-1:m(f)\leq2s\}$. By \eqref{eq:misspecification-contrast-moments}, every $g\in\mathcal G_s$ satisfies $|g|\leq1$ and $\int g^2\,\diff P\leq4s$. We use the Bernstein norm $\rho_4$, defined by
\begin{align*}
    \rho_4^2(g):=32\int\left(e^{|g|/4}-1-\frac{|g|}{4}\right)\,\diff P.
\end{align*}
Since $e^x-1-x\leq x^2$ for $0\leq x\leq1$,
\begin{align*}
    \rho_4^2(g)\leq2\int g^2\,\diff P\leq8s=R_s^2,\qquad g\in\mathcal G_s,
\end{align*}
which verifies~(5.30). Next, take the theorem's constant of \citet{vandeGeer2000} as $C_1=1$. Conditions~(5.31) and~(5.32) follow from
\begin{align*}
    \frac{a_s}{\sqrt n\,R_s^2}=\frac1{16}\leq\frac14,
    \qquad
    \frac{a_s}{\sqrt n\,R_s}=\frac{\sqrt s}{4\sqrt2}\leq\frac1{4\sqrt2}<8,
    \qquad 0<s\leq1.
\end{align*}
Choose the theorem's constant $C_0\geq\sqrt2\,C$, with $C$ as in that theorem. This verifies~(5.34). For~(5.33), Lemma~5.10 of \citet{vandeGeer2000} gives
\begin{align*}
    \Hc_{B,4}(u,\mathcal G_s,P)
    \leq \Hc_B(u/\sqrt2,\widetilde{\Fc}_\lambda,L_2(P)),
\end{align*}
where $\Hc_{B,4}$ denotes bracketing entropy computed using $\rho_4$ in place of the $L_2(P)$ norm. By monotonicity, the right-hand side is bounded by $\Hc_B(1/4,\widetilde{\Fc}_\lambda,L_2(P))\lesssim_{K_0}1$ whenever $u/\sqrt2\geq1/2$. Hence, by \eqref{eq:misspecification-bracketing-entropy} in Lemma~\ref{lem:misspecification-bracketing},
\begin{align*}
    \left\{\int_{a_s/(64\sqrt n)}^{R_s}
       \sqrt{\Hc_{B,4}(u,\mathcal G_s,P)}\,\diff u\right\}\vee R_s
    \lesssim_{K_0}\sqrt s\{1+\log(1/s)\}^2
     \lesssim_{K_0}\frac{a_s}{\sqrt K},
\end{align*}
where the last bound uses $ns\geq K(\log n)^4$. With $C_0$ fixed as above, choose $K$ sufficiently large that the last bound is at most $a_s/C_0$. This verifies~(5.33). Finally, $a_s^2/R_s^2=ns/32$, so~(5.35) gives \eqref{eq:misspecification-local-deviation} for $K\delta_n^2\leq s\leq1$.
\end{proof}

Using Lemmas~\ref{lem:misspecification-bracketing} and~\ref{lem:misspecification-local-deviation} and following the proof of Proposition~4.1 of \citet{Patilea2001}, we show that the estimated marginal density $\widehat f_n^+$ converges to the pseudo-true marginal density $\widetilde f_{\lambda_n}$, with $\lambda_n=\sqrt{\rho_n}$, in the $\H_0$ divergence. Note that, for $\czero=0$, the $\H_0^2$ bound in the following proposition requires only the moment and Fourier conditions in  Assumption~\ref{ass:moment-fourier} $(i)$ with $r=1$, without the additional density regularity conditions imposed in Theorem~\ref{thm:misspecification-true}.

\begin{prop}[Estimation of the pseudo-true marginal density]\label{prop:misspecification-marginal}
Suppose that $\czero=0$ and Assumption~\ref{ass:moment-fourier} (i) holds with $r=1$. Put $\lambda_n:=\sqrt{\rho_n}$, where $\rho_n$ is defined in \eqref{eq:rho_n}. There exists $n_0$, depending only on $\gammav$, such that, for $n\geq n_0$,
\begin{equation}\label{eq:misspecification-marginal-rate}
    \E_{G^*}\!\left[\H_0^2(\widehat f_n^+,\widetilde f_{\lambda_n}) \right] \lesssim_{\gammav}q_n :=\frac{(\log n)^4}{n}.
\end{equation}
If, in addition, $G^*$ admits a density $g^*$ satisfying the density regularity conditions in Theorem~\ref{thm:misspecification-true}, then
\begin{equation}\label{eq:misspecification-marginal-hellinger-rate}
    \E_{G^*}[\H^2(\widehat f_n^+,\widetilde f_{\lambda_n})]
       \lesssim_{\gammav,L}q_n.
\end{equation}
\end{prop}

\begin{proof}
We first control the marginal log-likelihood of $\widehat f_n^+$ relative to that of the pseudo-true density $\widetilde f_{\lambda_n}$. Since $\hat c_0^+\geq\lambda_n$ by \eqref{eq:deconvolution-floor}, we can write
\begin{equation}\label{eq:misspecification-floor-embedding}
\begin{aligned}
    \widehat Q_n:=\npmle^+\star N(0,(\hat c_0^+)^2-\lambda_n^2), \qquad \widehat g_n^+=\widehat Q_n\star\phi_{\lambda_n},\qquad \widehat f_n^+=p_{\widehat Q_n,\sqrt{1+\lambda_n^2}}\in\Fc_{\lambda_n}.
\end{aligned}
\end{equation}
Let $\mathcal T_n$ be the event in \eqref{eq:estimated-c-denoising-event} with $r=1$, and put $T_n:=K_R\log n$. Since $\czero=0$, this is
\begin{equation}\label{eq:misspecification-estimation-event}
\begin{aligned}
    \mathcal T_n &=\left\{0\leq\hat c_0^2\leq K_{\mathcal S}\rho_n,\quad \widehat C\leq K_{\mathcal S},\quad 0\leq L_n(0)-L_n(\hat c_0)\leq K_{\mathcal S}b_n\right\}\\
    &\qquad\cap\left\{\max_i|X_i|\leq T_n\right\},
\end{aligned}
\end{equation}
where $K_{\mathcal S}>0$ is the constant in \eqref{eq:modified-c-good-event}, $\widehat C$ is defined in \eqref{eq:estimated-c-buffer}, and $L_n(c)$ is the maximized marginal log-likelihood at smoothing parameter $c$, defined with $b_n$ in \eqref{eq:likelihood}. By \eqref{eq:estimated-c-denoising-good-event}, \eqref{eq:modified-c-good-event}, and \eqref{eq:deconvolution-floor}, we have, for all sufficiently large $n$,
\begin{equation}\label{eq:misspecification-good-event}
\begin{aligned}
    \P_{G^*}(\mathcal T_n^c)\lesssim_{\gammav}n^{-2},\qquad \lambda_n^2\leq (\hat c_0^+)^2\leq (K_{\mathcal S}\vee1)\rho_n\leq1, \qquad b_n\lesssim_{\gammav}(\log n)^{5/2},
\end{aligned}
\end{equation}
on $\mathcal T_n$. Then the likelihood bound in \eqref{eq:misspecification-estimation-event}, monotonicity of $L_n$, and $\widetilde f_{\lambda_n}\in\Fc_{\lambda_n}$ give
\begin{equation}\label{eq:misspecification-floor-likelihood}
\begin{aligned}
    \sum_{i=1}^n\log\widehat f_n^+(X_i)
       &=L_n(\hat c_0^+) =\min\{L_n(\hat c_0),L_n(\lambda_n)\}\\
       &\geq L_n(\lambda_n)-K_{\mathcal S}b_n\\
       &\geq L_n(\lambda_n)-d_n\\
       &\geq\sum_{i=1}^n\log\widetilde f_{\lambda_n}(X_i)-d_n
       \quad\text{on }\mathcal T_n,
\end{aligned}
\end{equation}
where $d_n:=K_1(\log n)^{5/2}$, with $K_1>0$ depending only on $\gammav$ and sufficiently large that $K_{\mathcal S}b_n\leq d_n$ on $\mathcal T_n$.

Next, we use the likelihood bound to derive an inequality for the estimation error in the $\H_0$ divergence. Recall $P_n$ and $m(f)$ from Lemma~\ref{lem:misspecification-local-deviation}, and $h_f$ from \eqref{eq:misspecification-bracketing-class}. By convexity, $(\widehat f_n^++\widetilde f_{\lambda_n})/2\in\Fc_{\lambda_n}$, so its marginal log-likelihood is at most $L_n(\lambda_n)$. Thus, on $\mathcal T_n$,
\begin{align*}
    \int(h_{\widehat f_n^+}-1)\,\diff P_n
    &\geq\int\log h_{\widehat f_n^+}\,\diff P_n\\
    &=\frac1n\sum_{i=1}^n\log
       \frac{\widehat f_n^+(X_i)}
       {\{\widehat f_n^+(X_i)+\widetilde f_{\lambda_n}(X_i)\}/2}\\
    &\geq\frac{L_n(\lambda_n)-d_n-L_n(\lambda_n)}{n}
     =-\frac{d_n}{n}.
\end{align*}
Here, the first inequality uses $\log y\leq y-1$, and the last uses \eqref{eq:misspecification-floor-likelihood}. Consequently,
\begin{equation}\label{eq:misspecification-basic-inequality}
\begin{aligned}
    m(\widehat f_n^+)
    &=\int(h_{\widehat f_n^+}-1)\,\diff(P_n-P)
      -\int(h_{\widehat f_n^+}-1)\,\diff P_n\\
    &\leq\int(h_{\widehat f_n^+}-1)\,\diff(P_n-P)
      +\frac{d_n}{n}
      \qquad\text{on }\mathcal T_n.
\end{aligned}
\end{equation}

Finally, we combine this inequality with Lemmas~\ref{lem:misspecification-bracketing} and~\ref{lem:misspecification-local-deviation} to bound the expected $\H_0^2$ error. Recall $\delta_n = (\log n)^2/\sqrt{n}$ from Lemma~\ref{lem:misspecification-local-deviation}. By Lemma~\ref{lem:misspecification-bracketing},
\begin{align*}
    \Jc_B(\delta_n,\widetilde{\Fc}_{\lambda_n},L_2(P))\lesssim_{\gammav}\delta_n\{\log(1/\delta_n)\}^2\lesssim_{\gammav}\sqrt n\delta_n^2,
\end{align*}
and $d_n/n=o(\delta_n^2)$. Thus, following the argument of Proposition~4.1 of
\citet{Patilea2001}, we obtain $\H_0^2(\widehat f_n^+,\widetilde f_{\lambda_n}) =O_P(\delta_n^2)$. Further, to bound $\E_{G^*}[\H_0^2(\widehat f_n^+,\widetilde f_{\lambda_n})]$, we derive exponential bounds for $\P_{G^*}\{\H_0^2(\widehat f_n^+,\widetilde f_{\lambda_n})>t,\, \mathcal T_n\}$ and integrate them (see also the paragraph immediately following Proposition 4.1 of \citet{Patilea2001}). Let $t\geq K(\delta_n^2+d_n/n)$ with $K$ sufficiently large. Then $d_n/n\leq t/2$, and Lemma~\ref{lem:misspecification-local-deviation} applies with $s=2^jt$ whenever $j\geq0$ and $2^jt<1$. On $\mathcal T_n\cap\{2^jt<m(\widehat f_n^+)\leq2^{j+1}t\}$, \eqref{eq:misspecification-basic-inequality} gives
\begin{align*}
    \int(h_{\widehat f_n^+}-1)\,\diff(P_n-P)
    &\geq m(\widehat f_n^+)-d_n/n\\
    &>2^jt-d_n/n\geq2^{j-1}t.
\end{align*}
Since $m(\widehat f_n^+)\leq2^{j+1}t$, applying \eqref{eq:misspecification-local-deviation} with $s=2^jt$ gives, for every $j\geq0$ with $2^jt<1$,
\begin{align*}
    &\P_{G^*}\{2^jt<m(\widehat f_n^+)\leq2^{j+1}t,\ \mathcal T_n\}\\
    &\quad\leq\P_{G^*}\left\{
       \sup_{\substack{f\in\Fc_{\lambda_n}:\ m(f)\leq2^{j+1}t}}
       \int(h_f-1)\,\diff(P_n-P)>2^{j-1}t\right\}
       \leq Ke^{-kn2^jt}.
\end{align*}
Because $m(f)\leq1$, the intervals $(2^jt,2^{j+1}t]$ with $2^jt<1$ cover all possible values of $m(\widehat f_n^+)>t$. Therefore, a union bound gives, for $t\geq K(\delta_n^2+d_n/n)$,
\begin{equation}\label{eq:misspecification-marginal-deviation}
\begin{aligned}
    \P_{G^*}\{m(\widehat f_n^+)>t,\ \mathcal T_n\} \leq\sum_{\substack{j\geq0:\\2^jt<1}}Ke^{-kn2^jt} \leq Ke^{-knt}.
\end{aligned}
\end{equation}
Here, the last inequality uses $2^j\geq j+1$ and $nt\geq K(\log n)^4$, after enlarging $K$ if necessary. Then, using $\H_0^2(f,\widetilde f_{\lambda_n})\leq m(f)\leq1$ from \eqref{eq:misspecification-contrast-moments} and the bound in \eqref{eq:misspecification-marginal-deviation}, we obtain
\begin{align*}
    \E_{G^*}[\H_0^2(\widehat f_n^+,\widetilde f_{\lambda_n})]
    &\leq\E_{G^*}[m(\widehat f_n^+)\1v_{\mathcal T_n}]+\P_{G^*}(\mathcal T_n^c)\\
    &=\int_0^1\P_{G^*}\{m(\widehat f_n^+)>t,\ \mathcal T_n\}\,\diff t
       +\P_{G^*}(\mathcal T_n^c)\\
    &\leq K(\delta_n^2+d_n/n)
       +\int_{K(\delta_n^2+d_n/n)}^\infty Ke^{-knt}\,\diff t
       +\P_{G^*}(\mathcal T_n^c)\\
    &\leq K(\delta_n^2+d_n/n)+\frac{K}{kn}+\P_{G^*}(\mathcal T_n^c)\\
    &\lesssim_{\gammav}\frac{(\log n)^4}{n}.
\end{align*}
The last inequality uses \eqref{eq:misspecification-good-event}, $\delta_n^2=(\log n)^4/n$, and $d_n/n=o(\delta_n^2)$. This proves \eqref{eq:misspecification-marginal-rate}.

Under the additional density regularity conditions in Theorem~\ref{thm:misspecification-true}, \eqref{eq:misspecification-projection-ratio} gives $\H^2(f,\widetilde f_{\lambda_n})\lesssim_L\H_0^2(f,\widetilde f_{\lambda_n})$ for every $f\in\Fc_{\lambda_n}$, so \eqref{eq:misspecification-marginal-rate} proves \eqref{eq:misspecification-marginal-hellinger-rate}.
\end{proof}

We are now ready to prove Proposition~\ref{prop:misspecification-pseudoprior}, which bounds the estimation error of the smooth NPMLE $\widehat g_n^+$ relative to the pseudo-true prior density $\widetilde g_{\lambda_n}$ with $\lambda_n = \sqrt{\rho_n}$.

\begin{proof}[Proof of Proposition~\ref{prop:misspecification-pseudoprior}]
Put $\lambda_n:=\sqrt{\rho_n}$. We apply Plancherel's theorem to bound the estimation error of $\widehat g_n^+$ relative to the pseudo-true prior density $\widetilde g_{\lambda_n}$. By \eqref{eq:misspecification-floor-embedding},
\begin{align*}
    \widehat g_n^+-\widetilde g_{\lambda_n}
       &=(\widehat Q_n-\widetilde H_{\lambda_n})\star\phi_{\lambda_n},\\
    \widehat f_n^+-\widetilde f_{\lambda_n}
       &=(\widehat Q_n-\widetilde H_{\lambda_n})
          \star\phi_{\sqrt{1+\lambda_n^2}},
\end{align*}
where $\widehat Q_n=\npmle^+\star N(0,(\hat c_0^+)^2-\lambda_n^2)$. Following the proof of Theorem~\ref{thm:moment-deconvolution}, for every $T>0$ and $0<\tau_n<\lambda_n$,
\begin{equation}\label{eq:misspecification-true-fourier-split}
\begin{aligned}
    \|\widehat g_n^+-\widetilde g_{\lambda_n}\|_{L_2}^2
       &=\frac1{2\pi}\int
          |\varphi_{\widehat Q_n}(t)-\varphi_{\widetilde H_{\lambda_n}}(t)|^2
          e^{-\lambda_n^2t^2}\,\diff t\\
       &\leq\underbrace{
          e^{T^2}\|\widehat f_n^+-\widetilde f_{\lambda_n}\|_{L_2}^2
          }_{\mathrm{(I)}}+\underbrace{
          e^{-(\lambda_n^2-\tau_n^2)T^2}
             \|(\widehat Q_n-\widetilde H_{\lambda_n})
                \star\phi_{\tau_n}\|_{L_2}^2
          }_{\mathrm{(II)}},
\end{aligned}
\end{equation}
where $\tau_n$ will be specified below. We will bound $\mathrm{(I)}$ and $\mathrm{(II)}$ in \eqref{eq:misspecification-true-fourier-split} separately.

To control $\mathrm{(I)}$, note that all densities in $\Fc_{\lambda_n}$ are bounded by $(2\pi)^{-1/2}$, so \eqref{eq:misspecification-marginal-hellinger-rate} in Proposition~\ref{prop:misspecification-marginal} yields
\begin{equation}\label{eq:misspecification-true-marginal-l2}
\begin{aligned}
    \E_{G^*}[\|\widehat f_n^+-\widetilde f_{\lambda_n}\|_{L_2}^2] \lesssim \E_{G^*}[\H^2(\widehat f_n^+,\widetilde f_{\lambda_n})] \lesssim_{\gammav,L}q_n,
\end{aligned}
\end{equation}
where $q_n=(\log n)^4/n$ is defined in \eqref{eq:misspecification-marginal-rate}. Thus, we have
\begin{equation}\label{eq:misspecification-true-I-bound}
    \E_{G^*}[\mathrm{(I)}]\lesssim_{\gammav,L}e^{T^2}q_n.
\end{equation}
Next, we control $\mathrm{(II)}$ using Lemma~\ref{lem:misspecification-smoothed-mixing}. For this, we first construct a high-probability event. By \eqref{eq:misspecification-true-projection-kl-bound} in Lemma~\ref{lem:misspecification-kl-approximation}, with $V_{\lambda_n}^2=2\log(1/\lambda_n)/\lambda_n^2$ and $\lambda_n^2=\rho_n$, we have
\begin{equation}\label{eq:misspecification-true-projection-kl}
\begin{aligned}
    \mathrm{KL}(f^*\parallel\widetilde f_{\lambda_n})
       &\lesssim_{\gammav,L}\lambda_n^4e^{-V_{\lambda_n}^2+KV_{\lambda_n}^{4/3}}\\
       &\leq\frac1n\exp\left\{
          \frac{(\log n)\log\log\log n}{\log\log n}
          +K(\log n)^{2/3}\right\}.
\end{aligned}
\end{equation}
Here, $K>0$ depends only on $K_0$ and $L$, and the final expression equals $n^{-1+o(1)}$ since both terms inside the braces are $o(\log n)$. Also, since $\mathrm{TV}\leq\sqrt2\,\H$, \eqref{eq:misspecification-marginal-hellinger-rate} in Proposition~\ref{prop:misspecification-marginal} gives $\E_{G^*}[\mathrm{TV}(\widehat f_n^+,\widetilde f_{\lambda_n})]\lesssim_{\gammav,L}\sqrt{q_n}$. Now, for a fixed $0<\kappa<1/4$, we define the high-probability event
\begin{equation}\label{eq:misspecification-true-marginal-event}
    \mathcal E_n:=\mathcal T_n\cap
       \left\{\mathrm{TV}(\widehat f_n^+,f^*)+
          \mathrm{TV}(\widetilde f_{\lambda_n},f^*)\leq n^{-\kappa}\right\},
\end{equation}
where $\mathcal T_n$ is defined in \eqref{eq:misspecification-estimation-event}. Note that, by Pinsker's inequality and \eqref{eq:misspecification-true-projection-kl}, we have $\mathrm{TV}(\widetilde f_{\lambda_n},f^*)\leq\{\mathrm{KL}(f^*\parallel\widetilde f_{\lambda_n})/2\}^{1/2}\leq n^{-1/2+o(1)}\leq n^{-\kappa}/4$ for all sufficiently large $n$. Hence, \eqref{eq:misspecification-good-event} and the preceding expected bounds give
\begin{equation}\label{eq:misspecification-true-marginal-event-probability}
\begin{aligned}
    \P_{G^*}(\mathcal E_n^c)
       &\leq\P_{G^*}(\mathcal T_n^c)
          +\P_{G^*}\!\left\{\mathrm{TV}(\widehat f_n^+,\widetilde f_{\lambda_n})
             >\frac{n^{-\kappa}}2\right\}\\
       &\leq\P_{G^*}(\mathcal T_n^c)
          +2n^\kappa\E_{G^*}[\mathrm{TV}(\widehat f_n^+,\widetilde f_{\lambda_n})]\\
       &\lesssim_{\gammav,L}n^{-2}+n^{\kappa-1/2}(\log n)^2\\
       &\lesssim_{\gammav,L}n^{-\kappa},
\end{aligned}
\end{equation}
where the last inequality uses $n^{2\kappa-1/2}(\log n)^2\to0$ since $\kappa<1/4$. 

Now, in \eqref{eq:misspecification-true-fourier-split}, we choose $\tau_n=(h\log n)^{-1/2}$, where $h$ is a fixed constant satisfying $0<h<\kappa<1/4$. Since $\lambda_n^2/\tau_n^2=h\log\log n\to\infty$, we have $\tau_n<\lambda_n$ for all sufficiently large $n$. For either $H=\widehat Q_n$ from \eqref{eq:misspecification-floor-embedding} or $H=\widetilde H_{\lambda_n}$, \eqref{eq:misspecification-true-marginal-event} gives, on $\mathcal E_n$,
\begin{align*}
    \mathrm{TV}\bigl(p_{H,\sqrt{1+\lambda_n^2}},f^*\bigr)\leq n^{-\kappa}.
\end{align*}
By the Fourier upper bound on $g^*$ in \eqref{eq:misspecification-true-envelopes} in Lemma~\ref{lem:misspecification-density-regularity}, we can apply \eqref{eq:misspecification-smoothed-mixing-bound} in Lemma~\ref{lem:misspecification-smoothed-mixing} with $g=g^*$, $\lambda=\lambda_n$, $\delta=n^{-\kappa}$, and $U=1/\tau_n$. This gives, on $\mathcal E_n$,
\begin{align*}
    \|H\star\phi_{\tau_n}\|_{L_2}^2
       \lesssim_{\gammav,L}1+\lambda_n e^{\lambda_n^2/(2\tau_n^2)}
          +\frac{n^{-\kappa}}{\tau_n}e^{(1+\lambda_n^2)/(2\tau_n^2)}
       \lesssim_{\gammav,L}1.
\end{align*}
Indeed, the last two terms are $o(1)$, since $\sqrt{\rho_n}(\log n)^{h/2}=o(1)$ and $\sqrt{\log n}\,n^{-\kappa+h/2}(\log n)^{h/2}=o(1)$. Applying this bound to both mixing distributions gives
\begin{equation}\label{eq:misspecification-true-smoothed-mixing}
    \|\widehat Q_n\star\phi_{\tau_n}\|_{L_2}^2
       +\|\widetilde H_{\lambda_n}\star\phi_{\tau_n}\|_{L_2}^2
       \lesssim_{\gammav,L}1\quad\text{on }\mathcal E_n.
\end{equation}
By \eqref{eq:misspecification-true-smoothed-mixing}, on $\mathcal E_n$,
\begin{align*}
    \mathrm{(II)} =e^{-(\lambda_n^2-\tau_n^2)T^2}
          \|(\widehat Q_n-\widetilde H_{\lambda_n})
             \star\phi_{\tau_n}\|_{L_2}^2 \lesssim_{\gammav,L}
          e^{-(\lambda_n^2-\tau_n^2)T^2}.
\end{align*}
Take $T^2=\log(1/q_n)/(1+\lambda_n^2-\tau_n^2)$, where $q_n=(\log n)^4/n$ is defined in \eqref{eq:misspecification-marginal-rate} in Proposition~\ref{prop:misspecification-marginal}. Combining \eqref{eq:misspecification-true-I-bound} with the preceding bound for $\mathrm{(II)}$ and using \eqref{eq:misspecification-true-fourier-split} yields
\begin{align*}
    \E_{G^*}\!\left[
       \|\widehat g_n^+-\widetilde g_{\lambda_n}\|_{L_2}^2\1v_{\mathcal E_n}\right]
       &\leq\E_{G^*}[\mathrm{(I)}]
          +\E_{G^*}[\mathrm{(II)}\1v_{\mathcal E_n}] \lesssim_{\gammav,L}q_n^{\alpha_n},\qquad \alpha_n:=\frac{\lambda_n^2-\tau_n^2}{1+\lambda_n^2-\tau_n^2}.
\end{align*}
Note that, since $\lambda_n^2=(\log\log n)/\log n$ and $\tau_n^2=1/(h\log n)$, we have
\begin{align*}
    \alpha_n\log(1/q_n)
       &=\frac{(\log\log n-h^{-1})(\log n-4\log\log n)}
                    {\log n+\log\log n-h^{-1}}\\
       &=\log\log n-h^{-1}+o(1).
\end{align*}
Thus, $q_n^{\alpha_n}\lesssim1/\log n$. Finally, since $\hat c_0^+\geq\lambda_n$ and $\|\widehat g_n^+-\widetilde g_{\lambda_n}\|_{L_2}^2\lesssim\lambda_n^{-1}$, \eqref{eq:misspecification-true-marginal-event-probability} gives
\begin{align*}
    \E_{G^*}\!\left[
       \|\widehat g_n^+-\widetilde g_{\lambda_n}\|_{L_2}^2\1v_{\mathcal E_n^c}\right]
       \lesssim \lambda_n^{-1}\P_{G^*}(\mathcal E_n^c)
       \lesssim_{\gammav,L}\lambda_n^{-1}n^{-\kappa}
        =o(1/\log n).
\end{align*}
Combining the bounds on $\mathcal E_n$ and $\mathcal E_n^c$ proves \eqref{eq:misspecification-true-pseudoprior}.
\end{proof}

\appsubsubsection{Approximation error of the pseudo-true prior density}\label{app:proof-misspecification-approximation}
\begin{proof}[Proof of Proposition~\ref{prop:misspecification-true-approximation}]
Write $f_g=g\star\phi_1$. For every $U>0$, Plancherel's theorem gives
\begin{equation}\label{eq:misspecification-approximation-split}
\begin{aligned}
    \|\widetilde g_\lambda-g\|_{L_2}^2
       &=\frac1{2\pi}\int
          |\varphi_{\widetilde g_\lambda}(t)-\varphi_g(t)|^2\,\diff t\\
       &\leq\underbrace{e^{U^2}\|\widetilde f_\lambda-f_g\|_{L_2}^2}_{\mathrm{(I)}}\\
       &\quad+\underbrace{\frac1\pi\int_{|t|>U}
          |\varphi_{\widetilde H_\lambda}(t)|^2e^{-\lambda^2t^2}\,\diff t}_{\mathrm{(II)}}
          +\underbrace{\frac1\pi\int_{|t|>U}|\varphi_g(t)|^2\,\diff t}_{\mathrm{(III)}}.
\end{aligned}
\end{equation}
Here, we used $\varphi_{\widetilde g_\lambda}-\varphi_g=e^{t^2/2}(\varphi_{\widetilde f_\lambda}-\varphi_{f_g})$ on $|t|\leq U$, and $|a-b|^2\leq2|a|^2+2|b|^2$ on $|t|>U$.

We first bound $\mathrm{(I)}$. Since both marginal densities are uniformly bounded, the Hellinger bound by KL gives $\|\widetilde f_\lambda-f_g\|_{L_2}^2\lesssim\mathrm{KL}(f_g\parallel\widetilde f_\lambda)$. Recall from Lemma~\ref{lem:misspecification-kl-approximation} that $V_\lambda^2=2\log(1/\lambda)/\lambda^2$. By \eqref{eq:misspecification-true-projection-kl-bound} in Lemma~\ref{lem:misspecification-kl-approximation}, we have $\mathrm{KL}(f_g\parallel\widetilde f_\lambda)\lesssim_{K_0,L}\lambda^4e^{-V_\lambda^2+KV_\lambda^{4/3}}$ for some $K>0$ depending only on $K_0$ and $L$. Choose a sufficiently large fixed $B_0$ and set $U^2=V_\lambda^2-B_0V_\lambda^{4/3}$. Then, for all sufficiently small $\lambda$,
\begin{align*}
    \mathrm{(I)}\lesssim e^{U^2}\mathrm{KL}(f_g\parallel\widetilde f_\lambda)\lesssim_{K_0,L}\lambda^4,
    \qquad \lambda^2U^2=2\log(1/\lambda)-o(1).
\end{align*}

Next, we bound $\mathrm{(II)}$. Note that, by Pinsker's inequality, $\mathrm{TV}(\widetilde f_\lambda,f_g)\leq\{\mathrm{KL}(f_g\parallel\widetilde f_\lambda)/2\}^{1/2}$. Under the density regularity conditions on $g$, we have $|\varphi_g(t)|\lesssim_{K_0,L}(1+t^2)^{-1}$ by \eqref{eq:misspecification-true-envelopes} in Lemma~\ref{lem:misspecification-density-regularity}. Thus, we can apply \eqref{eq:misspecification-smoothed-mixing-bound} in Lemma~\ref{lem:misspecification-smoothed-mixing} with $\delta=\{\mathrm{KL}(f_g\parallel\widetilde f_\lambda)/2\}^{1/2}$, since $U\geq1/\lambda$ for all sufficiently small $\lambda$. This gives
\begin{align*}
    \|\widetilde H_\lambda\star\phi_{1/U}\|_{L_2}^2
       &\lesssim_{K_0,L}\left\{1+\lambda e^{\lambda^2U^2/2}
          +U\{\mathrm{KL}(f_g\parallel\widetilde f_\lambda)\}^{1/2}e^{(1+\lambda^2)U^2/2}\right\}
        \lesssim1,
\end{align*}
where the choice of $B_0$ ensures that $U\{\mathrm{KL}(f_g\parallel\widetilde f_\lambda)\}^{1/2}e^{(1+\lambda^2)U^2/2}=o(1)$, while $\lambda e^{\lambda^2U^2/2}\leq1$. Since $U\geq1/\lambda$, for $|t|>U$ we have
\begin{align*}
    e^{-\lambda^2t^2}
       &=e^{-(\lambda^2-U^{-2})t^2}e^{-t^2/U^2}\\
       &\leq e^{-(\lambda^2-U^{-2})U^2}e^{-t^2/U^2}
        =e^{1-\lambda^2U^2}e^{-t^2/U^2}.
\end{align*}
Hence, applying Plancherel's theorem and using the preceding $L_2$ bound give
\begin{align*}
    \mathrm{(II)} &\leq\frac{e^{1-\lambda^2U^2}}\pi
          \int|\varphi_{\widetilde H_\lambda}(t)|^2e^{-t^2/U^2}\,\diff t\\
       &=2e^{1-\lambda^2U^2}\|\widetilde H_\lambda\star\phi_{1/U}\|_{L_2}^2
        \lesssim_{K_0,L}e^{1-\lambda^2U^2}\lesssim\lambda^2.
\end{align*}

Finally, using the Fourier upper bound $|\varphi_g(t)|\lesssim_{K_0,L}(1+t^2)^{-1}$ again, we obtain
\begin{align*}
    \mathrm{(III)}\lesssim_{K_0,L}\int_{|t|>U}(1+t^2)^{-2}\,\diff t
       \lesssim U^{-3}=o(\lambda^2),
\end{align*}
since $U\asymp\lambda^{-1}\sqrt{\log(1/\lambda)}$. Combining the bounds for $\mathrm{(I)}$, $\mathrm{(II)}$, and $\mathrm{(III)}$ in \eqref{eq:misspecification-approximation-split} proves \eqref{eq:misspecification-true-projection-bias}. 
\end{proof}

\appsubsection{Proof of Theorem~\ref{thm:Opt-Marginal}}\label{app:proof of thm:Opt-Marginal}
\begin{proof}
Let
\[
A_{>} := \{(x,\theta)\in\Xc\times\Theta : \pi(\theta\mid x)>k^*\},
\qquad
A_{=} := \{(x,\theta)\in\Xc\times\Theta : \pi(\theta\mid x)=k^*\}.
\]
We first claim that
\[
\P_G(\pi(\theta \mid X)>k^*) \le 1-\beta \le \P_G(\pi(\theta \mid X)\ge k^*).
\]
Note that $\pi(\theta \mid X) > 0$ almost surely under $\P_{G}$, so $\P_{G}(\pi(\theta \mid X) \ge t) \to 1$ as $t \downarrow 0$. The set in \eqref{eq:Marg-Const} contains some $t > 0$, and thus we have $k^* > 0$ because $0 < \beta < 1$. By definition of $k^*$, there exists a sequence $k_m \uparrow k^*$ such that
\[
\P_G(\pi(\theta\mid X)\ge k_m)\ge 1-\beta
\qquad \text{for all }m.
\]
Since
\[
\{\pi(\theta\mid X)\ge k^*\}=\bigcap_{m=1}^{\infty}\{\pi(\theta\mid X)\ge k_m\},
\]
continuity from above yields
\[
\P_G(\pi(\theta \mid X)\ge k^*)
=
\lim_{m\to\infty}\P_G(\pi(\theta\mid X)\ge k_m)
\ge 1-\beta.
\]
Next,
\[
\{\pi(\theta\mid X)>k^*\}=\bigcup_{m=1}^{\infty}\{\pi(\theta\mid X)\ge k^*+1/m\}.
\]
If $\P_G(\pi(\theta\mid X)>k^*)>1-\beta$, then by continuity from below there would exist some $m$ such that
\[
\P_G(\pi(\theta\mid X)\ge k^*+1/m)>1-\beta,
\]
contradicting the definition of $k^*$ as a supremum. Hence
\[
\P_G(\pi(\theta \mid X)>k^*) \le 1-\beta \le \P_G(\pi(\theta \mid X)\ge k^*).
\]

Now let
\[
\mu(dx,d\theta):=p_G(x)\,\diff\theta\,\diff x,
\qquad
\nu(dx,d\theta):=p(x,\theta)\,\diff\theta\,\diff x
= \pi(\theta\mid x)\,\mu(dx,d\theta).
\]
Since \(\nu(A_{>}) \le 1-\beta \le \nu(A_{>} \cup A_{=})\), we have
\[
0 \le 1-\beta-\nu(A_{>}) \le \nu(A_{=}).
\]
On \(A_{=}\), we have \(\diff\nu = k^*\,\diff\mu\), so \(\nu\) restricted to \(A_{=}\) is atomless because \(\mu\) is atomless. Therefore, by the standard splitting property of atomless measures, there exists a measurable set \(A^* \subseteq A_{=}\) such that
\[
\nu(A^*) = 1-\beta-\nu(A_{>}).
\]
Now define
\[
\widetilde A := A_{>} \cup A^*.
\]
Then $\nu(\widetilde A)=1-\beta$, and the induced set-valued rule is exactly
\[
\Ic^*(x) := \{\theta \in \Theta: \pi(\theta \mid x) > k^*\}\cup \{\theta\in\Theta:(x,\theta)\in A^*\}.
\]

It remains to prove optimality. Let $A \subseteq \Xc\times\Theta$ be any measurable set such that
\[
\nu(A)\ge 1-\beta=\nu(\widetilde A).
\]
Then
\[
\nu(A\setminus \widetilde A)\ge \nu(\widetilde A\setminus A).
\]
On $A\setminus \widetilde A$, we have $\pi(\theta\mid x)\le k^*$, because all points with $\pi(\theta\mid x)>k^*$ belong to $A_{>}\subseteq \widetilde A$, and only a subset of the boundary $\{\pi(\theta\mid x)=k^*\}$ is retained in $\widetilde A$. Hence
\[
\nu(A\setminus \widetilde A)
=
\int_{A\setminus \widetilde A} \pi(\theta\mid x)\,\diff \mu
\le
k^*\,\mu(A\setminus \widetilde A).
\]
Similarly, on $\widetilde A\setminus A$, we have $\pi(\theta\mid x)\ge k^*$, so
\[
\nu(\widetilde A\setminus A)
=
\int_{\widetilde A\setminus A} \pi(\theta\mid x)\,\diff\mu
\ge
k^*\,\mu(\widetilde A\setminus A).
\]
Combining the last three displays gives
\[
\mu(A\setminus \widetilde A)\ge \mu(\widetilde A\setminus A).
\]
Therefore
\[
\mu(A)
=
\mu(A\cap \widetilde A)+\mu(A\setminus \widetilde A)
\ge
\mu(A\cap \widetilde A)+\mu(\widetilde A\setminus A)
=
\mu(\widetilde A).
\]
Thus $\widetilde A$ minimizes $\mu(A)$ subject to $\nu(A)\ge 1-\beta$, so $\Ic^*$ solves~\eqref{eq:Opt-Marg-Set}. Finally, since $\nu(\widetilde A)=1-\beta$, we also have
\[
\P_G(\theta\in \Ic^*(X)) = 1-\beta.
\]
\end{proof}

\appsubsection{Proof of Theorem~\ref{thm:estimated-c-coverage}}\label{app:proof-estimated-c-coverage}
\begin{proof}[Proof of Theorem~\ref{thm:estimated-c-coverage}]
Conditional on $\Dc_n$, the rule $\hat\Ic_n$ has coverage $1-\beta$ under the fitted joint density $\widehat g_n^+(\theta)\phi(x-\theta)$ by \eqref{eq:Est-CI}. The true joint density is $g^*(\theta)\phi(x-\theta)$, and the total variation distance between these two joint densities equals $\mathrm{TV}(\widehat g_n^+,g^*)$. Thus,
\begin{align}\label{eq:coverage-tv-bound}
    \left|\P_{G^*}\{\theta\in\hat\Ic_n(X)\mid\Dc_n\}-(1-\beta)\right|
    &=\left|\iint\1v\{\theta\in\hat\Ic_n(x)\}\{g^*(\theta)-\widehat g_n^+(\theta)\}\phi(x-\theta)\,\diff\theta\,\diff x\right|\notag\\
    &\leq\mathrm{TV}(\widehat g_n^+,g^*).
\end{align}
By Lemma~\ref{lem:l2-to-tv} and \eqref{eq:moment-deconvolution-upper},
\begin{align*}
    \E_{G^*}[\mathrm{TV}(\widehat g_n^+,g^*)]
    &\leq K(\log n)^{1/(2r)}n^{-\alpha_0/2}\exp\{-k_1(\log n)^{r/2}/2\}+Kn^{-2}\\
    &\leq K n^{-\alpha_0/2}\exp\{-k(\log n)^{r/2}\}
\end{align*}
for some $k>0$, since $\log\log n=o\{(\log n)^{r/2}\}$. Taking expectations in \eqref{eq:coverage-tv-bound} proves \eqref{eq:estimated-c-coverage-rate}.
\end{proof}

\appsubsection{Proof of Theorem~\ref{thm:estimated-c-length}}\label{app:proof-estimated-c-length}
\begin{proof}[Proof of Theorem~\ref{thm:estimated-c-length}]
Throughout the proof, $K>0$ denotes a constant that depends only on the parameters in the theorem and may vary from line to line. Let $U_n=\sup_{u\in\Kc}|\widehat C_n(u)-C^*(u)|$. We first prove \eqref{eq:estimated-c-threshold-rate}. We claim that
\begin{align}\label{eq:estimated-c-threshold-claim}
    C_2\min\{|\hat k_n-k^*|,\delta_0\}\leq U_n.
\end{align}
If $|\hat k_n-k^*|<\delta_0$, then $\hat k_n\in\Kc$. By (C2) in \eqref{eq:excess length condition} and the calibration in \eqref{eq:Est-CI},
\begin{align*}
    C_2|\hat k_n-k^*|
    &\leq|C^*(\hat k_n)-C^*(k^*)|
     =|C^*(\hat k_n)-\widehat C_n(\hat k_n)|\leq U_n.
\end{align*}
If $\hat k_n\geq k^*+\delta_0$, monotonicity of $\widehat C_n$ implies
\begin{align*}
    U_n&\geq\widehat C_n(k^*+\delta_0)-C^*(k^*+\delta_0)\\
       &\geq C^*(k^*)-C^*(k^*+\delta_0)\geq C_2\delta_0.
\end{align*}
The case $\hat k_n\leq k^*-\delta_0$ follows similarly. This proves \eqref{eq:estimated-c-threshold-claim}.

Since $\czero>0$ and $\rho_n\to0$, the stabilized estimator $\hat c_0^+=\hat c_0\vee\sqrt{\rho_n}$ in \eqref{eq:deconvolution-floor} coincides with $\hat c_0$ on the event $\Sc_n$ defined in \eqref{eq:modified-c-good-event} for all sufficiently large $n$. Using \eqref{eq:one-sided-control}, we obtain, for all sufficiently large $n$,
\begin{align*}
    \P_{G^*}({\Sc_n^c})\lesssim n^{-2}, \qquad 
    \hat c_0^+=\hat c_0\geq\czero/\sqrt2\quad\text{on }{\Sc_n}.
\end{align*}
Also, note that, under the fitted model, $\theta\mid X,\xi$ is Gaussian with variance $\widehat\alpha_0^+={(\hat c_0^+)^2/\{1+(\hat c_0^+)^2\}}$. Thus, $\widehat\pi_n^+(\theta\mid x)\leq(2\pi\widehat\alpha_0^+)^{-1/2}$ for all $\theta,x\in\Real$. Similarly, \eqref{eq:full conditional theta} gives $\pi^*(\theta\mid x)\leq(2\pi\alpha_0)^{-1/2}$. Since $\hat c_0^+\geq\czero/\sqrt2$ on $\Sc_n$, this implies $0<\hat k_n,k^*\leq K$ and hence $|\hat k_n-k^*|\leq K$ on $\Sc_n$. By Lemma~\ref{lem:estimated-c-Un} below, $\E_{G^*}[U_n]\leq K\zeta_n^{1/2}$, where $\zeta_n$ is defined in \eqref{eq:coverage-rate}. Then \eqref{eq:estimated-c-threshold-claim} and Markov's inequality give
\begin{align*}
    \E_{G^*}\{|\hat k_n-k^*|\1v_{{\Sc_n}}\}
    &\leq C_2^{-1}\E_{G^*}[U_n]+K\P_{G^*}(U_n\geq C_2\delta_0)\\
    &\leq K\E_{G^*}[U_n]\leq K\zeta_n^{1/2}.
\end{align*}
On ${\Sc_n^c}$, the floor ${\hat c_0^+}\geq\sqrt{\rho_n}$ gives $|\hat k_n-k^*|\leq K\rho_n^{-1/2}$ for all sufficiently large $n$. Thus,
\begin{align*}
    \E_{G^*}\{|\hat k_n-k^*|\1v_{{\Sc_n^c}}\}
    &\leq K n^{-2}\rho_n^{-1/2}=o(\zeta_n^{1/2}).
\end{align*}
This proves \eqref{eq:estimated-c-threshold-rate}.

We next prove \eqref{eq:estimated-c-length-rate}. Since $\bigl||\hat\Ic_n(x)|-|\Ic^*(x)|\bigr|\leq|\hat\Ic_n(x)\Delta\Ic^*(x)|$, it suffices to bound the conditional expected symmetric difference. Put $\widetilde\Ic_n^*(x)=\{\theta:\pi^*(\theta\mid x)\geq\hat k_n\}$. Then
\begin{align*}
    \E_{G^*}\{|\hat\Ic_n(X)\Delta\Ic^*(X)|\mid\Dc_n\}
    &\leq\iint\1v\{\theta\in\hat\Ic_n(x)\Delta\widetilde\Ic_n^*(x)\}f^*(x)\,\diff\theta\,\diff x\\
    &\quad+\iint\1v\{\theta\in\widetilde\Ic_n^*(x)\Delta\Ic^*(x)\}f^*(x)\,\diff\theta\,\diff x\\
    &:=S_{1n}+S_{2n}.
\end{align*}
For $u\in\Kc$ and $t\in(0,t_1]$, we have $\pi^*(\theta\mid x)\geq u-t\geq k^*-\delta_0-t_1>0$ whenever $|\pi^*(\theta\mid x)-u|\leq t$. Thus, (C1) in \eqref{eq:excess length condition} gives
\begin{align}\label{eq:estimated-c-length-band-bound}
    \iint\1v\{|\pi^*(\theta\mid x)-u|\leq t\}f^*(x)\,\diff\theta\,\diff x
    \leq\frac{\P_{G^*}\{|\pi^*(\theta\mid X)-u|\leq t\}}{u-t}
       \leq\frac{C_1t}{k^*-\delta_0-t_1}.
\end{align}
On $F_n=\{|\hat k_n-k^*|\leq\delta_0\}$, we have $\hat k_n\in\Kc$. If the indicators defining $\hat\Ic_n(x)$ and $\widetilde\Ic_n^*(x)$ differ, then either $|\pi^*(\theta\mid x)-\hat k_n|\leq t$ or $|\widehat\pi_n^+(\theta\mid x)-\pi^*(\theta\mid x)|>t$. By \eqref{eq:estimated-c-length-band-bound}, for $t\in(0,t_1]$,
\begin{align*}
    S_{1n}&\leq \frac{C_1t}{k^*-\delta_0-t_1}+\frac2t\mathrm{wTV}(\widehat\pi_n^+,\pi^*).
\end{align*}
Since $|\hat k_n-k^*|\leq\delta_0\leq t_1$, \eqref{eq:estimated-c-length-band-bound} also gives $S_{2n}\leq C_1|\hat k_n-k^*|/(k^*-\delta_0-t_1)$. Therefore, on $F_n$,
\begin{align}\label{eq:set-boundary-bound}
    \E_{G^*}\{|\hat\Ic_n(X)\Delta\Ic^*(X)|\mid\Dc_n\}
    &\leq\frac{C_1}{k^*-\delta_0-t_1}\{t+|\hat k_n-k^*|\}+\frac2t\mathrm{wTV}(\widehat\pi_n^+,\pi^*).
\end{align}
Now, choose $t=\zeta_n^{1/2}\leq t_1$ for all sufficiently large $n$. Theorem~\ref{thm:estimated-c-posterior} and Markov's inequality give
$\mathrm{wTV}(\widehat\pi_n^+,\pi^*)=O_P(\zeta_n)$, while \eqref{eq:estimated-c-threshold-rate} gives $|\hat k_n-k^*|=O_P(\zeta_n^{1/2})$ and $\P_{G^*}(F_n^c)\to0$. This completes the proof.
\end{proof}

\begin{lemma}[Uniform coverage curve approximation]\label{lem:estimated-c-Un}
Under the conditions of Theorem~\ref{thm:estimated-c-posterior} and (C1) in \eqref{eq:excess length condition}, let $U_n=\sup_{u\in\Kc}|\widehat C_n(u)-C^*(u)|$. There is a constant $K>0$ such that, for all sufficiently large $n$,
\begin{align}\label{eq:estimated-c-Un}
    \E_{G^*}[U_n]\leq K\zeta_n^{1/2}.
\end{align}
\end{lemma}
\begin{proof}
For $u\in\Kc$, let $A_u=\{(\theta,x):\widehat\pi_n^+(\theta\mid x)\geq u\}$ and $B_u=\{(\theta,x):\pi^*(\theta\mid x)\geq u\}$. By \eqref{eq:Est-CI} and the definition of $C^*$ in \eqref{eq:excess length condition},
\begin{align*}
    |\widehat C_n(u)-C^*(u)|
    &\leq\left|\iint_{A_u}\{\widehat g_n^+(\theta)-g^*(\theta)\}\phi(x-\theta)\,\diff\theta\,\diff x\right|
       +\iint_{A_u\triangle B_u}g^*(\theta)\phi(x-\theta)\,\diff\theta\,\diff x\\
    &\leq\mathrm{TV}(\widehat g_n^+,g^*)
       +\iint_{A_u\triangle B_u}g^*(\theta)\phi(x-\theta)\,\diff\theta\,\diff x.
\end{align*}
If $(\theta,x)\in A_u\triangle B_u$ and $|\widehat\pi_n^+(\theta\mid x)-\pi^*(\theta\mid x)|\leq t$, then $|\pi^*(\theta\mid x)-u|\leq t$. By \eqref{eq:full conditional theta}, the oracle posterior $\pi^*$ is a Gaussian location mixture with component variance $\alpha_0$, so $\sup_{\theta,x}\pi^*(\theta\mid x)\leq(2\pi\alpha_0)^{-1/2}$. Using this bound, (C1) in \eqref{eq:excess length condition}, and Markov's inequality, we obtain, for $t\in(0,t_1]$,
\begin{align}\label{eq:estimated-c-coverage-curve-bound}
    |\widehat C_n(u)-C^*(u)|
    &\leq\mathrm{TV}(\widehat g_n^+,g^*)+C_1t
       +\frac1t\iint|\widehat\pi_n^+(\theta\mid x)-\pi^*(\theta\mid x)|\pi^*(\theta\mid x)f^*(x)\,\diff\theta\,\diff x\notag\\
    &\leq\mathrm{TV}(\widehat g_n^+,g^*)+C_1t
       +\frac{2}{\sqrt{2\pi\alpha_0}\,t}\mathrm{wTV}(\widehat\pi_n^+,\pi^*).
\end{align}
The right side is uniform over $u\in\Kc$. By \eqref{eq:wTV and TV}, Lemma~\ref{lem:l2-to-tv}, and \eqref{eq:moment-deconvolution-upper},
\begin{align}\label{eq:coverage-input-rates}
    \E_{G^*}[\mathrm{TV}(\widehat g_n^+,g^*)]
    +\E_{G^*}[\mathrm{wTV}(\widehat\pi_n^+,\pi^*)]\leq K\zeta_n.
\end{align}
Take $t=\zeta_n^{1/2}$, which belongs to $(0,t_1]$ for all sufficiently large $n$. Taking expectations in \eqref{eq:estimated-c-coverage-curve-bound} gives \eqref{eq:estimated-c-Un}.
\end{proof}

\appsubsection{Proof of Proposition~\ref{prop:upper confidence bound neighborhood}}\label{app:proof of prop:upper confidence bound neighborhood}
\begin{proof}
(i) Suppose that $\eta_n < 1/2$ and $\tilde{F} = H \star N(0, \sigma^2)$ satisfies $\KS(\F_n, \tilde{F}) \le \eta_n$. Note that the density of $\tilde{F}$ is bounded by $(\sqrt{2\pi}\sigma)^{-1}$. Also, for $X_{(1)} = \min_i X_i$ and $X_{(n)} = \max_i X_i$, we have $X_{(1)} < X_{(n)}$ almost surely and 
    \begin{align*}
        \tilde{F}(X_{(1)}) \le \eta_n \quad \text{and} \quad \tilde{F}(X_{(n)}) \ge 1-\eta_n
    \end{align*}
    because $\tilde{F}$ is continuous. This implies that
    \begin{align*}
        1- 2\eta_n \le \tilde{F}(X_{(n)}) - \tilde{F}(X_{(1)}) \le \frac{1}{\sqrt{2\pi}\sigma}(X_{(n)} - X_{(1)}).
    \end{align*}
    Therefore, it holds that
    \begin{align*}
        \sigma \le \frac{X_{(n)}-X_{(1)}}{\sqrt{2\pi}(1-2\eta_n)}.
    \end{align*}
    Since for every $\tilde{F}$ such that $\KS(\F_n, \tilde{F}) \le \eta_n$ satisfies the above condition on $\sigma$, we have
    \begin{align*}
        \sigma_0(\tilde{F}) \le \frac{X_{(n)} -X_{(1)}}{\sqrt{2\pi}(1-2\eta_n)}.
    \end{align*}
    Taking the supremum over all such $\tilde{F}$ gives 
    \begin{align*}
        \sigma_0(\F_n; \eta_n) \le \frac{X_{(n)} -X_{(1)}}{\sqrt{2\pi}(1-2\eta_n)} < \infty.
    \end{align*}
    This proves finiteness of $\sigma_0(\F_n;\eta_n)$ if $\eta_n < 1/2$. (ii) is immediate from the DKW inequality as mentioned right above the proposition.
\end{proof}

\appsubsection{Proof of Proposition~\ref{prop:scale-minimax-lower}}\label{app:proof-scale-minimax-lower}
\begin{proof}[Proof of Proposition~\ref{prop:scale-minimax-lower}]
Let $H_0=\mathrm{Laplace}(0,1)$ have density  $g(x)=e^{-|x|}/2$. Let $H_{1,n}$ have density $k_{d_n}$, where $k_d$ is the probability density constructed in Lemma~\ref{lem:ks-only-laplace-approximation} below and  $d_n:=\kappa\rho_n \to 0$ for a fixed $0<\kappa\leq1/12$. Define
\begin{align*}
    G_{0,n}:=H_0,\qquad G_{1,n}:=H_{1,n}\star N(0,d_n).
\end{align*}

We first verify that both $G_{0,n}$ and $G_{1,n}$ belong to $\mathcal G$. For every $p\geq1$, the Laplace distribution $H_0$ satisfies $\left\{\int|\xi|^p\,\diff H_0(\xi)\right\}^{1/p}=\{\Gamma(p+1)\}^{1/p} \le p$. By \eqref{eq:ks-only-kd-bound} in the proof of Lemma~\ref{lem:ks-only-laplace-approximation}, we have $g/2\leq k_d\leq2g$ for all sufficiently small $d$, which implies
\begin{align*}
    \left\{\int|\xi|^p\,\diff H_{1,n}(\xi)\right\}^{1/p}
    &\leq\{2\Gamma(p+1)\}^{1/p} \leq2p.
\end{align*}
Moreover, \eqref{eq:chf-kd} in the proof of Lemma~\ref{lem:ks-only-laplace-approximation} gives
\begin{align*}
    \varphi_{H_{1,n}}(t)
    \geq \frac{1}{1+t^2}
    \geq \frac{1}{2(1+|t|)^2}, \qquad t\in\Real,
\end{align*}
for all sufficiently large $n$. This Fourier lower bound also holds for $H_0$, since $\varphi_{H_0}(t)=(1+t^2)^{-1}$. Consequently, $H_0,H_{1,n}\in\mathcal H_{\gammav}$ with  $\gammav=(K_0,r,a,\vartheta,t_0)=(2,1,1/2,2,1)$ in Assumption~\ref{ass:moment-fourier} $(i)$.  Note that the Fourier lower bound rules out a positive Gaussian convolution factor in either $H_0$ or $H_{1,n}$. Hence, we have $\czero^2(G_{0,n})=0$ and $\czero^2(G_{1,n})=d_n$. Indeed, $G_{1,n}=H_{1,n}\star N(0,d_n)=H\star N(0,c^2)$ with $c^2>d_n$ would imply 
\begin{align*}
    |\varphi_{H_{1,n}}(t)| =|\varphi_H(t)|e^{-(c^2-d_n)t^2/2} \leq e^{-(c^2-d_n)t^2/2},
\end{align*}
contradicting its polynomial Fourier lower bound. Since $d_n\leq1$ for all sufficiently large $n$, both alternatives belong to $\mathcal G$.

Now, let $f_{j,n}:=p_{G_{j,n},1} = G_{j,n} \star N(0,1)$ and $P_{j,n}:=f_{j,n}^{\otimes n}$, $j=0,1$. Their observation densities are
\begin{align*}
    f_{0,n}=g\star\phi_1,\qquad f_{1,n}=k_{d_n}\star\phi_{\sqrt{1+d_n}}.
\end{align*}
Then Lemma~\ref{lem:ks-only-laplace-approximation} yields
\begin{align*}
    \|f_{0,n}-f_{1,n}\|_{L_1}\leq Kd_n\exp\left\{-\frac{\log(1/d_n)}{3d_n}\right\},
\end{align*}
where $K>0$ is an absolute constant. Since $\log(1/d_n)\geq \log\log n/2$ for all sufficiently large $n$ and $\kappa\leq1/12$, we have
\begin{align*}
    \|f_{0,n}-f_{1,n}\|_{L_1} &\leq Kd_n n^{-1/(6\kappa)} \leq Kd_n n^{-2}.
\end{align*}
Therefore, it holds that
\begin{align*}
    \mathrm{TV}(P_{0,n},P_{1,n}) \leq\frac n2\|f_{0,n}-f_{1,n}\|_{L_1} \longrightarrow0.
\end{align*}

For any estimator $\widetilde c_n$, consider the test $\psi_n:=\1v\{\widetilde c_n^2\geq d_n/2\}$. Then
\begin{align*}
    \max_{j=0,1} \P_{G_{j,n}}\!\left( |\widetilde c_n^2-\czero^2(G_{j,n})| \geq d_n/2 \right)
    &\geq\frac12 \{P_{0,n}(\psi_n=1)+P_{1,n}(\psi_n=0)\}\\
    &\geq\frac12 \{1-\mathrm{TV}(P_{0,n},P_{1,n})\} \geq\frac14
\end{align*}
for all sufficiently large $n$. Since both $G_{0,n}$ and $G_{1,n}$ belong to $\mathcal G$, this proves
\begin{align}\label{eq:scale-minimax-probability}
    \inf_{\widetilde c_n}\sup_{G\in\mathcal G} \P_G\!\left( |\widetilde c_n^2-\czero^2(G)| \geq\frac{\kappa}{2}\rho_n
    \right) \geq\frac14.
\end{align}
Finally,
\begin{align*}
    \inf_{\widetilde c_n}\sup_{G\in\mathcal G} \E_G\!\left[ |\widetilde c_n^2-\czero^2(G)| \right]
    &\geq\frac{d_n}{2} \inf_{\widetilde c_n}\sup_{G\in\mathcal G}\P_G\!\left(|\widetilde c_n^2-\czero^2(G)| \geq d_n/2 \right)\\
    &\geq\frac{d_n}{8} =\frac{\kappa}{8}\rho_n \gtrsim\rho_n.
\end{align*}
This proves \eqref{eq:scale-minimax-lower}.
\end{proof}

\begin{lemma}[Approximation of the marginal density under a Laplace prior]
\label{lem:ks-only-laplace-approximation}
Let $g(x)=e^{-|x|}/2$. There exist constants $d_0\in(0,1)$ and $K>0$ such that, for every $d\in(0,d_0)$, there is a probability density $k_d$ satisfying
\begin{align}\label{eq:ks-only-laplace-approximation}
    \bigl\|g\star\phi_1-k_d\star\phi_{\sqrt{1+d}}\bigr\|_{L_1}
    \leq Kd e^{-T^2/3}, \qquad T^2=\frac{\log(1/d)}{d}.
\end{align}
\end{lemma}

\begin{proof}
Note that the Laplace density $g$ satisfies the assumptions of Lemma~\ref{lem:misspecification-kl-approximation}. Indeed, $g$ is strictly positive, $\log g$ is locally absolutely continuous, its slope has total variation $2$, and $\{\int |x|^p g(x)\,\diff x\}^{1/p}=\{\Gamma(p+1)\}^{1/p}\leq p$ for every $p\geq1$. 

We use the arguments in Lemma~\ref{lem:misspecification-kl-approximation} to prove \eqref{eq:ks-only-laplace-approximation}. Set $\lambda=\sqrt d$ and take $k_d=k_{\lambda,g}$ as constructed in the proof of that lemma. By \eqref{eq:misspecification-candidate-relative-bound} in the proof of Lemma~\ref{lem:misspecification-kl-approximation}, $\sup_x |k_d(x)-g(x)|/g(x)\to0$ as $d\downarrow0$. Thus, for all sufficiently small $d$,
\begin{align}\label{eq:ks-only-kd-bound}
    \frac12 g(x)\leq k_d(x)\leq 2g(x), \qquad x\in\Real.
\end{align}
Since $\chi_{V_\lambda}(t)\geq0$ for real $t$, \eqref{eq:misspecification-chf-k} gives
\begin{align}\label{eq:chf-kd}
    \varphi_{k_d}(t)\geq\varphi_g(t)=\frac{1}{1+t^2},
    \qquad t\in\Real.
\end{align}

Finally, $V_\lambda^2=T^2$. Integrating \eqref{eq:misspecification-candidate-marginal-bound} and using $\int g=1$ gives
\begin{align*}
    \bigl\|g\star\phi_1-k_d\star\phi_{\sqrt{1+d}}\bigr\|_{L_1}
    &\leq Kd\exp\{-T^2/2+CT^{4/3}\}\\
    &\leq Kd e^{-T^2/3}
\end{align*}
for all sufficiently small $d$, since $T\to\infty$ as $d\downarrow0$.
\end{proof}

\appsubsection{Proof of Theorem~\ref{thm:ks-only-density}}\label{app:ks-only-lower-bound}
For $n\geq3$, define
\begin{align}\label{eq:ks-only-prior-sequence}
    G_n^*:=(1-2\varepsilon_n)\mathrm{Laplace}(0,1) +\varepsilon_n\delta_0 +\frac{\varepsilon_n}{2} (\delta_{R_n}+\delta_{-R_n}),
\end{align}
where
\begin{align*}
    \varepsilon_n:=\kappa\eta_n\log n,\qquad R_n:=4\log(1/\varepsilon_n),
\end{align*}
with $\eta_n = \sqrt{\log n /n}$ defined in \eqref{eq:eta-radius} and a sufficiently small constant $\kappa > 0$ chosen so that $\varepsilon_n\leq1/4$ for all $n\geq3$. Note that $\czero(G_n^*)=0$. We will first verify that (i) $G_n^*$ satisfies Assumption~\ref{ass:moment-fourier} $(i)$ for all $n\geq3$, with constants independent of $n$. We then show that (ii) the Laplace component in $G_n^*$ allows a marginal mixture with component variance $1+1/\log n$ to lie within the KS neighborhood, while (iii) the distant atoms at $\pm R_n$ in $G_n^*$ force a squared Hellinger error of order at least $\varepsilon_n/(\log n)^2\asymp(n\log n)^{-1/2}$ at this larger variance. Throughout the proof, $K>0$ denotes a constant that may change from line to line.

\begin{proof}
\noindent\emph{(i) Verification of Assumption~\ref{ass:moment-fourier} (i).}
Writing $z=\log(1/\varepsilon_n)/p$, we obtain
\begin{align*}
    \varepsilon_n^{1/p}R_n
    =4pze^{-z}\leq\frac{4p}{e},\qquad p\geq1.
\end{align*}
Thus, since the $\mathrm{Laplace}(0,1)$ distribution has $p$th absolute moment $\Gamma(p+1)\leq(2p)^p$ for $p\geq1$, we have
\begin{align*}
    \left\{\int|\xi|^p\,\diff G_n^*(\xi)\right\}^{1/p} &=\left\{(1-2\varepsilon_n)\Gamma(p+1)+\varepsilon_n R_n^p\right\}^{1/p}\\
    &\leq \Gamma(p+1)^{1/p}+\varepsilon_n^{1/p}R_n \leq \left(2+\frac4e\right)p\leq4p.
\end{align*}
Thus, \eqref{eq:simultaneous-moment-control} holds with $r=1$ and $K_0=4$, uniformly over $n\geq3$. Further,
\begin{align*}
    \varphi_{G_n^*}(t) =\frac{1-2\varepsilon_n}{1+t^2} +\varepsilon_n\{1+\cos(R_nt)\} \geq\frac1{2(1+|t|)^2}
\end{align*}
for all $n\geq3$, verifying \eqref{eq:fourier-lower-bound} with $a=1/2$, $\vartheta=2$, and $t_0=1$.

\noindent\emph{(ii) Feasibility of a larger-variance mixture in the KS neighborhood.}
Let $F_n^*=G_n^*\star N(0,1)$ have density $f_n^*$ and let $g(x)=e^{-|x|}/2$. Applying Lemma~\ref{lem:ks-only-laplace-approximation} with $d=d_n := 1/\log n$, we obtain a probability density $k_{d_n}$ such that, with $T^2=(\log n)\log\log n$,
\begin{align*}
    \bigl\|g\star\phi_1-k_{d_n}\star\phi_{\sqrt{1+d_n}}\bigr\|_{L_1}
    &\leq Kd_n\exp\left\{-\frac13(\log n)\log\log n\right\}\\
    &=Kd_n n^{-(\log\log n)/3}=o(\eta_n).
\end{align*}
The last equality follows from $\eta_n=\sqrt{\log n/n}$ and $\log\log n\to\infty$. Thus, replacing the Laplace density $g$ by $k_{d_n}$ allows us to increase the Gaussian component variance from $1$ to $1+d_n$, while approximating the Laplace marginal density $g\star\phi_1$ within $o(\eta_n)$ in $L_1$ distance and thus also in KS distance. 

Now, let $Q_n$ be the distribution with density
\begin{align*}
    q_n(x) &:=(1-2\varepsilon_n) (k_{d_n}\star\phi_{\sqrt{1+d_n}})(x) +\varepsilon_n\phi_{\sqrt{1+d_n}}(x)\\
    &\quad+\frac{\varepsilon_n}{2} \left\{\phi_{\sqrt{1+d_n}}(x-R_n) +\phi_{\sqrt{1+d_n}}(x+R_n)\right\}.
\end{align*}
That is, $Q_n$ is obtained by replacing the Laplace component in \eqref{eq:ks-only-prior-sequence} with the distribution having density $k_{d_n}$ and convolving the resulting distribution with $N(0,1+d_n)$. Note that the KS distance between $N(\mu,1)$ and $N(\mu,1+d_n)$ is at most $Kd_n$, uniformly in $\mu$. Thus, applying the triangle inequality gives
\begin{align*}
    \KS(Q_n,F_n^*) &\leq (1-2\varepsilon_n) \bigl\|k_{d_n}\star\phi_{\sqrt{1+d_n}} -g\star\phi_1\bigr\|_{L_1}\\
    &\quad+2\varepsilon_n\sup_{\mu\in\Real} \KS\{N(\mu,1+d_n),N(\mu,1)\}\\
    &\leq o(\eta_n)+K\varepsilon_n d_n =o(\eta_n)+K\kappa\eta_n \leq\eta_n/2.
\end{align*}
Here, the three atomic components $\delta_0$, $\delta_{R_n}$, and $\delta_{-R_n}$ have total weight $2\varepsilon_n$, and $\varepsilon_n d_n=\kappa\eta_n$ by their definitions. Choose $\kappa$ sufficiently small that $K\kappa\leq1/4$. For all sufficiently large $n$, the $o(\eta_n)$ term is also at most
$\eta_n/4$, proving the final inequality. Note that, by the DKW inequality,
\begin{align*}
    \P_{G_n^*}\{\KS(\F_n,F_n^*)\leq\eta_n/2\} \geq1-2n^{-1/2}\longrightarrow1.
\end{align*}
On this event $Q_n$ is feasible in \eqref{eq:envelope}, which implies $\sigma_0^2(\F_n;\eta_n)\geq1+d_n$. By \eqref{eq:estimated-KS-c}, this implies $\hat c_{\mathrm{KS}}^2\geq d_n$ on the same event. Hence,
\begin{align}\label{eq:ks-only-oversmoothing-event}
    \P_{G_n^*}(\hat c_{\mathrm{KS}}^2\geq d_n)\longrightarrow1.
\end{align}

\noindent\emph{(iii) Hellinger lower bound.}
Now, we will obtain a lower bound on the Hellinger distance between $f_n^* = G_n^* \star N(0,1)$ and Gaussian location mixtures with component variance at least $1+d_n$. For this, we will construct a function $h_d(x-R_n)$ whose expectation is nonpositive under every Gaussian location mixture with component variance at least $1+d$. Let
\begin{align*}
    a_d=\frac{1+d}{2+d},\qquad h_d(y)=(a_d-y^2)e^{-y^2/2}, \qquad d > 0.
\end{align*}
Note that, if $Y\sim N(\mu,s^2)$ and $s^2 \ge 1+d$, then for every $\mu \in \Real$,
\begin{align*}
    \E [h_d(Y)] =\E [e^{-Y^2/2}]\left\{a_d-\frac{s^2}{1+s^2}-\frac{\mu^2}{(1+s^2)^2}\right\}\leq0.
\end{align*}
This implies that $\int h_d(x-R_n)q(x)\,\diff x\leq0$ for every Gaussian location mixture $q$ with component variance at least $1+d$. In particular, since the Gaussian component variance of $\widehat f_n^{\mathrm{KS}}$ is $1+\hat c_{\mathrm{KS}}^2$, we have
\begin{align}\label{eq:ks-only-fitted-test-bound}
    \int h_{d_n}(x-R_n)\widehat f_n^{\mathrm{KS}}(x)\,\diff x \leq0,\qquad \text{on}~~\{\hat c_{\mathrm{KS}}^2\geq d_n\}.
\end{align}
This event has probability tending to one by \eqref{eq:ks-only-oversmoothing-event}.

In contrast, we show that $h_{d_n}(X_i-R_n)$ has positive expectation under the true marginal density. Recall that the true marginal density is
\begin{align*}
    f_n^*(x) =(1-2\varepsilon_n)(g\star\phi_1)(x) +\varepsilon_n\phi_1(x) +\frac{\varepsilon_n}{2}\left\{\phi_1(x-R_n)+\phi_1(x+R_n)\right\}.
\end{align*}
Note that the atom at $R_n$ in $G_n^*$ produces a Gaussian component $N(R_n,1)$ with weight $\varepsilon_n/2$ in the true marginal density $f_n^*$. Taking $d=d_n$, a direct calculation gives
\begin{align*}
    \frac{\varepsilon_n}{2} \int h_{d_n}(x-R_n)\phi_1(x-R_n)\,\diff x &=\frac{\varepsilon_n}{2\sqrt2} \left(a_{d_n}-\frac12\right) =\frac{\varepsilon_n d_n}{4\sqrt2(2+d_n)},\\
    \frac{\varepsilon_n}{2} \int h_{d_n}^2(x-R_n)\phi_1(x-R_n)\,\diff x &=\frac{\varepsilon_n}{2\sqrt3} \left(a_{d_n}^2-\frac{2a_{d_n}}3+\frac13\right) \lesssim \varepsilon_n.
\end{align*}
Also, note that $|h_d(y)|\leq(1+y^2)e^{-y^2/2}$ and $(g\star\phi_1)(x)\lesssim e^{-|x|}$. For $j=1,2$, a change of variables gives
\begin{align*}
    \int |h_{d_n}(x-R_n)|^j(g\star\phi_1)(x)\,\diff x &\lesssim e^{-R_n}\int(1+y^2)^j e^{-jy^2/2+|y|}\,\diff y \lesssim e^{-R_n},\\
    \int |h_{d_n}(y)|^j\phi_1(y-\mu)\,\diff y &\lesssim (1+|\mu|^{2j})\exp\left\{-\frac{\mu^2}{4}\right\} \lesssim e^{-\mu^2/8}.
\end{align*}
Using the expression for $f_n^*$ and the preceding bounds, with $\mu=-R_n$ and $\mu=-2R_n$ in the second bound, we obtain
\begin{align*}
    \int h_{d_n}(x-R_n)f_n^*(x)\,\diff x &=\frac{\varepsilon_n d_n}{4\sqrt2(2+d_n)} +O(e^{-R_n}+\varepsilon_ne^{-R_n^2/8}),\\
    \int h_{d_n}^2(x-R_n)f_n^*(x)\,\diff x &\lesssim \{\varepsilon_n+e^{-R_n}+\varepsilon_ne^{-R_n^2/8}\}.
\end{align*}
Since $R_n=4\log(1/\varepsilon_n)\sim2\log n$ with $\varepsilon_n=\kappa\eta_n\log n$ and $d_n = 1/\log n$,
\begin{align*}
    \frac{e^{-R_n}+\varepsilon_ne^{-R_n^2/8}}{\varepsilon_n d_n} =\frac{\varepsilon_n^3}{d_n}+\frac{e^{-R_n^2/8}}{d_n} \longrightarrow0.
\end{align*}
Combining the above results, we obtain
\begin{align}\label{eq:ks-only-test-bounds}
    \int h_{d_n}(x-R_n)f_n^*(x)\,\diff x \gtrsim \varepsilon_n d_n,\qquad \int h_{d_n}^2(x-R_n)f_n^*(x)\,\diff x \lesssim \varepsilon_n.
\end{align}

Finally, on the event $\{\hat c_{\mathrm{KS}}^2\geq d_n\}$, \eqref{eq:ks-only-fitted-test-bound}, \eqref{eq:ks-only-test-bounds}, and the Cauchy--Schwarz inequality give
\begin{align*}
    \varepsilon_n d_n
    &\lesssim \int h_{d_n}(x-R_n)
        \{f_n^*(x)-\widehat f_n^{\mathrm{KS}}(x)\}\,\diff x\\
    &= 2\int h_{d_n}(x-R_n)\sqrt{f_n^*(x)}
        \left\{\sqrt{f_n^*(x)}-\sqrt{\widehat f_n^{\mathrm{KS}}(x)}\right\}
        \,\diff x\\
    &\quad-\int h_{d_n}(x-R_n)
        \left\{\sqrt{f_n^*(x)}-\sqrt{\widehat f_n^{\mathrm{KS}}(x)}\right\}^2
        \,\diff x\\
    &\leq 2\sqrt{2}\left\{\int h_{d_n}^2(x-R_n)f_n^*(x)\,\diff x\right\}^{1/2}
        \H(\widehat f_n^{\mathrm{KS}},f_n^*)
        +2\|h_{d_n}\|_\infty\H^2(\widehat f_n^{\mathrm{KS}},f_n^*)\\
    &\lesssim \sqrt{\varepsilon_n}\,\H(\widehat f_n^{\mathrm{KS}},f_n^*)
        +\H^2(\widehat f_n^{\mathrm{KS}},f_n^*).
\end{align*}
Here, we used $a^2-b^2=2a(a-b)-(a-b)^2$ with $a^2 = f_n^*$ and $b^2 = \widehat f_n^{\mathrm{KS}}$ and $\|h_{d_n}\|_\infty\leq1+2/e$. Since $d_n\to0$, this implies $\H^2(\widehat f_n^{\mathrm{KS}},f_n^*)\gtrsim \varepsilon_n d_n^2$ on this event. Consequently,
\begin{align*}
    \E_{G_n^*}[\H^2(\widehat f_n^{\mathrm{KS}},f_n^*)] &\gtrsim \varepsilon_n d_n^2 \P_{G_n^*}(\hat c_{\mathrm{KS}}^2\geq d_n)\\
    &\gtrsim \frac{\eta_n}{\log n} \asymp(n\log n)^{-1/2}.
\end{align*}
This completes the proof.
\end{proof}

\appsubsection{Proof of Proposition~\ref{prop:floor-density-denoising}}\label{app:proof-floor-density-denoising}
\begin{proof}
Under $G^*=\delta_0$, we have $\theta_i=0$ almost surely and $X_i\overset{iid}{\sim}N(0,1)$. Note that $(\hat c_0^+)^2\geq\rho_n$ and the fitted marginal density is $\widehat f_n^+=\npmle^+\star\phi_{\sqrt{1+(\hat c_0^+)^2}}$, so
\begin{align*}
    \left|\varphi_{\widehat f_n^+}(1)\right| =\left|\varphi_{\npmle^+}(1)\right|e^{-\{1+(\hat c_0^+)^2\}/2} \leq e^{-(1+\rho_n)/2}.
\end{align*}
Since $\varphi_{f^*}(1)=e^{-1/2}$, the reverse triangle inequality gives
\begin{align*}
    \left|\varphi_{f^*}(1) -\varphi_{\widehat f_n^+}(1)\right| \geq e^{-1/2}\left(1-e^{-\rho_n/2}\right).
\end{align*}
For any densities $f,q$, the triangle inequality and Cauchy--Schwarz inequality give
\begin{align*}
    |\varphi_f(t)-\varphi_q(t)|
    &\leq\int_{\Real}|f-q|\\
    &\leq \left\{\int_{\Real}(\sqrt f-\sqrt q)^2\right\}^{1/2} \left\{\int_{\Real}(\sqrt f+\sqrt q)^2\right\}^{1/2}\\
    &\leq\sqrt{8\H^2(f,q)}.
\end{align*}
Thus, we have
\begin{align*}
    \H^2(\widehat f_n^+,f^*) \geq\frac{e^{-1}}8 \left(1-e^{-\rho_n/2}\right)^2.
\end{align*}
Taking expectations and using $1-e^{-\rho_n/2}\sim\rho_n/2$ proves \eqref{eq:floor-density-lower}.

Next, for denoising, write the EB posterior mean of $\theta_i$ given $X_i=x$ under the model fitted with $\hat c_0^+$ and $\npmle^+$ as
\begin{align*}
    m_n^+(x):=\frac{(\hat c_0^+)^2}{1+(\hat c_0^+)^2}x +\frac1{1+(\hat c_0^+)^2}\mu_n(x), \qquad \mu_n(x)&:=\E_{\npmle^+}[\xi\mid X=x].
\end{align*}
Here, the fitted conditional distribution of $X$ given $\xi$ is $N(\xi,1+(\hat c_0^+)^2)$. Then we have
\begin{align*}
    \mu_n'(x) &=\frac{\E_{\npmle^+}[\xi^2\mid X=x]-\mu_n^2(x)} {1+(\hat c_0^+)^2} =\frac{\Var_{\npmle^+}(\xi\mid X=x)}{1+(\hat c_0^+)^2}.
\end{align*}
Consequently,
\begin{align*}
    (m_n^+)'(x) &=\frac{(\hat c_0^+)^2}{1+(\hat c_0^+)^2} +\frac{\Var_{\npmle^+}(\xi\mid X=x)}{\{1+(\hat c_0^+)^2\}^2} \geq\frac{\rho_n}{1+\rho_n} =:a_n.
\end{align*}
Hence, we have $|\hat\theta_i^+-\hat\theta_j^+| \geq a_n|X_i-X_j|$ where $\hat\theta_i^+=m_n^+(X_i)$ and $\hat\theta_j^+=m_n^+(X_j)$ are the corresponding EB posterior means. Let $\bar\theta_n^+:=n^{-1}\sum_i\hat\theta_i^+$ and $\bar X:=n^{-1}\sum_iX_i$. Then, since $\theta_i^* = 0$ for all $i$,
\begin{align*}
    \frac1n\sum_{i=1}^n(\hat\theta_i^+ -\theta_i^*)^2
    &=\frac1n\sum_{i=1}^n(\hat\theta_i^+)^2 \\
    &\geq\frac1n\sum_{i=1}^n (\hat\theta_i^+-\bar\theta_n^+)^2\\
    &=\frac1{2n^2}\sum_{i,j=1}^n (\hat\theta_i^+-\hat\theta_j^+)^2\\
    &\geq\frac{a_n^2}{2n^2}\sum_{i,j=1}^n(X_i-X_j)^2\\
    &=\frac{a_n^2}{n}\sum_{i=1}^n(X_i-\bar X)^2.
\end{align*}
Since $\E_{G^*}[n^{-1}\sum_i(X_i-\bar X)^2]=1-1/n$, taking expectations proves \eqref{eq:floor-denoising-lower}. This completes the proof.
\end{proof}

\appsubsection{Proof of Theorem~\ref{thm:Opt-credible}}\label{app:proof of thm:Opt-credible}
\begin{proof}
Fix $x\in\mathcal X$, and define for measurable $A\subseteq \Theta$
\[
\nu_x(A):=\int_A \pi(\theta\mid x)\,\diff\theta.
\]
Since $\pi(\cdot\mid x)$ is a density, we have
\[
\nu_x(L_k(x)) \uparrow 1 \qquad \text{as } k\downarrow 0,
\]
so in particular $k(x)>0$. Let
\[
A_{>} := \{\theta\in\Theta:\pi(\theta\mid x)>k(x)\},
\qquad
A_{=} := \{\theta\in\Theta:\pi(\theta\mid x)=k(x)\}.
\]
By definition of $k(x)$, there exists a sequence $k_m\uparrow k(x)$ such that
\[
\mathbb P_G(L_{k_m}(x)\mid x)\ge 1-\beta
\qquad \text{for all }m.
\]
Since
\[
L_{k(x)}(x)=A_{>}\cup A_{=}=\bigcap_{m=1}^{\infty} L_{k_m}(x),
\]
continuity from above yields
\[
\mathbb P_G(L_{k(x)}(x)\mid x)
=
\lim_{m\to\infty}\mathbb P_G(L_{k_m}(x)\mid x)
\ge 1-\beta.
\]
Also,
\[
A_{>}=\bigcup_{m=1}^{\infty} L_{k(x)+1/m}(x).
\]
If
\[
\int_{A_{>}} \pi(\theta\mid x)\,\diff\theta > 1-\beta,
\]
then by continuity from below there would exist some $m$ such that
\[
\mathbb P_G(L_{k(x)+1/m}(x)\mid x)>1-\beta,
\]
contradicting the definition of $k(x)$ as a supremum. Therefore
\[
\int_{A_{>}} \pi(\theta\mid x)\,\diff\theta
\le 1-\beta
\le
\mathbb P_G(L_{k(x)}(x)\mid x).
\]

On $A_{=}$, the posterior density is identically $k(x)$, so posterior mass equals $k(x)$ times Lebesgue measure. Since Lebesgue measure is atomless, there exists a measurable set
\[
B_x\subseteq A_{=}
\]
such that
\[
\int_{B_x}\pi(\theta\mid x)\,\diff\theta
=
1-\beta-\int_{A_{>}}\pi(\theta\mid x)\,\diff\theta.
\]
Hence the set $\Ic_x$ defined in \eqref{eq:Opt-Cred-Set} has posterior probability exactly $1-\beta$.

Now let $\Ic\in\mathcal C_x$. Then
\[
\int_\Ic \pi(\theta\mid x)\,\diff\theta
\ge
1-\beta
=
\int_{\Ic_x} \pi(\theta\mid x)\,\diff\theta,
\]
so
\[
\int_{\Ic\setminus \Ic_x}\pi(\theta\mid x)\,\diff\theta
\ge
\int_{\Ic_x\setminus \Ic}\pi(\theta\mid x)\,\diff\theta.
\]
On $\Ic\setminus \Ic_x$, we have $\pi(\theta\mid x)\le k(x)$, whereas on $\Ic_x\setminus \Ic$, we have $\pi(\theta\mid x)\ge k(x)$. Therefore
\[
k(x)\,|\Ic\setminus \Ic_x|
\ge
\int_{\Ic\setminus \Ic_x}\pi(\theta\mid x)\,\diff\theta
\ge
\int_{\Ic_x\setminus \Ic}\pi(\theta\mid x)\,\diff\theta
\ge
k(x)\,|\Ic_x\setminus \Ic|.
\]
Since $k(x)>0$, it follows that
\[
|\Ic\setminus \Ic_x|\ge |\Ic_x\setminus \Ic|.
\]
Consequently,
\[
|\Ic|
=
|\Ic\cap \Ic_x|+|\Ic\setminus \Ic_x|
\ge
|\Ic\cap \Ic_x|+|\Ic_x\setminus \Ic|
=
|\Ic_x|.
\]
Thus $\Ic_x$ minimizes length over $\mathcal C_x$, i.e. it solves the pointwise optimization problem $\min_{\Ic\in\mathcal C_x} |\Ic|$.
\end{proof}

\appsubsection{Proof of Theorem~\ref{thm:hetero-estimated-c-rate}}\label{app:proof-hetero-estimated-c-rate}
To prove Theorem~\ref{thm:hetero-estimated-c-rate}, we first establish a heteroscedastic analogue of the KS inverse inequality in Lemma~\ref{lem:ks-inverse}.

\begin{lemma}[Inverse inequality in KS distance for the heteroscedastic model]\label{lem:hetero-inverse-scale}
Suppose that Assumption~\ref{ass:moment-fourier} and \eqref{eq:hetero hierarchical model} hold with $\underline{k}\leq\sigma_i^2\leq\bar{k}$ for all $i$. Let $\bar F_n$ and $M_n$ be defined in \eqref{eq:hetero-pooled-distribution}, and let $F=H\star N(0,c^2)\star M_n$, where $H\in\Pc(\Real)$ and $c\geq\czero$. There exist constants $C>0$ and $\delta_0\in(0,e^{-e})$, independent of $n$ and the variance sequence, such that, for $0<\delta\leq\delta_0$,
\begin{align}\label{eq:hetero-inverse-scale}
    \KS(F,\bar F_n)\leq\delta
    \quad\Longrightarrow\quad
    0\leq c^2-\czero^2\leq C\frac{\log\log(1/\delta)}{\log(1/\delta)}.
\end{align}
\end{lemma}

\begin{proof}[Proof of Lemma~\ref{lem:hetero-inverse-scale} under Assumption~\ref{ass:moment-fourier} (i)]
Suppose that $\KS(F,\bar F_n)\leq\delta$. By \eqref{eq:simultaneous-moment-control} and $\sigma_i^2\leq\bar{k}$, we have $\sup_i(\E|X_i|^p)^{1/p}\lesssim_{\czero,\gammav,\bar{k}}p^{1/r}$ for every $p\geq1$. As in the derivation of \eqref{eq:estimated-c-characteristic-bound} in the proof of Lemma~\ref{lem:ks-inverse}, we obtain
\begin{align}\label{eq:hetero-characteristic-bound}
    |\varphi_F(t)-\varphi_{\bar F_n}(t)|
    \lesssim_{\czero,\gammav,\bar{k}}\delta\left\{1+\{\log(1/\delta)\}^{1/r}|t|\right\},\qquad t\in\Real.
\end{align}
Put $q=1/r+(\vartheta+1)/2+1$ with $r$ and $\vartheta$ in Assumption~\ref{ass:moment-fourier} and choose $t>0$ such that
\begin{align*}
    t^2=\frac{2}{\czero^2+\bar{k}}\left\{\log(1/\delta)-q\log\log(1/\delta)\right\}.
\end{align*}
For sufficiently small $\delta$, we have $t\geq t_0$, $t\asymp\{\log(1/\delta)\}^{1/2}$, and $e^{-(\czero^2+\bar{k})t^2/2}=\delta\{\log(1/\delta)\}^q$. Since $\varphi_{M_n}(t)=n^{-1}\sum_i e^{-\sigma_i^2t^2/2}\geq e^{-\bar{k}t^2/2}>0$, \eqref{eq:fourier-lower-bound} gives
\begin{align*}
    |\varphi_{\bar F_n}(t)|
    &=e^{-\czero^2t^2/2}|\varphi_{\trueprior}(t)|\varphi_{M_n}(t)\\
    &\geq a e^{-(\czero^2+\bar{k})t^2/2}(1+t)^{-\vartheta}.
\end{align*}
Hence, by \eqref{eq:hetero-characteristic-bound} and the choice of $t$, for sufficiently small $\delta$,
\begin{align*}
    |\varphi_F(t)|\geq\frac12|\varphi_{\bar F_n}(t)|
    \geq\frac{a}{2}e^{-\czero^2t^2/2}(1+t)^{-\vartheta}\varphi_{M_n}(t).
\end{align*}
Since $F=H\star N(0,c^2)\star M_n$ and $|\varphi_H(t)|\leq1$, we have
\begin{align*}
    |\varphi_F(t)|
    &=|\varphi_H(t)|e^{-c^2t^2/2}\varphi_{M_n}(t)
    \leq e^{-c^2t^2/2}\varphi_{M_n}(t).
\end{align*}
Combining the last two displays and cancelling the positive factor $\varphi_{M_n}(t)$ gives
\begin{align*}
    e^{-c^2t^2/2}\geq\frac{a}{2}e^{-\czero^2t^2/2}(1+t)^{-\vartheta}.
\end{align*}
Taking logarithms and substituting the choice of $t$, as in the proof of Lemma~\ref{lem:ks-inverse}, proves \eqref{eq:hetero-inverse-scale}.
\end{proof}

\begin{proof}[Proof of Lemma~\ref{lem:hetero-inverse-scale} under Assumption~\ref{ass:moment-fourier} (ii)]
Suppose that $\KS(F,\bar F_n)\leq\delta$, and let $p$ and $\bar p_n^*$ be their densities. Since all Gaussian component variances are at least $\underline k$, the densities vanish at infinity and their derivatives have uniformly bounded $L_1$ norms. Hence, integration by parts gives
\begin{align*}
    \|p-\bar p_n^*\|_{L_2}^2
    &=-\int (F-\bar F_n)\{p'-(\bar p_n^*)'\}\,\diff x
      \lesssim_{\underline k}\delta.
\end{align*}

Next, set $T:=H\star N(0,c^2+\bar k)$ and $T^*:=H^*\star N(0,\czero^2+\bar k)$, with densities $p_T$ and $p_{T^*}$. Since
\begin{align*}
    \varphi_T(t)-\varphi_{T^*}(t)
    &=\frac{e^{-\bar k t^2/2}}{\varphi_{M_n}(t)}
      \{\varphi_F(t)-\varphi_{\bar F_n}(t)\},
\end{align*}
and $\varphi_{M_n}(t)\geq e^{-\bar k t^2/2}$, Plancherel's theorem gives $\|p_T-p_{T^*}\|_{L_2}\leq\|p-\bar p_n^*\|_{L_2}$. Now, take $I=[-\Lambda,\Lambda]$ with $\Lambda=R+\sqrt{2(\czero^2+\bar k)\log(1/\delta)}$. Then $T^*(I^c)\leq2\delta$. Since both densities integrate to one, $T(I^c)\leq T^*(I^c)+\int_I|p_T-p_{T^*}|\,\diff x$. Hence, it holds that
\begin{align*}
    \KS(T,T^*)
    \leq \mathrm{TV}(T,T^*) \leq \sqrt{2\Lambda}\,\|p_T-p_{T^*}\|_{L_2}+T^*(I^c)
      \leq \delta^{1/3}
\end{align*}
for all sufficiently small $\delta$. Applying Lemma~\ref{lem:ks-inverse} with radius $\delta^{1/3}$ and Gaussian scales $\sqrt{c^2+\bar k}$ and $\sqrt{\czero^2+\bar k}$ proves \eqref{eq:hetero-inverse-scale}.
\end{proof}

\begin{lemma}[Log-likelihood ratio bound for the classical heteroscedastic NPMLE]\label{lem:hetero-likelihood-gain}
Under \eqref{eq:simultaneous-moment-control} and \eqref{eq:hetero hierarchical model} with $\sigma_i^2\in[\underline{k},\bar{k}]\subset(0,\infty)$, for any sufficiently large fixed constant $B>0$ and all sufficiently large $n$,
\begin{align}\label{eq:hetero-likelihood-gain}
    \P\left\{L_n(0)-\sum_{i=1}^n\log f_i^*(X_i)>Bb_n\right\}\leq2n^{-3}.
\end{align}
\end{lemma}
\begin{proof}
We follow the covering argument in the proof of Lemma~\ref{lem:estimated-c-likelihood-gain}. For $\eta=n^{-3}$ and $R\geq\sqrt{\log n}$, Supplementary Lemma~5 of \citet{Soloff2025} gives an $\eta$-net of proper density vectors $(p_{H,\sigma_i})_{i=1}^n$ in the metric $\max_i\sup_{|x|\leq R}|p_{H,\sigma_i}(x)-p_{H',\sigma_i}(x)|$, with cardinality $N_R$ satisfying
\begin{align*}
    \log N_R\lesssim_{\underline{k},\bar{k}}\left(1+\frac{R}{\sqrt{\log n}}\right)(\log n)^2.
\end{align*}
For an ordinary heteroscedastic NPMLE $\widehat H_n^{\,0}$, the general maximum likelihood inequality gives
\begin{align*}
    p_{\widehat H_n^{\,0},\sigma_i}(X_i)\geq\frac1{n\sqrt{2\pi\sigma_i^2}}
    \geq\frac1{n\sqrt{2\pi\bar{k}}}.
\end{align*}
For each fixed vector $g_i$ in the net, independence gives $\E[\prod_i g_i(X_i)/f_i^*(X_i)]=1$, so \eqref{eq:estimated-c-likelihood-fixed-radius} holds with $f_i^*$ in place of $f^*$. The uniform moment bound in the preceding proof also gives $\P\{\max_i|X_i|>K_R(\log n)^{1/r}\}\leq n^{-3}$ for sufficiently large fixed $K_R$, as in \eqref{eq:estimated-c-observation-tail}. Then the arguments in the proof of Lemma~\ref{lem:estimated-c-likelihood-gain} proves \eqref{eq:hetero-likelihood-gain}.
\end{proof}

We now prove Theorem~\ref{thm:hetero-estimated-c-rate}.
\begin{proof}[Proof of Theorem~\ref{thm:hetero-estimated-c-rate}]
Let $\mathcal A_n=\{\KS(\F_n,\bar F_n)\leq\eta_n\}$. On $\mathcal A_n$, $\czero$ is admissible in \eqref{eq:hetero-neighborhood-estimator}, and any admissible distribution is within KS distance $2\eta_n$ of $\bar F_n$. Therefore, Lemma~\ref{lem:hetero-inverse-scale} gives, for a fixed $K>0$,
\begin{align}\label{eq:hetero-raw-scale-rate}
    \czero^2\leq\hat c_{\mathrm{KS}}^2\leq\czero^2+K\rho_n
    \qquad\text{on }\mathcal A_n,
\end{align}
while \eqref{eq:hetero-KS-concentration} gives $\P(\mathcal A_n^c)\leq2e n^{-2}$. 

We now analyze the modified neighborhood estimator $\hat c_0$ defined in~\eqref{eq:hetero-estimated-c}. For the cap $\tilde C_n$, the argument in the proof of Lemma~\ref{lem:estimated-c-cap} gives
\begin{align}\label{eq:hetero-cap-probability}
    \P(\tilde C_n<\czero)
    \leq 2\exp\left\{-k\frac{n}{(\log n)^2}\right\},
\end{align}
where $k>0$ is a constant.
On $\mathcal A_n\cap\{\tilde C_n\geq\czero\}$, set $\tilde C:=\{(\hat c_{\mathrm{KS}}^2\wedge\tilde C_n^2)-\czero^2\}/\rho_n\in[0,K]$. This choice is admissible in \eqref{eq:hetero-estimated-c-buffer} and yields the smoothing parameter $\czero$. Since $L_n(\czero)\geq\sum_i\log f_i^*(X_i)$, the same argument as in the proof of Theorem~\ref{thm:modified-c-rate} gives, on $\{L_n(0)-\sum_i\log f_i^*(X_i)\leq Bb_n\}$,
$\widehat C\leq B+K$, $0\leq L_n(0)-L_n(\hat c_0)\leq(B+K)b_n$, and $(\czero^2-(B+K)\rho_n)_+\leq\hat c_0^2\leq\czero^2+K\rho_n$.
Taking $B+K$ as the constant defining $\Sc_n$ and combining the three probability bounds proves \eqref{eq:hetero-one-sided-control}, since
\begin{align*}
    2e\,n^{-2}+2\exp\left\{-k\frac{n}{(\log n)^2}\right\}+2n^{-3}
    \leq 3e\,n^{-2}
\end{align*}
for all sufficiently large $n$. The remaining conclusions follow as in the homoscedastic proof.
\end{proof}

\appsubsection{Proof of Theorem~\ref{thm:het-density-denoising}}\label{app:proof-hetero-density-denoising}
\begin{proof}[Proof of Theorem~\ref{thm:het-density-denoising}]
We follow the proof of Theorem~\ref{thm:estimated-c-density-denoising}, using averaged squared Hellinger distance. Recall $r_n = (\log n)^{3/2+1/r}/n$ from \eqref{eq:estimated-c-rn}, and let
\begin{equation}\label{eq:hetero-density-event}
    \mathcal T_n:=\Sc_n\cap\{R_n\leq K_R(\log n)^{1/r}\},\qquad R_n:=\max_i|X_i|,
\end{equation}
where $\Sc_n$ is the event in Theorem~\ref{thm:hetero-estimated-c-rate} and $K_R$ is as in the proof of Lemma~\ref{lem:hetero-likelihood-gain}. Their probability bounds give
\begin{align}\label{eq:hetero-density-good-event}
    \P(\mathcal T_n^c)\leq3e n^{-2}+n^{-3}.
\end{align}
Let $\widehat G_n^{\mathrm{KW}}$ be an ordinary exact heteroscedastic NPMLE, and define $\overline G_n:=\widehat G_n\1v_{\mathcal T_n}+\widehat G_n^{\mathrm{KW}}\1v_{\mathcal T_n^c}$. Since $b_n\lesssim_{\czero,\gammav,\bar{k}}nr_n$ on $\mathcal T_n$, the likelihood bound in $\Sc_n$ and the optimality of the ordinary NPMLE give
\begin{align}\label{eq:hetero-density-likelihood-comparison}
    \sum_i\log p_{\overline G_n,\sigma_i}(X_i)
    \geq\sum_i\log f_i^*(X_i)-K_1nr_n
\end{align}
for a fixed $K_1>0$. Using \eqref{eq:estimated-c-true-prior-moments}, the same choices $S=\{0\}$ and $M=M_0(\log n)^{1/r}$ as in the homoscedastic proof give a rate function of order $r_n$ in Theorem 7 of \citet{Soloff2025}. Applying that theorem gives an expected averaged squared Hellinger bound of order $r_n$ and, for a fixed $K_2>0$,
\begin{align}\label{eq:hetero-density-hellinger-tail}
    \P\left\{\frac1n\sum_i\H^2(p_{\overline G_n,\sigma_i},f_i^*)>K_2t^2r_n\right\}
    \leq2n^{-t^2},\qquad t\geq1.
\end{align}
Since $p_{\overline G_n,\sigma_i}=\widehat f_{n,i}$ on $\mathcal T_n$ and $\H^2\leq1$, \eqref{eq:hetero-density-good-event} proves \eqref{eq:hetero density estimation}.

For denoising, the general maximum likelihood inequality used in Lemma~\ref{lem:hetero-likelihood-gain} gives $\widehat f_{n,i}(X_i)\geq\{n\sqrt{2\pi(\sigma_i^2+\hat c_0^2)}\}^{-1}$. Since $\hat c_0^2\leq\czero^2+1$ on $\mathcal T_n$ for large $n$ and $\overline G_n$ is the ordinary NPMLE on $\mathcal T_n^c$, we have
\begin{align}\label{eq:hetero-density-likelihood-floor}
    \min_i p_{\overline G_n,\sigma_i}(X_i)
    \geq\frac1{n\sqrt{2\pi(\czero^2+\bar{k}+1)}}.
\end{align}
This lower bound and the averaged Hellinger bound allow the same adaptation of the proof of Theorem~9 of \citet{Soloff2025} as in the homoscedastic case. Note that $\sigma_i^2$ are uniformly bounded and the EB posterior mean under $\overline G_n$ equals $\hat\theta_i$ on $\mathcal T_n$. Hence
\begin{align}\label{eq:hetero-density-good-event-regret}
    \E\left[\frac1n\sum_i(\hat\theta_i-\theta_i^*)^2\1v_{\mathcal T_n}\right]
    \lesssim_{\czero,\gammav,\underline{k},\bar{k}}r_n(\log n)^3.
\end{align}
Using $|\hat\theta_i|\leq R_n$ and the uniform moment bound $(\E|X_i|^p)^{1/p}\lesssim p^{1/r}$, the same H\"older argument as in~\eqref{eq:estimated-c-bad-event-regret} gives
\begin{align*}
    \E\left[\frac1n\sum_i(\hat\theta_i-\theta_i^*)^2\1v_{\mathcal T_n^c}\right]
    \lesssim_{\czero,\gammav,\bar{k}}(\log n)^{2/r}\{\P(\mathcal T_n^c)\}^{2/3}
    \lesssim_{\czero,\gammav,\bar{k}}n^{-1}.
\end{align*}
Combining these bounds proves \eqref{eq:hetero regret bound}.
\end{proof}

\appsubsection{Proof of Theorem~\ref{thm:het-prior}}\label{app:proof-het-prior}\label{app:proof of thm:hetero deconvolution upper bound}
\begin{proof}[Proof of Theorem~\ref{thm:het-prior}]
Define
\begin{align*}
    \widetilde f_n:=\widehat g_n^+\star\phi_{\sqrt{\bar{k}}},\qquad
    \widetilde f^*:=g^*\star\phi_{\sqrt{\bar{k}}}.
\end{align*}
Convolving each pair $(\widehat f_{n,i}^+,f_i^*)$ with $N(0,\bar{k}-\sigma_i^2)$ and using Hellinger contraction gives
\begin{align}\label{eq:hetero-common-noise-hellinger}
    \H^2(\widetilde f_n,\widetilde f^*)
    \leq\frac1n\sum_i\H^2(\widehat f_{n,i}^+,f_i^*).
\end{align}
Recall $\mathcal T_n$ from \eqref{eq:hetero-density-event}. On $\mathcal T_n$, Theorem~\ref{thm:hetero-estimated-c-rate} gives $\hat c_0^2\geq\czero^2-K\rho_n>\rho_n$ for all sufficiently large $n$, so $\hat c_0^+=\hat c_0$ and $\npmle^+=\npmle$. Hence, on $\mathcal T_n$, we have $\widehat f_{n,i}^+=\widehat f_{n,i}$ for every $i$. Thus, \eqref{eq:hetero-density-good-event}, \eqref{eq:hetero-density-hellinger-tail} with $t=\sqrt3$, and \eqref{eq:hetero-common-noise-hellinger} imply that
\begin{equation}\label{eq:hetero-prior-event}
    \mathcal E_n:=\mathcal T_n\cap\{\H^2(\widetilde f_n,\widetilde f^*)\leq3K_2r_n\}
\end{equation}
satisfies $\P(\mathcal E_n^c)\lesssim n^{-2}$, where $K_2$ is the constant in \eqref{eq:hetero-density-hellinger-tail}.

We first bound $\|\widehat g_n^+-g^*\|_{L_2}^2$ on $\mathcal E_n$. Applying $x\mapsto x/\sqrt{\bar{k}}$ to both $\widetilde f_n$ and $\widetilde f^*$ gives unit observation noise variance and leaves $\H(\widetilde f_n,\widetilde f^*)$ unchanged. Meanwhile, the true and estimated squared smoothing parameters become $\czero^2/\bar{k}$ and $\hat c_0^2/\bar{k}$, respectively. By Theorem~\ref{thm:hetero-estimated-c-rate}, their difference satisfies $|\hat c_0^2-\czero^2|/\bar{k}\lesssim\rho_n$ on $\mathcal E_n$. Since $\operatorname{supp}(\npmle)\subseteq[X_{(1)},X_{(n)}]$ and $R_n=\max_i|X_i|\lesssim(\log n)^{1/r}$ on $\mathcal E_n$, the rescaled fitted mixing distribution is supported on $[-R_n/\sqrt{\bar{k}},R_n/\sqrt{\bar{k}}]$.  Then following the argument starting at \eqref{eq:moment-deconvolution-plancherel}, with $m$ as in \eqref{eq:moment-deconvolution-low-cutoff}, gives
\begin{align*}
    \|\widehat g_n^+-g^*\|_{L_2}^2\1v_{\mathcal E_n}
    &\lesssim_{\czero,\gammav,\underline{k},\bar{k}}
    r_n^{\bar\alpha_0}(\log n)^K
    \exp\left\{-\frac{\bar{k}m}{\czero^2+\bar{k}}\right\},
\end{align*}
where $r_n = (\log n)^{3/2+1/r}/n$ is defined in~\eqref{eq:estimated-c-rn}, 
\begin{align*}
    \bar\alpha_0
    :=\frac{\czero^2/\bar{k}}{1+\czero^2/\bar{k}}
    =\frac{\czero^2}{\czero^2+\bar{k}},
\end{align*}
and $K>0$ depends only on $\czero$, $\gammav$, $\underline{k}$, and
$\bar{k}$. On $\mathcal E_n^c$, the floor $\hat c_0^+\geq\sqrt{\rho_n}$ gives, as in the homoscedastic proof,
\begin{align*}
    \E[\|\widehat g_n^+-g^*\|_{L_2}^2\1v_{\mathcal E_n^c}]
    &\lesssim \rho_n^{-1/2}\P(\mathcal E_n^c)\lesssim n^{-1}.
\end{align*}
Combining the two bounds gives the analogue of \eqref{eq:moment-deconvolution-intermediate-rate} with $\bar\alpha_0$ in place of $\alpha_0$ and $(\czero^2+\bar{k})/\bar{k}$ in place of $\sigma_0^2$. The rate calculation following \eqref{eq:moment-deconvolution-intermediate-rate} yields \eqref{eq:hetero MISE convergence rate}.
\end{proof}

\appsubsection{Proof of Theorem~\ref{thm:het-posterior}}\label{app:proof-het-posterior}
\begin{proof}[Proof of Theorem~\ref{thm:het-posterior}]
Apply \eqref{eq:wTV and TV} with $\phi_{\sigma_i}$ in place of $\phi$ and average over $i$. Total variation contraction under convolution gives
\begin{align}\label{eq:hetero-posterior-tv-bound}
    \overline{\mathrm{wTV}}(\widehat\pi_n^+,\pi^*)
    \leq\frac1n\sum_i\mathrm{TV}(\widehat f_{n,i}^+,f_i^*)+\mathrm{TV}(\widehat g_n^+,g^*)
    \leq2\,\mathrm{TV}(\widehat g_n^+,g^*).
\end{align}
Applying the argument of Lemma~\ref{lem:l2-to-tv} on $\mathcal E_n$ in \eqref{eq:hetero-prior-event}, with $\P(\mathcal E_n^c)\lesssim n^{-2}$, gives
\begin{align*}
    \E[\mathrm{TV}(\widehat g_n^+,g^*)]
    \lesssim_{\czero,\gammav,\underline{k},\bar{k}}
    (\log n)^{1/(2r)}\left\{\E\|\widehat g_n^+-g^*\|_{L_2}^2\right\}^{1/2}+n^{-2}.
\end{align*}
Theorem~\ref{thm:het-prior} now proves \eqref{eq:hetero-posterior-rate}, after absorbing the logarithmic factor into the exponential as in the homoscedastic proof.
\end{proof}
\end{document}